\documentclass{amsart}

\usepackage{amssymb, amsmath,mathtools,mathabx}
\usepackage{mathrsfs}
\usepackage{amscd}
\usepackage{verbatim}
\usepackage{stmaryrd}
\usepackage[dvipsnames]{xcolor}
\usepackage{a4wide}
\usepackage{tikz}
\usetikzlibrary{shapes.geometric, arrows}

\usepackage{enumerate}
\newcounter{nameOfYourChoice}
\mathtoolsset{showonlyrefs}

\usepackage[colorlinks,linkcolor={blue},citecolor={blue},urlcolor={purple},]{hyperref}

\usepackage[colorinlistoftodos,prependcaption,textsize=tiny]{todonotes}

\usepackage{forest}
\usepackage{tikz}
\tikzset{>=latex}

\usepackage{tabularx}
\newcolumntype{L}{>{\arraybackslash}X}
\usepackage{multirow}

\theoremstyle{plain}
\newtheorem{theorem}{Theorem}[section]
\theoremstyle{remark}
\newtheorem{remark}[theorem]{Remark}
\newtheorem{example}[theorem]{Example}
\theoremstyle{plain}
\newtheorem{corollary}[theorem]{Corollary}
\newtheorem{lemma}[theorem]{Lemma}
\newtheorem{proposition}[theorem]{Proposition}
\newtheorem{definition}[theorem]{Definition}

\newtheorem{assumption}[theorem]{Assumption}
\newtheorem{problem}{Problem}
\numberwithin{equation}{section}

\def\N{{\mathbb N}}
\def\Z{{\mathbb Z}}

\def\R{{\mathbb R}}
\def\C{{\mathbb C}}

\newcommand{\Di}{\mathcal{D}}
\newcommand{\E}{{\mathbf E}}
\renewcommand{\P}{{\mathbf P}}
\newcommand{\pr}{\mathbb{P}}
\newcommand{\F}{{\mathscr F}}
\newcommand{\Filtr}{\mathbb{F}}

\newcommand{\g}{\gamma}

\newcommand{\om}{\omega}
\renewcommand{\O}{\Omega}

\renewcommand{\a}{\kappa}

\newcommand{\Aop}{A}

\newcommand{\angH}{\omega_{H^{\infty}}}

\newcommand{\Dom}{\mathcal{O}}

\newcommand{\V}{\mathcal{V}}
\newcommand{\I}{I}

\newcommand{\Tor}{\mathbb{T}}
\newcommand{\T}{\mathbb{T}}

\newcommand{\loc}{{\rm loc}}
\newcommand{\tr}{\mathrm{tr}}

\newcommand{\calL}{{\mathscr L}}

\newcommand{\BIP}{\mathrm{BIP}}

\newcommand{\hz}{\prescript{}{0}{H}}
\newcommand{\Sz}{\prescript{}{0}{\mathrm{MR}}_X}

\newcommand{\Wz}{\prescript{}{0}{W}}

\newcommand{\Ls}{\mathbb{L}}
\newcommand{\Gs}{\mathbb{G}}

\newcommand{\Hs}{\mathbb{H}}

\newcommand{\HD}{\prescript{}{D}{H}}
\newcommand{\Dd}{\prescript{}{D}{\Delta}}
\newcommand{\WD}{\prescript{}{D}{W}}
\newcommand{\BD}{\prescript{}{D}{B}}

\newcommand{\Bs}{\mathbb{B}}

\newcommand{\Do}{\mathsf{D}}

\newcommand{\wt}{\widetilde}

\usepackage{stmaryrd}

\newcommand{\Tr}{\mathrm{Tr}}

\newcommand{\one}{{{\bf 1}}}
\newcommand{\embed}{\hookrightarrow}
\newcommand{\s}{\delta}
\newcommand{\lb}{\langle}
\newcommand{\rb}{\rangle}
\newcommand{\dps}{\displaystyle}

\renewcommand{\div}{\normalfont{\text{div}}}

\renewcommand{\l}{\langle}
\renewcommand{\r}{\rangle}

\newcommand{\reg}{\delta}

\newcommand{\norm}[1]{{\left\vert\kern-0.25ex\left\vert\kern-0.25ex\left\vert #1
    \right\vert\kern-0.25ex\right\vert\kern-0.25ex\right\vert}}

\renewcommand{\emptyset}{\varnothing}

\newcommand{\Progress}{\mathscr{P}}

\newcommand{\wh}{\widehat}

\newcommand{\rnoise}{g}

\newcommand{\btwod}{b}

\newcommand{\bm}{b}

\newcommand{\Borel}{\mathscr{B}}

\renewcommand{\S}{\mathcal{S}}

\def\XXint#1#2#3{{\setbox0=\hbox{$#1{#2#3}{\int}$ }
\vcenter{\hbox{$#2#3$ }}\kern-.6\wd0}}

\newcommand{\dd}{\mathrm{d}}
\newcommand{\Ext}{\mathrm{Ext}}

\allowdisplaybreaks

\begin{document}

\author{Antonio Agresti}
\address{Delft Institute of Applied Mathematics\\
Delft University of Technology \\ P.O. Box 5031\\ 2600 GA Delft\\The
Netherlands}
\curraddr{Department of Mathematics Guido Castelnuovo, Sapienza University of Rome,
P.le Aldo Moro 5, 00185 Rome, Italy}
\email{agresti.agresti92@gmail.com}

\author{Mark Veraar}
\address{Delft Institute of Applied Mathematics\\
Delft University of Technology \\ P.O. Box 5031\\ 2600 GA Delft\\The
Netherlands} \email{M.C.Veraar@tudelft.nl}

\dedicatory{Dedicated to Professor Da Prato, in admiration of his outstanding mathematical work.}

\thanks{The authors have received funding from the VICI subsidy VI.C.212.027 of the Netherlands Organisation for Scientific Research (NWO), and the first author  is a member of GNAMPA (IN$\delta$AM)}

\date\today


\title[Nonlinear SPDE\lowercase{s} and Maximal Regularity]{Nonlinear SPDE\lowercase{s} and Maximal Regularity:\\ An Extended Survey}



\begin{abstract}
In this survey, we provide an in-depth exposition of our recent results on the well-posedness theory for stochastic evolution equations, employing maximal regularity techniques. The core of our approach is an abstract notion of critical spaces, which, when applied to nonlinear SPDEs, coincides with the concept of scaling-invariant spaces. This framework leads to several sharp blow-up criteria and enables one to obtain instantaneous regularization results.
Additionally, we refine and unify our previous results, while also presenting several new contributions.

In the second part of the survey, we apply the abstract results to several concrete SPDEs. In particular, we give applications to stochastic perturbations of quasi-geostrophic equations, Navier-Stokes equations, and reaction-diffusion systems (including Allen--Cahn, Cahn--Hilliard and Lotka--Volterra models). Moreover, for the Navier--Stokes equations, we establish new Serrin-type blow-up criteria. While some applications are addressed using $L^2$-theory, many require a more general $L^p(L^q)$-framework. In the final section, we outline several open problems, covering both abstract aspects of stochastic evolution equations, and concrete questions in the study of linear and nonlinear SPDEs.
\end{abstract}

\maketitle
\tableofcontents

\section{Introduction}

In this survey, we give an exposition of the recent developments in \cite{AV19_QSEE_1, AV19_QSEE_2}, where we build a comprehensive framework for establishing local and global well-posedness for a class of It\^o stochastic parabolic evolution equations of the form (the reader is referred to Section \ref{ss:results_intro} below for the unexplained notation):
\begin{equation}
\label{eq:SEEintro}
\left\{
\begin{aligned}
&\dd u + A u \,\dd t = F(u)\, \dd t + (B u +G(u))\, \dd W,\\
&u(0)=u_{0}.
\end{aligned}
\right.
\end{equation}
Here, the term \emph{parabolic} refers to the property that the leading operators have suitable smoothing properties, which will be encoded in a \emph{maximal regularity} assumption. Notably, maximal regularity estimates are available for a large class of operators, extending beyond standard second-order heat-type operators to include many others, such as those arising in fluid dynamics.

To maintain clarity and simplicity, this survey focuses exclusively on the {\em semilinear} case, specifically equations of the form \eqref{eq:SEEintro}. This already includes a wide range of highly nontrivial models (e.g.\ Navier--Stokes and reaction-diffusion equations). However, it is worth noting that the references cited above also address the more general quasilinear setting, and the case where the coefficients depend on $(t,\omega)$ as well.

Additionally, this survey introduces several new abstract results and demonstrates the application of our framework to specific SPDEs. Finally, we provide a list of open problems to guide future research efforts.

\smallskip

The study of parabolic stochastic evolution equations has a rich history. While it is not feasible to provide a comprehensive overview of the literature, we will highlight several influential approaches that have significantly shaped our framework. References to additional approaches can be found in Subsection \ref{ss:literature_discussion}.

The semigroup approach is thoroughly explored in the monograph by Da Prato and Zabczyk \cite{DPZ} and the references therein. For the variational setting, key works include those of Pardoux \cite{Par75} and Krylov-Rozovskii \cite{KR79}, as well as the monograph by Liu and Röckner \cite{LR15}. Finally, the $L^p$-theory of Krylov \cite{Kry} is distributed across several papers, many of which are discussed in Subsection \ref{ss:furtherSMR}.
Our framework provides a bridge between, and in several cases, an extension of these foundational approaches and viewpoints to stochastic PDEs when restricted to the semilinear equation \eqref{eq:SEEintro}. This framework has led to significant new results for a variety of SPDEs, including:
\begin{itemize}
\item Reaction-diffusion such as Allen--Cahn, Cahn--Hilliard, Lotka--Volterra, and Gray-Scott.
\item Stochastic fluid dynamics, including the Navier--Stokes and primitive equations.
\end{itemize}
The main novelty and the key advantage of our framework lies in its ability to allow and capture \emph{critical} nonlinearities and data. Criticality is a well-established concept in PDE theory and mathematical physics, though its precise definition often depends on additional context. From a PDE point of view, a critical space and/or setting can be identified whenever the PDE under consideration admits scaling invariance. In this context, a space and/or setting is called critical if it respects the scaling invariance of the underlying equation. As discussed in Subsection \ref{ss:scaling_intro}, scaling invariance determines a family of invariant spaces rather than a single setting. For example, in the case of the 3D Navier--Stokes equations, both the Lebesgue space $L^3$ and Besov spaces $\dot{B}^{3/q-1}_{q,p}$ are scaling-invariant.

In addition to its natural connection with a given PDE, working with critical spaces offers several advantages. As a general principle, the critical setting provides the optimal framework for studying parabolic problems. More importantly, it ensures well-posedness for a wide range of data and offers sharp criteria for the explosion or blow-up of solutions to the corresponding PDE in finite time.
Blow-up criteria are particularly crucial in studying the global well-posedness of PDEs, especially when limited information is available about the solution's behaviour (such as energy estimates).

In our approach to stochastic evolution equations \eqref{eq:SEEintro}, we use a (relatively) abstract setting.
This has the benefit of clearly identifying the requirements for proving well-posedness results for parabolic SPDEs.
Additionally, within the context of critical spaces, such an abstract framework is valuable for addressing problems where global scaling invariance is not present
(e.g.\ SPDEs on domains and/or $x$-dependent coefficients). Further details will be provided in Subsection \ref{ss:scaling_intro}.

\smallskip

The rest of this section is structured as follows. First, we provide a simplified overview of the results from \cite{AV19_QSEE_1, AV19_QSEE_2} in the case of semilinear stochastic evolution equations \eqref{eq:SEEintro}. In Subsection \ref{ss:scaling_intro}, we demonstrate how the abstract results relate to scaling invariance or critical spaces for 3D Navier--Stokes equations with transport noise. Subsection \ref{ss:literature_discussion} offers a discussion of relevant literature for deterministic and stochastic evolution equations. Finally, in Subsection \ref{ss:historic_part_intro}, we discuss some historical context and alternative approaches to SPDEs. A complete overview of the current manuscript is presented in Subsection \ref{ss:overview}. Applications to specific SPDEs are discussed in Sections \ref{sec:appl} and \ref{sec:LpLq}.

\subsection{A glimpse into our framework}
\label{ss:results_intro}

In the study of parabolic equations, it is well-established that abstract methods can be highly effective.
When seeking local well-posedness and regularity results for both existing and new classes of SPDEs, verifying the assumptions of our framework often proves more fruitful than attempting to develop ad hoc methods for analyzing each SPDE individually. Moreover, the abstract conditions for global well-posedness derived through our approach frequently lead to significantly stronger results than those obtained through ad hoc techniques.

Before presenting the main result, we provide a more detailed description of the terms in \eqref{eq:SEEintro}. Throughout this discussion, we will use several concepts from operator theory and maximal regularity. For further details, the reader is referred to Sections \ref{sec:preliminaries} and \ref{sec:SMR}, respectively.

\subsubsection*{The setting}
The operator $A$ is assumed to be sectorial on a UMD Banach space $X_0$ of type 2 (e.g.\ $X_0 = L^q$ with $q\in [2, \infty)$) with domain $X_1=\Do(A)$ (with the graph norm). The linear part $(A,B)$  of \eqref{eq:SEEintro} is assumed to have so-called {\em stochastic maximal regularity}.
In most of the examples we have in mind, $A$ is a linear $2m$-th order differential operator (possibly an operator matrix), $F$ depends on the derivatives of $u$ up to the $(2m-1)$-th order, the operator $B$ is a linear $m$-th order differential operator, and $G$ depends on the derivatives of $u$ up to the $(m-1)$-th order.
The noise $W$ is modelled as a cylindrical Brownian motion on a Hilbert space $\mathcal{U}$, and it is coloured through the processes $B u$ and $G(u)$.

The nonlinearities $F$ and $G$ are assumed to be defined on an interpolation space $X_{\beta}$, which lies between $X_0$ and $X_1$. These nonlinearities are typically assumed to grow polynomially at a rate $\rho+1$, with a restriction on the pair $(\beta,\rho)$, and where $\beta\in (1/2,1)$. Typically, we assume the existence of a constant $C$ such that for all $u,v\in X_{\beta}$
\begin{align}\label{eq:Fsubcriintro}
\|F(u) - F(v)\|_{X_0}+
\|G(u) - G(v)\|_{\g(\mathcal{U},X_{1/2})}
\leq C(1+\|u\|_{X_{\beta}}^{\rho} + \|v\|_{X_{\beta}}^{\rho}) \|u-v\|_{X_{\beta}},
\end{align}
where $\rho\geq 0$. Here the space $\g(\mathcal{U},X_{1/2})$ is the set of $\g$-radonifying operators from $\mathcal{U}$ to $X_{1/2}$, and it coincides with the Hilbert-Schmidt operators when $X_{1/2}$ is a Hilbert space. The natural occurrence of $\g$-spaces in the context of stochastic analysis in a non-Hilbert space setting is discussed in Subsections \ref{subsec:gamma} and \ref{subsec:stochint}.

We are interested in solutions to \eqref{eq:SEEintro} with paths belonging to the weighted space
\begin{equation}
\label{eq:path_regularity_intro}
L^p_{\loc}([0,\sigma),t^\a\,\dd t ;X_1)\cap C([0,\sigma);X_{1-\frac{1+\a}{p},p}),
\end{equation}
where $p\geq 2$, $\a\in [0,\frac{p}{2}-1)\cup\{0\}$, $\sigma$ is a stopping time, and $X_{1-\frac{1+\a}{p},p}:=(X_0,X_1)_{1-\frac{1+\a}{p},p}$ is the real interpolation space.
The solution space \eqref{eq:path_regularity_intro} is natural when considering stochastic maximal $L^p$-regularity (see Subsection \ref{ss:stoch_max_reg}). Moreover, the space $X_{1-\frac{1+\a}{p},p}$ is optimal for handling pointwise evaluations (or traces) of solutions (see Subsection \ref{subsec:interp}). The role of the time weight $\a \geq 0 $ will become clear when addressing explosion criteria and the regularization properties of solutions to \eqref{eq:Fsubcriintro} (see Theorem \ref{thm:mainintro}).
We impose $\beta\in (1-\frac{1+\a}{p},1)$ to ensure that the nonlinearities $F$ and $G$ are \emph{rougher} than the trace space $X_{1-\frac{1+\a}{p},p}$. Note that this restriction is only technical: if \eqref{eq:Fsubcriintro} holds for some $\beta$, then it also holds for any $\beta'>\beta$.

On the parameters $p,\a,\beta$ and $\rho$, we impose the following condition:
\begin{align}\label{eq:subcriticalintro}
\frac{1+\kappa}{p}\leq \frac{(1+\rho)(1-\beta)}{\rho}.
\end{align}
The condition \eqref{eq:subcriticalintro} is a central assumption in our framework. If equality holds in \eqref{eq:subcriticalintro}, we will refer to the corresponding setting or space for the initial data as \emph{critical}.

As previously mentioned, the parameter $\rho$ is determined by the growth of the nonlinearity, while $\beta$ is typically derived from both the growth of the nonlinearity and the choice of underlying function spaces, in conjunction with Sobolev embedding. It is important to note that the inequality $1/p\leq \frac{1+\kappa}{p}\leq 1/2$ always holds. By choosing $p$ sufficiently large, the left-hand side of \eqref{eq:subcriticalintro} can be made arbitrarily small. The pair $(p,\kappa)$ is, in principle, flexible in applications; however, in Theorem \ref{thm:mainintro} below we require that $(A,B)$ possesses stochastic maximal $L^p$-regularity with weight $t^{\kappa}\,\dd t$.

\subsubsection*{Bird's-eye view of the framework}
Below we provide an overview of our framework. A conceptual map is presented in Figure \ref{fig:diagram_abstract}. We limit ourselves to providing a loose form of the local well-posedness, blow-up criteria, and instantaneous regularization for \eqref{eq:SEEintro}, which will be detailed in Sections \ref{sec:loc-well-posed} and \ref{sec:blowup}, see Theorems \ref{thm:localwellposed}, \ref{thm:criticalblowup} and \ref{thm:subcriticalblowup} there.

\begin{theorem}\label{thm:mainintro}
Let $X_1\hookrightarrow X_0$ be as above. Suppose that the linear part $(A,B)$ has stochastic maximal $L^p$-regularity with weight $t^{\kappa}\,\dd t$.
Suppose that the nonlinear part $(F,G)$ is as in \eqref{eq:Fsubcriintro} and that \eqref{eq:subcriticalintro} holds. Assume that $u_0\in L^0_{\F_0}(\Omega;X_{1-\frac{1+\kappa}{p},p})$. Then the following hold:
\begin{enumerate}[{\rm(1)}]
\item\label{it1:mainintro} {\rm (Local well-posedness)}
 \eqref{eq:SEEintro} has a unique maximal solution $u$ with lifetime $\sigma>0$ a.s.\ and
\[u\in L^p_{\loc}([0,\sigma),t^\a \,\dd t;X_1)\cap C([0,\sigma);X_{1-\frac{1+\kappa}{p},p})\cap  C((0,\sigma);X_{1-\frac{1}{p},p})\text{ a.s.}\]
\item\label{it2:mainintro} {\rm (Blow-up criteria)} The following hold a.s.\ on $\{\sigma<\infty\}$:
\vspace{0.1cm}
\begin{itemize}
\item $\displaystyle{\lim_{t\uparrow \sigma} u(t) \text{ does not exists in }X_{1-\frac{1+\a}{p},p}\,;}$
\vspace{0.1cm}
\item $\displaystyle{\sup_{t\in [0,\sigma)}\|u(t)\|_{X_{1-\frac{1+\kappa}{p},p}}+ \|u\|_{L^p(0,\sigma;X_{1-\frac{\kappa}{p}})}=+\infty\,;}$
\vspace{0.1cm}
\item $\displaystyle{\sup_{t\in [0,\sigma)}\|u(t)\|_{X_{1-\frac{1+\kappa}{p},p}}=+\infty}$ in the subcritical case.
\end{itemize}
\item\label{it3:mainintro} {\rm (Regularization)} Under relatively weak assumptions, but still keeping $u_0\in X_{1-\frac{1+\kappa}{p},p}$ a.s.,
\[
u \in C^{\theta-\varepsilon}_{\rm loc}((0,\sigma);X_{1-\theta}) \text{ a.s.\ for all } \theta\in (0,1/2), \varepsilon\in (0,\theta).
\]
\end{enumerate}
\end{theorem}

The result in \eqref{it1:mainintro} ensures the well-posedness of \eqref{eq:SEEintro}. Furthermore, when $\a>0$, it also shows instantaneous regularization of solutions to \eqref{eq:SEEintro}. Indeed, in this case, the solution belongs to $X_{1-\frac{1}{p},p}$, while the initial data is in $X_{1-\frac{1+\a}{p},p}$. This indicates an immediate improvement in the regularity of the solution, as the inclusion $X_{1-\frac{1}{p},p}\subseteq X_{1-\frac{1+\a}{p},p}$ is strict in the case $A$ is an unbounded operator (e.g.\ a differential operator).
The instantaneous regularization result in \eqref{it3:mainintro} extends this further, including the important case where $\a=0$. Its proof relies on the above observation of the instantaneous regularization when a non-trivial time weight is present. A detailed proof is provided in Subsection \ref{subsec:reg}. Finally, \eqref{it2:mainintro} gives several criteria for the explosion of solutions to \eqref{eq:SEEintro}, which can be used to establish the \emph{global well-posedness} of \eqref{eq:SEEintro} whenever (sufficiently strong) a priori estimates for its solution are available. Let us point out that the first criterion in \eqref{it2:mainintro} is new.
It is important to observe that by choosing $\a$ as large as possible, the conditions in \eqref{it2:mainintro} become easier to check.
For a detailed discussion, the reader is referred to Subsection \ref{subsec:blow_up}.

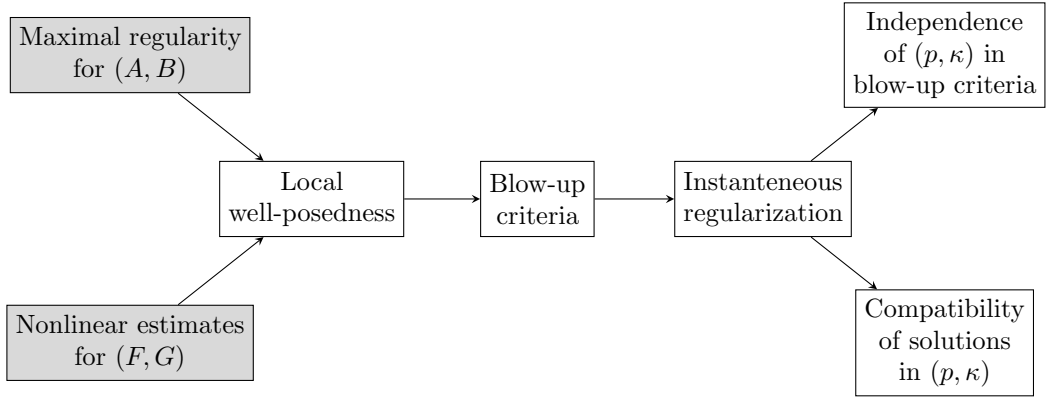
\begin{figure}[ht]
\centering
\begin{tikzpicture}[node distance=2cm, auto, >=latex',
    every node/.style={rectangle, draw, fill=white, minimum size=1cm, align=center}, 
    every edge/.style={draw, ->,>=stealth} 
]
    \node (A1) [fill=gray!30] {Maximal regularity \\ for $(A,B)$}; 
    \node (B) [below right of=A1, xshift=1cm, yshift = -0.5cm] {Local\\ well-posedness};
    \node (A2) [fill=gray!30, below left of=B, xshift=-1cm, yshift = -0.5cm] {Nonlinear estimates \\ for $(F,G)$}; 
    \node (C) [right of=B, xshift=1cm] {Blow-up\\ criteria};
    \node (D) [right of=C, xshift=1cm] {Instanteneous \\ regularization};
    \node (E1) [above right of=D, xshift=1cm, yshift=0.5cm] {Independence \\ of $(p,\kappa)$ in \\ blow-up criteria};
    \node (E2) [below right of=D, xshift=1cm, yshift=-0.5cm] {Compatibility \\ of solutions \\ in $(p,\kappa)$};

    \draw[->,>=stealth] (A1) -- (B);
    \draw[->,>=stealth] (A2) -- (B);
    \draw[->,>=stealth] (B) -- (C);
    \draw[->,>=stealth] (C) -- (D);
    \draw[->,>=stealth] (D) -- (E1);
    \draw[->,>=stealth] (D) -- (E2);

\end{tikzpicture}
\caption{Diagram describing our results. The grey boxes represent the assumptions, while the other boxes indicate the outputs of our framework. The arrows show the flow of implications in which one proves the corresponding result. The three central boxes correspond to Theorem \ref{thm:mainintro}, while the two on the right correspond to Corollaries \ref{cor:transfblowup} and \ref{cor:comp} presented later in the manuscript.}
\label{fig:diagram_abstract}
\end{figure}

\subsubsection*{Applications of our framework to SPDEs}
As discussed below \eqref{eq:SEEintro}, the framework developed in \cite{AV19_QSEE_1,AV19_QSEE_2} has been successfully applied to various concrete SPDEs. Some applications were presented in \cite[Sections 5-7]{AV19_QSEE_1} and \cite[Section 7]{AV19_QSEE_2}. Below, we provide additional references to the places where our framework has been applied, though it is worth noting that further applications to SPDEs are still being prepared. A list of open problems is presented at the end of the manuscript.

For reaction-diffusion equations (e.g.\ Allen--Cahn, Lotka--Volterra, and Gray-Scott equations), local and global well-posedness results can be found in \cite{AVreaction-local,AVreaction-global}. Regarding stochastic fluid dynamic models, \cite{AV20_NS} proves the well-posedness and sharp blow-up criteria of Navier--Stokes equations arising in the study of turbulence. For the 3D stochastic primitive equations, commonly used in atmospheric and oceanic dynamics, \emph{global} well-posedness under various assumptions is established in \cite{Primitive1,Primitive2,agresti2023primitive}. In \cite{AVvar}, we enhanced the variational setting and subsequently applied it to models such as the Cahn--Hilliard equation, tamed 3D Navier--Stokes equations, and more. This approach was further extended to the L\'evy setting with more flexibility in the conditions on $F$ in \cite{BGV}. The results from \cite{AVvar} were employed in \cite{BH24_stability} to study the stability of travelling waves in reaction-diffusion equations. Our results have also been applied to the thin-film equation, which is a quasilinear model (see \cite{AgrSau}).

Typical improvements in our results include: the ability to handle rough initial data, the derivation of optimal blow-up criteria for solutions, the handling of rougher noise (e.g.\ rough Kraichnan noise), the inclusion of superlinear diffusion, and the provision of instantaneous high-order regularity. It is important to note that in the context of stochastic reaction-diffusion equations and fluid dynamics, the roughness of the noise has significant physical relevance, while the sharpness of the blow-up criteria is often essential for proving global well-posedness.
Finally, it is worth mentioning that our framework has proven useful in the context of \emph{regularization by noise}, where stochastic perturbations improve the global well-posedness theory, as demonstrated in \cite{A24_delayed, agresti2024global}.

\smallskip

As Theorem \ref{thm:mainintro} and Figure \ref{fig:diagram_abstract} show, there are two essential ingredients for applying our results: maximal regularity estimates for $(A,B)$ and the nonlinear estimate \eqref{eq:Fsubcriintro}  for $(F,G)$. Typically, \eqref{eq:Fsubcriintro} is a relatively easy consequence of Sobolev embedding and H\"older's inequality. In contrast, the maximal regularity assumption on $(A,B)$ is more technical. However, there is now a well-established body of literature on maximal regularity, with corresponding estimates available for a broad range of situations, particularly when $B= 0$. For further details, the reader is referred to Section \ref{sec:SMR}. Some open problems on maximal regularity are also discussed at the end of this survey.

Finally, we mention that the diagram in Figure \ref{fig:diagram_abstract} can be further extended. Indeed, if one has sufficiently strong energy bounds for the solution that can be connected to the blow-up criteria, then global existence and uniqueness can also be established. This will be demonstrated in several concrete examples throughout the manuscript.

\subsection{Scaling and criticality for stochastic Navier--Stokes equations}
\label{ss:scaling_intro}
We now return to the concept of criticality in the PDE sense and try to connect it to the abstract setting outlined in the previous subsection in the special case of the Navier--Stokes equations with transport noise, which arise in the study of turbulent flows \cite{BCF91,DP24_two_scale,FL23_book,K68,MiRo04}.
In this discussion, we partly follow the approach outlined in \cite[Subsection 1.1]{AV20_NS}. The scaling arguments presented here can also be applied to other SPDEs such as stochastic reaction-diffusion equations, as discussed in Subsection \ref{sss:tracecriticalF_AC}.

Consider the following Navier--Stokes equations with transport noise on $\R^d$ for the unknown velocity field $u$ and pressures $P$ and $\wt{P}_n$,
\begin{equation}
\label{eq:NS_intro}
\textstyle  \dd u  =[\Delta u -\nabla P +(u\cdot\nabla) u]\, \dd t +\sum_{n\geq 1} \big[(b_n\cdot \nabla) u -\nabla \wt{P}_n\big]\,\dd W^n_t, \qquad \nabla \cdot u=0.
\end{equation}
Here, $(W^n)_{n\geq 1}$ is a family of standard independent Brownian motions and $(b_n)_{n\geq 1}\in \ell^2(\N_{\geq 1};\R^d)$.
As is well-known, the pressures $P$ and $\wt{P}_n$ are uniquely determined by $u$ through the divergence-free conditions, so we focus on the velocity field $u$ in the following analysis.
It is important to note that the above model is a simplified version of the physically relevant Navier--Stokes equations with transport noise, where the noise coefficients $(b_n)_{n\geq 1}$ in \eqref{eq:NS_intro}, are typically $x$-dependent. We will come back to this point below.

Next, we discuss the invariance of (local smooth) solutions to stochastic Navier-Stokes equations \eqref{eq:NS_intro} under the map $u\mapsto u_{\lambda}$, where $\lambda >0$ and
\begin{equation}
\label{eq:NS_scaling_map}
u_{\lambda}(t,x):=\lambda^{1/2}u(\lambda t,\lambda^{1/2} x),\qquad
 (t,x)\in \R_+\times\R^d.
\end{equation}
In the deterministic setting (e.g.\ $b_n\equiv 0$), the invariance of Navier-Stokes equations under the above mapping is well known.
In the PDE literature (see e.g.\ \cite{Can04,LePi,PW18,Trie13}), Banach spaces of functions (locally) invariant under the induced map on the initial data, i.e.\
$$
u_0\mapsto u_{0,\lambda}, \quad \text{ where } \quad u_{0,\lambda} :=\lambda^{1/2} u_{0}(\lambda^{1/2}\cdot ),
$$
are referred to as \textit{critical} for \eqref{eq:NS_intro}. Examples of such critical spaces include the Besov space
$
\dot{B}^{d/q-1}_{q,p}(\R^d;\R^d)
$
for $1<q,p<\infty$
and the Lebesgue space
$
L^d(\R^d;\R^d)
$. Indeed, these spaces satisfy the scaling
$$
\|u_{0,\lambda}\|_{\dot{B}^{d/q-1}_{q,p}(\R^d;\R^d)}\eqsim \|u_0\|_{\dot{B}^{d/q-1}_{q,p}(\R^d;\R^d)}, \qquad
\|u_{0,\lambda}\|_{L^d(\R^d;\R^d)}\eqsim \|u_0\|_{L^d(\R^d;\R^d)},
$$
where the implicit constants do not depend on $\lambda>0$.

For the stochastic Navier--Stokes equations \eqref{eq:NS_intro}, similar behaviour appears. More precisely, one can verify that if $u$ is a (local smooth) solution to \eqref{eq:NS_intro} on $\R^d$, then $u_{\lambda}$ is a (local smooth) solution to \eqref{eq:NS_intro} on $\R^d$, where the noise $(W^n)_{n\geq 1}$ is replaced by the \emph{scaled} noise $(\beta_{\cdot,\lambda}^n)_{n\geq 1}$, defined by $\beta_{t,\lambda}^n:=\lambda^{-1/2}W_{\lambda t}^n$ for $t\geq 0$ and $n\geq 1$.
Indeed, $\beta_{t,\lambda}^n$ are independent standard Brownian motions again, and for all $n\geq 1$, we have
\begin{align}
\label{eq:scaling_noise}
\textstyle \int_0^{t/\lambda} (b_n\cdot \nabla)u_{\lambda}(s,x) \,\dd  \beta_{s,\lambda}^n
&= \textstyle \lambda
\int_0^{t/\lambda}  (b_n\cdot \nabla)u(\lambda s,\lambda^{1/2} x)
\,\dd \beta_{s,\lambda}^n \\
\nonumber
&= \textstyle \lambda^{1/2}
\int_0^{t} (b_n\cdot \nabla)u(s,\lambda^{1/2} x)\, \dd W_{s}^n,
\end{align}
which matches the scaling of the deterministic nonlinearity:
\begin{align*}
\textstyle \int_0^{t/\lambda}\big(u_{\lambda}(s,x)\cdot \nabla\big)u_{\lambda}(s,x)\, \dd s
&= \textstyle \lambda^{1/2}\int_0^{t} \big(u(s,\lambda^{1/2}x)\cdot \nabla \big)u\big(s,\lambda^{1/2}x) \,\dd s.
\end{align*}
A similar scaling argument also applies to the other deterministic integrals.

Note that the above argument can also be applied in the important case where the $b_n$ are $x$-dependent. Indeed, if \eqref{eq:NS_intro} holds only on a ball $B_r(x_0)$, one can replace the rescaling in \eqref{eq:NS_scaling_map} by $u_{\lambda}(t,x)=\lambda^{1/2}u(\lambda t,x_0+ \lambda^{1/2}x)$ which results in a solution on $B_{r/\lambda^{1/2}}(0)$ of the Navier--Stokes equations. If we now let $\lambda\to \infty$ and we assume that $b_n(x)\approx b_n(x_0)$ if $x\approx x_0$ (which is the case for Kraichnan noise, see e.g.\ \cite[Proposition 2.1]{agresti2023primitive}), then the scaling argument above can be applied analogously.
Similar results hold if $\R^d$ is replaced by other domains.
\smallskip

In this survey, results on the Navier--Stokes equations with transport noise can be found in Subsections \ref{ss:SNS} and \ref{subsec:SNSRd} for the two-dimensional and three-dimensional cases, respectively. A partial discussion of the results proven in \cite{AV20_NS} in the case of periodic boundary conditions, is provided below Theorem \ref{t:NS_Rd_local}.
At this point, the reader may find it unclear how our framework captures the scaling invariance mentioned earlier. However, the key idea is that the assumptions in our main results, - specifically, the maximal regularity of $(A,B)$ and the mapping properties of $(F,G)$ - preserve the scaling (see Figure \ref{fig:diagram_abstract}), and our framework preserves this property. Further details are provided in Subsection \ref{subsec:SNSRd}.

Finally, it is worth mentioning that the $1/2$-scaling loss in the noise (see \eqref{eq:scaling_noise}), is responsible for the $1/2$-loss of smoothness in the maximal regularity for SPDEs, which will be discussed in Subsection \ref{sec:SMR}. This also motivates the appearance of $X_{1/2}$ for the $B$- and $G$-parts, see Subsection \ref{ss:stoch_max_reg} and \eqref{eq:Fsubcriintro}.

\subsection{Previous work on nonlinear equations through maximal $L^p$-regularity}
\label{ss:literature_discussion}
This subsection provides an overview of the literature on nonlinear evolution equations using the maximal $L^p$-regularity approach.
While we cannot give a comprehensive review here, we focus on works that are directly related to this framework for nonlinear evolution equations. Additional references for specific SPDEs discussed earlier will be found in the main body of the survey.

\subsubsection{Deterministic evolution equations}\label{subsec:detEE}
The well-posedness theory for quasilinear evolution equations has been an active area of research for several decades. The terminology surrounding the concepts of ``semi'', ``quasi'', and ``fully'' nonlinear equations varies across different research communities. In the discussion, we follow the terminology used in the standard references for deterministic evolution equations and PDEs \cite{Am, Fried, Henry,  Krylovnonlinear, LSU, Lun, pruss2016moving}.

Initially, quasilinear evolution equations were studied using the theory of evolution systems for non-autonomous settings (see e.g.\ \cite{Amann86, DPGsum, DPGeq, Sob66}). Inspired by \cite{DP86}, a more direct approach utilizing linearization techniques was introduced in \cite{ClLi}, and is based on maximal $L^p$-regularity (see also \cite{pruss2002maximal} for an overview on these topics). By this time maximal $L^p$-regularity for evolution equations was already known in various contexts \cite{DPGsum, DoreVenni, Lamberton, PrSo}.

A significant observation made in \cite{PruSim04} was that maximal $L^p$-regularity could also be considered with power weights $t^{\kappa}\,\dd t$ with $\kappa\in [0,p-1)$. This insight led to several variants of the theory. In \cite{PrussWeight1} it was used to broaden the class of admissible initial values and derive weaker abstract blow-up criteria for global well-posedness.
Building on this, \cite{PrussWeight2} introduced the more flexible condition \eqref{eq:Fsubcriintro} for the deterministic nonlinearity $F$. However, this came with the restriction that the condition \eqref{eq:subcriticalintro} holds with strict inequality, thus leading to a {\em subcritical} setting. Shortly thereafter, in \cite{addendum}, it was recognized that equality in \eqref{eq:subcriticalintro} could be allowed, thereby extending the framework to include the critical setting.

Applications to Navier--Stokes equations with Navier boundary conditions were presented in \cite{PW18}. Further results and applications can be found in \cite{CriticalQuasilinear}, where critical spaces have been identified for various concrete equations. Recently, in \cite{MatRobWal}, further progress has been made in a slightly different direction, and, in particular, no maximal regularity assumptions are needed there.

For a comprehensive overview of these topics, the reader is referred to the monographs \cite{pruss2016moving} and \cite{Analysis3}, as well as the survey \cite{Wilkesurvey}. Notably, in \cite[Chapters 17 and 18]{Analysis3}, the additional conditions on $A$ (BIP) and $X_0$ (UMD) that appeared in earlier works \cite{addendum,CriticalQuasilinear, Wilkesurvey} have been removed. The revised approach now relies solely on criticality and maximal $L^p$-regularity of the linear mapping $v\mapsto A(u_0)v$, where $u_0$ is the initial value.

\subsubsection{Stochastic evolution equations}
The paper \cite{MaximalLpregularity} and its extension to nonzero $B$ in \cite{VP18} have played a central role in finding the right definition of stochastic maximal $L^p$-regularity, a cornerstone assumption in Theorem \ref{thm:mainintro}. Both papers provide general classes of operators that satisfy stochastic maximal $L^p$-regularity.
Early contributions such as \cite{NVW11eq} for semilinear equations and \cite{Hornung} for quasilinear equations, played a vital role in shaping the formulation of Theorem \ref{thm:mainintro}. However, at the time these papers were written, several key aspects of the (deterministic) theory were not yet developed. Notably, criticality had not yet been formulated in an abstract framework, the operator $B$ was restricted to be zero or small, and no weighted theory had been established. Ultimately, the advancements in deterministic results discussed in Subsection \ref{subsec:detEE} were indispensable for the discovery and formulation of a critical stochastic framework. Without these developments, the progression to the current state of the theory would not have been possible.

\subsection{Other approaches to well-posedness for nonlinear SPDE}
\label{ss:historic_part_intro}
There exist many different approaches and viewpoints on SPDEs. Therefore, we cannot provide a complete overview here. However, we will at least present a selection of approaches and viewpoints. Where applicable, we connect these to our framework and highlight key differences.

\subsubsection{The semigroup approach to nonlinear stochastic evolution equations}
For a collection of references on this approach, the reader is referred to the monograph \cite{DPZ}.
Many papers in this approach focus on stochastic evolution equations in Hilbert spaces. While this framework is sufficient for many applications, it introduces limitations in the case of higher dimensions or rapidly growing nonlinearities. For example in the case of cubic nonlinearities in dimension three, an $L^2$-setting often becomes inadequate. The paper \cite{Brz2} extends the semigroup approach to $L^q$-spaces or Sobolev spaces with $q\in [2, \infty)$, thereby enabling far-reaching applications to nonlinear SPDEs.

\subsubsection{Weak solutions in probabilistic sense}

A powerful method for establishing the existence of solutions to SPDEs is the use of compactness techniques. In the deterministic setting, these methods are well-established and include the Schauder and Schaefer fixed point theorems (see \cite{EvansPDE}). However, these results do not guarantee uniqueness.
In the stochastic setting, compactness-based methods can be adapted to prove the existence of weak solutions. Compactness arguments imply tightness of the laws of approximate solutions, and Skorokhod's theorem is then used to construct a new probability space where a solution exists. For more details, see the monograph \cite{BFH}.

Usually, compactness is created through Sobolev embedding and parabolic regularization theory. Besides the examples in the above monograph, some far-reaching examples can be found in \cite{DHV16, RSZ}. In \cite{BecVer}, it is shown that in the variational framework, some compactness comes for free.

Ensuring strong existence (solutions on the original probability space) is often challenging. In certain cases, this can be achieved through uniqueness results using Yamada-Watanabe-type theorems in infinite-dimensional settings (see \cite{Kunze, Ondrejat04, rockner2008yamada, theewis2025yamada}).

\subsubsection{Random field approach}

The random field approach, originating from \cite{Walsh}, enables pointwise arguments and the use of scalar-valued stochastic calculus. As discussed in \cite{DalangQS}, this method is ultimately equivalent to the semigroup approach. Combining insights from both approaches often provides the best results.

\subsubsection{Rough path approach to PDEs}

Brownian motion can be viewed as a special case of rough paths, allowing certain SPDEs to be solved pathwise. This approach was first introduced in the context of ODEs by \cite{Lyons98}, and a detailed overview is provided in the lecture notes \cite{FrizHairer}. These notes also explore rough PDEs and provide references on this rapidly developing field.

Unlike traditional SPDE frameworks, the rough path approach does not require a semimartingale structure, making it applicable to a broader class of driving processes, such as fractional Brownian motion. A key advantage of this pathwise viewpoint is its ability to handle SPDEs that are classically ill-posed \cite{HairCom}. In some cases, renormalization is required to give a rigorous meaning to the equation \cite{HaiKPZ, HaiReg, hairer2014singular}.

Recent advancements have shown that even supercritical SPDEs can be included through suitable approximation methods (see \cite{CET, CGT} and the lecture notes \cite{cannizzaro2024lecture}). Another significant benefit of the rough path approach, even in ODE settings, is its direct construction of random dynamical systems for the nonlinear solution operator. For infinite-dimensional settings, see \cite{HEsseNeam, HesseNeam20, HesNeam22, HocHes}, and for applications in stability theory, refer to \cite{KueNea}.

While the rough path framework shares some connections with parabolic regularity theory, it primarily relies on Schauder estimates in space-time H\"older spaces. Since Brownian motion paths belong to $C^{1/2-\varepsilon}$ for any $\varepsilon>0$, there is often a loss in regularity. This can present challenges in obtaining classical solutions, especially in the critical settings considered in our framework.

\subsection{Overview of the results in the survey}
\label{ss:overview}

Our goal in this survey is to present the key ideas of our work within a simpler framework that remains sufficiently general to cover a wide range of applications. We hope these results will serve as a valuable resource for researchers working on SPDEs. We acknowledge that extracting the main insights from our earlier works can be challenging, as those papers were presented in their most general forms. Below, we outline the primary contributions of this survey and what readers can expect.

\smallskip
\noindent
\underline{Stochastic maximal regularity}. In Section \ref{sec:SMR}, we provide an overview of maximal $L^p$-regularity both in the deterministic and stochastic setting. In Proposition \ref{prop:localizationSMR}, we show that for linear equations, only the path properties of inhomogeneities influence the regularity. Furthermore, in Theorem \ref{thm:SMRHS} it is proved that in the Hilbert space setting $(A,0)$ always exhibits stochastic maximal $L^p$-regularity. This result generalizes the $p=2$ case of \cite{DPZ82} (see also \cite[Theorem 6.12(2)]{DPZ}), and builds on the deterministic case due to \cite{DeSimon} (see also \cite[Corollary 17.3.8]{Analysis3}).

\smallskip
\noindent
\underline{Local well-posedness, blow-up criteria, and regularity.}
In Sections \ref{sec:loc-well-posed} and \ref{sec:blowup}, we present a slightly more general version of Theorem \ref{thm:mainintro}. We relax the conditions on $F$ and $G$ in \eqref{eq:Fsubcriintro} by replacing the spaces $X_{\beta}$ by $X_{\beta,1}$. Moreover, Theorem \ref{thm:criticalblowup}\eqref{it1:criticalblowup} introduces a new blow-up criterion. In Subsection \ref{subsec:reg}, we provide self-contained proofs of three parabolic regularization theorems in the time variable. Applying these regularization theorems typically removes the critical nature of the problem, paving the way for the application of classical bootstrap methods. When combined with stochastic maximal $L^p$-regularity, this approach facilitates enhanced spatial integrability and regularity.

\smallskip
\noindent
\underline{Critical variational setting and its applications.}
Sections \ref{sec:Var} and \ref{sec:appl} focus on the variational setting with $p=2$ and are based on the results from \cite{AVvar}. We include additional results on higher-order smoothness and moments.  Section \ref{ss:Fluid} demonstrates that a broad class of fluid dynamical models fits within this framework, with detailed applications to two-dimensional Navier--Stokes, Boussinesq, and quasi-geostrophic equations.

\smallskip
\noindent
\underline{SPDEs in the $L^p(L^q)$-setting with either $p>2$ or $q>2$.}
Subsection \ref{ss:AllenCahnLpLq} provides a self-contained analysis of the Allen--Cahn equation on domains $\Dom\subseteq \R^3$. We explain why the $L^p(L^q)$-framework is necessary, and demonstrate how our abstract results yield global existence and uniqueness. This example serves as a blueprint for applying our techniques to other problems.

In Subsection \ref{ss:reaction}, we summarize several results on reaction-diffusion equations with transport noise and periodic boundary conditions. Moreover, we include applications to both coercive systems and non-coercive systems. In particular, we explain in detail the global well-posedness of the Lotka--Volterra model, which is non-coercive.

In Subsection \ref{ss:quasiGSLp}, we consider the so-called quasi-geostrophic equation used in fluid dynamics. Here we partly demonstrate the power of our setting. We identify the critical spaces and combine temporal regularization and classical bootstrap techniques to obtain higher-order smoothness and integrability. While the results here may be new, our primary aim is to demonstrate the practical implementation of our techniques in a relatively straightforward context.

Finally, Subsection \ref{subsec:SNSRd} discusses local well-posedness and regularity results for the stochastic Navier--Stokes equations with transport noise on $\R^d$. Moreover, for $d=2$, we connect these local results to the global ones already obtained in Subsection \ref{ss:SNS}.

\subsection{Notation}
\label{ss:notation_terminology}

Below we collect some of the notation which often appears in the paper.

\begin{flushleft}
$w^a_\kappa(t) = (t-a)^{\kappa}$, and $w_\kappa(t) = t^{\kappa}$, weights, see Subsection \ref{subsec:interp}.

$X_0$, $X_1$ Banach spaces such that $X_1\hookrightarrow X_0$ densely and continuously.

$X_{\theta,p} = (X_0, X_1)_{\theta,p}$ real interpolation, $X_{\theta} = [X_0, X_1]$ complex interpolation, see Subsection \ref{ss:interp}.

$L^p_{\rm loc}(\Dom)$: all $f:\Dom\to \R$ such that for all compact subsets $K\subseteq \Dom$, $f|_{K}\in L^p(K)$. The same notation will be used for other (weighted) function spaces.

$H^{s,q}$ Bessel potential, $B^{s}_{q,p}$ Besov, and $W^{k,p}$ Sobolev spaces, see Subsections \ref{subsec:interp} and \ref{appendix:extra}.

$\BD, \HD, \WD$, see Example \ref{ex:extrapolated_Laplace_dirichlet}.

$W^{k,p}_0(\Dom)$ closure of $C^\infty_c(\Dom)$ in $W^{k,p}(\Dom)$.

$\Wz^{1,p},\hz^{\theta,p}$ see Subsection \ref{subsec:interp}.

$C(\Dom)$ continuous functions, $C^{s}(\Dom)$ H\"older continuous functions.

$C_b(\Dom)$, $C^s_b(\Dom)$: their subspaces of bounded functions.

$C^{\theta_1,\theta_2}(I \times \Dom)= C^{\theta_1}(I; C(\overline{\Dom})) \cap C(\overline{I}; C^{\theta_2}(\Dom))$ parabolic H\"older spaces.

$f\in C^{\theta_1,\theta_2}_{\loc}(I \times \Dom)$ if $f|_{J\times K}\in C^{\theta_1,\theta_2}(J\times K)$ for all compact subsets $J\subseteq I$ and $K\subseteq \Dom$.

$\gamma(\mathcal{U},X)$ $\gamma$-radonifying operators, $\calL_2(\mathcal{U},K)$ Hilbert-Schmidt operators, see Subsection \ref{subsec:gamma}.

$L^0(S;X)$ strongly measurable functions on a measure space $(S,\mathcal{A}, \mu)$.

$W$ cylindrical Brownian motion on $\mathcal{U}$, see Subsection \ref{subsec:stochint}.

$\mathcal{U}$ real separable Hilbert space.

$(V, H, V^*)$ Gelfand triple of Hilbert spaces, see Subsection \ref{ss:settVHV}.

$(\cdot, \cdot)$ inner product; $\lb  \cdot, \cdot\rb$ duality.

$(\O,\mathcal{A},\P)$ probability spaces, $\E$ expectation, $\F_t$ filtration, $\Progress$ progressive $\sigma$-algebra.

$A\in \calL(X_1, X_0)$, $B\in \calL(X_1, \calL_2(\mathcal{U}, X_{1/2}))$ leading operators in deterministic and stochastic part, see Subsections \ref{ss:stoch_max_reg} and \ref{ss:setting}.

$F$ deterministic nonlinearity, $G$ stochastic nonlinearity, see Subsection \ref{ss:setting}.

$p$ integrability in time, $q$ integrability in space.

$\beta_j$ smoothness parameter, $\rho_j$ power in nonlinearity, see Subsection \ref{ss:setting}.

$P$, $P_n$ pressure term in fluid dynamics, $\pr$ Helmholtz projection.

$\Bs$, $\Hs$, $\Ls$: divergence free subspaces, see Subsections \ref{sss:Helm} and \ref{subsec:SNSRd}.
\end{flushleft}

\subsubsection*{Acknowledgements}
{The authors thank Udo B\"ohm, Foivos Evangelopoulos-Ntemiris, Fabian Germ, Daniel Goodair, Katharina Klioba, Theresa Lange, Emiel Lorist, Jan van Neerven, Martin Meyries, Floris Roo\-denburg, Esmée Theewis and the anonymous referee for helpful comments.}

\section{Preliminaries}
\label{sec:preliminaries}

\subsection{The $H^\infty$-calculus for sectorial operators}\label{ss:Hinfty}
The $H^\infty$-calculus is a powerful modern tool for analyzing evolution equations.
Monographs on the $H^\infty$-calculus include \cite{Haase:2, KuWe, Analysis2, Analysis3}. After a very brief introduction, we briefly mention the main concepts.

A well-known functional calculus is the one established for self-adjoint operators (or normal operators) in Hilbert spaces. However, many differential operators with nonconstant coefficients are not self-adjoint or normal. Furthermore, in a Banach space, the concepts of self-adjointness and normality are not defined. This has led to the development of other calculi over the past century. Motivated by an open problem of Kato, the $H^\infty$-calculus was introduced in \cite{McY}. Key results from this work were extended to Banach spaces in the influential paper \cite{CDMY}.

Let $A:\Do(A)\subseteq X\to X$ be a closed operator on a Banach space $X$. An operator $A$ is called {\em sectorial} if the domain and the range of $A$ are dense in $X$ and there exists $\phi\in (0,\pi)$ such that $\sigma(A)\subseteq \overline{\Sigma_{\phi}}$, where $ \Sigma_{\phi}:=\{z\in \C\,:\,|\arg z|< \phi\}$, and there exists $C>0$ such that
\begin{equation}
\label{eq:sectorial}
|\lambda|\|(\lambda-A)^{-1}\|_{\calL(X)}\leq C, \qquad \forall \lambda \in \C\setminus \overline{\Sigma_{\phi}}.
\end{equation}
The quantity $\om(A):=\inf\{\phi\in (0,\pi)\,:\,\eqref{eq:sectorial}\,\,\text{holds for some}\; C>0\}$ is called the {\em angle of sectoriality} of $A$.

For $\phi\in (0,\pi)$, the space $H^{\infty}_0(\Sigma_{\phi})$ consists of all holomorphic functions $f:\Sigma_{\phi}\to \C$ such that $|f(z)|\lesssim \min\{|z|^{\varepsilon},|z|^{-\varepsilon}\}$ for some $\varepsilon>0$. For $\phi>\om(A)$ and $f\in H^{\infty}_0(\Sigma_{\phi})$, we define $f(A)\in \calL(X)$ through a contour integral over $\Gamma:=\partial \Sigma_{\psi}$, where $\psi\in (\omega(A), \phi)$:
$$
\textstyle f(A):=\frac{1}{2\pi i} \int_{\Gamma} f(z)(z-A)^{-1}\,\dd z.
$$
The operator $A$ is said to have {\em a bounded $H^{\infty}(\Sigma_{\phi})$-calculus} if there exists $C>0$ such that
\begin{equation}
\label{eq:H_infinite_calculus}
\|f(A)\|_{\calL(X)}\leq C\|f\|_{H^{\infty}(\Sigma_{\phi})},\qquad \forall f \in H^{\infty}_0(\Sigma_{\phi}),
\end{equation}
where $\|f\|_{H^{\infty}(\Sigma_{\phi})}:=\sup_{z\in \Sigma_{\phi}}|f(z)|$.
The angle of the $H^{\infty}$-calculus of $A$ is defined as \[\angH(A):=\inf\{\phi\in (0,\pi)\,:\,\eqref{eq:H_infinite_calculus}\text{ holds for some }C>0\}.\]

Without imposing conditions beyond sectoriality, it is possible to construct an extended calculus that incorporates additional operators such as $A^{z}$ for $z\in \C$ (see \cite{Haase:2, Analysis3}). If an operator $A$ admits a bounded $H^\infty$-calculus with angle $\phi\in (0,\pi)$, then $A^{it}\in \calL(X)$ for $t\in \R$. Operators $A$ satisfying this property are said to have {\em bounded imaginary powers (BIP)}. On Hilbert spaces, the converse holds as well.

The $H^\infty$-calculus is a powerful abstract tool in the study of evolution equations, including applications to (stochastic) PDEs. Knowing that an operator $A$ has a bounded $H^\infty$-calculus, ensures the boundedness of many singular integral operators.

In particular, many sectorial operators $A$ encountered in applications and studied on $L^q(\Dom)$ possess a bounded $H^\infty$-calculus (or at least $\lambda+A$ does for sufficiently large $\lambda$). Examples include uniformly elliptic operators with Dirichlet or Neumann boundary conditions, or more generally Lopatinskii-Shapiro boundary conditions, as well as the Stokes operator. Typically, one only needs some regularity properties of the underlying domain $\Dom$ and (H\"older)-regularity of the coefficients. For further examples and detailed references, see the notes of \cite[Chapter 10]{Analysis2}.

\subsection{Complex and real interpolation}\label{ss:interp}
In the theory of evolution equations, both the complex and real interpolation methods play a significant role. For a comprehensive treatment of general interpolation theory, the reader is referred to the monographs \cite{BeLo, Tri95}. A concise introduction to complex and real interpolation can be found in \cite[Appendix C]{Analysis1}. Additionally, \cite{LiLo24} offers a novel perspective on interpolation theory, unifying many classical and non-classical methods under a single framework. Interestingly, the terms "complex" and "real" in these methods are largely historical and bear little connection to the respective scalar fields.

A pair of Banach spaces $(X_0, X_1)$ is called a {\em Banach couple} if both $X_0$ and $X_1$ embed continuously into a Hausdorff topological
vector space $\mathcal{H}$. This set-up allows one to define the intersection $X_0\cap X_1$ and sum $X_0+X_1$, both of which are Banach spaces as well.

The space $\mathcal{H}$ is always kept fixed, and all of the definitions (intersection, sum, interpolation, etc.) depend on it.
In many situations, there is a rather natural choice of $\mathcal{H}$. For our purposes, one can take it to be the space of measurable functions or the space of distributions, possibly vector-valued.

\subsubsection{Complex interpolation}
The complex interpolation method, introduced by Calder\'on \cite{Cal64}, relies on the theory of holomorphic functions on strips taking values in $X_0+X_1$ (see \cite{LiLo24} for an alternative definition). The resulting complex interpolation space is denoted by $[X_0, X_1]_{\theta}$ with $\theta\in [0,1]$. Here, $[X_0, X_1]_{0} = X_0$ and $[X_0, X_1]_{1} = X_1$. For simplicity, we use the shorthand notation $X_{\theta}:=[X_0, X_1]_{\theta}$. It is important to note that $[X_0, X_1]_{\theta} = [X_1, X_0]_{1-\theta}$.

The complex interpolation spaces are crucial in evolution equations and PDEs for several reasons (see \cite{TayPDE1,TayPDE2,TayPDE3,Ya}). One application is that fractional Sobolev spaces, also known as Bessel potential spaces (see \eqref{eq:Bessel}), can equivalently be defined by complex interpolation. More generally, if $A$ has (BIP), domains of fractional powers of sectorial operators can be identified with complex interpolation spaces in the following way. Suppose that $A\in \calL(X_1, X_0)$ is sectorial on $X_0$, and $X_1 = \Do(A)$. Then for $\alpha>0$ and $\theta\in (0,1)$, \[\Do(A^{\alpha \theta}) = [X_0, \Do(A^{\alpha})]_{\theta}\]
if $I+A$ has BIP (see \cite[Theorem 6.6.9]{Haase:2}, \cite[Theorem 15.3.9]{Analysis3}, \cite[1.15.3]{Tri95}). This result holds particularly if $A$ has a bounded $H^\infty$-calculus. It is well-known that on Hilbert spaces the above identification of the domains of fractional powers is equivalent to the boundedness of the $H^\infty$-calculus, and to BIP (see \cite[Theorem 7.3.1]{Haase:2}).

\subsubsection{Real interpolation}
Real interpolation was introduced by both Lions and Peetre in a series of papers predating the discovery of the complex interpolation method. The primary motivation arose from its applications to PDEs. For historical insights, see \cite{BeLo, Tri95}, and for PDE applications, consult \cite{DPGsum, DPGeq, Lun}. Real interpolation can be introduced in various equivalent ways, many of which are unified under the framework in \cite{LiLo24}.

The real interpolation space is denoted as $(X_0, X_1)_{\theta,p}$ for $\theta\in (0,1)$ and $p\in [1, \infty]$, or simply as $X_{\theta,p}$. The space $X_{\theta,p}$ becomes larger as $p$ increases. For many choices of a space $E$, the following embedding holds:
\begin{align}\label{eq:sandwichinterp}
X_{\theta,1} \hookrightarrow  E \hookrightarrow X_{\theta,\infty}.
\end{align}
This applies, for instance, to $E = X_{\theta,p}$ and $E = X_{\theta}$. It also applies if $X_1= \Do(A)$ and $E = \Do(A^{\theta})$ (see \cite[Sections 1.3.3, 1.10 and 1.15.2]{Tri95}). Furthermore, if $(X_0,X_1)$  are Hilbert spaces, then (see \cite[Corollary C.4.2]{Analysis1})
\begin{align}\label{eq:sandwichinterpHS}
X_{\theta,2} = X_{\theta}.
\end{align}

One key application of real interpolation is in describing the regularity of orbits of analytic semigroups (see \cite[Section 2.2.1]{Lun} or \cite[Theorem L.2.4]{Analysis3}): for $p\in (1, \infty)$ and $\kappa\in [0,p-1)$, one has
$t\mapsto t^{\kappa/p} A e^{-tA} x \in L^p(0,T;X)$ if and only if $x\in (X,\Do(A))_{1-\frac{1+\kappa}{p},p}$. Moreover, the following norm equivalence holds:
\begin{align}\label{eq:equivorbitreal}
\|x\|_X+\|t\mapsto t^{\kappa/p} A e^{-tA} x\|_{L^p(0,T;X)} \eqsim \|x\|_{(X,\Do(A))_{1-\frac{1+\kappa}{p},p}}.
\end{align}

\subsection{Function spaces and interpolation}\label{subsec:interp}

In stochastic evolution equations, fractional smoothness in both space and time is often required. While H\"older smoothness is very useful, working with optimal $L^p$-estimates necessitates $L^p$-variants of H\"older smoothness. A suitable class of function spaces for this purpose is the Bessel potential spaces, denoted by $H^{s, p}(0,T;X)$.
These spaces are vector-valued, reflecting the fact that, in the evolution equation approach to PDEs, the spatial variable is represented by the function space $X$. For a comprehensive discussion of these spaces and their connections to evolution equations, we refer the reader to the monographs \cite{Am, AmannII, Analysis1, Analysis3, pruss2016moving}.

To handle rough initial data and obtain regularization results, weighted versions of the Bessel potential spaces are required. The standard class of weights for defining such spaces is the Muckenhoupt $A_p$-class  for $p\in (1, \infty)$ (see \cite[Appendix J.2]{Analysis2} and \cite[Chapter 7]{Grafakos1}). In our work, we specifically consider power weights of the form $w_\kappa^{a}(t) = (t-a)^{\kappa}$ for $a\in \R$ and $\kappa\in \R$. It is known that $w_{\kappa}^a\in A_p$ if and only if $\kappa\in (-1,p-1)$. In the stochastic setting,
$L^p$-theory often necessitates working with the more restrictive $A_{p/2}$-weights due to the roughness of Brownian paths. However, since power weights $w_{\kappa}^a$ are primarily used in evolution equations to temper the roughness of initial data, we only need $\kappa\geq 0$. Consequently, the relevant range of $\kappa$ in this manuscript is $[0,p/2-1)\cup\{0\}$.

Let $I\subseteq \R$ be an open interval. For a measurable function $w:I\to [0,\infty)$ and strongly measurable function $v$, let
\begin{equation*}%
\textstyle \|v\|_{L^p(I,w;X)} := \big(\int_{I} \|v(t)\|^p w(t)\, \dd t\big)^{1/p}.
\end{equation*}
For $w\in A_p$ one has that $L^p(I,w;X)\subseteq L^1_{\rm loc}(\overline{I};X)$, meaning every $u\in L^p(I,w;X)$ is integrable on compact subsets of $\overline{I}$. Similarly, for $p\in (2, \infty)$ and $w\in A_{p/2}$, $L^p(I,w;X)\subseteq L^2_{\rm loc}(\overline{I};X)$. If $w=1$, we simply write $L^p(I;X)$.

The Bessel potential spaces are introduced through complex interpolation as in \cite{MS12} and \cite[Section 3.4.5]{pruss2016moving}. It is possible to give an equivalent definition through fractional powers and restrictions under the additional assumptions that $X$ is UMD (see Subsection \ref{ss:UMD_definition} for the definition).

To give the definition through complex interpolation let $-\infty\leq a<b\leq \infty$, $\I=(a,b)$, $p\in (1,\infty)$, $w\in A_p$, and $\theta \in (0,1)$. Define the {\em vector-valued Bessel potential spaces} by
\begin{align}\label{eq:Bessel}
H^{\theta,p}(\I,w;X)&:=[L^p(\I,w;X),W^{1,p}(\I,w;X)]_{\theta},
\end{align}
where $W^{1,p}(\I,w;X)$ is the Sobolev space of functions which satisfy  $u,u'\in L^p(\I,w;X)$. If $w=1$, we omit it from the notation.
If there exist constants $c,C>0$ such that $c\leq w(t)\leq C$ for all $t\in I$,  the weighted space $H^{\theta,p}(\I,w;X)$ and the unweighted space $H^{\theta,p}(\I;X)$ are isomorphic.

In \cite{AV19_QSEE_1, AV19_QSEE_2} the spaces $\hz^{\theta,p}(\I,w;X)$ are extensively used to obtain estimates with constants independent of the interval $I$. These spaces will also be used at various points in the current manuscript. They are defined  analogously to $H^{\theta,p}$ but with $W^{1,p}$ replaced by $\Wz^{1,p} = \{u\in W^{1,p}: u(0) = 0\}$. For details, we refer to the aforementioned references.

We will frequently rely on the following embedding results in which time and space are mixed.
\begin{proposition}[Trace embedding with fractional smoothness]
\label{prop:tracespace}
Assume that $p\in (1,\infty)$, $\a\in [0,p-1)$, $\theta\in (0,1]$ and let $0\leq a<b< \infty$. Let $X_0,X_1$ be Banach spaces such that $X_1\hookrightarrow X_0$. Then the following hold:
\begin{enumerate}[{\rm(1)}]
\item\label{it:trace_with_weights_Xap} If $\theta>(1+\a)/p$, then
$$
H^{\theta,p}(a,b,w_{\a}^a;X_{1-\theta})\cap L^p(a,b,w_{\a}^a;X_1)\hookrightarrow
C([a,b];X_{1-\frac{1+\a}{p},p});
$$
\item\label{it:trace_without_weights_Xp} If $\theta>1/p$, then
\begin{align*}
H^{\theta,p}(a,b,w_{\a}^a;X_{1-\theta})\cap L^p(a,b,w_{\a}^a;X_1)&\subseteq
H^{\theta,p}_{\rm loc}((a,b];X_{1-\theta})\cap L^p_{\rm loc}((a,b];X_1)
\subseteq
C((a,b];X_{1-\frac{1}{p},p}).
\end{align*}
\end{enumerate}
\end{proposition}

The above result follows from \cite{MV14,AV19} provided $X_1=\Do(A)$ and $A$ is a sectorial operator on $X_0$. The more general case follows from \cite[Theorem 1.2]{ALV21}. The result is optimal in the sense that in many cases, there exists a right-inverse for the trace map
\[\Tr:H^{\theta,p}(a,b,w_{\a}^a;X_{1-\theta})\cap L^p(a,b,w_{\a}^a;X_1)\to X_{1-\frac{1+\a}{p},p}\]
defined by $\Tr\, u = u(a)$. Related results can be found in the recent work \cite{choi2023trace}.

We will also need several other classes of function spaces on domains $\mathcal{O}\subseteq \R^d$ which could be either $\R^d$, the flat torus $\T^d$, a Lipschitz domain, or smoother domains. In some cases, no assumptions on $\mathcal{O}$ beyond openness will be required. We need the Bessel potential spaces $H^{s,q}(\Dom)$ which form a complex interpolation scale in the sense that \[H^{s, q}(\Dom) = [H^{s_0, q}(\Dom),H^{s_1, q}(\Dom)]_{\theta},\]
where $s_0, s_1\in \R$, $s=  (1-\theta) s_0 + \theta s_1$ and $q\in (1, \infty)$. For the Fourier analytic definition of the Bessel potential spaces $H^{s,q}(\R^d)$, see \cite{Tri83}. On domains $\Dom$ as above, $H^{s,q}(\Dom)$ is defined by restriction. For $s\in \N_0$ and $q\in (1, \infty)$,  $H^{s,q}(\Dom)$ coincides with the classical Sobolev space  $W^{s, q}(\Dom)$.

To describe the space of initial data, we will often employ Besov spaces $B^{s}_{q,p}(\Dom)$ for $s\in \R$ and $q,p\in [1, \infty]$. These spaces can also be defined Fourier analytically on $\R^d$ and by restriction for $\Dom\subseteq \R^d$. For $\theta\in (0,1)$ it holds that
\[B^{s}_{q,p}(\Dom) = (H^{s_0, q}(\Dom),H^{s_1, q}(\Dom))_{\theta,p}\]
where $s_0, s_1\in \R$, $s=  (1-\theta) s_0 + \theta s_1$, $q\in (1, \infty)$ and $p\in [1, \infty]$. By \eqref{eq:equivorbitreal} and Proposition \ref{prop:tracespace}, this explains the necessity of Besov spaces. Note that $B^{s}_{2,2}(\Dom) = H^{s,2}(\Dom)$, which is often denoted by $H^s(\Dom)$.

We will make extensive use of Sobolev embedding results:
\begin{align*}
H^{s_0,q_0}(\Dom)&\hookrightarrow H^{s_1,q_1}(\Dom) \ \text{for} \ q_0, q_1\in (1, \infty), \ \ \text{and}  \ \
B^{s_0}_{q_0,p}(\Dom)&\hookrightarrow B^{s_1}_{q_1,p}(\Dom)\ \text{for} \ p,q_0, q_1\in [1, \infty],
\end{align*}
which hold under the condition $s_0-\frac{d}{q_0}\geq s_1-\frac{d}{q_1}$ and $q_0\leq q_1$. No conditions on $\Dom$ are needed here since the spaces are defined by restriction from the $\R^d$-case. Similarly,
\begin{align*}
H^{s_0,q_0}(\Dom)\hookrightarrow C_b^{s_1}(\overline{\Dom}) \  \text{for} \ q_0\in (1, \infty), \ \ \text{and}  \ \  B^{s_0}_{q_0,p}(\Dom)\hookrightarrow C^{s_1}_b(\overline{\Dom})\  \text{for} \ p,q_0\in [1, \infty],
\end{align*}
if $s_0-\frac{d}{q_0}\geq s_1$ and $s_1\in (0,\infty)\setminus\N$. Here, $C_b^{s}(\overline{\Dom})$ denotes the space of bounded $\lfloor s \rfloor$-times continuously differentiable functions with all derivatives bounded and $(s-\lfloor s \rfloor)$-H\"older continuous.

All the results extend naturally to the $X$-valued setting. To prove the equivalence of the definitions through Fourier analysis and complex interpolation, $X$ must be a UMD space (see \cite[Proposition 2.7]{AV19_QSEE_1} for details and references), as discussed in the next subsection.

\subsection{UMD spaces}
\label{ss:UMD_definition}
For an introduction to UMD spaces and their historical development, the reader is referred to \cite[Chapter 4]{Analysis1} and \cite{Pisier16}.
The abbreviation {\em UMD} stands for unconditional martingale differences and describes a condition on a Banach space $X$. UMD spaces have nontrivial types and cotypes and are reflexive. Conversely, many classical reflexive spaces have UMD. Examples include $L^p$, Sobolev space $W^{s,p}$, Besov spaces $B^{s}_{p,q}$ etc.\ for $p,q\in (1, \infty)$ and $s\in \R$.

Although UMD is a probabilistic notion, it admits geometric and analytical characterizations (see \cite{Bu1} and \cite[Chapter 5]{Analysis1}). In particular, if $X$ has UMD, then many results from harmonic analysis extend to the vector-valued setting. For instance, the boundedness of the Hilbert transform, Mikhlin multiplier theorem, and Littlewood-Paley theorem hold in this context. Conversely, the boundedness of the Hilbert transform on $L^p(\R;X)$ already implies that $X$ is UMD.

The theory of Fourier multipliers plays a central role in function space theory, making UMD a frequently encountered condition in the study of vector-valued function spaces. This is one reason why UMD spaces are important for our analysis. Another reason is that if $X$ has UMD, there exist two-sided estimates for stochastic integrals in terms of $\gamma$-radonifying operators. These estimates are instrumental in obtaining sharp regularity bounds for solutions to SPDEs. Before going into detail on stochastic integration in Subsection \ref{subsec:stochint}, we first discuss this special class of operators.

\subsection{The space $\gamma(\mathcal{U},X)$}\label{subsec:gamma}
In the theory of stochastic evolution equations, it is often necessary to consider noise that takes values in an infinite-dimensional Hilbert space. An example of such a noise is space-time white noise. However, in many models, the noise has to be coloured to get the well-posedness of the equation in a classical sense.

The space on which the noise is modelled is the Hilbert space $\mathcal{U}$, which is often taken to be $\ell^2$. The space $X$ (or later $X_0$) represents the space in which the evolution equation is analyzed. If $X$ is a Hilbert space, one typically assumes processes take values in the space of {\em Hilbert-Schmidt operators} $\calL_2(\mathcal{U},X)$ (see \cite[Appendix D]{Analysis1}). In the Banach space setting, a different space of operators is required. This is the space of {\em $\g$-radonifying operators}, denoted by $\gamma(\mathcal{U},X)$. For a detailed discussion and an extensive historical overview of this concept, the reader is referred to \cite[Chapter 9]{Analysis2}. Below, we summarize the aspects of the theory that will be relevant for our purposes.

Let $(\g_i)_{i\geq 1}$ be a sequence of independent standard normal random variables on a probability space $(\O,\P)$ and $(u_i)_{i\geq 1}$ an orthonormal basis for $\mathcal{U}$. A bounded linear operator $R:\mathcal{U}\to X$ belongs to $\g(\mathcal{U},X)$ if $\sum_{i=1}^{\infty} \g_i Ru_i$ converges in $L^2(\Omega;X)$. In this case, the $\gamma$-radonifying norm of $R$ is defined  as
$$
\textstyle \|R\|_{\g(\mathcal{U},X)}:=\Big\|\sum_{i=1}^{\infty} \g_i Ru_i\Big\|_{L^2(\Omega;X)}.
$$
The space $\g(\mathcal{U},X)$ is an operator ideal. If $X$ is a Hilbert space, then $\g(\mathcal{U},X)$ coincides with the space of Hilbert-Schmidt operators $\calL_2(\mathcal{U},X)$ (see \cite[Proposition 9.1.9]{Analysis2}).

If $X = L^q(S)$ with $q\in [1, \infty)$, where $(S,\Sigma,\mu)$ is a measure space, then $\gamma(\mathcal{U},X) = L^q(S;\mathcal{U})$ (see \cite[Proposition 9.3.2]{Analysis2}).
Similarly, if $X = H^{s,q}(\Dom)$ for $s\in \R$ and $q\in (1, \infty)$, then $\gamma(\mathcal{U},X) = H^{s,q}(\Dom;\mathcal{U})$.

\subsection{Stochastic integration}\label{subsec:stochint}

Let $(\O,\mathcal{A},\P)$ denote a probability space with filtration $\Filtr=(\F_t)_{t\geq 0}$. A process $\phi:[0,T]\times\Omega \to X$ is called {\em strongly progressively measurable} if for all $t\in [0,T]$, $\phi|_{[0,t]}$ is strongly $\Borel([0,t])\otimes \F_t$-measurable (where $\Borel$ denotes the Borel $\sigma$-algebra). The $\sigma$-algebra generated by the strongly progressively measurable processes is denoted by $\Progress$ and is a subset of $\Borel([0,\infty))\otimes \F_{\infty}$.

To model the noise, we use a cylindrical Brownian motion, which is an isonormal Gaussian process on the Hilbert space $L^2(\R_+;\mathcal{U})$ (see \cite{Kal}).
\begin{definition}
\label{def:Cylindrical_BM}
Let $\mathcal{U}$ be a Hilbert space. A bounded linear operator $W_\mathcal{U}:L^2(\R_+;\mathcal{U})\rightarrow L^2(\Omega)$ is said to be a {\em cylindrical Brownian motion} in $\mathcal{U}$ if the following are satisfied:
\begin{itemize}
\item for all $f\in L^2(\R_+;\mathcal{U})$ the random variable $W_\mathcal{U}(f)$ is centered Gaussian;
\item for all $t\in \R_+$ and $f\in L^2(\R_+;\mathcal{U})$ with support in $[0,t]$, $W_\mathcal{U}(f)$ is $\F_t$-measurable;
\item for all $t\in \R_+$ and $f\in L^2(\R_+;\mathcal{U})$ with support in $[t,\infty]$, $W_\mathcal{U}(f)$ is independent of $\F_t$;
\item for all $f_1,f_2\in L^2(\R_+;\mathcal{U})$ we have $\E(W_\mathcal{U}(f_1)W_\mathcal{U}(f_2))=(f_1,f_2)_{L^2(\R_+;\mathcal{U})}$.
\end{itemize}
\end{definition}
Given $W_{\mathcal{U}}$, the process $t\mapsto W(t)h:=W_\mathcal{U}(\one_{(0,t]} h)$ is a Brownian motion for each $h\in \mathcal{U}$.

\begin{example}
\label{ex:BM}
Let $(W^n)_{n\geq 1}$ be independent standard Brownian motions. Then
\[\textstyle  W_{\ell^2} (f) = \sum_{n\geq 1} \int_{\R_+} \lb f, e_n\rb \, \dd W^n \ \ \ \text{in $L^2(\Omega)$}\]
defines a cylindrical Brownian motion in $\ell^2$, where $(e_n)_{n\geq 1}$ is the standard basis of $\ell^2$.
\end{example}

For $0\leq a<b\leq T$ and $A\in \F_a$, the stochastic integral of $\one_{(a,b]\times A}  h\otimes x$ is defined by
\begin{equation}
\textstyle \int_0^{T} \one_{(a,b]\times A}(r)  h\otimes x \,\dd W(r):= \one_{A} (W(b)h-W(a)h)\,x\,,
\end{equation}
and extended to adapted step processes of finite rank by linearity. Here, the tensor notation $h\otimes x$ is short-hand notation for the operator from $\mathcal{U}$ into $X$ given by $u\mapsto (u, h)_\mathcal{U} x$.

The space $L^p_{\Progress}((0,T)\times\Omega;\g(\mathcal{U},X))$ denotes the progressively measurable subspace of $L^p((0,T)\times\Omega;\g(\mathcal{U},X))$. It can be shown that this coincides with the closure of the adapted step processes of finite rank (see \cite[Proposition 2.10]{NVW1}). Moreover, the latter paper provides a characterization of stochastic integrability along with two-sided $L^p$-estimates under the assumption that $X$ is a UMD space. For details the reader is referred to \cite[(5.5)]{NVW13}. These sharp estimates for stochastic integrals play a crucial role in proving the maximal $L^p$-regularity result of Theorem \ref{thm:SMRHinfty} below.

The following proposition provides a simple sufficient condition for stochastic integrability. For the definition of (Rademacher) type $2$, see \cite[Chapter 7]{Analysis2}. Spaces of type $2$ include $L^p$, $W^{s,p}$, and $B^{s}_{p,q}$ etc.\ for $p,q\in [2, \infty)$ and $s\in \R$.
\begin{proposition}
\label{prop:Ito}
Let $T>0$ and let $X$ be a UMD Banach space with type $2$. Then for every $p\in [0,\infty)$, the mapping $G\mapsto \int_0^\cdot G\,\dd W$ extends to a continuous linear operator from $L^p_{\Progress}(\Omega;L^2(0,T;\g(\mathcal{U},X)))$ into $L^p(\Omega;C([0,T];X))$. Moreover, for $p\in(0,\infty)$ there exists a constant $C_{p,X}$ such that for all $G\in L^p_{\Progress}(\Omega;L^2(0,T;\g(\mathcal{U},X)))$, the following estimate holds
\begin{equation*}
\textstyle \E\sup_{0\leq  t\leq T}\Big\|\int_0^t G(s)\,\dd W(s)\Big\|_{X}^p \leq C_{p,X}^p \E\|G\|_{L^2(0,T;\g(\mathcal{U},X))}^p.
\end{equation*}
\end{proposition}
The above result is also valid for the larger class of martingale type $2$ spaces (see the survey \cite[Theorem 4.7]{NVW13} and \cite{Ondrejat04}). In \cite{Seid10} it was shown that the constant $C_{p,X}$ grows as $C \sqrt{p}$ for $p\to \infty$. An alternative proof based on exponential tail estimates can be found in \cite[Theorem 3.1]{cox2024sharp}.

It can be shown that if $\mathcal{U}$ is separable and $(u_n)_{n\geq 1}$ is an orthonormal basis for $\mathcal{U}$, then for all $p\in [0,\infty)$ and $G\in L^p_{\Progress}(\Omega;L^2(0,T;\g(\mathcal{U},X)))$, the following series representation holds with convergence in $L^p(\Omega;C([0,T];X))$:
\[\textstyle \int_0^\cdot G(s)\,\dd W(s) = \sum_{n\geq 1} \int_0^\cdot G(s) u_n \,\dd W(s) u_n.\]
For further details, the reader is referred to the survey \cite{NVW13}.
Many standard properties of the stochastic integral can be found there as well. For It\^o's formula, the reader is directed to \cite{BNVW08}, which holds even in the UMD setting. Further extensions of the above stochastic integration theory and alternative proofs can be found in \cite{Yaros20}.

It is important to note that some standard results that hold when $X$ is a Hilbert space do not extend to the setting of Proposition \ref{prop:Ito}. In particular, the martingale representation theorem requires the precise characterizations provided in \cite{NVW1}.

\section{Stochastic maximal regularity}\label{sec:SMR}
Regularity estimates for linear equations play a crucial role in the theory of nonlinear (S)PDE. In fact, the local well-posedness of nonlinear SPDEs can often be obtained through linearization procedures.

However, before going into stochastic maximal
$L^p$-regularity, we will first discuss some basic
$L^p$-theory for elliptic and parabolic partial differential equations (PDEs). This will serve to illustrate some of the techniques and difficulties that one can expect in the study of maximal regularity. Additionally, in Subsection \ref{subsub:MR}, we will focus on the concept of maximal
$L^p$-regularity in the deterministic setting.

\subsection{The deterministic case}
We start our discussion with a concrete example of $L^p$-theory in the context of elliptic equations. Similar results can be obtained in the scale of H\"older spaces and this leads to Schauder theory. For details on both settings, the reader is referred to the monographs \cite{GT83, KryHolder, Krylov2008_book}.

\subsubsection{The Laplace equation}
On $\R^d$, consider the following elliptic equation
\begin{equation}\label{eq:Laplace}
u - \Delta u = f,
\end{equation}
where $f\in L^p(\R^d)$ is a given function and $u$ is the unknown. The goal is to show $u\in W^{2,p}(\R^d)$ and estimate its norm. It is clear that if $u\in W^{2,p}(\R^d)$, then the left-hand side $u- \Delta u \in L^p(\R^d)$. However, the inverse problem - solving \eqref{eq:Laplace} given $f\in L^p(\R^d)$ - is much more delicate when $d\geq 2$. Indeed, it is unclear how to derive information on the mixed derivatives such as $\partial_1\partial_2 u$. The unique solvability is equivalent to having the estimate
\begin{equation}\label{eq:ellipticest}
\|u\|_{W^{2,p}(\R^d)} \leq C \|f\|_{L^p(\R^d)}
\end{equation}
whenever $u$ is a solution to \eqref{eq:Laplace}. Furthermore, the closedness of the operator $u\mapsto \Delta u$ on $L^p(\R^d)$ with domain $D(\Delta) = W^{2,p}(\R^d)$ is also equivalent to \eqref{eq:ellipticest}.

For $p=2$, the estimate \eqref{eq:ellipticest} can be proved through the isometric property of the Fourier transform. For $p=1$ and $p=\infty$, the estimate \eqref{eq:ellipticest} fails when $d\geq 2$, as the second order Riesz transform is unbounded on $L^1$ and $L^\infty$ (see \cite[p.\ 98]{Krylov2008_book} for a simple example in $d=2$). For $p\in (1, \infty)$, the estimate \eqref{eq:ellipticest} can be proved through Calder\'on--Zygmund theory for singular integral operators or Fourier multiplier theory. In fact, $L^p$-theory for \eqref{eq:Laplace} motivated the development of these theories.

More recent extensions of the above results involve variants of the equation \eqref{eq:Laplace} where $\Delta$ is replaced by a second-order operator with rough coefficients in either divergence or nondivergence form.

\subsubsection{Heat equation}
As a next step, we consider the zero-initial value problem for the heat equation on $\R^d$ over a time interval $(0,T)$, where $T\in (0,\infty)$ is fixed:
\begin{equation}
\label{eq:heat}
\partial_t u - \Delta u = f \ \  \text{with initial condition} \ \ u(0,\cdot)=0.
\end{equation}
In an $L^p$-setting, this problem allows us to investigate various estimates similar to the elliptic case we discussed earlier. However, in this parabolic setting, time regularity becomes equally important. Below we choose the spaces that are well-suited for the sequel.

Initially, we focus on the case where $f\in L^p((0,T)\times \R^d)$,
but it will later prove useful to have more flexibility, considering different integrability conditions in both time and space for the inhomogeneity $f$, specifically $f\in L^p(0,T;L^q(\R^d))$.

It is clear that $\partial_t - \Delta$ maps from $\prescript{}{0}{W}^{1,p}(0,T;L^q(\R^d))\cap L^p(0,T;W^{2,q}(\R^d))$ into $L^{p}(0,T;L^q(\R^d))$. The inverse problem - solving for $u$ in this space and proving the estimate
\begin{equation}\label{eq:parabest}
\|u\|_{W^{1,p}(0,T;L^q(\R^d))}+ \|u\|_{L^p(0,T;W^{2,q}(\R^d))} \leq C \|f\|_{L^p(0,T;L^q(\R^d))}
\end{equation}
is much harder. A possible approach is to consider the operator-theoretic viewpoint. Let $X_0 = L^q(\R^d)$ and $X_1= \Do(\Delta) = W^{2,q}(\R^d)$ and use the variation of constants formula for the solution $u:[0,T]\to X$ to \eqref{eq:heat} as
\[\textstyle  u(t) = \int_0^t e^{(t-s)\Delta} f(s)\,\dd s.\]
Since $\sup_{s\in [0,T]}\|e^{s\Delta}\|_{\calL(X_0)}\leq 1<\infty$, we can estimate
$\|u\|_{L^p(0,T;X_0)} \leq C \|f\|_{L^p(0,T;X_0)}$.
If $f\in C^1([0,T];X_1)$, then one can check that \eqref{eq:heat} holds in an a.e.\ sense. Therefore, to prove \eqref{eq:parabest}, by \eqref{eq:ellipticest} it suffices to show
\begin{equation}\label{eq:parabest2}
\|\Delta u\|_{L^p(0,T;X_0)} \leq C \|f\|_{L^p(0,T;X_0)}.
\end{equation}
Indeed, the missing estimate for $\partial_t u$ follows from $\partial_t u = \Delta u + f$ and the triangle inequality. Moreover, the result for general $f\in L^p(0,T;X_0)$ can be obtained by a density argument.
The estimate \eqref{eq:parabest2} can be interpreted as a singular integral operator with an operator-valued kernel, where $\|\Delta e^{t\Delta}\|_{\calL(X_0)}\leq Ct^{-1}$.
To prove the estimate \eqref{eq:parabest2}, one can employ vector-valued (i.e.\ with values in a Banach space) Fourier multiplier theory. For the latter, the reader is referred to \cite{ABK}, \cite[Section 5.6c]{Analysis1} and \cite[Section 13.3]{Analysis3} and here the UMD condition of $L^q$ plays an important role again.

\subsubsection{Abstract setting}\label{subsub:MR}
The framework used for the heat equation can be generalized to any linear operator $A$ acting on a Banach space $X_0$ with domain $\Do(A)\subseteq X_0$. Consider
\begin{equation}
\label{eq:heatabstract}
u' + A u = f \ \ \text{with initial condition} \ \ u(0,\cdot)=0.
\end{equation}
The operator $A$ is said to have {\em maximal $L^p$-regularity on $(0,T)$} if for every $f\in L^p(0,T;X_0)$, there exists a unique $u\in L^p(0,T;\Do(A))$ such that for all $s\in [0,T]$,
\[\textstyle  u(t) + \int_0^t A u(s) \, \dd s = \int_0^t f(s) \, \dd s,\]
and there is a constant $C$ independent of $f$ such that
\begin{equation}\label{eq:LpMRest}
\|u\|_{L^p(0,T;\Do(A))}\leq C\|f\|_{L^p(0,T;X_0)}.
\end{equation}
A function $u$ that satisfies the above-integrated form of the equation \eqref{eq:heatabstract} is called a {\em strong solution}. From this, we immediately obtain an estimate for $u'$, namely $u' = f - A u$.

Next, we formulate several permanence properties. For a more detailed discussion and further results, the reader is referred to \cite{Dore} (see also \cite[Section 17.2]{Analysis3}).
\begin{theorem}\label{thm:Dore}
Let $p\in [1, \infty)$. Let $A$ be a linear operator with domain $X_1 = \Do(A)\subseteq X$. Suppose that $A$ has maximal $L^p$-regularity on $(0,T)$. Then the following assertions hold:
\begin{enumerate}[{\rm(1)}]
\item $A$ is closed, i.e.\ $\Do(A)$ is a Banach space with its graph norm;
\item the mapping $f\mapsto u$ defines an isomorphism $\mathcal{M}:L^p(0,T;X_0)\to \prescript{}{0}{W}^{1,p}(0,T;X_0)\cap L^p(0,T;X_1)$;
\item for every $\lambda\in \C$, $\lambda+A$ has maximal $L^p$-regularity on $(0,T)$;
\item $-A$ generates an analytic semigroup;
\item $A$ has maximal $L^p$-regularity on $(0,T')$ for any $T'\in (0,\infty)$;
\item $A$ has maximal $L^q$-regularity on $(0,T)$ for all $q\in (1, \infty)$.
\end{enumerate}
\end{theorem}
Due to the independence of the time interval, we can and will simply write $A$ has maximal $L^p$-regularity.

\begin{remark}
Maximal $L^p$-regularity can also be extended to the case $T = \infty$. However, since we do not require this extension and since there are several options, the interested reader is referred to \cite[Section 17.2]{Analysis3} for details on this case.

It is important to note that the endpoint cases $p=1$ and $p=\infty$ present complications. Specifically, unbounded operators $A$ fail to have maximal $L^1$- or $L^\infty$-regularity on reflexive spaces $X_0$ as shown in a result by Baillon (see \cite[Section 17.4]{Analysis3}).
\end{remark}

If $-A$ generates an analytic semigroup, a strong solution of \eqref{eq:heatabstract} can be expressed as a {\em mild solution}, i.e.\ it can be written in the form
\[\textstyle u(t) = \int_0^t e^{-(t-s)A} f(s) \,\dd s.\]
Conversely, a mild solution for which $A u\in L^1(0,T;X_0)$ can be shown to be a strong solution. These results are discussed in \cite[Proposition 17.1.3]{Analysis3}.

Thus, when $\Do(A)$ is dense in $X_0$, proving maximal $L^p$-regularity for $p\in [1, \infty)$ reduces to showing the boundedness of the integral operator
\[\textstyle  A u(t) = J_Af(t) := \int_0^t A e^{-(t-s)A} f(s)\,\dd s,\]
for $f\in L^p(0,T;\Do(A))$, as a mapping from $L^p(0,T;X_0)$ into itself. The kernel $A e^{-tA}\in \calL(X_0)$ has a singularity of order $t^{-1}$ as $t\to 0$. As mentioned below \eqref{eq:parabest2}, one can use operator-theoretic tools and harmonic analysis to study the boundedness of the above integral operator.

The following result is due to \cite{DeSimon} (see also \cite[Corollary 17.3.8]{Analysis3}) and implies that analytic semigroup generators on  Hilbert spaces always lead to maximal $L^p$-regularity.
\begin{theorem}\label{thm:deSimon}
Let $X_0$ be a Hilbert space and suppose that $-A$ generates a strongly continuous analytic semigroup on $X_0$ and let $X_1= \Do(A)$. Then $A$ has maximal $L^p$-regularity for all $p\in (1, \infty)$.
\end{theorem}
Let us sketch the proof. One can equivalently consider $\lambda+A$ for sufficiently large $\lambda\geq 0$, to reduce to the exponentially stable case (see Theorem \ref{thm:Dore}). After that one can prove the result for $p=2$ by using the $L^2$-isometric property of the Fourier transform. Finally, the case $p\in (1, \infty)$ follows from Theorem \ref{thm:Dore} again.

As we will see later on, weights in time play an important role in the theory.
For the definition of maximal $L^p_{\kappa}$-regularity, one has to replace $L^p(0,T;X_0)$ in \eqref{eq:LpMRest} by $L^p(0,T, t^\a\,\dd t;X_0)$ on both sides.

\begin{theorem}[Weighted maximal $L^p$-regularity]\label{thm:MRweight}
Let $X_0$ be a Banach space and suppose that $-A$ generates a strongly continuous analytic semigroup. Let $p\in (1, \infty)$ and $\kappa\in (-1,p-1)$. Then $A$ has maximal $L^p$-regularity if and only if $A$ has maximal $L^p_{\kappa}$-regularity.
\end{theorem}

The above result is due to \cite{PruSim04}.
An alternative proof can be found in \cite[Section 17.2.e]{Analysis1}, where the cases $p=1$ and $p=\infty$ are also discussed. Moreover, it is shown there that maximal $L^p$-regularity can be obtained with arbitrary $A_p$-weights (i.e.\ $L^p(0,T;X_0)$ in \eqref{eq:LpMRest} is replaced by $L^p(0,T, w;X_0)$ with $w\in A_p$ on both sides). For further results in this direction, the reader is referred to \cite{ChillFio}.

For a long time, it was an open question whether being the generator of an analytic semigroup on, say, an $L^q$-space with $q\in (1, \infty)\setminus\{2\}$, is sufficient to guarantee maximal $L^p$-regularity for $p\in (1, \infty)$, a question often referred to as Brezis' question.  Some sufficient conditions for maximal $L^p$-regularity were found in \cite{DPGsum, DoreVenni, Lamberton, PrSo}. However, in general, the question was shown to have a negative answer (see \cite{KL00, KL02}),  although all known counterexamples are highly theoretical and not closely related to differential operators. Further extensions of these counterexamples were later found in \cite{Fack14} (see also \cite[Section 17.4.c]{Analysis3}). Around the same time that the first counterexamples appeared, a characterization of maximal $L^p$-regularity was established in \cite{Weis-maxreg,We} (see also \cite[Sections 17.3a]{Analysis3}) in terms of $R$-sectoriality.

After this discussion, we can return to the heat equation \eqref{eq:heat}. Using the above it can be shown that $-\Delta$ has maximal $L^p$-regularity on $L^q(\R^d)$ if and only if $p,q\in (1, \infty)$ (see \cite[Section 17.4bc]{Analysis3}).

\begin{remark}
In the introduction, we mentioned the work by Da Prato and Grisvard \cite{DPGsum}, which provides a maximal $L^p$-regularity result when $X_0$ is replaced by a real interpolation space $(X, \Do(A))_{\theta,p}$, with no conditions on $A$ other than it being the generator of an analytic semigroup. Their method is a ``sum of operators approach''. The result in \cite{Weis-maxreg,We} can also be proved by using a ``sum of operators'' approach, which was done in \cite{KWcalc} (see also \cite[Sections 17.3c]{Analysis3}).
\end{remark}

\subsection{Definition and basic properties in the stochastic setting}
\label{ss:stoch_max_reg}

In this subsection, we consider the following linear variant of the stochastic evolution equation \eqref{eq:SEEintro} on $[a,\tau]\subseteq [0,T]$
\begin{equation}
\label{eq:SEElin}
\left\{
\begin{aligned}
&\dd u + A u \,\dd t = f\,\dd t + (B u +g) \,\dd W,\\
&u(a)=u_a,
\end{aligned}
\right.
\end{equation}
and we define the notion of stochastic maximal regularity for the pair $(A,B)$. Here, $a\in [0,\infty)$ and $\tau$ is a stopping time. As in the deterministic case, it is primarily expressed through an a priori estimate, which plays a crucial role in establishing both local and global well-posedness results. The definition of stochastic maximal regularity will be more technical compared to the previous cases, but it is precisely what can be proved in many applications.

\smallskip

\noindent
\underline{Standing assumption for stochastic maximal regularity.}
Unless stated otherwise, here and below $X_0$ and $X_1$ are Banach spaces with UMD and type $2$, $X_1\hookrightarrow X_0$ and $A\in \calL(X_1, X_0)$ and $B\in \calL(X_1,\gamma(\mathcal{U},X_{1/2}))$. The equation \eqref{eq:SEElin} is to be understood in the It\^o sense.

\smallskip

Similar to deterministic maximal regularity, we are given the inhomogeneities $f$ and $g$, and aim to find necessary and sufficient conditions for the well-posedness of a strong solution $u$.
The concept of stochastic maximal regularity is a property of the pair $(A,B)$, as both are leading order operators, due to the irregularity of $W$. In applications, $A$ might represent a second-order elliptic operator, while $B$ could be of order one. A joint parabolicity condition is typically required to ensure well-posedness.

For instance, consider the following prototype example: for $u\in X_1=W^{2,q}(\R^d)$ (with $q\in [2, \infty)$),
\begin{equation}
\label{eq:AB_basic_example}
A u =- \Delta u , \qquad \text{ and }\qquad B u = ((b_n\cdot\nabla) u)_{n\geq 1},
\end{equation}
where $b\in\ell^2:=\ell^2(\N_{\geq 1})$ and $\mathcal{U}=\ell^2$. As we noted in Subsection \ref{ss:scaling_intro}, the transport term $(b_n\cdot\nabla) u$ has the same scaling as the Laplace operator in the deterministic part due to the scaling of the Brownian motion. To ensure stochastic maximal $L^2$-regularity with $X_0=L^2(\R^d)$ an additional assumption is required: there exists $\nu\in (0,2)$ such that for all $\xi\in \R^d$,
\begin{equation}
\label{eq:stochastic_parabolicity_example}
\textstyle  \sum_{n\geq 1} (b_n\cdot \xi)^2\leq \nu|\xi|^2.
\end{equation}
This condition, often referred to as  \emph{stochastic parabolicity}, arises from the energy estimate (see Theorem \ref{thm:varlin} below for an abstract version). It is also sufficient for obtaining $L^p(L^q)$-estimates in the case $p>2$ and $q\in [2, \infty)$, and in the presence of temporal weights. For more details, the reader is referred to Subsection \ref{ss:sufficient_B_non_zero}.

There are also many important examples where $B=0$ or $B$ is relatively small. While this case is simpler, it remains significant and nontrivial.
In case the reader is willing to use the definition of stochastic maximal $L^p$-regularity below as a black box, they could already go to Section \ref{sec:loc-well-posed}, where we start with the local existence and uniqueness theory for \eqref{eq:SEEintro}. In particular, if $B=0$ or $B$ is relatively small, it is possible to use the well-established $H^\infty$-calculus as a black box. Indeed, by Theorem \ref{thm:SMRHinfty} for every $A$ on a suitable space, having a bounded $H^\infty$-calculus implies that $(A,0)$ has stochastic maximal $L^p$-regularity in the sense of Definitions \ref{def:SMR} and \ref{def:SMRbullet} below.

\subsubsection{Strong solutions}

Before continuing, let us first define our solution concept.
\begin{definition}[Strong solutions]\label{def:strong}
Let $f\in L^1(a,\tau;X_0)$ a.s.\ and $g\in L^2(a,\tau;\gamma(\mathcal{U},X_{1/2}))$ a.s.\ both be progressively measurable.
A progressively measurable process $u:[a,\tau]\times \Omega\to X_0$ is called a {\em strong solution} to \eqref{eq:SEElin} on $[a,\tau]$ if $u\in C([a,\tau];X_0)\cap L^2(a,\tau;X_1)$ a.s., and a.s. for all $t\in [a,\tau]$,
\[\textstyle  u(t) - u_a + \int_a^t A u(s) \,\dd s = \int_a^t f(s)\,\dd s + \int_a^t (B u(s) +g(s)) \,\dd W(s).\]
\end{definition}
In principle, we could also allow $B u$ and $g$ to take values in $\gamma(U,X_0)$, but later on, we will only need the above setting. The stochastic integral exists in $X_0$ (and also in $X_{1/2}$) in the sense of Proposition \ref{prop:Ito}.

\subsubsection{Stochastic maximal $L^p$-regularity}
The definition of stochastic maximal $L^p$-regularity will be given in a weighted setting, as this framework leads to the most robust results for applications. Recall that for $a,\kappa\in \R$, $w^a_\kappa(t) = (t-a)^{\kappa}$ and $w_{\kappa}(t)= t^\kappa$.
In the context of weighted stochastic maximal $L^p$-regularity, the different scaling behaviour of Brownian motion (as discussed in Subsection \ref{ss:scaling_intro}), imposes additional restrictions compared to the deterministic setting as seen in Theorem \ref{thm:MRweight}. In particular, we again lose $1/2$-scaling, and therefore we consider only power weights in the $A_{p/2}$ class. For $p\in [2,\infty)$, this means we focus on $\a\in [0,\frac{p}{2}-1)\cup\{0\}$, where $\{0\}$ is added for the case $p=2$.
By H\"older's inequality, we have the embedding $L^p(a,b, w^{a}_{\kappa};Y)\hookrightarrow L^2(a,b;Y)$ for all $p\in [2, \infty)$ and $\kappa\in [0,p/2-1)\cup\{0\}$, which ensures that stochastic integrals with $g\in L^p(0,T,w_{\a};\g(\mathcal{U},Y))$ are well-defined, see Proposition \ref{prop:Ito}.

As will be discussed in Subsection \ref{ss:nonzeroIC}, it will be enough to consider zero initial data in \eqref{eq:SEElin}, i.e.\
\begin{equation}
\label{eq:SEElin0}
\left\{
\begin{aligned}
&\dd u + A u \,\dd t = f\,\dd t + (B u +g) \,\dd W,\\
&u(a)=0.
\end{aligned}
\right.
\end{equation}

After these preparations, we are ready to give the two central definitions.
\begin{definition}\label{def:SMR}
Let $p\in [2, \infty)$ and $\kappa\in [0,p/2-1)\cup \{0\}$.
We say $(A,B)\in \mathcal{SMR}_{p,\kappa}$ if for all $T\in (0,\infty)$ there exists a constant $C_T$ such that for all $a\in [0,T]$ and every stopping time $\tau:\Omega\to [a,T]$, every progressively measurable $f\in L^p(\O;L^p(a,\tau,w_{\kappa}^a;X_0))$ and $g\in  L^p(\O;L^p(a,\tau,w_{\kappa}^a;\g(\mathcal{U},X_{1/2})))$, there exists a unique strong solution $u$ to \eqref{eq:SEElin} on $[a,\tau]$ such that $u\in L^p(\Omega;L^p(a,\tau,w_{\a}^a;X_1))$, and moreover the following estimate holds
\begin{align}\label{eq:SMRest}
\|u\|_{L^p(\Omega;L^p(a,\tau,w_{\kappa}^a;X_1))} \leq C_T \|f\|_{L^p(\O;L^p(a,\tau,w_{\kappa}^a;X_0))} + C_T \|g\|_{L^p(\O;L^p(a,\tau,w_{\kappa}^a;\g(\mathcal{U},X_{1/2})))}.
\end{align}
\end{definition}

In the deterministic situation, estimates for $u'$ can be obtained directly from the equation. In the stochastic case, this seems not possible in general, and the time regularity needs to be studied separately. Time regularity estimates play an important role in nonlinear equations. For technical reasons, we distinguish between the cases $p>2$ and $p=2$.
\begin{definition}\label{def:SMRbullet}
\
\begin{itemize}
\item For $p\in (2, \infty)$ and $\kappa\in [0,p/2-1)$, we say $(A,B)\in \mathcal{SMR}_{p,\kappa}^{\bullet}$ if $(A,B)\in \mathcal{SMR}_{p,\kappa}$ and for all $\theta\in (0,1/2)$ for all $T\in (0,\infty)$ there exists a constant $C_{T,\theta}$ such that for all $[a,b]\subseteq [0,T]$, for every progressively measurable $f\in L^p(\O;L^p(a,b,w_{\kappa}^a;X_0))$ and $g\in  L^p(\O;L^p(a,b,w_{\kappa}^a;\g(\mathcal{U},X_{1/2})))$, the strong solution $u$ to \eqref{eq:SEElin0} on $[a,b]$ satisfies
\begin{align*}
\|u\|_{L^p(\Omega;H^{\theta,p}(a,b,w_{\kappa}^a;X_{1-\theta}))} \leq C_{T,\theta} \|f\|_{L^p(\O;L^p(a,b,w_{\kappa}^a;X_0))} + C_{T,\theta} \|g\|_{L^p(\O;L^p(a,b,w_{\kappa}^a;\g(\mathcal{U},X_{1/2})))}.
\end{align*}

\item We say $(A,B)\in \mathcal{SMR}_{2,0}^{\bullet}$ if $(A,B)\in \mathcal{SMR}_{2,0}$ and
for all $T\in (0,\infty)$ there exists a constant $C_{T}$ such that for all $[a,b]\subseteq [0,T]$, for every progressively measurable $f\in L^2(\O;L^2(a,b;X_0))$ and $g\in  L^2(\O;L^2(a,b;\g(\mathcal{U},X_{1/2})))$, the strong solution $u$ to \eqref{eq:SEElin0} on $[a,b]$ satisfies
\begin{align*}
\|u\|_{L^2(\Omega;C([a,b];X_{1/2}))} \leq C_{T} \|f\|_{L^2(\O;L^2(a,b;X_0))} + C_T \|g\|_{L^2(\O;L^2(a,b;\g(\mathcal{U},X_{1/2})))}.
\end{align*}
\end{itemize}
In case $(A,B)\in \mathcal{SMR}_{p,\kappa}^{\bullet}$ we will say that $(A,B)$ has stochastic maximal $L^p_{\kappa}$-regularity.
\end{definition}

\begin{remark}
Definitions \ref{def:SMR} and \ref{def:SMRbullet} are slightly more restrictive than the ones in \cite{AV19_QSEE_1}, as we require that the constant $C_T$ be chosen uniformly over all $a$ and $b$. However, this condition does not present any issues in our applications. An alternative approach would be to fix $a=0$ and quantify over all possible cylindrical Brownian motions on the given probability space. To see this, it suffices to translate the problem back to one that starts at zero.
\end{remark}

In the following remark, we explain how to reduce to $f=0$.
\begin{remark}\label{rem:MRf=0}
Clearly $(A,0)\in \mathcal{SMR}_{p,\kappa}$ implies that $A$ has maximal $L^p$-regularity. Indeed, this follows by taking $g=0$ and applying Theorems \ref{thm:Dore} and \ref{thm:MRweight}.

Conversely, if $A$ is known to have maximal $L^p$-regularity, then to check stochastic maximal $L^p_{\kappa}$-regularity for $(A,B)$ it suffices to consider the case where $f=0$. Uniqueness is clear. For existence, recall that maximal $L^p$-regularity implies the weighted variant maximal $L^p_{\kappa}$-regularity (see Theorem \ref{thm:MRweight}). Now, let $v\in L^p(a,b,w_{\kappa}^a;X_1)$ be the strong solution of $v'+Av = f$ with $v(a) = 0$. Let $z\in L^p(\Omega;L^p(a,b,w_{\a}^a;X_1))$ be the strong solution of $\dd z+Az\,\dd t = (Bz + Bv+ g)\,\dd W$ with $z(a)=0$, where $Bv + g$ serves as the stochastic inhomogeneity. Then it can be shown that $u = v+z$ is the desired strong solution to \eqref{eq:SEElin0}. To obtain the required estimate for $u$, it suffices to consider the estimate for $v$. By translation and restriction, it suffices to consider the problem on $(0,c)$ with $c=b-a\leq T$. By maximal $L^p_{\kappa}$-regularity, we obtain the estimate
\begin{align*}
\|v\|_{W^{1,p}(0,c,w_{\kappa};X_0)} + \|v\|_{L^p(0,c,w_{\kappa};X_1)}\leq C_T \|f\|_{L^p(0,c,w_{\kappa};X_0)},
\end{align*}
where $C_T$ does not depend on $c$. Therefore, for $p\in (1, \infty)$ following the argument in \cite[Proposition 2.8]{AV19_QSEE_1}, we obtain that for all $\theta\in (0,1)$,
\begin{equation}
\label{eq:interptrickHthetaMR}
\|v\|_{H^{\theta,p}(0,c,w_{\kappa};X_{1-\theta})}  \leq \|v\|_{\hz^{\theta,p}(0,c,w_{\kappa};X_{1-\theta})}  \leq C_E \|v\|_{W^{1,p}(0,c,w_{\kappa};X_0)} + C_E \|v\|_{L^p(0,c,w_{\kappa};X_1)},
\end{equation}
where the constant $C_E$ comes from the extension operator of \cite[Proposition 2.5]{AV19_QSEE_1} and is also independent of $c$. This completes the proof for $p>2$. For $p=2$, it remains to apply  Proposition \ref{prop:tracespace} with $\theta>1/2$, where the constant can again be chosen independent of $c$ by using the spaces $\hz^{\theta,p}(0,c,w_{\kappa};X_{1-\theta})$ instead (see \cite[Proposition 2.10]{AV19_QSEE_1}).
\end{remark}

\subsection{Nonzero initial values and inhomogeneities with moments}\label{ss:nonzeroIC}

Assuming stochastic maximal $L^p$-regularity, one can derive regularity for equations with nonzero initial data, as in \eqref{eq:SEElin}, even when $f$ and $g$ are not integrable over $\Omega$.
This will be frequently applied to establish the regularity of solutions to nonlinear equations.

\begin{proposition}[$\O$-Localization of inhomogeneities and initial data]\label{prop:localizationSMR}
Suppose $(A,B)\in \mathcal{SMR}_{p,\kappa}^{\bullet}$, where $-A$ is the generator of a strongly continuous analytic semigroup. Suppose that $u_a\in L^0_{\F_a}(\Omega;X_{1-\frac{1+\kappa}{p},p})$. Let $\tau$ be a stopping time with values in $[a,T]$. Let $f\in L^p(a,\tau,w_{\kappa}^a;X_0)$ a.s.\ and $g\in L^p(a,\tau,w_{\kappa}^a;\gamma(\mathcal{U},X_{1/2}))$ a.s.\ be strongly progressively measurable. Then there exists a unique strong solution $u$ to \eqref{eq:SEElin} on $[a,\tau]$ such that $u\in L^p(a,\tau,w_{\kappa}^a;X_1)$ a.s. Moreover, if $p>2$, then the following additional regularity holds a.s.\ for all $\theta\in [0,1/2)$:
\begin{align*}
u& \in H^{\theta,p}(a,\tau,w_{\kappa}^a;X_{1-\theta})\subseteq H^{\theta,p}_{\rm loc}((a,\tau];X_{1-\theta}),
\\ u& \in C([a,\tau];X_{1-\frac{1+\kappa}{p},p}), \ \ \text{and} \ \ u\in C((a,\tau];X_{1-\frac{1}{p},p}).
\end{align*}
If $p=2$, then one has $u \in C([a,\tau];X_{1/2})$ a.s.
\end{proposition}
\begin{proof}
The uniqueness assertion follows from Definition \ref{def:SMR} and the linearity of the problem. It remains to prove the existence and establish the stated regularity. We first consider $p>2$. For each $n\geq 1$, define a stopping time by
\[\tau_n = \inf\{t\in [a,\tau]: \|f\|_{L^p(a,t,w_{\kappa}^a;X_0)} + \|g\|_{L^p(a,t,w_{\kappa}^a;\gamma(\mathcal{U},X_{1/2}))}\geq n\},\]
where we set $\inf\emptyset =\tau$.
Let $f_n:=\one_{[a,\tau_n]} f\in L^p(\O;L^p((a,T),w_{\kappa}^a;X_0))$ and $g_n:=g\one_{[a,\tau_n]}\in L^p(\O;L^p((a,T),w_{\kappa}^a;\gamma(\mathcal{U},X_{1/2})))$.
By the assumption $(A,B)\in \mathcal{SMR}_{p,\kappa}^{\bullet}$, there exists a unique strong solution $u_n$ to \eqref{eq:SEElin0} with $(f,g)$ replaced by  $(f_n, g_n)$, such that $u_n\in L^p(\Omega;L^p(a,T,w_{\kappa}^a;X_1))$. Moreover, for all $\theta\in [0,1/2)$,
\begin{align*}
\|u_n\|_{L^p(\Omega;H^{\theta,p}(a,T,w_{\kappa}^a;X_{1-\theta}))} \leq C_{T,\theta} \|f_n\|_{L^p(\O;L^p((a,T),w_{\kappa}^a;X_0))} + C_T \|g_n\|_{L^p(\O;L^p(a,T,w_{\kappa}^a;\g(\mathcal{U},X_{1/2})))}.
\end{align*}
Since for $n\geq m$, $u_n$ also is a strong solution to \eqref{eq:SEElin0} on $[a,\tau_m]$ with $(f,g)$ replaced by  $(f_m, g_m)$, by uniqueness it follows that $u_n = u_m$ on $[a,\tau_m]$. Since for a.e. $\omega\in \O$, there exists an $n\geq 1$ such that $\tau_n(\omega)=\tau(\omega)$, it follows that $u := \lim_{n\to \infty} u_n$ a.s.\ exists, and a.s.\ for all $\theta\in [0,1/2)$, one has $u\in H^{\theta,p}(a,\tau,w_{\kappa}^a;X_{1-\theta})$. The remaining regularity assertions follow from Proposition \ref{prop:tracespace}.

To add the initial value, let $v(t) = e^{-(t-a)A} u_a$. Then on the interval $[a,T]$, $v$ is a strong solution to $v' + A v = 0$ with $v(a) = u_a$. Moreover, from \eqref{eq:equivorbitreal} it follows that
\begin{align*}
\|v\|_{L^p(a,T,w_{\kappa}^a;X_1)} \leq C_T \|u_a\|_{X_{1-\frac{1+\kappa}{p},p}}.
\end{align*}
The same estimate holds for $v'$ by $v' = -Av$. Therefore, from \eqref{eq:interptrickHthetaMR}, we conclude that $v\in H^{\theta,p}(a,T,w_{\kappa}^a;X_{1-\theta})$. By linearity, it follows that $u+v$ is a strong solution to \eqref{eq:SEElin} and it has the desired regularity.

The proof for $p=2$ follows analogously, with the only modification being the use of the fact that $e^{-tA}$ is strongly continuous on $X_{1/2}$.
\end{proof}

The initial value was added directly using the analytic semigroup generated by $A$. A different argument can be used when $A$ depends on $(t,\omega)$ as discussed in \cite[Propositions 3.10 and 3.12]{AV19_QSEE_1}.

\subsection{The case $B=0$}

In this subsection, we discuss several results in the case $B = 0$. By considering $g=0$, one sees that $(A,0)\in \mathcal{SMR}_{p,\kappa}$ implies that $A$ has maximal $L^p$-regularity. As mentioned in Remark \ref{rem:MRf=0}, if maximal $L^p$-regularity holds for the deterministic setting, we can reduce the problem to $f=0$. If, in addition, $-A$ generates a strongly continuous semigroup, then it is standard (see \cite[Section 7.5]{veraar2006stochastic}) to express the strong solution on $[a,b]$ as
\[\textstyle  u(t) = \int_a^t e^{-(t-s)A} g(s) \,\dd W(s), \ \ t\in [a,b],\]
which in turn establishes uniqueness. Thus, in order to prove $(A,0)\in \mathcal{SMR}_{p,\kappa}$, by translation (after applying the vector-valued Burkholder-Davis-Gundy inequality \cite[(5.5)]{NVW13}) and considering $\one_{(a,\tau)}g$, it suffices to prove the corresponding estimate on $(0,T)$. By density, it is enough to show that
there exists a constant $C_T$ such that for all $g\in L^p(\O;L^p(0,T,w_{\kappa};\g(\mathcal{U},X_{1})))$
\begin{align}\label{eq:SMRsimple}
\textstyle  \Big\|t\mapsto \int_0^t e^{-(t-s)A} g(s) \,\dd W(s)\Big\|_{L^p(\O;L^p(0,T,w_{\kappa};X_1))} \leq C \|g\|_{L^p(\O;L^p(0,T,w_{\kappa};\g(\mathcal{U},X_{1/2})))}.
\end{align}

The estimate \eqref{eq:SMRsimple} can be seen as a stochastic integral operator of convolution type, with kernel $e^{-tA}\in \calL(X_1, X_{1/2})$, which has a singularity of order $t^{-1/2}$ for $t\to 0$. The following extrapolation theorem can be proved in the same way as \cite[Theorem 8.2]{LoVer} by Lorist and the second named author, where a stochastic (and $A_{p/2}$-weighted) version of Calder\'on--Zygmund theory with operator-valued kernels was developed.
\begin{theorem}[Extrapolation]\label{thm:extrapol}
Let $X_0$ be a UMD space with type $2$. Let $p\in [2, \infty)$. Let $-A$ be the generator of a strongly continuous analytic semigroup with $\Do(A) = X_1$. If \eqref{eq:SMRsimple} holds for some $p\in [2, \infty)$ and $\kappa=0$, then \eqref{eq:SMRsimple} holds for all $p\in (2, \infty)$ and $\kappa\in [0,p/2-1)$.
\end{theorem}
The special case of power weights with fixed $p$ was also considered in \cite[Section 7]{AV19} and can be seen as a stochastic analogue of Theorem \ref{thm:MRweight}.

\subsubsection{The Hilbert space case}
The following result can be viewed as a stochastic version of Theorem \ref{thm:deSimon} on analytic generators on Hilbert spaces. In particular, it extends the work of Da Prato \cite{DPZ82} (see also \cite[Theorem 6.12(2)]{DPZ}) where it is shown that $(A,0)\in \mathcal{SMR}_{2,0}$.
\begin{theorem}\label{thm:SMRHS}
Suppose that $X_0$ is a Hilbert space, $X_1 = \Do(A)$, and $-A$ generates a strongly continuous analytic semigroup on $X_0$. Then for all $p\in [2, \infty)$ and $\kappa\in [0,p/2-1)\cup\{0\}$, one has $(A,0)\in \mathcal{SMR}_{p,\kappa}^{\bullet}$.
\end{theorem}
Before we turn to the general setting, we first present Da Prato's proof of $(A,0)\in \mathcal{SMR}_{2,0}$ in the case $e^{-tA}$ is exponentially stable, in which case the result even holds for $T=\infty$. By Theorem \ref{thm:deSimon} and Remark \ref{rem:MRf=0}, it suffices to consider $f=0$. Let $(u_n)_{n\geq 1}$ be an orthonormal basis for $\mathcal{U}$ (which can be assumed to be separable). Then by the It\^o isometry,
\begin{align*}
\textstyle  \Big\|t\mapsto \int_0^t A e^{-(t-s)A} g(s) \,\dd W(s)\Big\|_{L^2(\O;L^2(\R_+;X_0))}^2 &= \textstyle \E \sum_{n\geq 1}\int_{\R_+} \int_0^t \|A e^{-(t-s)A} g(s)u_n\|^2_{X_0} \,\dd s \,\dd t
\\ & = \textstyle \E \sum_{n\geq 1}\int_{\R_+} \int_{\R_+} \|A e^{-tA} g(s)u_n\|^2_{X_0} \,\dd t\,\dd s
\\ & \textstyle \stackrel{\eqref{eq:equivorbitreal}}{\eqsim} \E \sum_{n\geq 1}\int_{\R_+} \|g(s)u_n\|^2_{X_{1/2, 2}} \,\dd s
\\ & = \textstyle  \|g\|_{L^2(\O;L^2(\R_+;\calL_2(\mathcal{U},X_{1/2,2})))}^2
\\ & \textstyle \stackrel{\eqref{eq:sandwichinterpHS}}{\eqsim}
\|g\|_{L^2(\O;L^2(\R_+;\calL_2(\mathcal{U},X_{1/2})))}^2.
\end{align*}
Next, one could apply Theorem \ref{thm:extrapol} to extrapolate the above result to (weighted) $L^p$-spaces with $p>2$. However, since we want to prove time-regularity as well (see Definition \ref{def:SMRbullet}), we argue differently (see below Theorem  \ref{thm:SMRHinfty}).

\subsubsection{The role of the $H^\infty$-calculus}
Below, we present one of the cornerstone results in the theory of stochastic maximal $L^p$-regularity. The result provides a sufficient condition for stochastic maximal $L^p$-regularity in terms of the $H^\infty$-(functional) calculus of $A$ (see Subsection \ref{ss:Hinfty}).
The unweighted case of the next result was obtained by the second named author together with van Neerven and Weis in \cite{MaximalLpregularity}, and through a different method in \cite{NVW13}. The extension to the weighted setting is derived using a perturbation argument in \cite[Section 7]{AV19}.
\begin{theorem}\label{thm:SMRHinfty}
Let $(\mathcal{O}, \Sigma,\mu)$ be a $\sigma$-finite measure space and let $q\in [2, \infty)$. Suppose that $X_0 = L^q(\mathcal{O})$ or $X_0$ is isomorphic to a closed subspace of $L^q(\mathcal{O})$. If there exists a $\lambda\geq 0$ such that $\lambda+A$ has a bounded $H^\infty$-calculus of angle $<\pi/2$, then $(A,0)\in \mathcal{SMR}_{p,\kappa}^{\bullet}$  for all $p\in (2, \infty)$ and $\kappa\in [0,p/2 - 1)$. Moreover, if $q=2$, then additionally $(A,0)\in \mathcal{SMR}_{2,0}^{\bullet}$.
\end{theorem}

\begin{remark}
Theorem \ref{thm:SMRHinfty} is only formulated for the setting where $X_0$ is isomorphic to a closed subset of an $L^q$-space (e.g.\ a fractional Sobolev space). However, by \cite{NVW11} it is possible to extend the result to a broader class of spaces. It is an open problem to extend the result to arbitrary spaces $X_0$ with UMD and type $2$. In particular, we do not know whether the result holds for spaces such as $X_0 = L^2(\R;L^q(\R))$ with $q\in (2, \infty)$. On the one hand, this space is not so important for evolution equations. On the other hand, for a space like $X_0 = L^q(\R;L^2(\R))$ with $q\in  (2,\infty)$, it is possible to obtain stochastic maximal $L^p$-regularity from \cite{MaximalLpregularity,NVW13} and an extension of this observation was used by the first named author to study 3D primitive equations with rough transport noise, see \cite[Appendix B]{agresti2023primitive}.
It also remains unclear whether the assumption that $A$ has a bounded $H^\infty$-calculus can be relaxed. Some evidence that this might be possible is provided by Theorem \ref{thm:SMRHS} where no additional conditions are required in the Hilbert space case.
\end{remark}

After this preparation, we can now give the proof of the Hilbert space result of Theorem \ref{thm:SMRHS}.
\begin{proof}[Proof of Theorem \ref{thm:SMRHS}]
By Theorem \ref{thm:deSimon} and Remark \ref{rem:MRf=0} it suffices to consider $f=0$.

Let $A_0$ be an invertible positive self-adjoint operator on $X_0$ with $\Do(A_0) = X_1$. The existence of such an operator follows from \cite[Proposition 8.1.10]{TayPDE2}. By \cite[Proposition 10.2.23]{Analysis2}, $A_0$ has a bounded $H^\infty$-calculus of angle zero.

By translation and extending $g$ by zero, it remains to prove
\begin{align}\label{eq:SMRsimple2}
\|u\|_{L^p(\O;H^{\theta,p}(0,T,w_{\kappa};X_{1-\theta}))} \leq C_{T,\theta} \|g\|_{L^p(\O;L^p(0,T,w_{\kappa};\g(\mathcal{U},X_{1/2})))},
\end{align}
where $u$ is the solution to $\dd u + Au\, \,\dd t = g \,\dd W$ with $u(0) = 0$.
Now, we argue as in \cite[Theorem 3.9]{VP18}. Let $v\in L^p(\O;L^p(0,T,w_{\kappa};X_1))$ be the strong solution to $\dd  v + A_0 v\,\dd t = g \,\dd  W$ with $v(0) = 0$. Then, by Theorem \ref{thm:SMRHinfty}, for all $\theta\in [0,1/2)$,
\[\|v\|_{L^p(\O;H^{\theta,p}(0,T,w_{\kappa};X_{1-\theta}))} \leq C_{T,\theta}\|g\|_{L^p(\O;L^p(0,T,w_{\kappa};\g(\mathcal{U},X_{1/2})))}.\]
Since $A$ has maximal $L^p_{\kappa}$-regularity by Theorems \ref{thm:deSimon} and \ref{thm:MRweight}, there is a unique $z\in L^p(\O;L^p(0,T,w_{\kappa};X_1))$ which is a strong solution to $z' + A z = (A_0 - A)v$ with $z(0) = 0$,
and the following estimate holds
\begin{align*}
\|z\|_{L^p(\O;L^p(0,T,w_{\kappa};X_{1}))} \leq C
\|(A - A_0)v\|_{L^p(\O;L^p(0,T,w_{\kappa};X_0))}&\leq C\|v\|_{L^p(\O;L^p(0,T,w_{\kappa};X_1))}
\\ & \leq C\|g\|_{L^p(\O;L^p(0,T,w_{\kappa};\g(\mathcal{U},X_{1/2})))}.
\end{align*}
Here, the process $z$ can be shown to be progressively measurable.
Moreover, as in \eqref{eq:interptrickHthetaMR} we see that the same type of estimate holds for $\|z\|_{L^p(\O;H^{\theta,p}(0,T,w_{\kappa};X_{1-\theta}))}$.

It follows that $u = v+z$ is in $L^p(\O;L^p(0,T,w_{\kappa};X_1))$ and is the required strong solution. The estimate \eqref{eq:SMRsimple2} follows by combining the estimates for $v$ and $z$.
\end{proof}

\begin{remark}
The proof technique of Theorem \ref{thm:SMRHS} can be extended to operators $A$ that depend on time and $\Omega$. However, it is crucial to assume that $A$ has {\em deterministic} maximal $L^p$-regularity, which is often a nontrivial condition. The technique can even be extended to the Banach space setting if one knows that there exists an operator $A_0\in \calL(X_1, X_0)$ such that $A_0$ has a bounded $H^\infty$-calculus of angle $<\pi/2$ (for details, see the transference result \cite[Theorem 3.9]{VP18}).
\end{remark}

\subsection{Sufficient conditions in case $B\neq 0$}
\label{ss:sufficient_B_non_zero}
If $B\neq 0$, determining whether $(A,B)\in \mathcal{SMR}_{p,\kappa}^{\bullet}$ can be quite complicated. An exception is the variational setting if $p=2$ and $\kappa=0$. In this case the classical coercivity condition on $(A,B)$ provides a sufficient condition for $(A,B)\in \mathcal{SMR}_{2,0}^{\bullet}$:

\begin{theorem}[Variational setting]\label{thm:varlin}
Let $(V,H,V^*)$ be a Gelfand triple as in Section \ref{sec:Var}. Let $X_0=V^*$ and $X_1=V$. Suppose that there exist $\theta,M>0$ such that, for all $v\in V$,
$$
\langle A v,v \rangle -\tfrac{1}{2}\|B v\|_{\calL_2(\mathcal{U},X_{1/2})}\geq  \theta \|v\|^2_{V}-M\|v\|_H^2.
$$
Then $(A,B)\in \mathcal{SMR}_{2,0}^{\bullet}$.
\end{theorem}
It is well-known that the coercivity condition implies that $-A$ generates a strongly continuous analytic semigroup on $V^*$ (and then also on $X_{1/2}$). For instance, this can be deduced from Theorem \ref{thm:Dore}.

Note that the condition of the above theorem leads to \eqref{eq:stochastic_parabolicity_example} when $(A,B)$ is as in \eqref{eq:AB_basic_example} with $V=W^{1,2}(\R^d)$ and $H=L^2(\R^d)$.
For the standard proof the reader is referred to \cite[Chapter 4]{LR15} and \cite[Lemma 4.1]{AVvar}. Further details on the variational setting are given in Section \ref{sec:Var}.

There are several concrete situations in which $(A,B) \in \mathcal{SMR}_{p,\kappa}$, and the reader may consult a selection of them in Subsection \ref{ss:furtherSMR}. The first results in an $L^p$-setting with $p>2$ were obtained by Krylov in \cite{Kry96,Kry}, where $A$ is a second order elliptic differential operator on $\R^d$, and $B$ a first order operator. In Subsection \ref{ss:secondorderconcrete}, we discuss the joint parabolicity condition (stochastic parabolicity) on the pair $(A,B)$ in this concrete setting.

Using the following simple result (see \cite[Proposition 3.8]{AV19_QSEE_1} for the proof), one can transfer $(A,B) \in \mathcal{SMR}_{p,\kappa}$ to $(A,B) \in \mathcal{SMR}_{p,\kappa}^{\bullet}$.
\begin{proposition}[Transference]\label{prop:transferenceSMRbullet}
Suppose that $(A,B) \in \mathcal{SMR}_{p,\kappa}$. If there is a closed operator $A_0$ such that $\Do(A_0) = X_1$ and $(A_0, 0)\in \mathcal{SMR}_{p,\kappa}^{\bullet}$, then $(A,B) \in \mathcal{SMR}_{p,\kappa}^{\bullet}$.
\end{proposition}
Often one chooses $A_0$ to be an operator to which Theorem \ref{thm:SMRHinfty} is applicable. For example, if $A$ is a second-order differential operator with space-dependent coefficients and boundary conditions, one could take $A_0=-\Delta$ with the same boundary conditions, and then show that $A_0$ has a bounded $H^\infty$-calculus of angle $<\pi/2$. A similar argument was used in the proof of Theorem \ref{thm:SMRHS}, where $A_0$ was constructed more abstractly.

\smallskip

Thus, in many cases, to obtain $(A,B)\in \mathcal{SMR}^\bullet_{p,\kappa}$, it suffices, by Proposition \ref{prop:transferenceSMRbullet} to check that $(A,B)\in \mathcal{SMR}_{p,\kappa}$. For this, one typically requires tools from harmonic analysis, stochastic calculus, PDE theory, and perturbation theory (see \cite[Theorem 3.2]{AV21_SMR_torus}). A stochastic method of continuity (see \cite[Proposition 3.13]{AV19_QSEE_2}) ensures that it is enough to assume that $u$ satisfies \eqref{eq:SEElin0}, and to prove the a priori estimate \eqref{eq:SMRest}. For a concrete differential operator on a domain $\Dom\subseteq \R^d$ one typically follows the roadmap outlined below.
\begin{itemize}
\item Find a starting point $(A_0, 0)$ for the stochastic method of continuity, reducing the problem to the a priori estimate \eqref{eq:SMRest};
\item By localization it is enough to consider the half-space and whole-space cases;
\item Reduce to the case of constant coefficients using a freezing argument leveraging the smoothness of the coefficients;
\item After that there are two options:
\begin{itemize}
  \item Transform $(A,B)$ into $(\wt{A}, 0)$ via a Doss-Sussman argument (see \cite[Section 3.5]{VP18}). Show that the resulting problem $(\wt{A}, 0)$ has stochastic maximal regularity via the transference result \cite[Theorem 3.9]{VP18} and deterministic maximal regularity.
  \item Show stochastic maximal regularity using stochastic calculus.
\end{itemize}
\end{itemize}
An example where this roadmap was applied can be found in \cite{AV21_SMR_torus}.
We derived stochastic maximal $L^p$-regularity in an $L^q$-setting for $(A,B)$ where $A$ is a second-order operator and $B$ is a first-order operator.
A similar method was used for the Stokes operator in \cite{AV20_NS} and a fourth-order operator in \cite{AgrSau}.

\subsection{Further references}\label{ss:furtherSMR}
In this subsection, we provide a list of references on concrete situations in which stochastic maximal regularity holds (or does not hold) and a selection of related works. We focus solely on Gaussian noise and the $(t,\omega)$-independent setting as in the rest of the paper. However, these results can typically be extended to cases where the dependency in $(t,\omega)$ is progressively measurable. We will not consider lower-order terms, since they do not add anything to the discussion below and they can be often added afterwards via perturbation \cite[Theorem 3.2]{AV21_SMR_torus}.
We restrict ourselves to the real-valued case, though some remarks on the system case can be found in Subsections \ref{ss:secondorderconcrete}, \ref{eq:SPDEdomain} and \ref{ss:counter}.
This subsection is not intended to be a comprehensive survey of the rapidly growing literature on this topic. Since Theorems \ref{thm:SMRHS} and \ref{thm:SMRHinfty} provide general conditions for $(A,0)\in \mathcal{SMR}_{p,\kappa}^{\bullet}$, we will not revisit that case. However, further nontrivial results in this setting are discussed later in Subsections \ref{sssfurther:B=0} and \ref{sss:SMRdiff}.

From this point onward, we focus on the case where $B\neq 0$. The case where $B$ is relatively small is easy. Indeed, if there exists a $\delta>0$ small enough and $C_{\delta}$ such that for all $u\in X_1$,
$\|B u\|_{\gamma(\mathcal{U},X_{1/2})}\leq \delta \|u\|_{X_1} + C_{\delta} \|u\|_{X_0}$, then $(A,B)\in \mathcal{SMR}_{p,\kappa}^{\bullet}$ follows from a perturbation argument, see the above-mentioned perturbation result or \cite{Brz1}, \cite[Theorem 6.24]{DPZ}, \cite{Fla90} and \cite[Theorem 4.5]{NVW11eq}.

\subsubsection{Second order operators on $\R^d$ with $B\neq 0$}\label{ss:secondorderconcrete}
Consider a second order operator $A$ and first order operator $B$, where $A$ is in {\em nondivergence form}:
\[\textstyle  A u = \sum_{i,j=1}^d a_{i,j} \partial_i \partial_j u \ \ \text{and} \ \ (B
u)_n = \sum_{i=1}^d b_{n}^i \partial_i u,\]
where $a_{i,j} = a_{j,i}$ and $b^i_n$ are real-valued functions defined on $\R^d$. In this context, the stochastic parabolicity condition (see \eqref{eq:AB_basic_example}) has a similar form. Indeed, letting $\sigma_{i,j} = \sum_{n\geq 1} b_{n}^i b^j_{n}$, the {\em stochastic parabolicity condition} for the above pair of operators reads as follows:

There exists a $\theta>0$ such that
\begin{align}\label{eq:stochpara}
\textstyle \sum_{i,j=1}^d [a_{i,j}(x) -\frac12\sigma_{i,j}(x)]\xi_i \xi_j
\geq \theta |\xi|^2\quad \text{ for all }x,\xi\in \R^d.
\end{align}
In the case where $A$ is in {\em divergence form}, i.e.\ $A u(x) = \sum_{i,j=1}^d \partial_i [a_{i,j} \partial_j u]$ the same condition \eqref{eq:stochpara} applies, and most of the results below hold in both situations.

Under suitable boundedness conditions on the coefficients and \eqref{eq:stochpara}, one can easily check that for the divergence form case, the coercivity condition of Theorem \ref{thm:varlin} holds if $V$ is a closed subspace of $H^1(\R^d)$, and $H = L^2(\R^d)$. Therefore, $(A,B)\in \mathcal{SMR}_{2,0}^{\bullet}$ for the choice $X_1 = V$ and $X_0 = V^*$.

For the case $\Dom = \R^d$, it was shown in a series of papers by Krylov \cite{Kry94a,Kry96,Kry,Krylov-div} that under suitable regularity conditions on the coefficients, the stochastic parabolicity condition \eqref{eq:stochpara} implies that for $X_0 = H^{s,q}(\R^d)$ and $X_1 = H^{s+2,q}(\R^d)$ with $s\in \R$, one has $(A,B)\in \mathcal{SMR}_{p,0}$ if $p=q\in [2, \infty)$. For $x$-independent coefficients, a similar result was established for $p\geq q\geq 2$ in \cite{Kry00} (note that we switched roles of $p$ and $q$). When $p=q$, extensions to systems with certain diagonal structures were considered in \cite{KimLeesystems,MiRo}, and for the stochastic Stokes system (also referred to as ``turbulent Stokes system'') in \cite{M02}.

Under similar regularity conditions, and an additional continuity condition at $\infty$, it was shown in \cite{AV21_SMR_torus} that one has $(A,B)\in \mathcal{SMR}_{p,\kappa}^{\bullet}$ for all $p\in (2, \infty)$, $q\in [2, \infty)$, and $\kappa\in [0,p/2-1)$. Moreover, the case $p=q=2$ and $\kappa=0$ is also included. Both divergence and nondivergence form systems are covered in the previous work. These results also hold on $\T^d$ and can be extended to smooth manifolds without boundary. Similar results for the stochastic Stokes system can be found in \cite{AV20_NS}.

\subsubsection{Second order operators on domains with $B\neq 0$}\label{eq:SPDEdomain}
In case $\Dom$ is either the half space $\R^d_+ = \R_+\times \R^{d-1}$ or a bounded smooth domain, sufficient conditions for $(A,B)\in \mathcal{SMR}_{p,0}$ have been established in a series of papers.
In \cite{KryW22theory}, it was shown that weighted Sobolev spaces $W^{n,2}$ can be used to handle the blow-up of derivatives near the boundary, thereby avoiding the need for compatibility conditions on the data, which had appeared in earlier works such as \cite{Brz1, Brz2, Fla90}. In this context, the stochastic parabolicity condition \eqref{eq:stochpara} remains suitable for achieving higher-order regularity for linear equations. Further extensions of these results were obtained in \cite{KimLeeC1}.

Under similar conditions, $L^p$-theory on the half-line and half-space was developed in \cite{KryLothalfline,KryLot} for coefficients independent of $x$. The half space case with VMO (vanishing mean oscillation) coefficients was addressed in \cite{Kry-divhalf}. Extensions to $C^1$ and Lipschitz domains and weighted Sobolev spaces were provided in \cite{Kim04a,Kim04b,Kim05,Kim08,KimLeeC1}. In these papers, the space integrability is taken as $q=p$. The form of the regularity estimates differs slightly from our definition of $\mathcal{SMR}_{p,0}$. The reader is referred to \cite[Lemma 6.11]{AV19_QSEE_1} and \cite{lindemulder2024functional, LV18} for an explanation of how the two formulations can be connected.

Finally, we note that in \cite{Du2018}, under the geometric condition $b\cdot n = 0$ on $\partial \Dom$, it was shown that certain results valid in the full space setting can be extended to domains with boundaries.

\subsubsection{Further related results and counterexamples for $B\neq 0$}\label{ss:counter}
The paper \cite[Theorem 5.3]{Krylov03}, discusses restrictions on weighted function spaces for considering SPDEs on domains with boundary conditions.

Extending the stochastic parabolicity condition \eqref{eq:stochpara} to systems or higher-order equations is not straightforward.
In \cite{KimLeesystems}, a system variant of $(A,B)$ is constructed (and extended in \cite{DuLiuZhang}) for which the $L^2$-theory holds, but the $L^p$-theory breaks down if $p$ becomes large. Similar behaviour was observed for a different class of examples in \cite{BrzVer11}. In the latter paper, the first order derivative in $B$ is replaced by a scalar multiple of $(-\Delta)^{1/2}$ on the torus. A similar phenomenon is expected for higher-order equations.
Some of the above-mentioned issues are caused by the lack of integrability in $\Omega$ of the solution.

In the variational setting, some abstract theory on higher order moments was developed in \cite{NeeSis} and further improved in \cite{GHV}. The latter paper also unifies results on systems \cite{DuLiuZhang} and higher order equations \cite{wang2019schauder}. We should note that the emphasis in the last two papers is on Schauder's theory for SPDEs.
Finally, for measurable coefficients satisfying \eqref{eq:stochpara}, an $L^\infty$-bound was obtained in \cite{DG15_boundedness} using Moser iterations.

\subsubsection{The case $B = 0$}\label{sssfurther:B=0}
The $H^\infty$-calculus of Theorem \ref{thm:SMRHinfty} provides a general framework for having $\mathcal{SMR}_{p,\kappa}$. For a comprehensive list of examples of operators with a bounded $H^\infty$-calculus, the reader is referred to \cite{Analysis2}. In particular, the following class is included: all positive contraction semigroups on $L^q$ which have a bounded analytic extension to a sector. The positivity and contractivity can be relaxed to regular contractivity \cite[Theorem 4.2.21]{fackler2015regularity}.

However, the $H^\infty$-calculus cannot be directly applied if $A$ depends on $t$ or even $(t,\omega)$. Fortunately, there is a simple trick to reduce to the time-independent situation provided that $A$ has deterministic maximal $L^p$-regularity (see \cite[Theorem 3.9]{VP18}).
On the other hand, Theorem \ref{thm:SMRHS} indicates that the $H^\infty$-calculus is not strictly necessary, suggesting that there are aspects of the theory that are yet to be fully understood. Additionally, the condition on $X_0$ in Theorem \ref{thm:SMRHinfty} might not hold in some cases. Fortunately, it can be replaced with a more flexible condition as shown in \cite{NVW11} and the proofs in \cite{MaximalLpregularity, NVW13}. This allows $X_0$ to be a closed subspace of Besov spaces $B^{s}_{q,r}$ or Triebel-Lizorkin spaces $F^{s}_{q,r}$ with $q,r\in (2, \infty)$.

Motivated by the above discussion, there is room left for improvement in the theory. It is also worth mentioning that for concrete examples, the $H^\infty$-calculus is not established yet.

In the case where $B = 0$, various forms of (weighted) extrapolation results for stochastic maximal regularity can be found in \cite{KimKimExtr,lorist2021vector, LoVer}. When combining these results, one can obtain $L^p(L^q)$-theory from $L^2$-theory. The results in \cite{LoVer} in particular imply the general extrapolation result of Theorem \ref{thm:extrapol}.
In the setting of time-dependent operators $A$ such an extrapolation result does not hold. Indeed, even for elliptic operators in divergence form with time-dependent coefficients in the deterministic setting on $L^2$, there are counterexamples to maximal $L^p$-regularity for $p>2$ (see \cite{BMV}).

For domains with lower regularity (polygonal, wedges) sufficient conditions for $(A,0)\in \mathcal{SMR}_{p,\kappa}$ have been obtained in \cite{cioica2020lp, CioicaKimLeeLindner18, CioicaKimLee18,KimLeeSeo,LoVer}.

\subsubsection{Stochastic maximal regularity in different scales}\label{sss:SMRdiff}
As in the deterministic setting, regularity estimates can be considered across various scales, each offering a unique insight into the possible solution space for a given (S)PDE. Two examples of such scales, which have been extensively studied and discussed above, are the H\"older and $L^p$-scales. In all of the references below, we assume $B=0$ or that it is sufficiently small.

In the $L^p$-framework, one can also consider real interpolation spaces instead of the complex ones we considered. The reader is referred to \cite{Brz1,BH09,DL98} and \cite{LoVer} for more information.

Another extension of $L^2$-theory can be given in terms of the $\gamma$-spaces $\gamma(L^2(0,T;H),X_0)$ as shown in \cite{NVWgamma}. The approach provides different information about the solution and includes $X_0 = L^q(\Dom)$ with $q$ in the full reflexive range $(1, \infty)$. In this case, one obtains $L^q(\Dom;L^2(0,T))$-estimates of the solution and its derivatives. This framework was extended to $L^q(\Dom;L^p(0,T))$-estimates in \cite{antoni2017regular}, and can also be obtained through the extrapolation theory of \cite{lorist2021vector}. Estimates in $L^q(\Dom\times(0,T);L^p(\Omega))$ were obtained in \cite{MR4385410}.

Another scale of interest is the parabolic tent space, with several results on second-order SPDEs obtained in \cite{AusNeePorTent, AusPorTent, VP18}. One key advantage of this setting is that it only requires measurability in the $x$-variable. In some cases, the $H^\infty$-calculus can also provide results in this direction in $L^q$-spaces.

\subsubsection{Volterra equations}
For stochastic maximal regularity results on Volterra equations and nonlocal operators, the reader is directed to \cite{choi2024sobolev,DeLo09,DL13,KimKimLim,KimParkRyu,LoVer}  and references therein. Notably, \cite{KimKimLim} also incorporates a transport noise term.

\section{Local existence, uniqueness and regularity}\label{sec:loc-well-posed}
In this section, we discuss local existence and uniqueness for the following stochastic evolution equation:
\begin{equation}
\label{eq:SEE}
\left\{
\begin{aligned}
&\dd u + A u \,\dd t = F(u)\,\dd  t + (B u +G(u)) \,\dd  W,\\
&u(0)=u_{0}.
\end{aligned}
\right.
\end{equation}
The nonlinearities $F$ and $G$ are assumed to be locally Lipschitz on certain intermediate spaces with their local Lipschitz constants exhibiting controlled growth. The interplay between the smoothness of these intermediate spaces and the growth of the Lipschitz constants determines the classification of the problem into sub-critical, critical, and super-critical regimes.

For the leading operators $A$ and $B$, we assume that the pair $(A,B)$ satisfies stochastic maximal $L^p$-regularity. Consequently, the natural path space in which we seek a solution to \eqref{eq:SEE} is:
\begin{equation}
\label{eq:max_reg_space}
\textstyle \bigcap_{\theta\in [0,1/2)}H^{\theta,p}(0,T,w_{\a};X_{1-\theta}) \subseteq L^p(0,T,w_{\a};X_1)\cap C([0,T];X_{1-\frac{1+\a}{p},p}),
\end{equation}
where $w_{\kappa}(t) = t^{\kappa}$ and $\kappa\in [0,\frac{p}{2}-1)$ in case $p>2$. For $p=2$, we instead use the space $L^2(0,T;X_1)\cap C([0,T];X_{\frac{1}{2}})$.

In most situations, it suffices to work with the right-hand side of \eqref{eq:max_reg_space}. However, in certain cases, achieving optimal space-time regularity, as provided by the left-hand side of \eqref{eq:max_reg_space}, becomes essential. According to Proposition \ref{prop:tracespace}, the so-called trace space $X_{1-\frac{1+\kappa}{p},p}$ is the optimal space for the initial data $u_0$ when searching for solutions with paths in the LHS of \eqref{eq:max_reg_space}.
The introduction of time weights $\kappa$ adds flexibility, enabling the analysis of qualitative properties of solutions to \eqref{eq:SEE} such as blow-up behaviour and regularity. These aspects are explored further in Section \ref{sec:blowup} below.

\subsection{Main assumptions}\label{ss:setting}
Recall that $X_{\theta}=[X_0,X_1]_{\theta}$ denotes the complex interpolation space, while $X_{\theta,r}= (X_0, X_1)_{\theta,r}$ denotes the real interpolation space for $\theta\in (0,1)$ and $r\in [1, \infty]$. Furthermore, assume $A\in \calL(X_1, X_0)$, and $B\in \calL(X_1, \calL_2(\mathcal{U}, X_{1/2}))$. Additional assumptions on $A$ and $X_1$ are provided below. In Remark \ref{rem:comparison_comments_partI}, we discuss which of these conditions can be omitted or generalized.

\begin{assumption}\label{ass:FGcritical}
Let $X_0$ be a UMD space with type $2$, and suppose that there is a $\lambda_0\geq 0$ such that $\lambda_0+A$ is a sectorial operator on $X_0$ and set $X_1 = \Do(A)$.
Let $p\in [2, \infty)$, $\kappa\in [0,p/2-1)\cup\{0\}$.
The mappings
$$
F:X_1\to X_0\quad  \text{ and } \quad G:X_1\to \gamma(\mathcal{U},X_{\frac12})
$$
satisfy the following local Lipschitz condition: For each $n\geq 1$, there exists $L_n>0$ such that for all $u,v\in X_1$ with $\|u\|_{X_{1-\frac{1+\kappa}{p},p}}, \|v\|_{X_{1-\frac{1+\kappa}{p},p}}\leq n$,
\[\textstyle \|F(u) - F(v)\|_{X_0} + \|G(u) - G(v)\|_{\gamma(\mathcal{U},X_\frac12)}
\leq L_n \sum_{j=1}^{m} (1+\|u\|_{X_{\beta_j,1}}^{\rho_j} + \|v\|_{X_{\beta_j,1}}^{\rho_j}) \|u-v\|_{X_{\beta_j,1}},\]
where $\beta_j\in (1-\frac{1+\kappa}{p},1)$ and $\rho_j\geq 0$  satisfy
\begin{align}\label{eq:subcritical}
\frac{1+\kappa}{p}\leq \frac{(1+\rho_j)(1-\beta_j)}{\rho_j} \ \text{ for all } j\in \{1, \ldots, m\}.
\end{align}
\end{assumption}

The condition \eqref{eq:subcritical} imposes an upper bound on $\frac{1+\a}{p}$, which, in turn, determines a lower bound for the smoothness of the trace space $X_{1-\frac{1+\a}{p},p}$ where the initial data $u_0$ is taken. These bounds depend solely on the parameters of the local Lipschitz condition for $F$ and $G$.
Even in the deterministic case, \eqref{eq:subcritical} is known to be sharp for local existence and uniqueness (see \cite[Theorem 2.4]{CriticalQuasilinear}).

The conditions on the deterministic part $F$ and the stochastic part $G$ are very similar; however, $G$ is required to have more space regularity (it is $X_{1/2}$-valued) due to the reduced parabolic regularization of the noise term. To handle infinite-dimensional noise, we assume $G$ takes values in the space of $\gamma$-radonifying operators $\gamma(\mathcal{U},X_\frac12)$ (see Subsection \ref{subsec:gamma}).
Given the constraints on $\kappa$ and $p$, we always have $X_{1-\frac{1+\kappa}{p},p}\subseteq X_{\frac12}$. Illustrative examples that clarify the above are discussed in Subsection \ref{ss:tracecriticalF}.

\smallskip

We define the couple $(p,\a)$ or the setting $(X_0,X_1,p,\a)$ as \underline{critical} (resp., \underline{subcritical}) for \eqref{eq:SEE} if \eqref{eq:subcritical} holds with equality for some $j$ (respectively, with strict inequality for all $j$). Similar terminology extends naturally to the trace space $X_{1-\frac{1+\a}{p},p}$. To provide an intuitive understanding of  \eqref{eq:subcritical}, we first rewrite it as
\begin{equation}
\label{eq:subcritical2}
\text{(a)}\quad \rho_j \Big(\beta_j - 1 + \frac{1+\kappa}{p}\Big) + \beta_j \leq 1, \qquad \text{(b)}\quad \beta_j -\Big(1 - \frac{1+\a}{p}\Big) \leq \frac{1}{\rho_j+1}\frac{1+\a}{p}.
\end{equation}

Part $\text{(a)}$ of \eqref{eq:subcritical2} can be roughly interpreted as follows. Due to parabolic regularity theory, the smoothness of the solution is always one order higher than that of the inhomogeneity or nonlinearity. This additional regularity is leveraged to control the nonlinearities. The estimates for $F$ and $G$ consist of two components: a ``Lipschitz constant'' which grows as a power of $\rho_j$, and a ``difference'' part, both measured in the $X_{\beta_j,1}$-norm.
\begin{itemize}
\item
In the first part of $\text{(a)}$, the condition is derived from the distance of
$\beta_j$ to the smoothness parameter of the trace space $X_{1-\frac{1+\kappa}{p},p}$, modulated by  the power $\rho_j$.
\item In the second part of $\text{(a)}$, $\beta_j\in (0,1)$ appears with a coefficient of one due to the contribution of the ``difference part''.
\end{itemize}
The condition \eqref{eq:subcritical2} appeared in \cite{AV19_QSEE_1,CriticalQuasilinear} in a more general form. Specifically, $\beta_j-1+\frac{1+\a}{p}$ was replaced by $\varphi_j-1+\frac{1+\a}{p}$ where $\varphi_j\in [\beta_j,1)$. For the latter, the above interpretation of the critical condition carries over verbatim.

Formula $\text{(b)}$ says that the roughness of the nonlinearity, represented by $\beta_j$, can exceed the one of the trace space $1-\frac{1+\a}{p}$ by a factor $\frac{1}{\rho_j+1}<1$ appearing in front of the quantity $\frac{1+\a}{p}$.

\smallskip

In the following remark, we offer an alternative perspective on the criticality condition \eqref{eq:subcritical}.
\begin{remark}[Growth condition and criticality]\label{rem:growth}
Note that Assumption \ref{ass:FGcritical} (with $v=0$) implies the following growth condition: For all $n\geq 1$ and $u\in X_1$ satisfying $\|u\|_{X_{1-\frac{1+\kappa}{p},p}}\leq n$, one has
\begin{align}\label{eq:growthFcGc}
\|F(u)\|_{X_0} + \|G(u)\|_{\gamma(\mathcal{U},X_\frac12)} \lesssim_n \sum_{j=1}^m (1+\|u\|_{X_{\beta_j,1}}^{\rho_j+1}).
\end{align}
An important consequence of the above is that for all $T<\infty$ and $u\in L^p(0,T,w_{\kappa};X_1)$ that satisfy $\sup_{t\in [0,T]}\|u(t)\|_{X_{1-\frac{1+\kappa}{p},p}}\leq n$, $\|F(u)\|_{X_0}$ and $\|G(u)\|_{\gamma(\mathcal{U},X_{\frac12})}$ are in $L^p(0,T,w_{\kappa})$. We discuss this in some detail, as it reveals another interpretation of the (sub)criticality condition \eqref{eq:subcritical}. By Assumption \ref{ass:FGcritical},
\begin{align}
\label{eq:funnyXrole}
\|F(u)\|_{L^p(0,T,w_{\kappa};X_0)} + \|G(u)\|_{L^p(0,T,w_{\kappa};\gamma(\mathcal{U},X_{\frac12}))} &\lesssim_n \sum_{j=1}^m \Big\| \big(1+\|u\|_{X_{\beta_j,1}}^{\rho_j+1}\big)\Big\|_{L^p(0,T,w_{\kappa})}
\\
\nonumber
& \lesssim_n\sum_{j=1}^m (T^{\frac{1+\kappa}{p}} + \|u\|_{L^{p (\rho_j+1)}(0,T,w_{\kappa};X_{\beta_{j},1})}^{\rho_j+1}).
\end{align}
Assuming  $\frac{1+\kappa}{p}\leq  \frac{(1+\rho_j)(1-\beta_j)}{\rho_j}$ for a fixed $j$, Lemma \ref{lem:funnyXembedding} below shows
\begin{align}\label{eq:interpestfunnyX}
\|u\|_{L^{p (\rho_j+1)}(0,T,w_{\kappa};X_{\beta_j,1})}^{\rho_j+1}\lesssim T^{\varepsilon_j/p} \|u\|_{L^p(0,T,w_{\kappa};X_1)}^{\theta_j(\rho_j+1)}\|u\|_{L^\infty(0,T;X_{1-\frac{1+\kappa}{p},p})}^{(1-\theta_j)(\rho_j+1)},
\end{align}
where $\varepsilon_j\geq 0$  and $\theta_j\in (0,1)$ are such that $(1-\theta_j)(\rho_j+1)\leq 1$. Furthermore, when $(p,\a)$ are subcritical, then $\varepsilon_j>0$ and the exponent on the $L^p(0,T,w_{\kappa};X_1)$-norm of $u$ is \emph{strictly} less than $1$. In the critical case, this exponent equals $1$. In particular, if $(p,\a)$ are critical, the constant on the RHS\eqref{eq:interpestfunnyX} does not tend to $0$ as $T\downarrow 0$ and therefore the existence of solutions to \eqref{eq:SEE} based on fixed point methods is quite delicate, see \cite{AV19_QSEE_1,CriticalQuasilinear}.

The estimate \eqref{eq:interpestfunnyX} together with \cite[Proposition 2.10]{AV19_QSEE_1} can be used to avoid \cite[Lemma 4.9]{AV19_QSEE_1} and thus slightly improve the setting in \cite{AV19_QSEE_1}, where $X_{\beta_{j}}$ was used instead of $X_{\beta_{j},1}$ in the assumptions on $F$ and $G$. A similar improvement can be obtained in \cite{CriticalQuasilinear}, where the condition called (S) before Remark 1.1 can be omitted. This refinement is elaborated in detail in \cite[Section 18.2]{Analysis3}.
\end{remark}

\begin{lemma}[Critical interpolation estimate]\label{lem:funnyXembedding}
Suppose that $p\in [2, \infty)$ and $\kappa\in [0,p/2-1)\cup\{0\}$.
Let $\beta\in (1-\frac{1+\kappa}{p},1)$, $\rho\geq 0$ be such that $\frac{1+\kappa}{p}\leq  \frac{(1+\rho)(1-\beta)}{\rho}$. Then there are constants $C$ and $\varepsilon\geq 0$ independent of $T$ such that for all $u\in L^p(0,T,w_{\kappa};X_1)\cap L^\infty(0,T;X_{1-\frac{1+\kappa}{p},p})$ one has
\[\|u\|_{L^{p (\rho+1)}(0,T,w_{\kappa};X_{\beta,1})}^{\rho+1}\leq C T^{\varepsilon/p} \|u\|_{L^\infty(0,T;X_{1-\frac{1+\kappa}{p},p})}^{(1-\theta)(\rho+1)} \|u\|_{L^p(0,T,w_{\kappa};X_1)}^{\theta (\rho+1)},\]
where $\theta = 1-\frac{1-\beta}{\frac{1+\kappa}{p}}\in (0,1)$  satisfies
$\theta(\rho+1)\leq 1$. Finally, $[\varepsilon>0$ and $\theta(\rho+1)<1]$ if and only if $\frac{1+\kappa}{p}< \frac{(1+\rho)(1-\beta)}{\rho}$.
\end{lemma}
Actually, the proof below holds for any $p\in (1, \infty)$ and $\kappa\in [0,p-1)$.
\begin{proof}
The above estimate can be extracted from the proof in \cite[Lemma 18.2.7]{Analysis3}. We present a detailed proof here since it is important to see how the (sub)criticality condition enters.

The reiteration theorem for real interpolation (see \cite[(L.2) and Theorem L.3.1]{Analysis3}) gives that $X_{\beta,1} = (X_{1-\frac{1+\kappa}{p},p},X_1)_{\theta,1}$ with $\theta = 1-\frac{1-\beta}{\frac{1+\kappa}{p}}\in(0,1)$, and therefore
$\|x\|_{X_{\beta,1}}\leq C \|x\|^{1-\theta}_{X_{1-\frac{1+\kappa}{p},p}} \|x\|_{X_1}^{\theta}$ for $x\in X_1$.
Applying the latter with $x = u(t)$ and taking $L^{p(\rho+1)}(0,T,w_{\kappa})$-norms we obtain
\begin{align*}
\textstyle \|u\|_{L^{p (\rho+1)}(0,T,w_{\kappa};X_{\beta,1})}^{p(\rho+1)} &\leq \textstyle C^{p(\rho+1)} \|u\|_{L^\infty(0,T;X_{1-\frac{1+\kappa}{p},p})}^{(1-\theta)p(\rho+1)} \int_0^T \|u(t)\|_{X_1}^{\theta p(\rho+1)} t^{\kappa} \,\dd t \\ &\leq \textstyle C^{p(\rho+1)} \|u\|_{L^\infty(0,T;X_{1-\frac{1+\kappa}{p},p})}^{(1-\theta)p(\rho+1)}   T^{\varepsilon} \|u\|_{L^p(0,T,w_{\kappa};X_1)}^{\theta p(\rho+1)},
\end{align*}
where in the last step we used H\"older's inequality with $\theta(\rho+1) + \frac{1}{r} = 1$ and $\varepsilon = \frac{1+\kappa}{r}$. Note that thanks to $\frac{1+\kappa}{p}\leq \frac{(1+\rho)(1-\beta)}{\rho}$ one has $\theta(\rho+1)\leq 1$, so that H\"older's inequality can be applied. For the final assertion note that $\theta = 1-\frac{1-\beta}{\frac{1+\kappa}{p}}< \frac{1}{\rho+1}$.
\end{proof}

\begin{remark}[Comments and comparison with the assumptions in \cite{AV19_QSEE_1}]\
\label{rem:comparison_comments_partI}
\begin{itemize}
\item The assumptions on $A$ and $X_1$ in Assumption \ref{ass:FGcritical} can be removed in case there exists a sectorial operator $A_0$ on $X_0$ such that $\Do(A_0)=X_1$.
\item
As in \cite{AV19_QSEE_1}, the results in the current manuscript extend in case the mappings $A,B,F$ and $G$ are $(t,\omega)$-dependent in a progressively measurable way, provided the stochastic maximal $L^p$-regularity assumption holds and the estimates of Assumption \ref{ass:FGcritical} and \eqref{eq:growthFcGc} hold uniformly in $(t,\omega)$.
\item
In Assumption \ref{ass:FGcritical} we used the real interpolation space $X_{\beta_j,1}$ to formulate the estimate for $F$ and $G$. In contrast, in \cite{AV19_QSEE_1} uses the complex interpolation spaces $X_{\beta_j}$, which leads to a stronger condition because $X_{\beta_j,1}\hookrightarrow X_{\beta_j}$, see \eqref{eq:sandwichinterp}.
\item
In \cite{AV19_QSEE_1}, we assumed that $F=F_L+F_{{\rm c}}+ F_{\Tr}$ and $G=G_L+G_{{\rm c}}+ G_{\Tr}$, where $(F_c,G_c)$ are the ``critical parts'' as described in Assumption \ref{ass:FGcritical}, $(F_{\Tr},G_{\Tr})$ are the ``trace parts'' that are locally Lipschitz on the trace space $X_{1-\frac{1+\a}{p},p}\to X_0\times \g(\mathcal{U},X_{1/2})$, and $(F_L,G_L)$ are globally Lipschitz mapping $X_1\to  X_0\times \g(\mathcal{U},X_{1/2})$ (with relatively small Lipschitz constant).
In this manuscript, the trace parts are covered under the current assumptions, though distinguishing them may be conceptually useful. The global Lipschitz parts $(F_L,G_L)$ are omitted here, as they are typically unnecessary for applications.
\end{itemize}
\end{remark}

\subsection{Local existence and uniqueness}

\begin{definition}
Suppose Assumption \ref{ass:FGcritical} is satisfied for some $p\in [2, \infty)$ and $\kappa\in [0, p/2-1)\cup\{0\}$. A pair $(u,\sigma)$ is called an {\em $L^p_{\kappa}$-strong solution} of \eqref{eq:SEE} if $\sigma:\Omega\to [0,\infty)$ is a stopping time and $u:[0,\sigma]\to X_0$ is a strongly progressively measurable process such that
\[u\in L^p(0,\sigma,w_{\kappa};X_1)\cap C([0,\sigma];X_{1-\frac{1+\kappa}{p},p}),\]
and the following identity holds a.s. for all $t\in [0,\sigma]$:
\begin{equation}\label{eq:stronsol}
\textstyle u(t) - u_0 + \int_0^t A u(s) \,\dd s = \int_0^t F(u(s)) \,\dd s + \int_{0}^t \one_{[0,\sigma]}(s) [B u(s) + G(u(s))] \,\dd  W(s).
\end{equation}
\end{definition}
The integrals in \eqref{eq:stronsol} are well-defined. From Remark \ref{rem:growth} and the assumptions on $u$, we see that $F(u)\in L^p(0,\sigma,w_{\kappa};X_0)$ and $G(u) \in L^p(0,\sigma,w_{\kappa};\gamma(\mathcal{U},X_{1/2}))$ a.s. It remains to observe that due to the restrictions on $\kappa$ we have $L^p(0,\sigma,w_{\kappa})\hookrightarrow L^2(0,\sigma)$ by H\"older's inequality. Progressive measurability of $F(u)$ and $G(u)$ holds as well, and thus the Bochner integral and stochastic integral in \eqref{eq:stronsol} are well-defined (see Subsection \ref{subsec:stochint}).

\begin{definition}\label{def:localsol}
Suppose Assumption \ref{ass:FGcritical} is satisfied for some $p\in [2, \infty)$ and $\kappa\in [0, p/2-1)\cup\{0\}$.
\begin{enumerate}[{\rm(1)}]
\item\label{it1:localsol} A pair $(u,\sigma)$ is called an {\em $L^p_{\kappa}$-local solution} to \eqref{eq:SEE} if $\sigma:\Omega\to [0,\infty]$ is a stopping time and $u:[0,\sigma)\to X_0$ is a strongly progressively measurable process, and there exists an increasing sequence of stopping times $(\sigma_{n})_{n\geq 1}$ such that $\lim_{n\to \infty} \sigma_n = \sigma$ a.s., and $(u|_{[0,\sigma_n]}, \sigma)$ is a $L^p_{\kappa}$-strong solution to \eqref{eq:SEE}. The sequence $(\sigma_{n})_{n\geq 1}$  is called a {\em localizing sequence} for $(u, \sigma)$.
\item\label{it2:localsol} An $L^p_{\kappa}$-local solution $(u,\sigma)$ to \eqref{eq:SEE} is called {\em unique} if for every $L^p_{\kappa}$-local solution $(v,\tau)$ one has that a.s.\ $u=v$ on $[0,\sigma\wedge \tau)$.
\item\label{it3:localsol} An $L^p_{\kappa}$-local solution $(u,\sigma)$ to \eqref{eq:SEE} is called {\em $L^p_{\kappa}$-maximal} if for any other unique $L^p_{\kappa}$-local solution $(v, \tau)$ to \eqref{eq:SEE} one has that a.s.\ $\tau\leq \sigma$ and $u = v$ on $[0,\tau)$.
\item An $L^p_{\kappa}$-local solution $(u,\sigma)$ of \eqref{eq:SEE} is called {\em global} if $\sigma=\infty$ a.s.
\end{enumerate}
\end{definition}
If the word ``unique'' is left out in Definition \ref{def:localsol}\eqref{it3:localsol}, then one obtains an equivalent definition, see \cite[Remark 5.6]{AV19_QSEE_2}. As explained in the latter remark, in the quasilinear case (not considered in current manuscript) extra conditions are required.

We can now state the main local existence and uniqueness result.
\begin{theorem}[Local existence and uniqueness]\label{thm:localwellposed}
Let $p\in [2, \infty)$ and $\kappa\in [0,p/2-1)\cup \{0\}$ and suppose that Assumption \ref{ass:FGcritical} holds. Assume that $A\in \calL(X_1,X_0)$ and $B\in \calL(X_{1},\gamma(\mathcal{U},X_{1/2}))$ satisfy
\[(A,B)\in \mathcal{SMR}_{p,\kappa}^{\bullet}.\]
Then for every $u_0\in L^0_{\F_0}(\Omega;X_{1-\frac{1+\kappa}{p},p})$, there exists an $L^p_{\kappa}$-maximal solution $(u,\sigma)$ to \eqref{eq:SEE} with $\sigma>0$ a.s. Moreover, the following properties hold:
\begin{enumerate}[{\rm(1)}]
\item\label{it1:localwellposed} {\em (Regularity)} For each localizing sequence $(\sigma_n)_{n\geq 1}$ for $(u,\sigma)$ one has
\begin{itemize}
  \item if $p>2$ and $\kappa\in [0, p/2-1)$, then for all $n\geq 1$ and all $\theta\in [0,1/2)$
  \[u\in H^{\theta,p}(0,\sigma_n,w_{\kappa};X_{1-\theta})\cap C([0,\sigma_n];X_{1-\frac{1+\kappa}{p},p}) \ \ \text{a.s.},\]
  and one has the following instantaneous regularity
  \begin{align}\label{eq:instant}
  u\in H^{\theta,p}_{\rm loc}((0,\sigma);X_{1-\theta})\cap C((0,\sigma);X_{1-\frac{1}{p},p}) \ \ \ \text{a.s.}
  \end{align}
  \item if $p=2$, then for all $n\geq 1$,
  \[u\in L^2(0,\sigma_n;X_{1})\cap C([0,\sigma_n];X_{\frac12}) \ \ \text{a.s.}\]
\end{itemize}
\item\label{it2:localwellposed} {\em (Localization)} Let $v_0\in L^0_{\F_0}(\Omega;X_{1-\frac{1+\kappa}{p},p})$. If $(v,\tau)$ is the $L^p_{\kappa}$-maximal solution to \eqref{eq:SEE} with initial value $v_0$, then a.s.\ on the set $\{u_0=v_0\}$ one has $\tau=\sigma$ and $u = v$ on $[0,\sigma\wedge \tau)$.
\end{enumerate}
\end{theorem}

The results mentioned above were established in \cite[Theorem 4.8]{AV19_QSEE_1}, which also addressed the quasilinear case and considered $(t,\omega)$-dependent coefficients. The proof relies on a Banach fixed point argument applied to a truncated version of the equation, combined with suitable stopping time techniques. For detailed proofs the reader is referred to \cite[Sections 4.3-4.5]{AV19_QSEE_1}.

Assumption \ref{ass:FGcritical} is slightly weaker than the corresponding conditions in \cite{AV19_QSEE_1}, owing to the use of the real interpolation spaces $X_{\beta_j,1}$. However, the arguments from the earlier work extend directly to this setting by leveraging Lemma \ref{lem:funnyXembedding}.
One of the strengths of Theorem \ref{thm:localwellposed} is that it is very general but still powerful/optimal in concrete situations. We will demonstrate its applicability to a wide range of parabolic SPDEs to obtain local existence and uniqueness.

One can already see some of the uses of the weight $w_{\kappa}(t) = t^{\kappa}$ in Theorem \ref{thm:localwellposed}. It can be used to enlarge the class of initial values one can consider by choosing $\kappa\in [0,p/2-1)$ large. Moreover, in the latter case, we see the following parabolic instantaneous regularization effect: If $\a>0$ and $X_1$ is strictly contained in $X_0$,
\begin{equation}
\label{eq:improved_regularity_small_times}
u(t)\in X_{1-\frac{1}{p},p}\subsetneq X_{1-\frac{1+\a}{p},p} \text{ a.s.\ on $0<t<\sigma$, even if } u_0\in X_{1-\frac{1+\a}{p},p}\text{ a.s.\ }
\end{equation}
Theorem \ref{thm:parabreg} further explores this regularization effect,
demonstrating that it can be significantly extended as a consequence of blow-up criteria. For a comprehensive discussion, see Section \ref{sec:blowup}.
\smallskip

Finally, we establish the continuity of solutions with respect to the initial data
$u_0$ in the path space defined in \eqref{eq:max_reg_space}. The proof follows directly from the arguments in \cite[Proposition 2.9]{AVreaction-local}. Together with Theorem \ref{thm:localwellposed}, this continuity result ensures the \emph{local well-posedness} of \eqref{eq:SEE} in $X_{1-\frac{1+\a}{p},p}$.

\begin{proposition}[Local continuity]
\label{prop:local_continuity_SEE}
Let the assumptions of Theorem \ref{thm:localwellposed} hold and assume $u_0\in L^p(\O;X_{1-\frac{1+\a}{p},p})$.
Let $(u,\sigma)$ be the maximal $L^p_\a$-solution to \eqref{eq:SEE}. There exist constants $C_0,T_0,\varepsilon_0>0$ and stopping times $\sigma_0,\sigma_1\in (0,\sigma]$ a.s.\ for which the following assertion holds:

For each $v_0\in L^p_{\F_0}(\O;X_{1-\frac{1+\a}{p},p})$ with
$\E\|u_0-v_0\|_{X_{1-\frac{1+\a}{p},p}}^p\leq \varepsilon_0$,
the maximal $L^p_\a$-solution $(v, \tau)$ to \eqref{eq:SEE} with initial data $v_0$ has the property that there exists a stopping time $\tau_0\in (0,\tau]$ a.s.\ such that for all $t\in [0,T_0]$ and $\gamma>0$, one has
\begin{align}
\label{eq:local_continuity_1}
\P\Big(\sup_{r\in [0,t]}\|u(r)-v(r)\|_{X_{1-\frac{1+\a}{p},p}}\geq \gamma, \  \sigma_0\wedge \tau_0>t\Big)
\leq \frac{C_0}{\gamma^p}
\E \|u_0-v_0\|_{X_{1-\frac{1+\a}{p},p}}^p,&\\
\label{eq:local_continuity_2}
\P\Big(\|u-v\|_{L^p(0,t,w_{\a};X_1)} \geq\gamma, \  \sigma_0\wedge \tau_0>t\Big)
\leq \frac{C_0}{\gamma^p}
\E \|u_0-v_0\|_{X_{1-\frac{1+\a}{p},p}}^p,&\\
\label{eq:local_continuity_3}
\P(\sigma_0\wedge \tau_0\leq t)
\leq C_0
\big[\E \|u_0-v_0\|_{X_{1-\frac{1+\a}{p},p}}^p+ \P(\sigma_1\leq t)\big].&
\end{align}
\end{proposition}

The stopping time $\tau_0$ depends on $(u_0, v_0)$. To some extent, the estimates in the above result  \eqref{eq:local_continuity_1}-\eqref{eq:local_continuity_2} show that $(u,\sigma)$
depends continuously on the initial data $u_0$. Meanwhile, \eqref{eq:local_continuity_3} provides a measure of the size of the length of the time interval on which the continuity estimates \eqref{eq:local_continuity_1}-\eqref{eq:local_continuity_2} hold.
A key point is that the right-hand side of \eqref{eq:local_continuity_3} depends on $v_0$, but not on $v$.
In particular, $\{\tau_0\leq t\}$ has small probability if $t\sim 0$ and $v_0$ is close to $u_0$.

For future reference, note that the observation in \cite[Remark 3.4]{AVreaction-local} extends trivially to stochastic evolution equations of the form \eqref{eq:SEE}.

\subsection{Criticality in polynomial nonlinearities}\label{ss:tracecriticalF}

Here we illustrate how to check Assumption \ref{ass:FGcritical} in a common scenario. Specifically, we consider the nonlinearity of Allen--Cahn-type.
This example will also serve as a basis for addressing the Allen--Cahn equation
in Subsections \ref{ss:AllenCahnLpLq}.

\subsubsection{Allen--Cahn-type nonlinearity}
\label{sss:tracecriticalF_AC}
In this example, we demonstrate how to verify Assumption \ref{ass:FGcritical} for the following second-order parabolic PDE, often referred to as \emph{Allen--Cahn} equation
 \begin{equation}
 \label{eq:AC_example_nonlinearity}
 \textstyle \dd u   =  \big( \Delta u +u-u^3\big) \, \dd t
 + \sum_{n\geq 1}  \big[(b_{n}\cdot \nabla) u+ g_{n}(\cdot,u)\big] \, \dd W^n_t \text{ on }\Dom, \quad u=0\text{ on }\partial\Dom,
\end{equation}
 where $\Dom$ is a bounded smooth domain in $\R^d$ with $d\geq 2$. A detailed investigation of this equation is provided in Subsection \ref{ss:AllenCahnLpLq}. The primary motivation for considering \eqref{eq:AC_example_nonlinearity} is that the deterministic Allen--Cahn equation with leading order nonlinearity,  $\partial_t u=\Delta u-u^3$, shares the same local scaling as the Navier--Stokes equations (see Subsection \ref{ss:scaling_intro} and in particular \eqref{eq:NS_scaling_map}). Based on the argument in Subsection \ref{ss:scaling_intro}, we expect critical spaces of the form $B^{d/q-1}_{q,p}$. Below, we verify that Theorem \ref{thm:localwellposed} identifies these spaces correctly.

To simplify the analysis, we consider the weak PDE setting: for $q\in [2, \infty)$, define
\begin{equation}
\label{eq:choice_X0_X1_allencahn}
X_0 = H^{-1,q}(\Dom)\quad \text{ and }\quad X_1=H^{1,q}_0(\Dom)=\{u\in H^{1,q}(\Dom)\,:\, u|_{\partial\Dom}=0\}.
\end{equation}
Here, $X_1$ captures the second-order nature of the SPDE, as it includes two more weak derivatives than $X_0$. The goal of this subsection is to investigate the criticality of the Allen--Cahn-type nonlinearity $F(u)=\pm u^3$ (variants with $F(u) = \pm u^n$ or other functions are possible as well).

To study the mapping property of $u\mapsto F(u)$, note the following inclusion
\begin{equation}
\label{eq:inclusion_AC_intermediate_space_example}
X_{\beta,1}\embed X_{\beta}=[X_0,X_1]_{\beta}\embed [H^{-1,q}(\Dom),H^{1,q}(\Dom)]_{\beta} = H^{-1+2\beta,q}(\Dom).
\end{equation}
Now, we perform the main estimate. Using Sobolev embeddings, we know $L^r(\Dom)\hookrightarrow H^{-1,q}(\Dom)$ for all $r\in (1,\infty)$ satisfying $- \frac{d}{r}\geq -1-\frac{d}{q}$. Since $|F(u) - F(v)|  \leq C (u^2+v^2) |u-v|$, H\"older's inequality implies
\begin{align*}
\|F(u) - F(v)\|_{X_0}\lesssim \|F(u) - F(v)\|_{L^r(\Dom)} &
\lesssim \big\|   (u^2+v^2) |u-v| \big\|_{L^r(\Dom)}
\\ & \lesssim\|u^2+v^2\|_{L^{3r/2}(\Dom)} \|u-v\|_{L^{3r}(\Dom)}
\\ & \lesssim(\|u\|_{L^{3r}(\Dom)}^2+\|v\|_{L^{3r}(\Dom)}^2) \|u-v\|_{L^{3r}(\Dom)}.
\end{align*}
Again, using Sobolev embeddings, $H^{-1+2\beta,q}(\Dom)\hookrightarrow L^{3r}(\Dom)$ holds for all $\beta\in (0,1)$ such that $-1+2\beta -\frac{d}{q}\geq -\frac{d}{3r}$. Combining this with \eqref{eq:inclusion_AC_intermediate_space_example}, we obtain
\[\|F(u) - F(v)\|_{X_0} \lesssim (\|u\|_{X_{\beta,1}}^2+\|v\|_{X_{\beta,1}}^2) \|u-v\|_{X_{\beta,1}}.\]
This matches the form required in Assumption \ref{ass:FGcritical} with $m=1$ and $\rho = \rho_1 = 2$.
However, the choice $\beta_1=\beta$ requires $\beta>1-\frac{1+\a}{p}$. Otherwise, if $\beta<1-\frac{1+\a}{p}$, we select $\beta_1=1-\delta \frac{1+\a}{p}$ with $\delta\in (\frac{\rho}{\rho+1},1)$, ensuring $\beta_1>1-\frac{1+\a}{p}$ and the condition \eqref{eq:subcritical} holds with the strict inequality (i.e.\ is subcritical), see \eqref{eq:subcritical2}$\text{(b)}$.

To analyze the criticality in \eqref{eq:subcritical}, we need to specify the values of $r$ and $\beta$ in the above construction.
To this end, we distinguish the following cases.
Moreover, as it turns out below, to ensure $\beta<1$, in all cases, we have to assume $q>\frac{d}{2}$.

\begin{enumerate}
\item \emph{Case I: $q>\frac{d}{2}$, and either $q\neq 2$ or $d\neq 2$}. In this case, we can choose $r:=\frac{dq}{q+d}$ corresponding to a \emph{sharp} Sobolev embedding $L^r\embed H^{-1,q}$, as the relation $-\frac{d}{r}=-1-\frac{d}{q}$ holds.
Next, for $\beta$, we observe that the sharp Sobolev embedding in $H^{-1+2\beta,q}\embed L^{3r}$ is achieved when
$\beta=\frac{1}{3}+\frac{d}{3q}$. Note that $\beta<1$ since $q>\frac{d}{2}$.
\item \emph{Case II: $q= 2$ and $d= 2$}. In this case, for any $\varepsilon>0$, we can choose $r=1+\varepsilon$ and correspondingly $\beta=\frac{2}{3}+\delta_\varepsilon$, where $\delta_\varepsilon>0$ and $\lim_{\varepsilon\to 0}\delta_\varepsilon=0$.
\end{enumerate}

Next, we examine the form of the condition \eqref{eq:subcritical} in these settings. Before proceeding, it is useful to recall from the scaling argument in Subsection \ref{ss:scaling_intro} that we expect the space for initial data in Theorem \ref{thm:localwellposed} to be given by a Besov-type space with smoothness $\frac{d}{q}-1$ and macroscopic integrability $q$. Note that the condition $q>\frac{d}{2}$ is natural because, when $q\leq \frac{d}{2}$, the smoothness of the critical space is $\frac{d}{q}-1\geq 1$ which equals or exceeds the smoothness of $X_1$ as in \eqref{eq:choice_X0_X1_allencahn}. Moreover, since
$
X_{1-\frac{1+\a}{p},p}
$
is a Besov space with smoothness $1-2\frac{1+\a}{p}$, the condition $\frac{1+\a}{p}<\frac{1}{2}$ implies that the smoothness of the trace space cannot be lower than $0$. Therefore, the expected smoothness for the trace space is $\frac{d}{q}-1$, and this leads to the restriction $q<d$. This formal reasoning can be made precise by focussing on Case I, where sharp Sobolev embeddings are applied (preserving the scaling).
In Case I, the condition \eqref{eq:subcritical} becomes
\begin{equation}
\label{eq:critical_condition_AC}
\frac{1+\kappa}{p}\leq \frac{3(1-\beta)}{2}.
\end{equation}
If equality holds in this condition, then automatically $\beta>1-\frac{1+\a}{p}$. We now examine when equality holds in \eqref{eq:critical_condition_AC} for $\beta$ of Case I. In this case, \eqref{eq:critical_condition_AC} becomes
\begin{equation}
\label{eq:critical_condition_AC_2}
\frac{1+\kappa}{p}= 1-\frac{d}{2q}.
\end{equation}
Since $\frac{1+\kappa}{p}<\frac{1}{2}$, this forces $q<d$, as anticipated. Therefore, in Case I, we find that if
\begin{equation}
\label{eq:condition_parameters_AC}
\tfrac{d}{2}<q<d,\quad\text{ and either $q>2$ or $d\neq 2$},
\end{equation}
then Theorem \ref{thm:localwellposed} ensures local well-posedness with initial data in
\begin{align*}
X_{1-\frac{1+\kappa}{p},p} = \BD^{1 - 2\frac{1+\kappa}{p}}_{q,p}(\Dom)
\stackrel{\eqref{eq:critical_condition_AC}}{=} \BD^{\frac{d}{q}-1}_{q,p}(\Dom),
\end{align*}
where $\BD$ is a Besov space with Dirichlet boundary conditions as discussed in Example \ref{ex:extrapolated_Laplace_dirichlet}. The above trace space has the correct scaling for the Allen--Cahn equation, as explained in Subsection \ref{ss:scaling_intro}. Moreover, by \eqref{eq:BD_identifications}, if we assume $q\in (d-1,d)$, then
$\BD^{\frac{d}{q}-1}_{q,p}(\Dom)=B^{\frac{d}{q}-1}_{q,p}(\Dom)$, i.e.\ no boundary conditions are required for the initial data $u_0$.

Local well-posedness can still be established using Theorem \ref{thm:localwellposed} for all the other cases not considered in \eqref{eq:condition_parameters_AC} but still within Cases I and II.
In these situations, however, the space for the initial data does not exhibit the natural scaling of the Allen--Cahn nonlinearity. As noted earlier, the choice \eqref{eq:condition_parameters_AC} gives the full set of critical spaces achievable within the weak PDE framework, i.e., with the choice in \eqref{eq:choice_X0_X1_allencahn}. As we will see in Subsection \ref{ss:AllenCahnLpLq}, a broader range of critical spaces can be obtained by choosing $X_j = \HD^{2j-\delta,q}(\Dom)$ for $\delta\in [1,2)$ and $j\in \{0, 1\}$.

\smallskip

Finally, let us turn our attention to the diffusion coefficients $g_n(\cdot,u)$. Until now, we have focused primarily on the deterministic nonlinearity. At this stage, it is useful to examine if there is a \emph{critical} growth of $g_{n}$ for which the diffusion term $g_n(\cdot,u)\,\dd W^n_t$ exhibits the same ``roughness'' as $-u^3\,\dd t$. More precisely, we seek $m>1$ such that if $g_n(\cdot,u)=|u|^m$ for some $n\geq 1$, then the term $g_n(\cdot,u)\,\dd W^n_t$ scales similarly to $u^3\,\dd t$. This question can be answered by recalling that the deterministic Allen--Cahn equation is invariant under the Navier--Stokes scaling in \eqref{eq:NS_scaling_map}, i.e.\ for $\lambda>0$,
\begin{equation*}
u_{\lambda}(t,x)=\lambda^{1/2}u(\lambda t,\lambda^{1/2} x),\qquad
 (t,x)\in \R_+\times\R^d,
\end{equation*}
where we (roughly) rescale a ball in the domain to the whole space. Arguing as in \eqref{eq:scaling_noise}, on $\R^d$, the above scaling gives:
\begin{equation}
\label{eq:matching_scaling_AllenCahn}
\begin{aligned}
\textstyle \int_0^{t/\lambda} (u_\lambda(s,x))^3\, \dd s
&= \textstyle \lambda^{1/2}\int_0^{t} (u(s,\lambda^{1/2}x))^3\, \dd s,\\
\textstyle \int_0^{t/\lambda} g_n(\cdot,u_{\lambda}(s,x)) \,\dd  \beta_{s,\lambda}^n
&= \textstyle \lambda^{(m-1)/2}
\int_0^{t} |u(s,\lambda^{1/2} x)|^m \,\dd  W^n_t,
\end{aligned}
\end{equation}
where $\beta_{n,\lambda}=\lambda^{-1/2}W_{\lambda t}^n$ is the rescaled Brownian motion. Hence, the critical growth for the diffusion is $m=2$. In case $(g_n(\cdot,u))_{n\geq1 }$ grows more than quadratically, the prefactor in the scaling for $g$ in \eqref{eq:matching_scaling_AllenCahn} will have a power larger than $\frac{1}{2}$. Therefore, in the latter case, as $\lambda\to \infty$, the stochastic part dominates the deterministic term, and the space of initial data for local well-posedness will only depend on the diffusion $g$. However, we expect that global well-posedness will be problematic in that case.

We leave it to the reader to check that if $g_n:\R\to\ell^2$ satisfies the quadratic type bound $\|(g_n(y)-g_n(y'))_{n\geq 1}\|_{\ell^2}\lesssim (1+|y|+|y'|)|y-y'|$, then $G(u)=(g_n(u))_{n\geq 1}$ satisfies the required estimate in Assumption \ref{ass:FGcritical}.
Further details can be found in Subsection \ref{ss:AllenCahnLpLq}.

\smallskip

The critical relationship between drift and diffusion growth is not limited to the Allen--Cahn situation.
Indeed, it extends to more general polynomial-type nonlinearities as outlined in \cite[Subsection 5.3.2]{AV19_QSEE_1}. In the context of reaction-diffusion equations, the optimal balance between drift and diffusion is captured in Assumption \ref{ass:reaction_diffusion_global}\eqref{it:growth_nonlinearities}.
Additionally, we note tay conservative terms of the form $F(u) = \div\, \Phi(u)$ where $\Phi$ is a vector field, can also be included. Specifically, for the Allen--Cahn equation \eqref{eq:AC_example_nonlinearity}, quadratic growth for $\Phi$ is critical. For other reaction-diffusion equations, this condition is discussed in \cite[Assumption 2.1(4)]{AVreaction-local}).
Conservative terms frequently arise in fluid dynamics.
For further details on these arguments, the reader is encouraged to explore Subsection \ref{ss:CahnHilliard}, Remark \ref{rem:gradientformburger}, and Subsection \ref{subsec:SNSRd}.

\section{Blow-up criteria and instantaneous parabolic regularization}\label{sec:blowup}

\subsection{Blow-up criteria}
\label{subsec:blow_up}
Theorem \ref{thm:localwellposed} provides a robust framework for proving the existence of $L^p_{\kappa}$-maximal (unique) solution $(u,\sigma)$ to \eqref{eq:SEE}. In applications to SPDEs, the question of whether global well-posedness holds  (i.e.\ $\sigma=\infty$ a.s.), is crucial. The occurrence of explosion or blow-up, where $\P(\sigma<\infty)>0$, can correspond to the ``unphysical'' behaviour of the underlying SPDE. Addressing this question is challenging, as our setting also includes complex problems like the 3D Navier--Stokes equations, for which global well-posedness remains largely open \cite{fefferman2000existence}.

Global well-posedness is often tied to subtle energy balances, which in the context of PDEs are expressed through a priori bounds on the lifetime $[0,\sigma)$ of the $L^p_{\kappa}$-maximal solution $(u,\sigma)$.
While abstract theory alone may not fully capture these intricate energy dynamics - often associated with ``sign'' or ``dissipative'' conditions - it is nevertheless effective for establishing general conditions under which a blow-up can occur. Combining such abstract results with energy estimates frequently leads to global well-posedness, see e.g.\ Section \ref{sec:Var} and \cite{Primitive1,Primitive2,AgrSau,AVvar, AVreaction-global}.

A blow-up criterion for $(u,\sigma)$ can be expressed as:
\begin{equation}
\label{eq:blow_up_criteria_X}
\P(\sigma<\infty,\, u\in \mathscr{E}_{\sigma}) = 0,
\end{equation}
where $\mathscr{E}_{\sigma}$ represents a specific property of $u$ or a function space to which $u$ might belong. The subscript $\sigma$ stresses that this property involves the entire lifetime $[0,\sigma)$.

According to Theorem \ref{thm:localwellposed}, the condition $u\in \mathscr{E}_{\sigma}$ essentially imposes constraints on $u(t)$ as $t\approx \sigma$.
From this criterion, one can deduce $\sigma=\infty$ a.s.\ if one can prove $u\in \mathscr{E}_{\sigma}$ a.s.\ on $\{\sigma<\infty\}$.
The latter requires some structure of the SPDE under consideration, which is often encoded in energy estimates.
The efficiency of a blow-up criterion depends on the choice of $\mathscr{E}$. In particular, a blow-up criterion is more practical when the condition $u\in \mathscr{E}_{\sigma}$ is less restrictive. For example, if $\mathscr{E}$ is a function space, it is desirable for $\mathscr{E}$ to be as large as possible.
However, the function space $\mathscr{E}$ cannot be too rough, as the assertion $u\in \mathscr{E}_{\sigma}$ has to prevent possible ill-behaviour of the solutions (e.g.\ explosion or loss of regularity in finite time). As one can show \cite[Subsection 2.2]{CriticalQuasilinear}, the critical setting introduced in Section \ref{sec:loc-well-posed} provides the ``optimal'' framework for establishing blow-up criteria, at least in an abstract sense.

This section is organized as follows.
Below we review the results from \cite[Section 4]{AV19_QSEE_2}, where blow-up criteria for quasilinear stochastic evolution equations were studied. Additionally, we present new results and provide some simplified proofs, which are given in Subsection \ref{ss:proofs_blow_up}. Finally, in Subsection \ref{subsec:reg}, we discuss how blow-up criteria can also be employed to improve the regularity of the $L^p_{\kappa}$-maximal solution $(u,\sigma)$ following the approach of \cite[Section 6]{AV19_QSEE_2}.

The main result of this subsection reads as follows.
\begin{theorem}[Critical case]\label{thm:criticalblowup}
Suppose that the conditions of Theorem \ref{thm:localwellposed} hold with $p\in [2, \infty)$ and $\kappa\in [0,p/2-1)\cup\{0\}$, and let $(u,\sigma)$ be the $L^p_{\kappa}$-maximal solution to \eqref{eq:SEE}.
Then
\begin{enumerate}[{\rm (1)}]
\item\label{it1:criticalblowup}
$\dps \P\big(\sigma<\infty,\, \lim_{t\uparrow \sigma} u(t) \ \text{exists in $X_{1-\frac{1+\kappa}{p},p}$}\big) = 0$;
\item\label{it2:criticalblowup} $\dps \P\big(\sigma<\infty,\, \sup_{t\in [0,\sigma)} \|u(t)\|_{X_{1-\frac{1+\kappa}{p},p}}+ \|u\|_{L^p(0,\sigma;X_{1-\frac{\kappa}{p}})} <\infty \big) = 0$.
\end{enumerate}
\end{theorem}

Note that the norms in \eqref{it2:criticalblowup} are well-defined since, by Theorem \ref{thm:localwellposed} and Sobolev embedding with weights (see \cite[Proposition 2.7]{AV19_QSEE_1}):
\begin{align*}
u\in H^{\frac{\kappa}{p},p}_{\rm loc}([0,\sigma), w_{\kappa};X_{1-\frac{\kappa}{p}})\subseteq L^p_{\rm loc}([0,\sigma);X_{1-\frac{\kappa}{p}})\text{ a.s.}
\end{align*}

Theorem \ref{thm:criticalblowup} should be compared with \cite[Theorem 4.10]{AV19_QSEE_2}. Note that \eqref{it1:criticalblowup} is new, and
\eqref{it2:criticalblowup} improves \cite[Theorem 4.10(3)]{AV19_QSEE_2}, as Assumption \ref{ass:FGcritical} is (slightly) weaker compared to the one used in \cite{AV19_QSEE_2}. The deterministic version of  \eqref{it1:criticalblowup} can be found in \cite[Corollary 2.3]{CriticalQuasilinear}. However, our proof differs from the one in \cite{CriticalQuasilinear}, as the approach used in the latter heavily relies on the invariance under time translation of the deterministic version of \eqref{eq:SEE}. Our proof also provides new insights in the case of time-dependent leading operators $(A,B)$. The latter is also new in the deterministic case. Finally, \eqref{it1:criticalblowup} also extends to the quasilinear setting of \cite[Theorem 4.9]{AV19_QSEE_2}. As the proof below shows, one needs an additional (but a rather mild) assumption, i.e.\ the existence of a suitable extension operator for the couple $(X_0,X_1)$. The reader is referred to Step 1 in the proof of  Theorem \ref{thm:criticalblowup}\eqref{it1:criticalblowup} for details.

The blow-up criterion \eqref{it1:criticalblowup} will be applied to 3D Navier--Stokes equations in Subsection \ref{subsec:SNSRd} leading to an ``endpoint'' version of the Serrin type criteria proven in \cite[Theorem 2.9]{AV20_NS}. Moreover, \eqref{it1:criticalblowup} can also be useful when proving bootstrapping results, see Subsection \ref{subsec:reg} and \cite[Section 6]{AV19_QSEE_2}.
Item \eqref{it2:criticalblowup} can be often checked in applications, as the $L^\infty$-norm is easier to check through energy estimates, see e.g.\ \cite{AVreaction-global}.

Next, let us discuss the sharpness of the conditions in Theorem \ref{thm:criticalblowup}. Note that \eqref{it1:criticalblowup} and \eqref{it2:criticalblowup} are equivalent to \eqref{eq:blow_up_criteria_X} with $\mathscr{E}_{t}=C([0,t];X_{1-\frac{1+\a}{p},p})$ and $\mathscr{E}_{t}=L^\infty(0,t;X_{1-\frac{1+\a}{p},p})\cap L^p(0,t;X_{1-\frac{\a}{p}})$, respectively. All the previously mentioned spaces have \emph{space-time Sobolev index}\footnote{The space-time Sobolev index of $L^p(0,t;X_{\theta,r})$ is $\theta-\frac{1}{p}$.} given by $1-\frac{1+\a}{p}$ and therefore they are sharp in the case that the critical condition \eqref{eq:subcritical} holds with equality for some $j$.

It is unclear to us if Theorem \ref{thm:criticalblowup}\eqref{it2:criticalblowup} holds with the $\sup$-norm or equivalently if Theorem \ref{thm:criticalblowup}\eqref{it2:criticalblowup} holds without the term $\|u\|_{L^p(0,\sigma;X_{1-\frac{\kappa}{p}})}$. On this point, comments are given below Problem \ref{prob:NS_blow_up} in the special case of the 3D Navier--Stokes equations.
In the subcritical case, this holds as the following result proven in \cite[Theorem 4.10(3)]{AV19_QSEE_2} shows. We provide a shorter proof in Subsection \ref{ss:proofs_blow_up} below.

\begin{theorem}[Subcritical case]\label{thm:subcriticalblowup}
Suppose that the conditions of Theorem \ref{thm:localwellposed} hold with $p\in [2, \infty)$ and $\kappa\in [0,p/2-1)\cup\{0\}$, and let $(u,\sigma)$ be the $L^p_{\kappa}$-maximal solution to \eqref{eq:SEE}. Then
\[\P\big(\sigma<\infty,\, \sup_{t\in [0,\sigma)} \|u(t)\|_{X_{1-\frac{1+\kappa}{p},p}} <\infty\big) = 0 \ \ \text{if $(p,\kappa)$ is noncritical}.\]
\end{theorem}

Although the above function space $\mathscr{E}_t=L^\infty(0,t;X_{1-\frac{1+\a}{p},p})$ has space-time Sobolev index $1-\frac{1+\a}{p}$, it is \underline{not} sharp since the couple $(p,\a)$ is noncritical.

\smallskip

Before going into the proofs, it is important to note that there are also other blow-up criteria, often referred to as Serrin-type criteria, due to their resemblance to the Serrin condition for the Navier--Stokes equations (see, e.g., \cite{LePi}). In these criteria, the supremum term in Theorem \ref{thm:criticalblowup}\eqref{it2:criticalblowup} can be omitted. For stochastic variants of this criterion, the reader is referred to \cite[Theorem 4.11]{AV19_QSEE_2}. However, as follows from the proof of that result, it is crucial to work with the complex interpolation spaces $X_{\beta_j}$ rather than $X_{\beta_j,1}$ in Assumption \ref{ass:FGcritical} unless $\kappa=0$ and $\rho_j\leq 1$ for all $j$.

\begin{remark}\label{rem:unifombddblow}
To check global well-posedness through any of the stated blow-up criteria, it suffices to consider uniformly bounded $u_0:\Omega\to X_{1-\frac{1+\kappa}{p},p}$. This follows from a localization argument based on Theorem \ref{thm:localwellposed}\eqref{it2:localwellposed}. The reader is referred to \cite[Proposition 4.13]{AV19_QSEE_2} for details.
One advantage of assuming integrable initial data
$u_0$ is that it allows for the application of moment estimates for the solutions.
\end{remark}

Finally, we recall a fundamental property of the time $\sigma$ which follows from any of the blow-up criteria (see \cite[Proposition 4.12]{AV19_QSEE_2}).
\begin{proposition}[Predictability of $\sigma$]\label{prop:predsigma}
Suppose that the conditions of Theorem \ref{thm:localwellposed} hold with $p\in [2, \infty)$ and $\kappa\in [0,p/2-1)\cup\{0\}$, and let $(u,\sigma)$ be the $L^p_{\kappa}$-maximal solution to \eqref{eq:SEE}. Then for any localizing sequence $(\sigma_n)_{n\geq 1}$ for $(u,\sigma)$, one has that for all $n\geq 1$,
\[\P(\sigma<\infty, \sigma_n = \sigma) = 0.\]
\end{proposition}

\subsection{Proof of Theorems \ref{thm:criticalblowup} and \ref{thm:subcriticalblowup}}
\label{ss:proofs_blow_up}
We begin by proving Theorem \ref{thm:subcriticalblowup}, which introduces some tools and techniques that will also be used in the proof of  Theorem \ref{thm:criticalblowup}.

To prove these results, we rely on the following simpler blow-up criterion:
\begin{equation}
\label{eq:blow_up_intermediate}
\P\big(\sigma<T,\, \sup_{t\in [0,\sigma)}\|u(t)\|_{X_{1-\frac{1+\a}{p},p}}+ \|u\|_{L^p(0,\sigma,w_{\a};X_1)}<\infty\big)=0.
\end{equation}
This criterion follows from \cite[Theorem 4.10(1)]{AV19_QSEE_2} and Lemma \ref{lem:funnyXembedding}. We briefly outline the proof of this result. Using stochastic maximal $L^p$-regularity, one can show that on the set appearing in \eqref{eq:blow_up_intermediate}, the limit $\lim_{t\uparrow \sigma} u(t)$ exists. The proof that the probability vanishes then proceeds by contradiction. By restarting the equation at the random time $\sigma$ on some nontrivial subset of $A\subseteq \Omega$, we can extend the lifetime of the solution $u$ on this set $A$ and this leads to a contradiction with the maximality of $\sigma$. The details can be found in \cite[Sections 5.1 and 5.2]{AV19_QSEE_2}.

\subsubsection{Proof of Theorem \ref{thm:subcriticalblowup}}
Before going into the proof, let us comment on the main observation behind the proof of Theorem \ref{thm:subcriticalblowup}.  As in \cite{AV19_QSEE_2}, we argue by contradiction and derive an a priori estimate which contradicts a stronger blow-up criterion (which was proven in \cite{AV19_QSEE_2} arguing by contradiction with the maximality of the $L^p_\a$-maximal solution $(u,\sigma)$ to \eqref{eq:SEE}).
For simplicity, we assume that Assumption \ref{ass:FGcritical} holds with $m=1$ and we set $\rho:=\rho_1$, $\beta:=\beta_1$.
In case of subcritical nonlinearities, by Remark \ref{rem:growth} and Lemma \ref{lem:funnyXembedding}, there exists $\delta\in (0,1)$ such that
\begin{align}
\label{eq:estimate_FG_utau_subcritical}
\|F(u)\|_{L^p(0,\sigma,w_\a;X_0)}
&+\|G(u)\|_{L^p(0,\sigma,w_\a;\g(\mathcal{U},X_{1/2}))}\\
\nonumber
&\leq C\big[1+\|u\|_{L^\infty(0,\sigma;X_{1-\frac{1+\a}{p},p})}^{\rho+1-\delta}\|u\|_{L^p(0,\sigma,w_{\a};X_1)}^{\delta}\big]\\
\nonumber
&\leq C_{\varepsilon}+C_{\varepsilon}\|u\|_{L^\infty(0,\sigma;X_{1-\frac{1+\a}{p},p})}^{(\rho+1-\delta)/(1-\delta)}+\varepsilon \|u\|_{L^p(0,\sigma,w_{\a};X_1)}.
\end{align}
In particular, if $\|u\|_{L^\infty(0,\sigma;X_{1-\frac{1+\a}{p},p})}$ is known to be bounded, then one can estimate $\|u\|_{L^p(0,\sigma,w_{\a};X_1)}$ by combining stochastic maximal $L^p$-regularity and the above with  $\varepsilon$ sufficiently small.

\begin{proof}[Proof of Theorem \ref{thm:subcriticalblowup}]
By Remark \ref{rem:unifombddblow} we may assume $u_0$ is uniformly bounded in $X_{1-\frac{1+\kappa}{p},p}$. It is enough to show that for any $T<\infty$,
\begin{equation}
\label{eq:blow_up_sup_T_sup}
\P\big(\sigma<T,\, \sup_{t\in [0,\sigma)} \|u(t)\|_{X_{1-\frac{1+\a}{p},p}}<\infty\big) = 0.
\end{equation}
For simplicity, we assume that Assumption \ref{ass:FGcritical} holds with $m=1$. The general case is analogous.

We prove \eqref{eq:blow_up_sup_T_sup} by contradiction. Thus, suppose that the LHS\eqref{eq:blow_up_sup_T_sup} has a positive probability. In particular, there exists a set $\wt{\O}$ of positive probability and $M\geq 1$ such that
$$
\sup_{\wt{\O}}\sup_{t\in [0,\sigma)} \|u(t)\|_{X_{1-\frac{1+\a}{p},p}}<M,
$$
where we may assume $M$ is so large that $\|u_0\|_{X_{1-\frac{1+\a}{p},p}}<M$ a.s.
Now, we perform the main estimate as described in \eqref{eq:estimate_FG_utau_subcritical}. Note that, a priori, $u$ might not be $L^p(w_\a)$-integrable up to random time $\sigma$. Therefore, we need to localize in time. For $\ell\geq 1$, set
$$
\tau_\ell:=\inf\Big\{t\in [0,\sigma)\,:\,\sup_{s\in [0,t)} \|u(s)\|_{X_{1-\frac{1+\a}{p},p}}\geq M,\ \ \text{or} \ \   \|u\|_{L^p(0,t,w_{\a};X_1)}\geq \ell\Big\},
$$
where $\inf\emptyset:=\sigma$. Note that $(\tau_\ell)_{\ell\geq 1}$ is an increasing sequence of stopping times and therefore $\tau=\lim_{\ell\to \infty}\tau_\ell$ is a stopping time as well. Moreover, it holds that $\P(\tau=\sigma)\geq \P(\wt{\O})>0$.
By the stochastic maximal $L^p$-regularity of $(A,B)$ and \cite[Proposition 3.10]{AV19_QSEE_1}, there exists $K\geq 1$ such that, for all $\ell\geq 1$ and $\varepsilon>0$,
\begin{align*}
&\|u\|_{L^p(\O; L^\infty(0,\tau_\ell;X_{1-\frac{1+\a}{p},p}))}+
\|u\|_{L^p(\O\times (0,\tau_\ell),w_{\a};X_1)}\\
&\leq K\big(
\|u_0\|_{L^p(\O;X_{1-\frac{1+\a}{p},p})}+
\|F(u)\|_{L^p(\O\times (0,\tau_\ell),w_\a;X_0)}
+\|G(u)\|_{L^p(\O\times (0,\tau_\ell),w_\a;\g(\mathcal{U},X_{1/2}))}\big)\\
&\leq K_{\varepsilon}M + KC_{\varepsilon} M^{(\rho+1-\delta)/(1-\delta)}
+ K\varepsilon  \|u\|_{L^p(\O\times (0,\tau_\ell),w_{\a};X_1)},
\end{align*}
where in the last step we used \eqref{eq:estimate_FG_utau_subcritical} and the definition of $\tau_\ell$. Choosing $\varepsilon=(2K)^{-1}$ and letting $\ell\to \infty$, Fatou's lemma gives
$
\|u\|_{L^p(\O; L^\infty(0,\tau;X_{1-\frac{1+\a}{p},p}))}+
\|u\|_{L^p(\O\times (0,\tau),w_{\a};X_1)}<\infty.
$
Since $\P(\tau=\sigma)\geq \P(\wt{\O})>0$ by construction, it follows that
$$
\P\big(\sigma<T,\, \sup_{t\in [0,\sigma)}\|u(t)\|_{X_{1-\frac{1+\a}{p},p}}+ \|u\|_{L^p(0,\sigma,w_{\a};X_1)}<\infty\big)
\geq \P(\wt{\O})>0.
$$
The above leads to the required contradiction with \eqref{eq:blow_up_intermediate}.
\end{proof}

\subsubsection{Proof of Theorem \ref{thm:criticalblowup}}
\label{sss:proof_blow_up_criteria_lim}
Next, we prove Theorem \ref{thm:criticalblowup}.
We begin with part \eqref{it1:criticalblowup}.
As in the previous subsection, we first outline the main idea behind the proof which again starts by examining Lemma \ref{lem:funnyXembedding}.
As above, we assume that Assumption \ref{ass:FGcritical} holds with $m=1$, and we let $\rho:=\rho_1$ and $\beta:=\beta_1$. Without loss of generality, we suppose that $\rho>0$. By contradiction,
if Theorem \ref{thm:criticalblowup}\eqref{it1:criticalblowup} does not hold, then
there exists a set of positive probability $\wt{\O}$ on which $u\in C([0,\sigma];X_{1-\frac{1+\a}{p},p})$. Now, for a random time $\lambda<\sigma$ and any mapping $u_\lambda$,
we can write (recall $w_{\kappa}^\lambda(t) = (t-\lambda)^{\kappa}$)
\begin{align}
\label{eq:estimate_FG_utau}
&\|F(u)\|_{L^p(\lambda,\sigma,w_\a^\lambda;X_0)}+\|G(u)\|_{L^p(\lambda,\sigma,w_\a^\lambda;\g(\mathcal{U},X_{1/2}))}\\
\nonumber
&\stackrel{(i)}{\leq} C (1+\|u\|_{L^{p (\rho+1)}(\lambda,\sigma,w_{\kappa}^\lambda;X_{\beta,1})}^{\rho+1}) \\
\nonumber
&\leq C (1+\|u-u_\lambda\|_{L^{p (\rho+1)}(\lambda,\sigma,w_{\kappa}^\lambda;X_{\beta,1})}^{\rho+1}
+\|u_\lambda\|_{L^{p (\rho+1)}(\lambda,\sigma,w_{\kappa}^\lambda;X_{\beta,1})}^{\rho+1})  \\
\nonumber
&\stackrel{(ii)}{\leq} C(1+\|u-u_\lambda\|_{L^\infty(\lambda,\sigma;X_{1-\frac{1+\a}{p},p})}^{\rho}\|u-u_\lambda\|_{L^p(\lambda,\sigma,w_{\a}^\lambda;X_1)} +
\|u_\lambda\|_{L^{p(\rho+1)}(\lambda,\sigma,w_{\a}^\lambda;X_{\beta,1})}^{\rho+1}),
\end{align}
where the constant $C$ varies from line to line. In $(i)$ we used Remark \ref{rem:growth} and in $(ii)$ Lemma \ref{lem:funnyXembedding}.
Now, mimicking the argument in \eqref{eq:estimate_FG_utau_subcritical}, on $\wt{\O}$, we want to choose $u_\lambda$ such that $\|u-u_\lambda\|^\rho_{L^\infty(\lambda,\sigma;X_{1-\frac{1+\a}{p},p})}$ is as small as needed. If this condition is satisfied, the quantity in front of $\|u-u_\lambda\|_{L^p(\lambda,\sigma,w_{\a}^\lambda;X_1)}$ can be made arbitrarily small, and this term can then be absorbed on the left-hand-side of a corresponding estimate as in the proof of Theorem \ref{thm:subcriticalblowup}. Since $u\in C([0,\sigma];X_{1-\frac{1+\a}{p},p})$ on $\wt{\O}$, one might intuitively choose $u_\lambda=u(\lambda,\om)\in X_{1-\frac{1+\a}{p},p}$ for $\lambda\approx \sigma(\om)$ and $\om\in \wt{\O}$. However, this approach fails because the last term on the RHS of the previous expression is ill-defined as $u(\lambda,\om)\not\in L^{p(\rho+1)}(\lambda,t,w_\a^\lambda;X_{\beta,1})$ for any $t>\lambda$. To resolve this issue, we instead work with suitable extensions of $u(\lambda,\om)$ with $\lambda\approx \sigma(\om)$, as outlined in Step 1 of the following proof.

\begin{proof}[Proof of Theorem \ref{thm:criticalblowup}\eqref{it1:criticalblowup}]
As in the proof of Theorem \ref{thm:subcriticalblowup}, we may assume $u_0$ is bounded. Moreover, it is enough to show that for any $T\in (0,\infty)$,
\begin{equation}
\label{eq:blow_up_sup_T}
\P\big(\sigma<T,\, \lim_{t\uparrow \sigma} u(t) \ \text{exists in $X_{1-\frac{1+\kappa}{p},p}$}\big) = 0.
\end{equation}
Moreover, as above, we assume that Assumption \ref{ass:FGcritical} holds with $m=1$ and we argue by contradiction with the blow-up criterion \eqref{eq:blow_up_intermediate}.
For the sake of clarity, we divide the proof into several steps. In the first step, we introduce and analyze the properties of the extension operator mentioned above.

\emph{Step 1: Let $0\leq \mu_1<\tau_1<\infty$ and $y\in C([\mu_1,\tau_1]; X_{1-\frac{1+\a}{p},p})$. For all $\tau\in [\mu_1,\tau_1)$, set
$$
y_\tau(t):= \Ext (y(\tau))(t-\tau), \ \  t\in [\tau,\infty),
$$
where $\Ext$ is the extension operator defined as
\begin{align*}
\Ext: X_{1-\frac{1+\a}{p},p} \to W^{1,p}(\R_+,w_\a;X_0)\cap L^p(\R_+,w_{\a};X_1), \qquad
 \Ext\, x(t):= (1+t(\lambda+A))^{-1}x.
\end{align*}
Then, it holds that}
$$
\lim_{\tau\to \tau_1} \sup_{\tau<t<\tau_1}\|y (t)-y_\tau (t)\|_{X_{1-\frac{1+\a}{p},p}}=0.
$$
Let us mention that the explicit form of the extension operator used here does not play any role below. However, this is the standard choice in case $\Do(A)=X_1$ for some sectorial operator $A$ on $X_0$.
Now, by Proposition \ref{prop:tracespace}, $\Ext \, x\in C([0,\infty);X_{1-\frac{1+\a}{p},p})$. Moreover, by continuity, for all $\tau<\tau_1$ there exists $\tau_0\in [\tau,\tau_1]$ such that
$$
\sup_{\tau<t<\tau_1}\|y (t)-y_\tau (t)\|_{X_{1-\frac{1+\a}{p},p}}= \|y (\tau_0)-y_\tau (\tau_0)\|_{X_{1-\frac{1+\a}{p},p}}.
$$
As $\Ext\,x (0)=x$, we have
\begin{align*}
y(\tau_0)-y_\tau (\tau_0)
=[y(\tau_0)-y(\tau_1)] &+ [\Ext (y(\tau_1))(0) - \Ext (y(\tau_1))(\tau_0-\tau)]\\
&+ [\Ext(y(\tau_1)- y(\tau))(\tau_0-\tau) ].
\end{align*}
Hence, the claim of Step 1 follows from the above, the continuity of $\Ext$ and
 $y\in C([\tau_0,\tau_1];X_{1-\frac{1+\a}{p},p})$.

\emph{Step 2: Setting up the proof by contradiction.}
By contradiction, we assume that LHS\eqref{eq:blow_up_sup_T} is positive. In particular, there exists a set of positive probability $\O_0\in \F_\sigma$ and a constant $M\geq 0$ such that on $\Omega_0$
\begin{equation*}
\sup_{t\in [0,\sigma)}\|u(t)\|_{X_{1-\frac{1+\a}{p},p}}<M
\quad \text{ and }\quad \lim_{n\to \infty}\sup_{\sigma_n<t<\sigma}\|u(t)-u(\sigma_n)\|_{X_{1-\frac{1+\a}{p},p}}=0,
\end{equation*}
where $(\sigma_n)_{n\geq 1}$ is a localizing sequence for the $L^p_\a$-maximal solution $(u,\sigma)$. Again we may assume $\|u_0\|_{X_{1-\frac{1+\a}{p},p}}<M$ a.s.
Next,
by combining \cite[Lemma 2.7]{AV19_QSEE_2} and Egorov's theorem, there exists a subset $\O_0' \subseteq \O_0$ of positive probability and a sequence of stopping times $(\sigma_n')_{n\geq 1}$ such that $\sigma_n'(\O)$ is a discrete set contained in $[0,T]$, $\sigma_n'<\sigma$ a.s.\ on $\O_0'$,  $\lim_{n\to \infty}\sigma_n'=\sigma$ a.s.\ on $\O_0'$ and
\begin{equation*}
\lim_{n\to \infty}\sup_{\O_0'}\sup_{\sigma_n'<t<\sigma}\|u(t)-u(\sigma_n')\|_{X_{1-\frac{1+\a}{p},p}}=0.
\end{equation*}
The reader is also referred to \cite[Step 1, Theorem 4.9(1)]{AV19_QSEE_2} for an analogous situation.

The discreteness of $\sigma_n'(\O)$ will be used in Step 3 when dealing with stochastic maximal $L^p$-regularity estimates. The use of discrete stopping times through \cite[Proposition 3.11]{AV19_QSEE_2}, allows us to avoid technical issues related to stochastic maximal $L^p$-regularity with time weights located at a (non-discrete) random time.

Coming back to the content of Step 2, note that the above and a final application of Egorov's theorem and Step 1 yield
\begin{align}
\label{eq:condition_1_blow_up}
\sup_{\wt{\O}}\sup_{t\in [0,\sigma)} \|u(t)\|_{X_{1-\frac{1+\a}{p},p}}&\leq M,
\ \ \text{and} \ \ \lim_{n\to \infty} \sup_{\wt{\O}}\sup_{t\in [\sigma_n',\sigma)} \|u(t)-u_{\sigma_n'}(t)\|_{X_{1-\frac{1+\a}{p},p}}=0,
\end{align}
where $\wt{\O}\subseteq \O_0'$ is a set of positive probability and
$u_{\sigma'_n}:=\one_{\{\sigma>\sigma_n'\}}\Ext(u(\sigma_n'))$.
Recall that $\{\sigma>\sigma_n'\}\in \F_{\sigma_n'}$ due to \cite[Lemmas 9.1 and 9.5]{Kal}. In particular, the process $\one_{[\sigma_n',\infty)}
u_{\sigma'_n}$ is progressively measurable.
Finally, by \eqref{eq:condition_1_blow_up} and the boundedness of $\Ext$,
\begin{equation}
\label{eq:usigmaprime_bounds}
\|u_{\sigma'_n}\|_{W^{1,p}(\sigma_n',\infty,w_{\a}^{\sigma_n'};X_0)\cap L^p(\sigma_n',\infty,w_{\a}^{\sigma_n'};X_1)}\lesssim_{A,p,\a} M.
\end{equation}

\emph{Step 3: Conclusion.} The idea is to prove that
\begin{equation}
\label{eq:blow_up_contradiction_1}
\P\big(\sigma<T,\,
\sup_{t\in [0,\sigma)}\|u(t)\|_{X_{1-\frac{1+\a}{p},p}}+ \|u\|_{L^p(0,\sigma,w_{\a};X_1)}<\infty\big)\geq \P(\wt{\O}).
\end{equation}
Thus, if $\P(\wt{\O})>0$, then the above contradicts \eqref{eq:blow_up_intermediate}. By the pathwise regularity in Theorem \ref{thm:localwellposed}, to check \eqref{eq:blow_up_contradiction_1} it suffices to show the existence of a stopping time $\lambda$ such that $\lambda<\sigma$ a.s.\ and
\begin{equation}
\label{eq:blow_up_contradiction_2}
\P\big(\sigma<T,\,
\sup_{t\in [\lambda,\sigma)}\|u(t)\|_{X_{1-\frac{1+\a}{p},p}}+ \|u\|_{L^p(\lambda,\sigma,w^\lambda_{\a};X_1)}<\infty\big)\geq \P(\wt{\O}).
\end{equation}
In the following, we focus on the proof of \eqref{eq:blow_up_contradiction_2}. From \eqref{eq:condition_1_blow_up} it follows that for each $\varepsilon>0$ there exists $N(\varepsilon)\geq 1$ for which
\begin{equation}
\label{eq:sigma'sigmaareclose}
\sup_{\wt{\O}}\sup_{t\in [\sigma_{N(\varepsilon)}',\sigma)} \|u(t)-u_{\sigma_{N(\varepsilon)}'}(t)\|_{X_{1-\frac{1+\a}{p},p}}
<\varepsilon.
\end{equation}
To simplify the notation, we set $\lambda_\varepsilon:=\sigma_{N(\varepsilon)}'$ and $\mathcal{V}_{\varepsilon}:=\{\sigma>\lambda_\varepsilon\}$.
Note that, for each $\varepsilon>0$, the process $v:=\one_{\mathcal{V}_\varepsilon} u|_{[\lambda_\varepsilon,\sigma)}$ is a local $L^p_\a$-solution to
\begin{equation}
\label{eq:SEE_blow_up_proof}
\left\{
\begin{aligned}
&\dd v +A v\, \dd t = \one_{\mathcal{V}_\varepsilon \times[\lambda_\varepsilon,\sigma)} F(u)\,\dd t + (Bv + \one_{\mathcal{V}_\varepsilon \times[\lambda_\varepsilon,\sigma)}G(u))\,\dd W,\\
&v(\lambda_\varepsilon)=\one_{\mathcal{V}_\varepsilon} u(\lambda_\varepsilon).
\end{aligned}
\right.
\end{equation}
Now, the idea is to apply stochastic maximal $L^p$-regularity and argue as in \eqref{eq:estimate_FG_utau} to control the nonlinearities $F(u)$ and $G(u)$. However, in principle, the $L^p(w_\a;X_1)$-norms are not finite up to the stopping time $\sigma$ even on $\wt{\O}$. Thus, another localization argument is needed. For each $\ell\geq 1$, set
\begin{align*}
\tau_{\varepsilon,\ell}:=\inf\Big\{t\in [\lambda_\varepsilon,\sigma)\,:\,
\sup_{t\in [\lambda_\varepsilon,\sigma)} \|u(t)-u_{\lambda_\varepsilon}(t)\|_{X_{1-\frac{1+\a}{p},p}}\geq  \varepsilon \ \text{or} \ \|u\|_{L^p(\lambda_\varepsilon,t,w_{\a}^{\lambda_\varepsilon};X_1)}\geq \ell\Big\} \wedge T,
\end{align*}
where $\inf\emptyset:=\sigma\wedge T$ on $\mathcal{V}_\varepsilon$, and $\tau_{\varepsilon,\ell}=\lambda_{\varepsilon}$ otherwise. Note that $\tau_{\varepsilon,\ell}$ are monotone in $\ell\geq 1$, and converge pointwise to a stopping time $\tau_\varepsilon$ as $\ell\to \infty$. Moreover, due to \eqref{eq:sigma'sigmaareclose} and $\wt{\O}\subseteq \O'_0\subseteq \bigcap_{n\geq 1}\{\sigma>\sigma_n'\}$ (the latter by construction), for all $\varepsilon>0$,
\begin{equation}
\label{eq:tauvarespilon_sigma_positive_probability}
\wt{\Omega} \subseteq \{\tau_\varepsilon=\sigma\wedge T\}\cap \mathcal{V}_\varepsilon.
\end{equation}
Now, by \eqref{eq:SEE_blow_up_proof}, the stochastic maximal $L^p_\a$-regularity of $(A,B)$ and
\cite[Proposition 3.11]{AV19_QSEE_2}, there exists a $K>0$ independent of $\ell\geq 1$ such that
\begin{align*}
&\|u\|_{L^p(\mathcal{V}_\varepsilon;L^\infty(\lambda_\varepsilon,\tau_{\varepsilon,\ell};X_{1-\frac{1+\a}{p},p}))}
+ \|u\|_{L^p (\mathcal{V}_\varepsilon\times (\lambda_{\varepsilon},\tau_{\varepsilon,\ell}),w_{\a}^{\lambda_{\varepsilon}};X_1)}\\
&\leq K\big( \| u(\lambda_\varepsilon)\|_{L^p(\mathcal{V}_\varepsilon;X_{1-\frac{1+\a}{p},p})}
+  \|F(u)\|_{L^p(\mathcal{V}_\varepsilon\times (\lambda_{\varepsilon},\tau_{\varepsilon,\ell}),w_\a^{\lambda_{\varepsilon}};X_0)}
+\|G(u)\|_{L^p(\mathcal{V}_\varepsilon\times(\lambda_{\varepsilon},\tau_{\varepsilon,\ell}),w_\a^{\lambda_{\varepsilon}};\g(\mathcal{U},X_{1/2}))} \big)\\
&\stackrel{(i)}{\leq} CK M
+ C K \big(\E \big[\one_{\mathcal{V}_\varepsilon}\|u-u_{\lambda_{\varepsilon}}\|_{L^\infty(\lambda_{\varepsilon},\tau_{\varepsilon,\ell};X_{1-\frac{1+\a}{p},p})}^{\rho p}\|u-u_{\lambda_{\varepsilon}}\|_{L^p(\lambda_{\varepsilon},\tau_{\varepsilon,\ell},w_\a^{\lambda_{\varepsilon}};X_1)}^p\big]\big)^{1/p} \\
&\stackrel{(ii)}{\leq} CK M
+ C K \varepsilon^\rho\|u-u_{\lambda_{\varepsilon}}\|_{L^p(\mathcal{V}_\varepsilon\times(\lambda_{\varepsilon},\tau_{\varepsilon,\ell}),w_\a^{\lambda_{\varepsilon}};X_1)},
\end{align*}
where in $(i)$ we argued as in \eqref{eq:estimate_FG_utau} and used \eqref{eq:usigmaprime_bounds} as well as Lemma \ref{lem:funnyXembedding}, while $(ii)$ follows from the definition of $\tau_{\varepsilon,\ell}$.
Now, as $\rho>0$ by assumption, choosing $\varepsilon=\varepsilon_\star:= (2CK)^{-1/\rho}$ leads to
\begin{align}
\label{eq:estimate_absod_1halfpart}
\|u\|_{L^p(\mathcal{V}_\varepsilon;L^\infty(\lambda_{\varepsilon_\star},\tau_{\varepsilon_\star,\ell};X_{1-\frac{1+\a}{p},p}))}
&+ \|u\|_{L^p (\mathcal{V}_\varepsilon\times (\lambda_{\varepsilon_\star},\tau_{\varepsilon_\star,\ell}),w_{\a}^{\lambda_{\varepsilon_\star}};X_1)}\\
\nonumber
&\leq CK M
+\tfrac{1}{2}\|u\|_{L^p(\mathcal{V}_\varepsilon\times (\lambda_{\varepsilon_\star},\tau_{\varepsilon_\star,\ell}),w_\a^{\lambda_{\varepsilon_\star}};X_1)},
\end{align}
where we used again the uniform control in $\varepsilon$ in \eqref{eq:usigmaprime_bounds} and Lemma \ref{lem:funnyXembedding}. From the definition of $\tau_{\varepsilon,\ell}$, it follows that $\|u\|_{L^p(\mathcal{V}_\varepsilon\times(\lambda_{\varepsilon},\tau_{\varepsilon,\ell}),w_\a^{\lambda_{\varepsilon}};X_1)}<\infty$. Hence, absorbing the last term on RHS\eqref{eq:estimate_absod_1halfpart} in the corresponding LHS and letting $\ell\to \infty$, Fatou's lemma implies
\begin{equation*}
\|u\|_{L^p(\mathcal{V}_\varepsilon;L^\infty(\lambda_{\varepsilon_\star},\tau_{\varepsilon_\star};X_{1-\frac{1+\a}{p},p}))}
+ \|u\|_{L^p (\mathcal{V}_\varepsilon\times (\lambda_{\varepsilon_\star},\tau_{\varepsilon_\star}),w_{\a}^{\lambda_{\varepsilon_\star}};X_1)}<\infty.
\end{equation*}
Thus, \eqref{eq:blow_up_contradiction_2} with $\lambda=\lambda_{\varepsilon_\star}$ follows from the above and \eqref{eq:tauvarespilon_sigma_positive_probability}. As explained at the beginning of Step 3, this concludes the proof of Theorem \ref{thm:criticalblowup}\eqref{it1:criticalblowup}.
\end{proof}

To prove Theorem \ref{thm:criticalblowup}\eqref{it2:criticalblowup} one can, in principle, follow the argument in \cite{AV19_QSEE_2}. However, since the conditions on $F$ and $G$ in Assumption \ref{ass:FGcritical} are weaker since they are formulated in terms of the real interpolation spaces $X_{\beta_j,1}$ instead of the complex interpolation spaces $X_{\beta_j}$, we need a variant of \cite[Lemma 5.11]{AV19_QSEE_2}, where on the left-hand side of the estimates real interpolation is used. For unexplained notation, the reader is referred to
\cite[Section 5]{AV19_QSEE_2}.

\begin{lemma}[Interpolation inequality]\label{lem:interpolationineqMR0}
Let $p\in (1, \infty)$, $\kappa\in [0,p-1)$, $\psi\in (1-\frac{1+\a}{p},1)$, and set $\zeta=(1+\a)/\big(\psi-1+\frac{1+\a}{p}\big)$. Then there exists a $\theta_0\in [0,\frac{1+\a}{p})$ such that for all $\theta\in [\theta_0, 1)$, there is a constant $C>0$ such that the following estimate holds for all $0\leq a<b\leq T$ and all $f\in \Sz^{\theta,\a}(a,b)\cap L^{\infty}(a,b;X_{1-\frac{1+\kappa}{p},p})\cap L^p(a,b;X_{1-\frac{\a}{p}})$,
\begin{equation}
\label{eq:interpolation_inequality_serrin_zeta}
\|f\|_{L^{\zeta}(a,b,w_{\a};X_{\psi,1})}\leq C \|f\|_{L^{\infty}(a,b;X_{1-\frac{1+\kappa}{p},p})}^{1-\phi}\|f\|_{\Sz^{\theta,\a}(a,b)}^{(1-\delta)\phi}\|f\|_{L^{p}(a,b;X_{1-\frac{\a}{p}})}^{\delta\phi},
\end{equation}
where we can take $\delta\in (0,1]$ and $\phi\in [0,1]$ such that
\begin{equation}\label{eq:identity1mindeltaphi}
(1-\delta)\phi \leq \frac{p}{1+\a}\Big(\psi-1+\frac{1+\a}{p}\Big).
\end{equation}
\end{lemma}
\begin{proof}
As in \cite[Lemma 5.11]{AV19_QSEE_2}, by interpolation it suffices to consider $\theta=\theta_0$. Let $\mu$ be such that $\mu>\psi$ and $\mu\in (1-\frac{\kappa}{p}, 1)$. Then by the reiteration theorem $\|x\|_{X_{\psi,1}}\leq C \|x\|_{X_{1-\frac{1+\kappa}{p},p}}^{(1-\lambda)} \|x\|_{X_{\mu}}^{\lambda}$,
for all $x\in X_{1}$, where $\lambda\in (0,1)$ satisfies $\psi = (1-\lambda) ( 1-\frac{1+\kappa}{p}) + \lambda \mu$.  Therefore,
\begin{align*}
\|f\|_{L^{\zeta}(a,b,w_{\a};X_{\psi,1})} \leq C \|f\|_{L^\infty(a,b;X_{1-\frac{1+\kappa}{p},p})}^{1-\lambda} \|f\|_{L^{\lambda \zeta}(a,b,w_{\a};X_{\mu})}^{\lambda}.
\end{align*}
Note that $\lambda \zeta = (1+\a)/\big(\mu-1+\frac{1+\a}{p}\big)$.
Thus, we have reduced the problem to estimating $\|f\|_{L^{\lambda \zeta}(a,b,w_{\a};X_{\mu})}$. The latter case has already been considered in \cite[Lemma 5.11(1)]{AV19_QSEE_2}, where the estimate was proved (where $\phi=1$) and
\[(1-\delta) \leq \frac{p}{1+\a}\Big(\mu-1+\frac{1+\a}{p}\Big) = \frac{p}{\lambda \zeta},
\]
and thus we obtain \eqref{eq:interpolation_inequality_serrin_zeta} with $\phi = \lambda$  and thus \eqref{eq:identity1mindeltaphi} follows from the definition of $\zeta$.
\end{proof}

\subsection{Instantaneous regularization}\label{subsec:reg}

Instantaneous regularization phenomena are a well-known feature of deterministic parabolic PDEs. A classical example is the heat equation, whose solutions become smooth for any positive time $t>0$, regardless of the initial data's regularity. Similar results hold for many deterministic nonlinear PDEs.

For nonlinear (S)PDEs, a standard approach to obtaining regularity is through a bootstrapping argument. The idea is as follows: given $u$ with some initial regularity, one first determines the regularity of $f:=F(u)$ and $g:=G(u)$ and then applies parabolic smoothing to deduce a new regularity estimate for $u$.
If this yields an improvement, the process can be iterated to further enhance the regularity of $u$.

However, in critical settings, this classical bootstrapping method \emph{fails}. Specifically, if $u$ belongs to a scaling-invariant space, then the regularity of the inhomogeneities $f=F(u)$ and $g=G(u)$ is just sufficient to recover the initial regularity of $u$, but not to improve it. In other words, when the regularity is critical or scaling-invariant, bootstrapping does not enhance
the smoothness of $u$ (see Lemma \ref{lem:funnyXembedding}). Concrete instances of this breakdown can be found in \cite[Section 1.4]{AV19_QSEE_2}, \cite[Theorem 2.12]{AV20_NS}, and \cite[Section 7]{AVreaction-local}.

In the deterministic setting, a way to break the limitation of the classical bootstrapping method was introduced by Angenent in \cite{Angenent90Ann,Angenent90proc} (see also \cite[Chapter 5]{pruss2016moving} and \cite{GGHHK20}). This approach, often referred to as the ``parameter trick'', involves first proving \emph{high-order time} regularity, which is then transferred to spatial smoothness using elliptic regularity. Unfortunately, this technique \emph{cannot} be applied to SPDEs because high-order time regularity fails due to the temporal roughness of the noise.

In \cite[Section 6]{AV19_QSEE_2}, we developed a \emph{new} method for bootstrapping regularity in both time and space by exploiting the instantaneous regularization properties of \emph{weighted} function spaces (see Proposition \ref{prop:tracespace}\eqref{it:trace_without_weights_Xp} and \eqref{eq:improved_regularity_small_times}) and the blow-up criteria as in Subsection \ref{subsec:blow_up}. These results have been successfully applied to various equations, including the 2D Navier--Stokes equations \cite[Theorem 2.4]{AV20_NS}, reaction-diffusion equations \cite[Theorem 2.7]{AVreaction-local}, 3D primitive equations \cite[Theorem 3.7]{agresti2023primitive}, and thin-film equations \cite[Proposition 2.13]{AgrSau}. The key idea is to introduce a time weight $\alpha>0$ to balance the desired gain in integrability $r>p$. In critical settings, this adjustment does not disrupt the scaling (i.e.\ the space-time Sobolev index) of the maximal regularity space \eqref{eq:max_reg_space}. The gain in integrability $r>p$ corresponds to choosing $\alpha>0$ such that
\begin{equation}
\label{eq:pr_equal_sobolev_index}
\frac{1}{p}=\frac{1+\alpha}{r}.
\end{equation}
The usefulness of these weights arises from the fact that $w_{\alpha}(t)=t^\alpha$ only ``acts'' at $t=0$. Indeed, by Proposition \ref{prop:tracespace}\eqref{it:trace_without_weights_Xp} and H\"older's inequality,
$$
\textstyle \bigcap_{\theta\in [0,1/2)} H^{\theta,r}(0,t,w_{\alpha};X_{1-\theta}) \subseteq L^r(\varepsilon,t;X_{1-\theta})\cap C([\varepsilon,t];X_{1-\frac{1}{r},r})
$$
for all $\varepsilon\in (0,t)$.
Here, $L^r(\varepsilon,t;X_{1})\cap C([\varepsilon,t];X_{1-\frac{1}{r},r})$ has space-time Sobolev index $1-\frac{1}{r}>1-\frac{1}{p}$ as $r>p$. Thus, for positive times, this method yields a gain in integrability/smoothness.  After this, we can apply the above-mentioned classical bootstrap method. To make the above picture rigorous, one must ensure  ``compatibility'' between the settings $(p,0)$ and $(r,\alpha)$ through blow-up criteria such as the ones in Theorem \ref{thm:criticalblowup}. For the proofs, the reader is referred to Theorem \ref{thm:parabreg2} below or \cite[Proposition 6.8]{AV19_QSEE_2} for the general case.
Finally, let us mention that the important case $p=2$ - relevant for problems like the 2D Navier--Stokes equations  \cite[Theorem 2.4]{AV20_NS} and for 3D primitive equations \cite[Theorem 3.7]{agresti2023primitive} - is technically more challenging and it will not be fully discussed here. When $p=2$, \eqref{eq:pr_equal_sobolev_index} cannot be satisfied for any weight $\alpha\in [0,\frac{r}{2}-1)$ admissible in the stochastic maximal $L^r$-regularity (see Definitions \ref{def:SMR} and \ref{def:SMRbullet}). To overcome this difficulty, in  \cite[Proposition 6.8]{AV19_QSEE_2}, we compensate the increment of time Sobolev index $-\frac{1+\alpha}{r}>-\frac{1}{2}=-\frac{1}{p}$ by ``penalizing'' the spatial regularity $X_{1-\theta}$. This involves ``shifting'' the scale downwards by an amount $\delta:=\frac{1}{2}-\frac{1+\alpha}{r}>0$ to ensure that the space-time Sobolev index of new weighted spaces remains the same as the one of $L^2(0,T;X_1)\cap C([0,T];X_{1/2})$, i.e.\ $\frac{1}{2}$.

After bootstrapping time regularity, the problem transitions to a noncritical setting, where the classical bootstrap method can be applied to achieve spatial regularity. In \cite[Theorem 6.3]{AV19_QSEE_2} we presented an abstract result for this purpose. However, based on experience with its application, we now find it better to perform spatial bootstrapping ``manually''. Moreover, on unbounded domains, the abstract results often impose unnecessary restrictions. Below, we illustrate the classical bootstrapping argument for spatial regularity in a concrete problem, detailed in Proposition \ref{prop:QSregularity}, particularly in Steps 2 and 3. Moreover, the same method is applied to the Navier--Stokes equations on $\R^d$ in Theorem \ref{t:NS_Rd_local}.

\smallskip

Before diving into the above-mentioned unweighted situation, we first discuss the bootstrap in a critical regime where a time weight is already present. This allows us to introduce the basic technical ideas in a more elementary situation.
  The following result on instantaneous time regularization, is a special case of \cite[Corollary 6.5 and Proposition 6.8 with $\delta=0$]{AV19_QSEE_2}.

\begin{theorem}[Parabolic regularization in time for $\kappa>0$]\label{thm:parabreg}
Suppose that the conditions of Theorem \ref{thm:localwellposed} hold with $p\in (2, \infty)$ and $\kappa\in (0,p/2-1)$, and let $(u,\sigma)$ be the $L^p_{\kappa}$-maximal solution to \eqref{eq:SEE}. Suppose that $(A,B)\in \mathcal{SMR}_{r,\alpha}^{\bullet}$ for all $r\in (2, \infty)$ and $\alpha\in [0,r/2-1)$. Then the following path regularity holds a.s.\
\begin{align*}
u\in H^{\theta,r}_{\rm loc}((0,\sigma);&X_{1-\theta})\cap C^{\theta-\varepsilon}_{\rm loc}((0,\sigma);X_{1-\theta}), \ \ \theta\in [0,1/2), r\in [2, \infty), \varepsilon\in (0,\theta).
  \end{align*}
\end{theorem}
\begin{proof}
The H\"older regularity follows from the $H^{\theta,r}$-regularity and Sobolev embedding into $C^{\theta-\frac1r}$ (see \cite[Corollary 14.4.27]{Analysis3}). So it remains to prove the $H^{\theta,r}$-regularity.

In the following, it suffices to consider $r\gg p$ large. Fix an arbitrary $\varepsilon>0$. Note that there exists $\alpha\in [0,r/2-1)$ such that
\begin{equation}
\label{eq:regularity_pkappa}
\frac{1}{p}<\frac{1+\alpha}{r}
<\frac{1+\a}{p}.
\end{equation}
Note that, by Theorem \ref{thm:localwellposed},
\begin{equation}
\label{eq:initial_data_adding_weights_lemma}
 \one_{\V} u(\varepsilon)\in L^0_{\F_{\varepsilon}}(\O;X_{1-\frac1p,p})
\subseteq
L^0_{\F_{\varepsilon}}(\O;X_{1-\frac{1+\alpha}{r},r}),
\end{equation}
where $\V:=\{\sigma>\varepsilon\}$. Again, by
Theorem \ref{thm:localwellposed}, there exists an $L_\a^{p}$-maximal (resp., $L^r_{\alpha}$-maximal) local solution $(v,\tau)$ (resp., $(\wh{v},\wh{\tau})$) to \eqref{eq:SEE} with initial data $\one_{\V}u(\varepsilon)$ at time $\varepsilon$. In particular, a.s.,
\begin{equation}
\label{eq:v_regularization_X_proof}
 \textstyle \wh{v}\in \bigcap_{\theta\in [0,1/2)}H^{\theta,r}(\varepsilon,\wh{\tau},w_{\alpha};X_{1-\theta})\subseteq
  \bigcap_{\theta\in [0,1/2)}
  H^{\theta,r}_{\rm loc}([\varepsilon,\wh{\tau});X_{1-\theta}),
\end{equation}
where the inclusion follows from \cite[Proposition 2.1(1)]{AV19_QSEE_2}.

By \eqref{eq:v_regularization_X_proof} and the arbitrariness of $\varepsilon>0$, it is enough to show that
$$
\sigma=\tau=\wh{\tau} \text{ a.s.\ on }\V\quad \text{ and }
\quad
u=v=\wh{v}\text{ a.e.\ on }[\varepsilon,\tau)\times \V.
$$
Now, we split the proof into two steps. Before doing so, let us comment on the role of $(v,\tau)$, while the one of $(\wh{v},\wh{\tau})$ is evident. The maximal $L^p_\a$-maximal solution $(v,\tau)$ is used because the restriction of $(u,\sigma)$ to times $t\geq \varepsilon$, i.e.\ $(\one_{\V}u,\sigma \one_{\V} + \varepsilon \one_{\V^{{\rm c}}})$, is not maximal anymore.
However, maximality is needed to connect the settings $(p,\a)$ and $(r,\alpha)$. Therefore, we need an intermediate maximal solution $(v,\tau)$ to adjust this.

\emph{Step 1:
$\sigma=\tau$ a.s.\ on $\V$ and $
u=v$ a.e.\ on $[\varepsilon,\tau)\times \V$.}
Note that
$(\one_{\V}u,\sigma \one_{\V} + \varepsilon \one_{\V^{{\rm c}}})$ is also an $L^p_{\kappa}$-local solution to \eqref{eq:SEE} with initial time $\varepsilon$. Thus, by maximality of $\tau$, we conclude that $\sigma \leq \tau$ and $u=v$ a.e.\ on $\V\times [\varepsilon, \sigma)$. On the other hand, on the set $\{\sigma>\tau\}$ by Theorem \ref{thm:localwellposed} we have that
$\lim_{t\uparrow \tau} v(t) = \lim_{t\uparrow \tau} u(t)$ exists in $X_{1-\frac1p,p}\hookrightarrow X_{1-\frac{1+\kappa}{p},p}$. Therefore, by Theorem \ref{thm:criticalblowup},
\begin{align*}
\P(\sigma>\tau) &=
\P\big(\sigma>\tau, \lim_{t\uparrow \tau} v(t) \ \text{exists in $X_{1-\frac{1+\kappa}{p},p}$}\big)
 \leq \P\big(\tau<\infty,\, \lim_{t\uparrow \tau} v(t) \ \text{exists in $X_{1-\frac{1+\kappa}{p},p}$}\big)  = 0.
\end{align*}
Hence, the claim of Step 1 is proved.

\emph{Step 2:
$\tau=\wh{\tau}$ a.s.\ on $\V$ and $
v=\wh{v}$ a.e.\ on $[\varepsilon,\tau)\times \V$.}
The proof of Step 2, is similar to the one of Step 1. Indeed, let us begin by noticing that the second inequality in \eqref{eq:regularity_pkappa} and \eqref{eq:v_regularization_X_proof} implies
\begin{equation}
\label{eq:regularity_v_adding_weight_lemma}
\widehat{v}\in C([\varepsilon,\widehat{\tau});X_{1-\frac{1+\alpha}{r},r})\subseteq C([\varepsilon,\widehat{\tau});X_{1-\frac{1+\a}{p},p})
\end{equation}
as $X_1\embed X_0$. Moreover, from H\"older's inequality \cite[Proposition 2.1(3)]{AV19_QSEE_2} and \eqref{eq:regularity_pkappa}, it follows that $(v,\tau)$ is an $L^p_\kappa$-local solution to \eqref{eq:SEE} with initial time $\varepsilon$.
Hence, from the maximality of $\tau$, it follows that
$
\widehat{\tau}\leq \tau$ a.s.\ on $\V$ and $
\wh{v}=v$ a.e.\ on $[\varepsilon,\widehat{\tau})\times \V.
$
As above, it remains to show $\P(\tau>\widehat{\tau})=0$. By the second inequality in \eqref{eq:regularity_pkappa} and \eqref{eq:instant} in Theorem \ref{thm:localwellposed}\eqref{it1:localwellposed}, it follows that
$\lim_{t\uparrow \tau} \widehat{v}(t) = \lim_{t\uparrow \tau} v(t)$ exists in $X_{1-\frac1p,p}\hookrightarrow X_{1-\frac{1+\alpha}{r},r}$ a.s.\ on $\{\tau>\widehat{\tau}\}$. Therefore, by Theorem \ref{thm:criticalblowup} applied to the $L^r_\alpha$-maximal solution $(v,\tau)$,
\begin{align*}
\P(\tau>\widehat{\tau})
&= \P\Big( \tau>\widehat{\tau}, \lim_{t\uparrow \tau} \wh{v}(t) \ \text{exists in $X_{1-\frac{1+\alpha}{r},r}$}\Big) \leq \P\Big(\tau<\infty,\, \lim_{t\uparrow \tau} \wh{v}(t) \ \text{exists in $X_{1-\frac{1+\alpha}{r},r}$}\Big)  = 0.
\end{align*}
This proves the claim of Step 2 and concludes the proof.
\end{proof}

As explained below \eqref{eq:pr_equal_sobolev_index}, the case $\kappa=0$ is more involved. Indeed, the instantaneous regularization \eqref{eq:instant} provides some extra room as was used in \eqref{eq:regularity_pkappa} and \eqref{eq:regularity_v_adding_weight_lemma}, and does not hold if $\kappa=0$. Under the slightly more restrictive additional condition \eqref{eq:condFcextrakap0}, we will still obtain the following result, which slightly extends \cite[Proposition 6.8 with $\delta=0$]{AV19_QSEE_2}.

\begin{theorem}[Parabolic regularization in time for $\kappa=0$]\label{thm:parabreg2}
Suppose that the conditions of Theorem \ref{thm:localwellposed} hold with $p\in (2, \infty)$ and $\kappa=0$, and let $(u,\sigma)$ be the $L^p_{0}$-maximal solution to \eqref{eq:SEE}. Suppose that
for each $n\geq 1$ there is a constant $C_{n}>0$ such that for all $\|x\|_{X_{1-\frac{1}{p},\infty}}\leq n$,
\begin{align}\label{eq:condFcextrakap0}
\textstyle  \|F(x)\|_{X_{0}} + \|G(x)\|_{\g(\mathcal{U},X_{1/2})}\leq C_{n} \sum_{j=1}^{m}(1+\|x\|_{X_{\beta_j,1}}^{\rho_j+1}),
\end{align}
where $\beta_j,\rho_j$, and $m$ are as in Assumption \ref{ass:FGcritical}.
Suppose that
$(A,B)\in \mathcal{SMR}_{r,\alpha}^{\bullet}$ for all $r\in (2, \infty)$ and $\alpha\in [0,r/2-1)$. Then the following path regularity holds a.s.\
\begin{align}
\label{eq:1r-instantp=2}
  u\in H^{\theta,r}_{\rm loc}((0,\sigma);&X_{1-\theta})\cap C^{\theta-\varepsilon}_{\rm loc}((0,\sigma);X_{1-\theta}), \ \ \theta\in [0,1/2), r\in [2, \infty), \varepsilon\in (0,\theta).
  \end{align}
\end{theorem}

Note that \eqref{eq:condFcextrakap0} does not follow from Assumption \ref{ass:FGcritical} due to the second exponent in the real interpolation space $X_{1-\frac{1}{p},\infty}$ used in the condition $\|x\|_{X_{1-\frac{1}{p},\infty}}\leq n$.

\begin{proof}
Step 1 in the proof of Theorem \ref{thm:parabreg} can be repeated verbatim. As for the second step, let $r\gg p$ and suppose that $\alpha\in (0,\frac{r}{2}-1)$ satisfies \eqref{eq:pr_equal_sobolev_index}. Now, to repeat Step 2 of Theorem \ref{thm:parabreg}, we first need to check that $(\wh{v},\wh{\tau})$ is an $L^{p}_{0}$-local solution to \eqref{eq:SEE}. Note that \eqref{eq:v_regularization_X_proof} still holds in this case and by the trace embedding it follows that $\wh{v}\in C([\varepsilon,\tau);X_{1-\frac{1}{p},r})$ a.s., see Proposition \ref{prop:tracespace}.

We claim that one still has
\begin{equation}
\label{eq:pathwise_regularity_claimed_proof}
\wh{v}\in L^p_{\rm loc}([\varepsilon,\wh{\tau});X_1)\cap C([\varepsilon,\wh{\tau});X_{1-\frac1p,p}).
\end{equation}
As soon as we know this, the second part of the proof of Theorem \ref{thm:parabreg} can be repeated literally as
$
X_{1-\frac{1}{p},p}=X_{1-\frac{1+\alpha}{r},p}\subseteq X_{1-\frac{1+\alpha}{r},r}.
$
To prove \eqref{eq:pathwise_regularity_claimed_proof}, we use maximal $L^p$-regularity once more. Indeed, let $(\wh{\tau}_n)_{n\geq 1}$ be a localizing sequence for $\wh{\tau}$. By Proposition \ref{prop:localizationSMR} it suffices to check that $F(\wh{v})\in L^p(\varepsilon, \wh{\tau}_n; X_0)$ and $G(\wh{v})\in L^p(\varepsilon, \wh{\tau}_n; \gamma(\mathcal{U},X_{1/2}))$ a.s. By \eqref{eq:condFcextrakap0} and \eqref{eq:funnyXrole} it remains to check that $\wh{v}\in  L^{p (\rho_j+1)}(\varepsilon,\wh{\tau}_n;X_{\beta_{j},1})$ a.s.\ for every $j\in \{1, \ldots,m\}$. Fix $j$ and $\theta_j\in (0,1-\beta_j)$. Note that $\theta_j\in (0,\frac{1}{2})$, as $1-\beta_j<\frac1p = \frac{1+\alpha}{r}$ and $\alpha<\frac{r}{2}-1$. By \eqref{eq:v_regularization_X_proof}, a.s., for all $n\geq 1$,
\begin{align*}
\wh{v}\in H^{\theta_j,r}(\varepsilon,\wh{\tau}_n,w_{\alpha}^{\varepsilon};X_{1-\theta_j})\subseteq L^{r_j}(\varepsilon,\wh{\tau}_n;X_{1-\theta_j}),
\end{align*}
where by Sobolev embedding with $-\frac{1}{r_j}=\theta_j - \frac{1+\alpha}{r}$, see e.g.\ \cite[Proposition 2.7]{AV19_QSEE_1}. Thus, a.s.,
\begin{align*}
\wh{v}\in L^{r_j}(\varepsilon,\wh{\tau}_n;X_{1-\theta_j})\cap C([\varepsilon,\wh{\tau}_n];X_{1-\frac1p,r}) \subseteq L^{p (\rho_j+1)}(\varepsilon,\wh{\tau}_n;X_{\beta_{j},1}),
\end{align*}
where the latter can be proved similarly as in Lemma \ref{lem:funnyXembedding} using condition \eqref{eq:subcritical} for $\kappa=0$.
\end{proof}

After either Theorem \ref{thm:parabreg} or \ref{thm:parabreg2} has been applied, one can often apply the more classical procedure to bootstrap regularity in space, see the discussion at the beginning of the section. We will explain such a procedure in detail in Subsection \ref{ss:quasiGSLp} in a particular example.

\begin{remark}
\
\begin{enumerate}[{\rm (1)}]
\item We do not know whether Theorem \ref{thm:parabreg} holds in the important situation where $p=2$ and $\kappa=0$. However, as explained below \eqref{eq:pr_equal_sobolev_index}, this can be fixed by stepping $\delta:=\frac{1+\alpha}{r}-\frac{1}{2}>0$ downwards in the spatial smoothness $(X_0,X_1)$, see \cite[Proposition 6.8]{AV19_QSEE_2}.
\item In light of the new result of Theorem \ref{thm:criticalblowup}\eqref{it1:criticalblowup}, the non-criticality condition in \cite[Theorem 6.3(2)]{AV19_QSEE_2} can be removed.
\end{enumerate}
\end{remark}

In the subcritical regime, the standard bootstrap procedure can also be applied for $p=2$ and $\kappa=0$. This is the content of the following result.

\begin{proposition}[Parabolic regularization in time for $p=2$ and $\kappa=0$]\label{prop:parabreg3}
Suppose that the conditions of Theorem \ref{thm:localwellposed} hold with $p=2$, $\kappa=0$ and let $(u,\sigma)$ be the $L^2_0$-maximal solution to \eqref{eq:SEE}. Assume that there exists $\varphi\in (0,1)$ such that, for all $x\in X_1$ satisfying $\|x\|_{X_{1/2}}\leq n$,
\begin{equation}
\label{eq:subcriticality_p2_proof_boostrap}
\|F(x)\|_{X_0}+ \|G(x)\|_{\g(\mathcal{U},X_{1/2})}\lesssim_n 1+\|x\|_{X_1}^\varphi.
\end{equation}
Finally, suppose that $u_0\in L^0_{\F_0}(\O;X_{\frac12+\varepsilon})$ for some $\varepsilon>0$ and
$(A,B)\in \mathcal{SMR}_{r,\alpha}^{\bullet}$ for all $r\in (2, \infty)$ and $\alpha\in [0,r/2-1)$. Then the following path regularity holds a.s.\
\begin{align*}
u\in H^{\theta,r}_{\rm loc}((0,\sigma);&X_{1-\theta})\cap C^{\theta-\varepsilon}_{\rm loc}((0,\sigma);X_{1-\theta}), \ \ \theta\in [0,1/2), r\in [2, \infty), \varepsilon\in (0,\theta).
  \end{align*}
\end{proposition}

The condition \eqref{eq:subcriticality_p2_proof_boostrap} does not imply that either $F$ or $G$ is sublinear due to the $n$-dependent implicit constant. Moreover, arguing as in Lemma \ref{lem:funnyXembedding}, one can check that \eqref{eq:subcriticality_p2_proof_boostrap} holds in case the setting $(p,\a)=(2,0)$ is \emph{subcritical} for \eqref{eq:SEE}, i.e.\ \eqref{eq:subcritical} holds with the strict inequality for all $j$.
The argument below also works for $p>2$. However, we do not state this here, as Theorems \ref{thm:parabreg} and \ref{thm:parabreg2} cover the cases $p>2$ and also include the critical setting.

\begin{proof}
For each $n\geq 1$, let
$$
\sigma_n:=\inf\{t\in [0,\sigma)\,:\, \|u(t)\|_{X_{1/2}}+\|u\|_{L^2(0,t;X_1)}\geq n \}\wedge n,
$$
where $\inf\emptyset :=\sigma$. It is clear that $(\sigma_n)_{n\geq 1}$ is a localizing sequence of stopping times for $(u,\sigma)$. Let $f_n = \one_{[0,\sigma_n]}F(u)$ and $g_n =   \one_{[0,\sigma_n]}G(u)$. Then by \eqref{eq:subcriticality_p2_proof_boostrap}, a.s.,
\begin{align*}
\textstyle
\|f_n\|_{L^{2/\varphi}(0,\sigma_n;X_0)}^{2/\varphi}+
\|g_n\|_{L^{2/\varphi}(0,\sigma_n;\g(\mathcal{U},X_{1/2}))}^{2/\varphi}
\lesssim_n \int_0^{\sigma_n} \|u(t)\|^{2}_{X_1}\,\dd t \lesssim_n 1.
\end{align*}
Set $\wh{p} := 2/\varphi>2$ and let $\wh{\kappa}\in (0,\wh{p}/2-1)$ be such that $\frac{1+\wh{\kappa}}{\wh{p}}>\frac12-\varepsilon$. Thus, $u_0\in X_{1-\frac{1+\wh{\kappa}}{\wh{p}},\wh{p}}$ a.s. Now, let $v_n$ be the $L^{\wh{p}}_{\wh{\kappa}}$-solution to the following linear initial value problem:
\begin{align*}
\dd v + A v\,\dd t = f_n \,\dd t + (B v +g_n)\,\dd W,\qquad
v(0)=u_{0}.
\end{align*}
Clearly, $v_n$ is also an $L^{2}$-solution to the above problem. Since $u$ is also an $L^2$-solution on $[0,\sigma_n]$, it follows from $\mathcal{SMR}_{2,0}$ that $u=v_n$ a.e.\ on $\O\times [0,\sigma_n]$. Thus, $u$ is an $L^{\wh{p}}_{\wh{\kappa}}$-solution to \eqref{eq:SEE} on $[0,\sigma_n]$.
Since $n$ was arbitrary, $(u,\sigma)$ is a local $L^{\wh{p}}_{\wh{\kappa}}$-solution to \eqref{eq:SEE}. Moreover, it is maximal since $L^{\wh{p}}_{\wh{\kappa}}$-solutions are also $L^2$-solutions. As the conditions of Theorem \ref{thm:localwellposed} hold with $(p,\kappa)=(2,0)$ by assumption, they also hold with $(p,\kappa)$ replaced by $(\wh{p},\wh{\kappa})$ in \eqref{eq:subcritical} (use $\frac{1+\wh{\kappa}}{\wh{p}}<\frac12$). Hence, $(u,\sigma)$ is the unique $L^{\wh{p}}_{\wh{\kappa}}$-solution of \eqref{eq:SEE} given by Theorem \ref{thm:localwellposed}. The stated regularity assertions are now immediate from Theorem \ref{thm:parabreg}.
\end{proof}

Thanks to the regularization results of Theorems \ref{thm:parabreg} and \ref{thm:parabreg2}, or Proposition \ref{prop:parabreg3}, one can ``transfer'' blow-up criteria to a rougher setting. This, in particular, allows one to show the global well-posedness of SPDEs for very rough initial data in case it holds for smooth data. For instance, this procedure has been applied to 2D Navier--Stokes equations in \cite[Theorem 2.12]{AV20_NS} and to reaction-diffusion equations in \cite[Theorem 3.2 and 5.2]{AVreaction-global}.

\begin{corollary}[Transference of blow-up criteria]\label{cor:transfblowup}
Suppose that the conditions of Theorem \ref{thm:localwellposed} hold and let $(u,\sigma)$ be the $L^p_\a$-maximal local solution to \eqref{eq:SEE} provided there. Suppose that for some $r\geq p$, the solution satisfies
\begin{equation}
\label{eq:instantaneous_regularization_r}
u\in  H^{\theta,r}_{\loc}((0,\sigma);X_{1-\theta}) \text{ a.s.\ for all }\theta\in [0,1/2).
\end{equation}
Let $\alpha\in [0,r/2-1)$ be such that $\frac{1+\alpha}{r}\leq \frac{1+\kappa}{p}$.
Then the following assertions hold:
\begin{enumerate}[{\rm(1)}]
\item\label{it1:transfblowup} For all $s>0$,
\begin{align*}
&\P\big(\sigma<\infty,\, \lim_{t\uparrow \sigma} u(t) \ \text{exists in $X_{1-\frac{1+\alpha}{r},r}$}\big) = 0;\\
&\P\big(s<\sigma<\infty,\, \sup_{t\in [s,\sigma)} \|u(t)\|_{X_{1-\frac{1+\alpha}{r},r}} + \|u\|_{L^r(s,\sigma;X_{1-\frac{\alpha}{r}})} <\infty\big)= 0.
\end{align*}
\item\label{it2:transfblowup} If either $(p,\kappa)$ is critical and $\frac{1+\alpha}{r}<\frac{1+\kappa}{p}$, or $(p,\kappa)$ is noncritical, then for all $s>0$,
\[\P\big(s<\sigma<\infty,\,\sup_{t\in [s,\sigma)} \|u(t)\|_{X_{1-\frac{1+\alpha}{r},r}} <\infty\big) = 0.\]
\end{enumerate}
\end{corollary}
The regularization condition \eqref{eq:instantaneous_regularization_r} can be verified through either
Theorem \ref{thm:parabreg}-\ref{thm:parabreg2} and Proposition \ref{prop:parabreg3}, or \cite[Proposition 6.8]{AV19_QSEE_2} in case $p=2,\a=0$, and the setting is critical.

In the above the parameter $s>0$ is necessary because $u_0$ may not belong to $X_{1-\frac{1+\a}{p},p}$, which would make the norms in expressions like \eqref{it2:transfblowup} ill-defined when $s=0$.
However, these norms are well-defined for positive times $s>0$ thanks to the assumed instantaneous regularization.
Note that in \eqref{it1:transfblowup} we automatically avoid $s=0$ since $\sigma>0$ a.s.\ by Theorem \ref{thm:localwellposed}.

The advantage of Corollary \ref{cor:transfblowup} over the blow-up criteria in Subsection \ref{subsec:blow_up} is that the initial pair $(p,\a)$ does not influence the explosion of the solution. In other words, the global well-posedness of \eqref{eq:SEE} is independent of the specific setting considered. Furthermore, one benefit of \eqref{it2:transfblowup} is that $(p, \kappa)$ could be critical, but only the supremum needs to be estimated when applying the result. This approach, for example, appears in \cite[Theorem 3.2]{AVreaction-global}.

In many situations, the argument used in the proof of Corollary \ref{cor:transfblowup} can be extended further. Indeed, for many SPDEs also high-order integrability/smoothness can be established (see references below \eqref{eq:pr_equal_sobolev_index}).
Additionally, one can prove that the blow-up criteria can be transferred between different settings, as demonstrated in \cite[Theorem 2.10]{AVreaction-local}, \cite[Theorem 2.9]{AV20_NS}, \cite[Theorem 3.9]{agresti2023primitive}, and \cite[Proposition 1.8]{AgrSau}.

\begin{proof}[Proof of Corollary \ref{cor:transfblowup}]
The argument below is a variation of the one given in \cite[Theorem 2.10]{AVreaction-local}. We only provide the details for the first statement in \eqref{it1:transfblowup}, the others follow similarly by applying Theorem \ref{thm:criticalblowup}\eqref{it2:criticalblowup} or Theorem \ref{thm:subcriticalblowup}  instead of Theorem \ref{thm:criticalblowup}\eqref{it1:criticalblowup}.

First in \eqref{it1:transfblowup}: Let us begin by recalling that, by Theorem \ref{thm:criticalblowup}\eqref{it1:criticalblowup},
\begin{align}\label{eq:keyblowup}
\P\big(\sigma<\infty,\, \lim_{t\uparrow \sigma} u(t) \ \text{exists in $X_{1-\frac{1+\kappa}{p},p}$}\big) = 0.
\end{align}
Next, we will use the $(r,\alpha)$-setting. Fix $s>0$ and let $\V:=\{\sigma>s\}$.
Note that from the instantaneous regularization assumption (i.e.\ either Theorem \ref{thm:parabreg} or \ref{thm:parabreg2} or Proposition \ref{prop:parabreg3} hold), $u\in C_{\rm loc}((0,\sigma);X_{1-\theta})$ for any $\theta\in (0,1/2)$. In particular, $\one_{\V}u(s)\in L^0_{\F_s}(\Omega;X_{1-\frac{1+\alpha}{r},r})$. By Theorem \ref{thm:localwellposed} (up to translation), we can find an $L^r_{\alpha}$-maximal solution $(v,\tau)$ to \eqref{eq:SEE} with initial value $\one_{\V}u(s)$ at initial time $s$. Again, by the instantaneous regularization assumption,
\begin{align}\label{eq:regvblow}
v\in H^{\theta,\ell}_{\rm loc}((s,\tau);X_{1-\theta})\cap C^{\theta-\varepsilon}_{\rm loc}((s,\tau);X_{1-\theta}) \ \text{a.s.} \ \text{for all $\theta\in [0,1/2)$, $\ell>2$, $\varepsilon\in (0,\theta)$}.
\end{align}
Moreover, due to Theorem \ref{thm:criticalblowup}\eqref{it1:criticalblowup},
\begin{align}\label{eq:blowupv}
\P\big(\tau<\infty,\, \lim_{t\uparrow \sigma} v(t) \ \text{exists in $X_{1-\frac{1+\alpha}{r},r}$}\big) = 0.
\end{align}
To conclude it suffices to show that
$\tau = \sigma$ on $\V$ and $u = v$ on $[s,\sigma)\times \V$.
To prove the latter, we argue as in the proofs of the instantaneous regularization results.
Indeed, note that $(u \one_{\V\times [s,\sigma)},\sigma \one_{\V}+ s\one_{\V^{{\rm c}}})$ is an $L^p_{\kappa}$-local solution to the same equation as $(v,\tau)$. Therefore, by maximality of $\tau$, we obtain $\sigma\leq \tau$ a.s.\ on $\V$ and $u = v$ on $[s,\sigma)\times \V$.
Therefore, it is enough to show $\P(s<\sigma<\tau) = 0$.
On the set $\{s<\sigma<\tau\}$, \eqref{eq:regvblow} gives that $u = v\in C((s,\sigma];X_{1-\theta})$ a.s.\ for all $\theta\in[0,1/2)$.
Thus,
\begin{align*}
\P(s<\sigma<\tau) & =
\P\big(s<\sigma<\tau,\lim_{t\uparrow \sigma} u(t) \ \text{exists in $X_{1-\frac{1+\a}{p},p}$}\big)
\\ &
\leq \P\big(\sigma<\infty,\,
\lim_{t\uparrow \sigma} u(t) \ \text{exists in $X_{1-\frac{1+\a}{p},p}$}\big) = 0,
\end{align*}
where in the last equality we used \eqref{eq:keyblowup}.
\end{proof}

We conclude this section by addressing the question of whether two solutions arising from two different choices of $(p,\a)$ actually lead to the same solution, i.e.\ whether different settings are \emph{compatible}. This is the right time to explore this question, as the instantaneous regularization results presented above suggest that different choices of the parameters $(p,\a)$ lead to the same regularity at least for positive times.

The result below is a version of \cite[Proposition 3.5]{AVreaction-local} for \eqref{eq:SEE}. Note that \cite[Proposition 3.5]{AVreaction-local} is not included in the following result, as in \cite{AVreaction-local} we also vary the spatial regularity (which here is encoded in $X_\theta$).
Of course, we could also cover \cite[Proposition 3.5]{AVreaction-local} by upgrading the level of generality and also varying the ground space $(X_0,X_1)$, and hence obtain compatibility in terms of all parameters $(p,\a,X_0,X_1)$. However, we prefer to present the following compatibility result in a simplified setting.

\begin{corollary}[Compatibility]\label{cor:comp}
Suppose that the conditions of Theorem \ref{thm:localwellposed} hold for two pairs of parameters $(p_1,\a_1)$ and $(p_2,\a_2)$, and let $(u_1,\sigma_1)$ and $(u_2,\sigma_2)$ be the corresponding maximal $L^{p_i}_{\a_i}$-solutions for $i=1$ and $i=2$. Assume that $p_1\leq p_2$ and that the solution $(u_1,\sigma_1)$ instantaneously regularizes to
\begin{equation}
\label{eq:instantaneous_regularization_p2}
u_1\in  H^{\theta,p_2}_{\loc}((0,\sigma_1);X_{1-\theta}) \text{ a.s.\ for all }\theta\in [0,1/2).
\end{equation}
Then $\sigma_1=\sigma_2$ a.s.\ and $u_1=u_2$ a.e.\ on $[0,\sigma_1)\times \O$.
\end{corollary}

The assumption \eqref{eq:instantaneous_regularization_p2} is only for positive times and allows us to connect the $(p_1,\a_1)$-setting to the $(p_2,\a_2)$-one. We do not expect compatibility to hold without any sort of instantaneous regularization. As before the regularization can be checked through either
Theorem \ref{thm:parabreg}-\ref{thm:parabreg2}, and Proposition \ref{prop:parabreg3} or \cite[Proposition 6.8]{AV19_QSEE_2} in case $p=2,\a=0$, and the setting is critical. From the proof below, one can check that the condition \eqref{eq:instantaneous_regularization_p2} can be relaxed. It is only used to prove that $(u_1,\sigma_1)$ is a local $L^{p_2}_{\a_2}$-solution to \eqref{eq:SEE} in the $(p_2,\a_2)$-setting. Indeed, by Lemma \ref{lem:funnyXembedding}, the assumption \eqref{eq:instantaneous_regularization_p2} can be replaced by $u_1\in \bigcap_{i\in \{1,\dots,m\}} L^{p_2(\rho_i+1)}_{\loc}((0,\sigma_1);X_{\beta_i})$ a.s.

\begin{proof}
As the assumptions of Theorem \ref{thm:localwellposed} hold for $(p_i,\a_i)$ with $i\in \{1,2\}$, it follows that $u_{0}\in \cap_{i\in \{1,2\}} X_{1-\frac{1+\a_i}{p_i},p_i}$ a.s. By localization (i.e.\ Theorem \ref{thm:localwellposed}\eqref{it2:localwellposed}), it is enough to consider $u_0\in \cap_{i\in \{1,2\}} L^{p_i}(\O; X_{1-\frac{1+\a_i}{p_i},p_i})$.
We claim that there exists a stopping time $\tau\in (0,\sigma_1\wedge \sigma_2)$ a.s.\ and
\begin{equation}
\label{eq:equality_u1u2_local_time}
u_1=u_2 \text{ a.e.\ on }[0,\tau)\times \O.
\end{equation}
Due to Proposition \ref{prop:local_continuity_SEE} and comments below it, \cite[Remark 3.4]{AVreaction-local} extends to the stochastic evolution equation \eqref{eq:SEE}, the claim \eqref{eq:equality_u1u2_local_time} follows verbatim from the arguments in Steps 1 and 2 of \cite[Proposition 3.5]{AVreaction-local}.
Next, we show how \eqref{eq:equality_u1u2_local_time} leads to the claim of Corollary \ref{cor:comp}. By \eqref{eq:instantaneous_regularization_p2} and the claim \eqref{eq:equality_u1u2_local_time}, we obtain
$$
u_1 \in  H^{\theta,p_2}_{\rm loc}([0,\sigma_1),w_{\a_2};X_{1-\theta}) \text{ a.s.\ for all }\theta\in [0,1/2).
$$
In particular, $(u_1,\sigma_1)$ is a local $L^{p_2}_{\a_2}$-solution to \eqref{eq:SEE}. By the maximality of $(u_2,\sigma_2)$, we obtain that $\sigma_1\leq \sigma_2$ a.s.\ and $u_1=u_2$ a.e.\ on $[0,\sigma_1)\times \O$. It remains to show $\sigma_1=\sigma_2$. Note that, on $\{\sigma_1<\sigma_2\}$, $u_1=u_2\in C((0,\sigma_1];X_{1-\frac{1}{p_2},p_2})$. Since $p_2\geq p_1$, it follows that $\lim_{t\uparrow \sigma_1} u_1(t)$ exists in $X_{1-\frac{1}{p_1},p_1}=:Y$ on $\{\sigma_1<\sigma_2\}$, we find
\begin{align*}
\P(\sigma_1<\sigma_2)
&=
\P\big(\sigma_1<\sigma_2,  \, \lim_{t\uparrow \sigma_1} u_1(t) \text{ exists in $Y$ }\big)\leq \P\big(\sigma_1<\infty,  \, \lim_{t\uparrow \sigma_1} u_1(t) \text{ exists in $Y$}\big)=0,
\end{align*}
where in the last step we used Theorem \ref{thm:criticalblowup}\eqref{it1:criticalblowup}.
\end{proof}

\section{Critical variational setting}\label{sec:Var}
In this section, we focus on the variational setting, for which a collection of references can be found in the introduction. The strength of this approach comes from the fact that it immediately gives global well-posedness under a coercivity condition.

Traditionally, the variational setting relies on a monotonicity condition imposed on the coefficients. However, this condition is violated in many practical examples. To address this limitation, we present a variant of the critical variational framework introduced in \cite{AVvar}, which circumvents the need for monotonicity assumptions by instead imposing a suitable local Lipschitz condition.
For completeness, let us recall that $F:V\to V^*$ is {\em weakly monotone} if there exists a constant $C$ such that $\lb u-v,F(u) - F(v)\rb\leq C\|u-v\|_H$ for all $u,v\in V$. For example, it fails for the Allen--Cahn equation in the strong setting, the Cahn--Hilliard equation, and the 2D Navier--Stokes equations (see Section \ref{sec:appl}).

In Theorem \ref{thm:varloc}, we establish local well-posedness in the variational setting, while Theorem \ref{thm:varglobal} provides global well-posedness under a coercivity condition. We emphasize that we do not assume compactness of the embedding $V\hookrightarrow H$. As before, for simplicity, we restrict our presentation to the case where the problem is $(t,\omega)$-independent and focus exclusively on the semilinear setting.

Recent work in \cite{BGV} extends the results of \cite{AVvar} to the case of L\'evy noise, additionally allowing for singular drifts and introducing a more flexible local Lipschitz condition.

\subsection{Setting}\label{ss:settVHV}
Let $(V, H, V^*)$ be a Gelfand triple of real Hilbert spaces, i.e.\ $V\hookrightarrow H\hookrightarrow V^*$ are dense and continuous, and the duality pairing between $V$ and $V^*$ satisfies $\lb v, x\rb = (v,x)_H$ for all $v\in V$ and $x\in H$. Moreover, it holds that $[V^*, V]_{1/2}  =(V^*, V)_{1/2,2} = H$. For further details, see \cite{AVvar}, and  \cite[Section 2.1]{theewis2024large}.
In relation to the framework of Subsection \ref{ss:setting}, we set $X_0 = V^*$, $X_1 = V$, and consequently, $X_{1/2} = H$.

We adopt the notations $V_{\beta,1} = (V^*,V)_{\beta,1}$ and $V_{\beta} = [V^*,V]_{\beta}$ with the corresponding norms given by
\[\|x\|_{\beta,1} = \|x\|_{(V^*,V)_{\beta,1}} \ \ \text{and}\ \ \|x\|_{\beta} = \|x\|_{[V^*,V]_{\beta}}.
\]
\begin{assumption}\label{ass:varsetting}
Let $A\in \calL(V, V^*)$ and $B\in \calL(V, \calL_2(\mathcal{U},H))$. Suppose that the following hold:
\begin{enumerate}[{\rm (1)}]
\item\label{it1:varsetting} There exist $\theta>0$ and $M\geq 0$ such that for all $u\in V$,
\[\lb u, Au\rb -\tfrac12 \|B u \|^2_{\calL_2(\mathcal{U},H)}\geq \theta \|u\|_V^2 - M \|u\|_H^2;\]
\item\label{it2:varsetting} For each $n\geq 1$, there exists an $L_n\geq 0$ such that for all $u,v\in V$ with $\|u\|_{H}, \|v\|_{H}\leq n$
\[\textstyle \|F(u) - F(v)\|_{V^*} + \|G(u) - G(v)\|_{\calL_2(\mathcal{U},H)} \leq L_n \sum_{j=1}^{m} (1+\|u\|_{\beta_j,1}^{\rho_j} + \|v\|_{\beta_j,1}^{\rho_j}) \|u-v\|_{\beta_j,1},\]
where $\beta_j\in (1/2,1)$ and $\rho_j\geq 0$  satisfy
\begin{align}\label{eq:subcriticalvar}
(2\beta_j-1)(\rho_j+1)\leq 1, \ \ j\in \{1, \ldots, m\}.
\end{align}
\end{enumerate}
\end{assumption}

Part \eqref{it1:varsetting} above asserts that the {\em linear part} of \eqref{eq:SEE} is coercive, which in turn ensures that $(A,B)$ has stochastic maximal $L^2$-regularity (see Theorem \ref{thm:varlin}). Moreover, from the discussion following that theorem, it follows that $-A$ generates a strongly continuous analytic semigroup on $V^*$.
Part \eqref{it2:varsetting} is equivalent to the condition $2\beta_j  \leq \frac{2+\rho_j}{\rho_j+1}$,
and is a special case of Assumption \ref{ass:FGcritical} with $p=2$ and $\kappa=0$. We emphasize that no further growth bounds are assumed on $F$ and $G$.

\subsection{Main results}
Below we leave out the prefix $L^2_0$ in the solution concept, as we consistently assume $p=2$ and $\kappa=0$. Consequently, whenever we refer to a maximal solution in this section, we mean an $L^2_0$-maximal solution.

\subsubsection{Local existence, uniqueness and regularity}
The following theorem gives very general conditions for local existence and uniqueness, and provides three blow-up criteria for global existence.
\begin{theorem}[Local existence, uniqueness, and blow-up]\label{thm:varloc}
Suppose that Assumption \ref{ass:varsetting} holds. Then for each $u_0\in L^0_{\F_0}(\Omega;H)$ there exists a maximal solution $(u,\sigma)$ to \eqref{eq:SEE} with $\sigma>0$ a.s.\ and
\begin{align*}
\P\big(\sigma<\infty, \sup_{t\in [0,\sigma)}\|u(t)\|_H + \|u\|_{L^2(0,\sigma;V)}<\infty\big) &= 0;
 \\ \P\big(\sigma<\infty, \lim_{t\uparrow \sigma} u(t) \ \text{exists in $H$}\big) & = 0;
\\ \P\big(\sigma<\infty, \sup_{t\in [0,\sigma)}\|u(t)\|_H <\infty\big) &= 0 \ \text{if the setting is noncritical}.
\end{align*}
\end{theorem}
\begin{proof}
From Assumption \ref{ass:varsetting}\eqref{it1:varsetting} (and the discussion below it), it follows that $A$ is a sectorial operator of angle $<\pi/2$, and $(A,B)\in \mathcal{SMR}_{2,0}^{\bullet}$. Therefore, the local existence and uniqueness follow from Theorem \ref{thm:localwellposed} with $p=2$ and $\kappa=0$. The blow-up criteria are immediate from Theorem \ref{thm:criticalblowup}.
\end{proof}

The following result establishes regularization properties in the subcritical setting, assuming additional stochastic maximal regularity conditions on the linear part $(A,B)$. It follows directly from Proposition \ref{prop:parabreg3}.
\begin{corollary}[Regularity in the subcritical setting]\label{cor:strongvarsub}
Suppose that Assumption \ref{ass:varsetting} holds. Suppose that there is a $\varphi\in (0,1)$ such that for every $n\geq 1$ there is a constant $C_n\geq 0$ such that
\begin{align*}
\textstyle \|F(v)\|_{V^*} + \|G(v)\|_{\gamma(\mathcal{U},H)} \leq C_n (1+\|v\|_{V}^\varphi), \ \ \text{for all $v\in V$ with $\|v\|_{H}\leq n$}.
\end{align*}
Suppose there is an $\varepsilon\in (0,1)$ such that $u_0\in L^0_{\F_0}(\Omega;[H,V]_{\varepsilon})$ a.s.\ and $(A,B)\in \mathcal{SMR}_{r,\alpha}^{\bullet}$ for all $r\in (2, \infty)$ and $\alpha\in [0,r/2-1)$. Then the solution $(u,\sigma)$ provided by Theorem \ref{thm:varloc} satisfies
\begin{align*}
  u&\in H^{\theta,r}_{\rm loc}((0,\sigma);[V^*,V]_{1-\theta})\cap C^{\theta-\varepsilon}_{\rm loc}((0,\sigma);[V^*,V]_{1-\theta}), \ \ r\in (2, \infty), \theta\in (0,1/2), \varepsilon\in (0,\theta).
  \end{align*}
\end{corollary}

\subsubsection{Global well-posedness under a coercivity condition}
The following is the primary global well-posedness result in the variational framework. Under a coercivity condition, it provides global existence, uniqueness, and continuity with respect to the initial data. The result can be proved in the same way as \cite[Theorems 3.5 and 3.8]{AVvar}. While the coercivity condition holds for many equations, it also has limitations. In Sections \ref{sec:appl} and \ref{ss:LV}, we present both coercive and noncoercive examples for which we derive global well-posedness. For the noncoercive examples, we cannot rely on Theorem \ref{thm:varglobal}.
\begin{theorem}[Global well-posedness]\label{thm:varglobal}
Suppose that Assumption \ref{ass:varsetting} holds, and there exist $\theta'>0$, $M'\geq 0$ such that
\[\lb v,Av - F(v)\rb -  \tfrac{1}{2} \|B v + G(v)\|_{\calL_2(\mathcal{U},H)}^2 \geq \theta' \|v\|_V^2 - M'\|v\|^2_H-M', \ \ v\in V.\]
Then for every $u_0\in L^0_{\F_0}(\Omega;H)$ there exists a unique global solution $u\in L^2_{\rm loc}([0,\infty);V)\cap C([0,\infty);H)$ to \eqref{eq:SEE}. Moreover, for every $p\in (0,2)$ and $T>0$ there are constants $C_{T}$ and $C_{p,T}$ such that
\begin{align*}
\textstyle \sup_{t\in[0,T]}\E\|u(t)\|^2_{H} + \E\int_0^T \|u(t)\|^2_V \,\dd  t & \leq C_T(1+\E\|u_0\|_{H}^2),
\\ \textstyle  \E \sup_{t\in [0,T]}\|u(t)\|_{H}^{p} + \E\Big|\int_0^T \|u(t)\|^2_V \,\dd t\Big|^{p/2} & \leq C_{p,T} (1+\E\|u_0\|_{H}^{p}).
\end{align*}
Furthermore, the following continuous dependency on the initial data holds: if $u_0^n \in L^0_{\F_0}(\Omega;H)$ are such that $\|u_0-u_0^n\|_H\to 0$ in probability, then for every $T\in (0,\infty)$,
\[\|u - u^n\|_{L^2(0,T;V)} + \|u - u^n\|_{C([0,T];H)}\to 0 \ \ \text{in probability},\]
where $u^n$ is the unique global solution to \eqref{eq:SEE} with initial data $u_0^n$.
\end{theorem}
It is unclear whether the above maximal estimate holds for $p=2$ (see Problem \ref{prob:p=2var}). However, Subsection \ref{sss:p=2} provides sufficient conditions under which this is possible.

\begin{remark}[Sufficient condition for coercivity]\label{rem:sufficientcoercivity}
Suppose that $\|G(v)\|_{\calL_2(\mathcal{U},H)}\leq C(1+\|v\|_{\beta})$ for some $\beta\in (0,1)$. To check the coercivity condition of Theorem \ref{thm:varglobal} it is enough to check
$\lb v, F(v)\rb \leq M''(\|v\|^2_H+1)$ for all $v\in V$. Indeed, for every $\varepsilon>0$ one has
\begin{align*}
\|B v + G(v)\|_{\calL_2(\mathcal{U},H)}^2 & \leq  (1+\varepsilon)\|B v\|^2_{\calL_2(\mathcal{U},H)} + C_{\varepsilon} \|G(v)\|_{\calL_2(\mathcal{U},H)}^2
\\ & \leq \|B v\|^2_{\calL_2(\mathcal{U},H)}  + \varepsilon\|B\|^2 \|v\|_V^2 + 2C_{\varepsilon} C^2(1+\|v\|_{{\beta}}^2).
\end{align*}
By Assumption \ref{ass:varsetting}\eqref{it1:varsetting} and interpolation estimates we obtain that for every $\delta>0$
\begin{align*}
\lb &v,Av - F(v)\rb -  \tfrac{1}{2} \|B v + G(v)\|_{\calL_2(\mathcal{U},H)}^2 \\ & \geq
\lb v,Av \rb - \tfrac12 \|B v\|^2_{\calL_2(\mathcal{U},H)}  - \tfrac{\varepsilon}{2}\|B\|^2 \|v\|_V^2 -\delta C_{\varepsilon}C^2 \|v\|^2_V - C_{\varepsilon}  C_{\delta} \|v\|_{H}^2 - M''(\|v\|^2_H+1)
\\ & \geq (\theta-\tfrac{\varepsilon}{2}\|B\|^2 + \delta C_{\varepsilon}C^2)\|v\|_V^2 - C_{\delta, \varepsilon}( \|v\|_H^2+1).
\end{align*}
It remains to choose $\varepsilon>0$ small enough, and after that $\delta>0$ small enough.
\end{remark}

\begin{remark}[Large deviations]
Recently, in the setting of Theorem \ref{thm:varglobal}, Theewis and the second named author established a large deviation principle in \cite{theewis2024large}. Notably, no additional conditions are imposed, allowing for a broad class of equations to be covered. This includes all applications from \cite[Section 4]{AVvar} as well as those in Section \ref{sec:appl}, except for part of Subsection \ref{ss:otherfluid}. In particular, the result extends to large deviations on unbounded domains with transport noise - a previously unknown case that is especially relevant in fluid dynamics \cite{SS06, ChueshovMillet, DuMi}.
\end{remark}

\subsection{Higher order moments}\label{ss:highermoments}
In this subsection, we explain how one can extend the $L^p$-moment bounds of Theorem \ref{thm:varglobal} to other values of $p\geq 2$.
\subsubsection{Second order moments}\label{sss:p=2}
One can take $p=2$ in Theorem \ref{thm:varglobal} in case the coercivity is strengthened as follows: there exists an $\eta>0$, $\theta'>0$, $M'\geq 0$ such that for all $v\in V$
\[\lb v,Av - F(v)\rb -  \big(\tfrac{1}{2}+\eta\big) \|B v + G(v)\|_{\calL_2(\mathcal{U},H)}^2 \geq \theta' \|v\|_V^2 - M'\|v\|^2_H - M'.\]
Such an $\eta$ can always be obtained from the usual coercivity condition in case
$\|G(v)\|_{\calL_2(\mathcal{U},H)}\leq C(1+\|v\|_{V})$. Indeed,
\begin{align*}
\|B v + G(v)\|_{\calL_2(\mathcal{U},H)}^2 &\leq 2\|B v\|^2_{\calL_2(\mathcal{U},H)} + 2\|G(v)\|_{\calL_2(\mathcal{U},H)}^2
 \leq (2\|B\|^2+2C^2)\|v\|_V^2 + 2C^2.
\end{align*}
Thus, taking $\eta$ small enough, the coercivity remains essentially unharmed.

\subsubsection{$L^p(\Omega)$-moments with $p>2$}\label{sss:p>2}
Estimates for higher-order moments with $p\in [2, \infty)$
can also be derived; however, they do not follow automatically, as explained in \cite{BrzVer11} for a specific class of examples.

Suppose that $B(v)^* v =0$ for all $v\in V$ and
$\lim_{\|v\|_V\to \infty} \frac{\|G(v)\|_{\calL_2(\mathcal{U},H)}}{\|v\|_V} = 0$. Then there exist constants $C$ and $C_T$ independent of $p\in [2, \infty)$ such that for all $p\in [2, \infty)$, $T\in (0, \infty)$
\begin{align}\label{eq:preciseestpind}
\Big\|\sup_{t\in [0,T]}\|u(t)\|_{H}\Big\|_{L^p(\Omega)} + p^{-1/2} \|u\|_{L^p(\Omega;L^2(0,T;V))} & \leq C(C_T+\|u_0\|_{L^p(\Omega;H)}).
\end{align}
Indeed, under this condition, for every $\varepsilon>0$ we can find $C_{\varepsilon}$ such that $\|G(v)\|_{\calL_2(\mathcal{U},H)}\leq \varepsilon \|v\|_V + C_{\varepsilon}$. This bound is sufficient to verify the required coercivity condition from \cite[Corollary 3.4]{GHV}, which then yields the desired result.

\subsubsection{Exponential moments}\label{sss:exp}
Exponential moments are useful for studying convergence rates in numerical schemes and this was used for stochastic Navier--Stokes equations in \cite{BeMil19,BesMil22}.  From \eqref{eq:preciseestpind} we can obtain exponential moment bounds in the abstract setting of Theorem \ref{thm:varglobal} under the additional hypothesis of Subsection \ref{sss:p>2}. Indeed, letting $p\to \infty$ in \eqref{eq:preciseestpind}, it follows that $\big\|\sup_{t\in [0,T]}\|u(t)\|_{H}\big\|_{L^\infty(\Omega)} \leq C(C_T+\|u_0\|_{L^\infty(\Omega;H)})$ if $u_0\in L^\infty(\Omega;H)$. Moreover, letting $M_{u_0}:=C(C_T+\|u_0\|_{L^\infty(\Omega;H)})$, for all $\varepsilon\in (0,(2M_{u_0}^{2} e)^{-1})$, we have
\begin{align*}
\E \exp\Big(\varepsilon\|u\|_{L^2(0,T;V)}^2\Big)& \leq \sum_{n=0}^\infty \frac{\varepsilon^{n} \E\|u\|_{L^2(0,T;V)}^{2n}}{n!}
 \leq \sum_{n=0}^\infty \frac{(2\varepsilon M_{u_0}^{2} n)^n}{n!}\leq
 \sum_{n=0}^\infty (2\varepsilon M_{u_0}^{2} e)^n = \frac{1}{1-2\varepsilon M_{u_0}^{2} e},
\end{align*}
where we applied the standard estimate $\frac{n^n}{n!}\leq e^n$.

\subsubsection{Regularization through a stronger setting}
In Theorems \ref{thm:parabreg}, \ref{thm:parabreg2} and Proposition \ref{prop:parabreg3} we have explored various regularization results. Here we present a different type of regularization result using a stronger Gelfand triple  $(\wt{V}, \wt{H}, \wt{V}^*)$ such that $\wt{H} = V$ and $\wt{V}^* = H$. When coercivity holds in the $(\wt{V}, \wt{H}, \wt{V}^*)$-setting, one can apply Theorem \ref{thm:varglobal} directly in this smoother framework (see e.g.\ the Allen--Cahn case in Subsection \ref{sss:AllenCahnstrong}). However, in many examples coercivity fails in the stronger setting - such as for Cahn--Hilliard (Subsection \ref{ss:CahnHilliard}), and in fluid dynamics (Subsection \ref{ss:Fluid}).

The bootstrap result will be formulated under the assumption that the solution exists globally in the $(V, H, V^*)$-setting. For example, this holds when the conditions of Theorem \ref{thm:varglobal} are satisfied. The current formulation has the advantage that it can be iterated.
\begin{theorem}[Regularity through the strong setting]\label{thm:strongvar}
Suppose that Assumption \ref{ass:varsetting} holds. Suppose the solution provided by Theorem \ref{thm:varloc} exists globally, i.e.\ $\sigma = \infty$ a.s. Assume that $A$ and $B$ also define bounded operators  $A\in \calL(\wt{V},H)$ and $B\in \calL(\wt{V}, \calL_2(\mathcal{U},V))$, and that $(A,B)$ and $(F, G)$ satisfy Assumption \ref{ass:varsetting} in the case $(V, H, V^*)$ is replaced by $(\wt{V}, \wt{H}, \wt{V}^*)$.
Suppose that there exist constants $\alpha_1, \alpha_2\in [0,1)$, $\gamma_1, \gamma_2\geq 0$, $C\geq 0$ such that
\begin{align*}
\|F(v)\|_{H}& \leq C (1+\|v\|_{\wt{V}})^{\alpha_1}(1+\|v\|_{V})^{2-2\alpha_1} (1+\|v\|_{H})^{\gamma_1}, &v \in V,
\\ \|G(v)\|_{\calL_2(\mathcal{U},V)}& \leq C(1+\|v\|_{\wt{V}})^{\alpha_2}(1+\|v\|_{V})^{2-2\alpha_2}(1+\|v\|_{H})^{\gamma_2}, & v\in V.
\end{align*}
Let $u_0\in L^0_{\F_0}(\Omega;V)$. Then $u\in L^2_{\rm loc}([0,\infty);\wt{V})\cap C([0,\infty);V)$ a.s.
\end{theorem}
Note that the stated regularity can often be further improved by using Corollary \ref{cor:strongvarsub} in the $(\wt{V}, \wt{H}, \wt{V}^*)$-setting since the latter typically is subcritical.
\begin{proof}
It follows from Theorem \ref{thm:varloc} that there exists a maximal solution $(\wt{u}, \wt{\sigma})$ to \eqref{eq:SEE} in the $(\wt{V}, \wt{H})$-setting. Since it is also a local solution in the $(V, H,V^*)$-setting, by uniqueness we obtain $\wt{u} = u$ on $[0,\wt{\sigma})$. We claim that $\wt{\sigma} = \infty$ a.s. To show this we will use the blow-up criteria of Theorem \ref{thm:varloc} in the $(\wt{V}, \wt{H})$-setting.

Fix $T\in (0,\infty)$. Let $(\tau_k)_{k\geq 1}$ be a localizing sequence for $(\wt{u}, \wt{\sigma}\wedge T)$. Then, in particular, $u\in C([0,\tau_k];H)\cap L^2(0,\tau_k;V)$ a.s. Let
\[\textstyle
\sigma_{k} = \inf\Big\{t\in [0,\tau_k]: \|u(t)-u_0\|_{H}\geq k \ \text{or} \ \int_0^t \|u(s)\|_{V}^2 \,\dd  s\geq k \Big\}, \]
where we set $\inf\emptyset  = \tau_k$. Then $(\sigma_k)_{k\geq 1}$ is a localizing sequence for $(\wt{u}, \wt{\sigma})$ as well. Letting $u^k(t) = u(t\wedge \sigma_k)$ we have $u^k\in L^2(\Omega;C([0,T];H))\cap L^2(\Omega;L^2(0,\sigma_k;V))$ and a.s.\ for all $t\in [0,T]$,
\begin{align*}
\textstyle u^{k}(t) = u_0- \int_0^t \one_{[0,\sigma_k]}(s) [A u^k(s) - F(u^k(s))]   \,\dd s + \int_0^t \one_{[0,\sigma_k]}(s) [B u^k(s) + G(u^k(s))] \,\dd  W(s).
\end{align*}

By It\^o's formula (see Lemma \ref{lem:ItoHilbert}) applied to $\frac12\|\cdot\|^2_{V}= \frac12\|\cdot\|^2_{\wt{H}}$ we obtain
\begin{align*}
\textstyle \frac12\|u^{k}(t) \|_{V}^2 &  + \textstyle \int_0^t \one_{[0,\sigma_k]}(s)  \mathcal{E}_{u^k}(s)\,\dd s
 =\frac12\|u_0\|_{V}^2 + \int_0^t \one_{[0,\sigma_k]}(s) [Bu^k(s) + G(u^k(s))]^* u \,\dd W(s),
\end{align*}
where
\[\mathcal{E}_{u^k}(s) = \langle u^k(s), Au^k(s) -F(u^k(s)) \rangle_{\wt{V}, \wt{V}^*} - \frac12\|B u^k(s) + G(u^k(s))\|_{\calL_2(\mathcal{U},V)}^2.\]

By the conditions on $F$ and Young's inequality we obtain for any $\varepsilon>0$
\begin{align*}
\lb u^k, F(u^k)\rb_{\wt{V}, \wt{V}^*} & \leq \|u^k\|_{\wt{V}} \|F(u^k)\|_{\wt{V}^*}
\\ & \leq
C(1+\|u^k\|_{\wt{V}})^{\alpha_1+1} (1+\|u^k\|_{V})^{2-2\alpha_1} (1+\|u^k\|_{H})^{\gamma_1}
\\ & \leq \varepsilon(1+\|u^k\|_{\wt{V}}^{2}) + C_{1,\varepsilon}(1+\|u^k\|_{V})^4
(1+\|u^k\|_{H})^{\wt{\gamma_1}},
\end{align*}
where $\wt{\gamma}_1 = \frac{2\gamma_1}{1-\alpha_1}$. For the It\^o correction for any $\varepsilon>0$ we have
\begin{align*}
\frac12\|B u^k + G(u^k)\|_{\calL_2(\mathcal{U},V)}^2& \leq (\frac12+\varepsilon)\|B u^k\|_{\calL_2(\mathcal{U},V)}^2 + C_{2,\varepsilon}\|G(u^k)\|_{\calL_2(\mathcal{U},V)}^2.
\end{align*}
By the conditions on $G$ and Young's inequality again, we obtain for any $\delta>0$
\begin{align*}
\|G(u^k)\|_{\calL_2(\mathcal{U},V)}^2\leq \delta(1+\|u^k\|_{\wt{V}}^{2}) + C_{1,\delta}(1+\|u^k\|_{V})^4
(1+\|u^k\|_{H})^{\wt{\gamma}_2},
\end{align*}
where $\wt{\gamma}_2 = \frac{2\gamma_2}{1-\alpha_2}$.

Therefore, by the above first choosing $\varepsilon>0$ and then $\delta>0$ small enough, and by the coercivity condition on $(A,B)$ in the $(\wt{V}, \wt{H}, \wt{V}^*)$-setting (with constants $(\wt{\theta}, \wt{M})$), we find that
\begin{equation*}
\begin{aligned}
\textstyle \frac{1}{2} \|u^{k}(t) \|_{V}^2  & \textstyle + \frac{\wt{\theta}}{4} \int_0^t \one_{[0,\sigma_k]}(s)\|u^k(s)\|_{\wt{V}}^2 \,\dd s
\\ \leq & \textstyle \frac{1}{2}\|u_0\|_{H}^2 + C_{\wt{\theta}}\int_{0}^t (1+\|u^k(s)\|_{V})^4 (1+\|u^k(s)\|_{H})^{\wt{\gamma}} \,\dd s
+ M_{\wt{\theta}} \int_0^t (\|u^k(s)\|_{V}^2+1) \,\dd s \\ & \textstyle \qquad +\int_0^t \one_{[0,\sigma_k]}(s) [Bu^k(s) +G(u^k(s))] ^* u^k(s) \,\dd W(s).
\end{aligned}
\end{equation*}
Letting
\begin{align*}
a(t)&= \textstyle 2C_{\wt{\theta}}\int_0^t (1+\|u(s)\|_{V})^2 (1+\|u(s)\|_{H})^{\wt{\gamma}}\,\dd s,
\\ Z^k(t) &= \textstyle \frac{1}{2} (1+\|u^{k}(t) \|_{V}^2)   + \frac{\wt{\theta}}{4} \int_0^t \one_{[0,\sigma_k]}(s)\|u^k(s)\|_{\wt{V}}^2 \,\dd s,
\\ \xi(t)&= \textstyle \frac{1}{2}(\|u_0\|_{H}^2+1) + M_{\wt{\theta}} \int_0^t (\|u(s)\|_{V}^2+1) \,\dd s,
\\ \zeta(t)&=\textstyle  \frac12(1+\sup_{s<\wt{\sigma}\wedge t} \|u(s)\|_V^2) + \frac{\wt{\theta}}{4}\int_0^{t\wedge \sigma} \|u(s)\|_{\wt{V}}^2 \,\dd s,
\end{align*}
we obtain that a.s. for all $t\in [0,T]$,
\begin{align*}
\textstyle Z^k(t)\leq  \xi(t) + \int_0^t Z^k(s) a(s) \,\dd s
+ \int_0^t \one_{[0,\sigma_k]}(s) [Bu^k(s) +G(u^k(s))]^* u^k(s) \,\dd W(s).
\end{align*}
Since $u\in L^2(0,T;V)\cap C([0,T];H)$ a.s., the processes $a$ and $\xi$ are continuous, increasing, and finite a.s.
Applying the stochastic Gronwall Lemma \ref{lem:Gronwall}  we find that
\begin{align*}
\textstyle \P(\sup_{t\in [0,T]}Z^k(t)>\mu)\leq \frac{e^{R}}{\mu} \E(\xi(T)\wedge \lambda) + \P(\xi(T)\geq \lambda) + \P(a(T)>R),
\end{align*}
where $\mu,\lambda,R>0$ are arbitrary. Letting $k\to \infty$ we find that
\begin{align}\label{eq:gronwalltailbound}
\textstyle \P(\zeta(T)>\mu)\leq \frac{e^{R}}{\mu} \E(\xi(T)\wedge \lambda) + \P(\xi(T)\geq \lambda) + \P(a(T)>R).
\end{align}
Letting $\mu\to \infty$, $\lambda\to \infty$ and finally $R\to \infty$, it follows that $\zeta(T)<\infty$ a.s. Therefore, after letting $T\to \infty$, we can conclude that a.s.\ on the set $\{\wt{\sigma}<\infty\}$ we have $\sup_{t<\wt{\sigma}} \|u(t)\|_V^2 + \int_0^{\wt{\sigma}} \|u(s)\|_{\wt{V}}^2 ds<\infty$, and thus by Theorem \ref{thm:varloc}, $\wt{\sigma}=\infty$ a.s., which proves the claim.

Now, the desired regularity is immediate from the claim and the fact that $(\wt{u},\wt{\sigma})$ is a maximal solution in the $(\wt{V}, \wt{H}, \wt{V}^*)$-setting.
\end{proof}

\begin{remark}\label{rem:integstrong}
No $L^p(\Omega)$-integrability is stated in Theorem \ref{thm:strongvar}. In general, we do not expect this to hold. The tail bound in \eqref{eq:gronwalltailbound} can be used to obtain some integrability properties of $u$ in the $(\wt{V}, \wt{H}, \wt{V}^*)$-setting if additionally the conditions of Theorem \ref{thm:varglobal} hold. For simplicity, assume $u_0\in L^2(\Omega;H)$. Then $\xi(T)\in L^1(\Omega)$ and thus after letting $\lambda\to \infty$ in \eqref{eq:gronwalltailbound} we find that
\[\textstyle \P(\zeta(T)>\mu)\leq \frac{e^{R}}{\mu} \E(\xi(T)) + \P(a(T)>R).\]
This still leads to very poor integrability results due to the factor $e^R$ unless $a(T)$ has exponential integrability properties. The latter is the case if $u_0\in L^\infty(\Omega;H)$, $B(v)^*v = 0$ for all $v\in V$, and $\lim_{\|v\|_V\to \infty} \frac{\|G(v)\|_{\calL_2(\mathcal{U},H)}}{\|v\|_V} = 0$ by \eqref{eq:preciseestpind}. To see what integrability can be deduced for $\zeta(T)$, let $\varepsilon>0$ be such that $\E \exp(\varepsilon a(T))<\infty$ (therefore $\varepsilon$ might depend on $\|u_0\|_{L^\infty(\Omega;V)}$ and $T$, see Subsection \ref{sss:exp}). Then letting $R = \log(\mu^{b})$ with $b\in (0,1)$ we find that
\begin{align*}
\textstyle \P(\zeta(T)>\mu)
\leq \mu^{b-1} \E(\xi(T)) + \P(a(T)>\log(\mu^{b}))
\leq \mu^{b-1} \E(\xi(T)) + \frac{1}{\mu^{\varepsilon b}} \E \exp(\varepsilon a(T)). \end{align*}
Multiplying by $\mu^{p-1}$ with $p\in (0,1)$ and integrating over $[0, \infty)$ we obtain that
\begin{align*}
\textstyle \E\zeta(T)^p \leq 1+ \int_1^\infty \P(\zeta(T)>\mu) \,\dd \mu \leq 1
+ \int_1^\infty\big[ \mu^{p+b-2} \E(\xi(T)) + \mu^{p-1-\varepsilon b} \E \exp(\varepsilon a(T))\big] \,\dd \mu.
\end{align*}
The latter is finite if $p<\min\{1-b, \varepsilon b\}$. Letting $b = \frac{1}{\varepsilon+1}$ we obtain that, for all $p\in (0,\frac{\varepsilon}{\varepsilon+1})$,
\[\textstyle \E \sup_{t\in [0,T]}\|u(t)\|_{V}^{2p} + \E \Big(\int_0^T \|u(t)\|_{\wt{V}}^{2} \,\dd t\Big)^{p/2}<\infty.\]
\end{remark}

\section{Selected applications through $L^2$-theory}\label{sec:appl}

In this section, we will explain several examples which can be covered through the $L^2$-theory presented in Section \ref{sec:Var}.
Along the way, we will highlight some of the limitations of this setting, which will motivate our treatment of certain examples using stochastic maximal  $L^p$-regularity in an $L^q$ or $H^{s,q}$-space (see Section \ref{sec:LpLq}).

As discussed in Section \ref{sec:Var}, in the variational setting, a coercivity condition ensures global well-posedness (see Theorem \ref{thm:varglobal}). In Subsection \ref{ss:AllenCahn}, we apply this result to the stochastic Allen--Cahn equation. We first examine the equation in the analytical weak setting for $d=1$, and then for $d\leq 4$ in the strong setting, which fails to be weakly monotone (see the beginning of Section \ref{sec:Var}).

In Subsection \ref{ss:CahnHilliard}, we consider the stochastic Cahn--Hilliard equation for $d\in \{1, 2\}$. The latter also fails to be weakly monotone, and turns out to be critical for $d=2$.

In Subsection \ref{ss:Fluid}, we apply the results to an abstract fluid dynamics model and discuss consequences for 2D stochastic Navier--Stokes equations and Boussinesq equations, which are critical but not weakly monotone.

In some cases where the coercivity condition fails, the $L^2$-setting can still give local well-posedness. Obtaining global well-posedness for these problems can be quite challenging. We will indicate some of the ideas needed for the 3D stochastic primitive equation in Subsection \ref{ss:otherfluid}. Moreover, in Subsection \ref{ss:LV}, we study a Lotka--Volterra model in detail. The $L^p$-setting plays a key role in this context, as it is used to ensure positivity of the solution, which is crucial for global well-posedness.

\subsection{Stochastic Allen--Cahn equation with transport noise}\label{ss:AllenCahn}
In this subsection, we consider the following Allen--Cahn type equation on a domain $\Dom\subseteq \R^d$:
\begin{align}\label{eq:AllenCahn}
\dd  u  & =  \textstyle \big( \Delta u +u-u^3\big) \,\dd  t+ \sum_{n\geq 1}  \big[(b_{n}\cdot \nabla) u+ g_{n}(\cdot,u)\big] \, \,\dd W^n_t\text{ on }\Dom,\quad u=0 \text{ on }\partial\Dom,
\end{align}
and initial value $u(0,\cdot) = u_0$. In Subsection \ref{sss:AllenCahnweak} we consider Dirichlet boundary conditions, and in Subsection \ref{sss:AllenCahnstrong} periodic boundary conditions.
The $(W^n)_{n\geq 1}$ are independent standard Brownian motions.
The quadratic diffusion assumption is optimal from a scaling point of view, see Subsection \ref{sss:tracecriticalF_AC}.

\subsubsection{Weak setting}\label{sss:AllenCahnweak}

To analyze \eqref{eq:AllenCahn} in the so-called weak setting, we let $V = H^1_0(\Dom)$ and $H = L^2(\Dom)$. In this way $V^* = H^{-1}(\Dom)$, and we set
\begin{align*}
\lb v,A u\rb  = (\nabla u, \nabla v)_{L^2(\Dom)},   \ \ \text{and} \ \  \lb v, F(u)\rb = (v, u)_{L^2(\Dom)}-(v, u^3)_{L^2(\Dom)}.
\end{align*}
For the noise, let $\mathcal{U} = \ell^2$ and let $(e_n)_{n\geq 1}$ be its standard orthonormal basis. Then $W_\mathcal{U}(\one_{[0,t]} e_n) = w^n_t$ uniquely extends to a cylindrical Brownian motion on $\mathcal{U}$. Let
\begin{align}\label{eq:defBGAllenCahn}
\textstyle (B u) e_n= \sum_{j=1}^d b_{n}^j \partial_j u, \ \ \text{and} \ \  G(u)e_n = g_n(\cdot,u).
\end{align}
To check Assumption \ref{ass:varsetting} for $F$, note that as in Subsection \ref{sss:tracecriticalF_AC},
\begin{align*}
\|F(u) - F(v)\|_{V^*} &\lesssim \|u-v\|_{L^2(\Dom)} + (\|u\|_{L^{3r}(D)}^2+\|v\|_{L^{3r}(D)}^2) \|u-v\|_{L^{3r}(D)}
\\ & \lesssim (1+\|u\|_{\beta}^2+\|v\|_{\beta}^2) \|u-v\|_{\beta},
\end{align*}
where $-1 - \frac{d}{2} \leq -\frac{d}{r}$, with $r\in (1, 2)$, and where $2\beta \geq \frac{d}{2} + 1-\frac{d}{3r}$. Just based on the $F$-part we see that $\rho = \rho_1 = 2$ in Assumption \ref{ass:varsetting}. Therefore, the (sub)criticality condition \eqref{eq:subcriticalvar} becomes $\beta\leq 2/3$. Taking $\beta=2/3$, we see that for $d=1$ we can take $r\in (1, 2)$ arbitrarily. For $d=2$, the lower bound on $\beta$ implies $r\leq 1$, which is not admissible in the above Sobolev embedding. Recently, a variation of the above method was introduced that allows $d=2$ in the weak setting (see \cite{BGV}). In the case $d\geq 3$, the best choice is to define $r$ by $-1 - \frac{d}{2} = -\frac{d}{r}$. One can check that this gives $2\beta\geq \frac{2}{3} + \frac{d}{3}$, which is false for $d\geq 3$. Therefore, from now on we assume $d=1$.

Suppose that  $b\in L^\infty(\Dom;\ell^2)$ and that the parabolicity condition \eqref{eq:stochastic_parabolicity_example} holds, i.e.\ $\sup_{\Dom}\|b\|_{\ell^2}<\nu$ for some constant $\nu<2$. Finally, for simplicity, we assume that $g:\Dom\times \R\to \ell^2$ is globally Lipschitz in the second variable and $g(\cdot, 0)\in L^2(\Dom;\ell^2)$ (see Remark \ref{rem:AC_variational_quadratic_growth} below for weaker conditions).
These conditions imply that Assumption \ref{ass:varsetting}\eqref{it1:varsetting} holds, since
\begin{align*}
\lb u, Au\rb -\frac12 \|B u \|^2_{\calL_2(\mathcal{U},H)} = \|\nabla u\|_{L^2(\Dom)}^2 - \frac{1}{2} \|b \nabla u\|_{L^2(\Dom)}^2\geq \nu \|\nabla u\|_{L^2(\Dom)}^2\geq \nu \|u\|_{V}^2 - M_{\nu} \|u\|_H^2.
\end{align*}
Moreover, one can check that $G$ is Lipschitz from $H$ into $\calL_2(\mathcal{U},H)$ and thus Assumption \ref{ass:varsetting}\eqref{it2:varsetting} holds as well by the above discussion. Finally, since
\begin{equation}
\label{eq:dissipation_AC_1D}
\lb u, F(u)\rb = \|u\|_{L^2(\Dom)}^2-(u, u^3)_{L^2(\Dom)}\leq \|u\|_{H}^2,
\end{equation}
the coercivity condition of Theorem \ref{thm:varglobal} follows from Remark \ref{rem:sufficientcoercivity}.
Hence, Theorem \ref{thm:varglobal} yields

\begin{theorem}[Allen--Cahn in the weak setting]\label{thm:AllenCahnd=1}
Let $d=1$ and suppose that $\sup_{\Dom}\|b\|^2_{\ell^2}<2$, $g(\cdot, 0)\in L^2(\Dom;\ell^2)$ and there is a constant $L$ such that
\[\|g(x,y)-g(x,y')\|_{\ell^2}\leq L|y-y'|, \ \ \  x\in \Dom, y,y'\in \R.\]
Then for every $u_0\in L^0_{\F_0}(\Omega;L^2(\Dom))$, \eqref{eq:AllenCahn} has a unique global solution $u\in L^2_{\rm loc}([0,\infty);H^1_0(\Dom))\cap C([0,\infty);L^2(\Dom))$ a.s.
\end{theorem}

Additionally, the estimates and continuity properties stated in Theorem \ref{thm:varglobal} hold, where $L^2$-moments are also included by Subsection \ref{sss:p=2}. Since the setting of Theorem \ref{thm:AllenCahnd=1} is subcritical, in some cases, we can obtain further regularity. Indeed, for instance, if $b=0$, then by Theorem \ref{thm:SMRHS} and the text below Theorem \ref{thm:varlin}, $(A,0)\in \mathcal{SMR}_{p,\kappa}^{\bullet}$ for all $p\in [2, \infty)$ and $\kappa\in [0,p/2-1)\cup\{0\}$.
Therefore, if $u_0\in H^{\delta}(\Dom)$ for some $\delta>0$, Proposition \ref{prop:parabreg3} and Sobolev embeddings show that
\begin{align*}
  u&\in L^{r}_{\rm loc}((0,\sigma);\HD^{1}(\Dom))\cap C^{\theta/2,\theta}_{\rm loc}((0,\sigma)\times \overline{\Dom}), \quad r\in [2, \infty), \theta\in (0,1).
 \end{align*}
To include higher dimensions, one could choose one of the following options:
\begin{enumerate}[{\rm (1)}]
\item\label{it1:AllenCahnstrong} consider the strong setting $V = H^{2}(\Dom)\cap H^1_0(\Dom)$, $H = H^{1}_0(\Dom)$ and $V^* = L^2(\Dom)$;
\item\label{it2:AllenCahnLpLq} use $L^p(L^q)$-theory with $p>2$ and $q\geq 2$.
\end{enumerate}

A disadvantage of \eqref{it1:AllenCahnstrong} is that it leads to compatibility conditions for $b$ and $g$ (for more details see the comments at the beginning of Subsection \ref{ss:AllenCahnLpLq}).
Option \eqref{it2:AllenCahnLpLq} will be the start of our discussion on $L^p(L^q)$-theory in Subsection \ref{ss:AllenCahnLpLq}.

\begin{remark}[Critical quadratic diffusion]
\label{rem:AC_variational_quadratic_growth}
As discussed at the end of Subsection \ref{sss:tracecriticalF_AC}, a scaling argument suggests that a quadratic growth for $g$ is natural. The reader can check that Theorem \ref{thm:AllenCahnd=1} also holds if $b=0$ and the measurable mapping $g:\Dom\times \R\to \ell^2$ satisfies $g(\cdot,0)\in L^\infty(\Dom;\ell^2)$, and moreover there exists $L>0$ such that, for all $x\in \Dom$ and $y,y'\in\R$,
\begin{align*}
\|g(x,y)-g(x,y')\|_{\ell^2}&\leq L (1+|y|+|y'|)|y-y'|,\\
\|g(x,y)\|_{\ell^2} &\leq L (1+|y|)+ \sqrt{2}|y|^2.
\end{align*}
Let us remark that the first in the above serves to check Assumption \ref{ass:varsetting} for $G$, while the second allows one to check the coercivity condition of Theorem \ref{thm:varglobal} (therefore using the dissipative effect of $-u^3$ which was not used in
\eqref{eq:dissipation_AC_1D}). The factor $\sqrt{2}$ is optimal for the coercivity to hold.

It is also possible to deal with transport noise and quadratic nonlinearities simultaneously. To check coercivity one can argue as in \cite[Lemma 3.3 and 3.4]{AVreaction-global}, see also Theorem 1.1 there.
\end{remark}

\subsubsection{Strong setting}\label{sss:AllenCahnstrong}
For simplicity, suppose that $\Dom$ is the $d$-dimensional torus $\T^d$. We explore what changes in the strong setting when considering the Allen--Cahn equation. Alternatively, one could consider $\R^d$ or a $d$-dimensional manifold without boundary.  A domain with boundary conditions could also be considered, provided additional compatibility conditions are met.

\begin{assumption}
\label{ass:AC}
Let $b^j=(b_n^j)_{n\geq 1}:\Tor^d\to \ell^2$ for $j\in \{1, \ldots,d\}$, and $g=(g_n)_{n\geq 1}:\Tor^d\times \R\to \ell^2$ be $\Borel(\Tor^d)$- and $\Borel(\Tor^d)\otimes\Borel( \R)$-measurable maps, respectively. Assume that
\begin{enumerate}[{\rm(1)}]
\item\label{it:AC_1} There exists $\nu\in (0,2)$ such that for all $x\in \Tor^d$,
\begin{equation*}
\textstyle \sum_{n\geq 1} \sum_{i,j=1}^d b_n^i(x) b_n^j(x) \xi_i \xi_j \leq \nu |\xi|^2 \ \ \text{ for all }\xi\in \R^d;
\end{equation*}
\item\label{it:AC_2} $\|b^j\|_{W^{1,d+\delta}(\Tor^d;\ell^2)}\leq M$ for every $j\in \{1,\dots,d\}$;

\item\label{it:AC_3} $g\in C^1(\T^d\times\R;\ell^2)$, and there exists a constant $L$ such that
    \[\|g(x,y)-g(x,y')\|_{\ell^2}+\|\partial_y g(x,y)-\partial_y g(x,y')\|_{\ell^2}\leq L|y-y'|, \ \ \  x\in \Dom, y,y'\in \R.\]
\end{enumerate}
\end{assumption}

As noted in Remark \ref{rem:AC_variational_quadratic_growth}, quadratic growth for the diffusion $g$ is possible. For further details see \cite[Assumption 5.16(5)]{AVvar} and the comments provided below it. For simplicity, we do not pursue this here.

\smallskip

Let $V = H^2(\T^d)$, $H = H^1(\T^d)$. It can be verified that $V^* = L^2(\T^d)$ with respect to the pairing of $H$. Thus, for $u\in V$ and $v\in V^*$ we set
$\lb u, v\rb = -(\Delta u, v)_{L^2(\T^d)} + (u, v)_{L^2(\T^d)}$. Let $A = -\Delta$ and $F(u ) =u-u^3$, and define $B$ and $G$ as in \eqref{eq:defBGAllenCahn}. Checking Theorem \ref{thm:varglobal} in this case is conceptually similar to the previous case. However, since the duality pairing and the space $H = H^1(\T^d)$ contain one more derivative, the calculation becomes more cumbersome. To check the coercivity condition, derivatives must be taken into account, and the function $G$ is no longer globally Lipschitz because of the chain rule.

We begin by examining the mapping properties of $F$. This becomes simpler since $V^* = L^2(\T^d)$. Using $|F(u) - F(v)|  \leq 2(u^2+v^2) |u-v|$ and applying H\"older's inequality we obtain
\begin{align*}
\|F(u) - F(v)\|_{V^*} &\lesssim \|u-v\|_{L^2(\T^d)} + (\|u\|_{L^{6}(\T^d)}^2+\|v\|_{L^{6}(\T^d)}^2) \|u-v\|_{L^{6}(\T^d)}
\\ & \lesssim (1+\|u\|_{H^{2\beta_1}(\T^d)}^2+\|v\|_{H^{2\beta_1}(\T^d)}^2) \|u-v\|_{H^{2\beta_1}(\T^d)},
\\ & \lesssim (1+\|u\|_{\beta_1}^2+\|v\|_{\beta_1}^2) \|u-v\|_{\beta_1},
\end{align*}
where we used $[V^*, V]_{\beta_1} = H^{2\beta_1}(\T^d)$  and Sobolev embedding with $2\beta_1-\frac{d}{2}\geq -\frac{d}{6}$ (see Subsection \ref{subsec:interp} for the definition of the fractional Sobolev spaces). On the other hand, we already saw that \eqref{eq:subcriticalvar} gives $\beta_1\leq \frac{2}{3}$ as (sub)criticality condition. Thus we find that $d\in \{1, \ldots, 4\}$ are admissible and subcritical for $F$.

The $G$-term was globally Lipschitz in the weak setting. This no longer holds in the strong setting, since we need to estimate the $H^1$-norm of $G$. By Assumption \ref{ass:AC}\eqref{it:AC_3} for $u,v\in V$  one has
\begin{align*}
\|&\nabla [G(u) - G(v)]\|_{L^2(\T^d;\ell^2)} = \|\partial_x g(\cdot, u) - \partial_x g(\cdot, v)\|_{L^2(\T^d;\ell^2)} + \|\partial_y g(\cdot, u) \nabla u - \partial_y g(\cdot, v) \nabla v]\|_{L^2(\T^d;\ell^2)}
\\ & \leq L\|u-v\|_{L^2(\T^d)} +\|\partial_y g(u)\|_{L^\infty(\T^d;\ell^2)} \|\nabla u - \nabla v\|_{L^2(\T^d)} + \|[\partial_y g(u)  - \partial_yg(v)] \nabla u\|_{L^2(\T^d;\ell^2)}
\\ & \leq L\|u-v\|_{L^2(\T^d)} + L \|\nabla u - \nabla v\|_{L^2(\T^d)} + L\||u-v|  |\nabla u|\|_{L^2(\T^d)}.
\end{align*}
By H\"older's inequality and Sobolev embedding (using again $d\leq 4$) we have
\begin{align*}
\||u-v|  |\nabla u|\|_{L^2(\T^d)}  & \leq \|u-v\|_{L^8(\T^d)}  \|\nabla u\|_{L^{\frac83}(\T^d)}
\lesssim  \|u-v\|_{H^{3/2}(\T^d)} \|u\|_{H^{3/2}(\T^d)}.
\end{align*}
The estimate for $\|G(u) - G(v)\|_{L^2(\T^d;\ell^2)}$ was already observed before. Therefore, it follows that
\begin{align*}
\|G(u) - G(v)\|_{\calL_2(\ell^2,H^{1}(\T^d))}  \leq (1+\|u\|_{\frac34}) \|u-v\|_{\frac34},
\end{align*}
and thus we can set $\rho_2=1$ and $\beta_2 = \frac{3}{4}$, which is critical for $G$ and thus for $(F,G)$ (if $d=4$).

Next, we check Assumption \ref{ass:varsetting}\eqref{it1:varsetting}. First, observe that
\[\textstyle \partial_k (B u) e_n = \partial_k [b_n \cdot \nabla u] = \sum_{j=1}^d b_{n}^j \partial_k \partial_j u + \sum_{j=1}^d \partial_k b_{n}^j \partial_j u.\]
To estimate $\frac12\|\nabla B u\|_{L^2(\T^d;\ell^2)}^2$ we split into the above two parts. By Assumption \ref{ass:AC}\eqref{it:AC_1}
\begin{align*}
\textstyle \sum_{k=1}^d\int_{\T^d} \sum_{n\geq 1} \Big|\sum_{j=1}^d  b_{n}^j \partial_j \partial_k u\Big|^2  \,\dd x
 \leq \nu \int_{\T^d} \sum_{k,j=1}^d |\partial_j \partial_k u|^2\,\dd x
 = \nu \|\Delta u\|^2_{L^2(\T^d)}.
\end{align*}
For the second term, one can check that
\begin{align*}
\textstyle \sum_{k=1}^d \int_{\T^d} \sum_{n\geq 1} \Big|\sum_{j=1}^d  \partial_k b_{n}^j \partial_j u\Big|^2  \,\dd x
 &\stackrel{(i)}{\leq}  \textstyle  \sum_{j=1}^d  \|\nabla b_{n}^j|_{L^{d+\delta}(\T^d;\ell^2)}^2  \|\nabla u\|^2_{L^{r}(\T^d)}
\\ &\stackrel{(ii)}{\leq} \textstyle \sum_{j=1}^d  \|b_{n}^j\|_{W^{1,d+\delta}(\T^d;\ell^2)}^2  \|\nabla u\|^2_{H^{\mu}(\T^d)},
\\ &\stackrel{(iii)}{\leq} \textstyle \alpha \|u\|_{H^2(\T^d)} +  C_{\alpha} \|u\|^2_{H^{1}(\T^d)},
\end{align*}
where in $(i)$ we applied Cauchy-Schwarz and H\"older's inequalities with $\frac{1}{d+\delta} + \frac{1}{r} = \frac12$,
and in $(ii)$ Sobolev embedding with $\mu-\frac{d}{2} \geq -\frac{d}{r}$ which is possible for some $\mu\in (0,1)$. Moreover, in $(iii)$ we used standard interpolation estimates and arbitrary $\alpha>0$. Using $(a+b)^2 \leq (1+\varepsilon)a^2 + C_{\varepsilon} b^2$, we can conclude that
\begin{align*}
\textstyle \frac12\|\nabla B u\|_{L^2(\T^d;\ell^2)}^2 \leq \frac{(1+\varepsilon)\nu -\alpha}{2} \|\Delta u\|^2_{L^2(\T^d)} + \frac{C_{\varepsilon}}{2} \sum_{j=1}^d \|b_{n}^j\|_{W^{1,d+\delta}(\T^d;\ell^2)}^2  \|u\|^2_{H^2(\T^d)}.
\end{align*}
Similarly, one can check that
$\frac12\|B u\|_{L^2(\T^d;\ell^2)}^2\leq \frac{\nu}{2} \|\nabla u\|_{L^2(\T^d)}^2$.
From these estimates one can deduce that $B\in \calL(V, \calL_2(\mathcal{U},H))$ and by taking $\varepsilon$ so small that $\frac{(1+\varepsilon)\nu-\alpha}{2} = (1-\theta)$ with $\theta>0$, and using
$\lb u, Au\rb = \|\Delta u\|^2_{L^2(\T^d)} + \|\nabla u\|_{L^2(\T^d)}^2$,
we find that
\begin{align*}
\textstyle
\lb u, Au\rb -\frac12 \|B u \|^2_{\calL_2(\mathcal{U},H)}\geq \theta \|\Delta u\|^2_{L^2(\T^d)} + \theta\|\nabla u\|_{L^2(\T^d)}^2 - \frac{C_{\alpha,\varepsilon} }{2} \|u\|^2_{H^1(\T^d)}
= \theta \|u\|_V^2 - \frac{C_{\alpha,\varepsilon}}{2} \|u\|_H^2
\end{align*}

To check the coercivity condition of Theorem \ref{thm:varglobal} note that the following subcritical growth estimate holds:
\begin{align*}
\|G(v)\|_{\calL_2(\ell^2,H)}^2 & = \big\| g(\cdot,v)\big\|_{L^2(\T^d;\ell^2)}^2+  \big\| \partial_x g(\cdot, v) + \partial_y g(\cdot, v) \nabla v \big\|_{L^2(\T^d;\ell^2)}^2
\\ & \leq  C(1+\|v\|_{L^2(\T^d)})^2 + L \|\nabla v\|_{L^2(\T^d)}^2
\leq C'(1+\|v\|_{H})^2.
\end{align*}
Therefore, by Remark \ref{rem:sufficientcoercivity} it suffices to prove that $\lb v, F(v)\rb \leq M''(\|v\|^2_H+1)$ for all $v\in C^2(\T^d)$.
For the linear part in $F$, this is obvious. For the cubic part, this follows from
\begin{align*}
\textstyle  \lb v, -v^3\rb = (\Delta v, v^3)_{L^2} - (v, v^3)_{L^2} = - 3\int_{\T^d} |\nabla v|^2 v^2 \,\dd x  - \int_{\T^4} v^4 \,\dd x\leq 0.
\end{align*}
Now, Theorem \ref{thm:varglobal} implies that
\begin{theorem}[Allen--Cahn in the strong setting]\label{thm:AllenCahnstrong}
Let $d\in \{2, 3, 4\}$ and suppose that the above conditions hold on $(b, g)$. Then for every $u_0\in L^0_{\F_0}(\Omega;H^1(\T^d))$ there exists a unique global solution $u$ such that $u\in L^2_{\rm loc}([0,\infty);H^2(\T^d))\cap C([0,\infty);H^1(\T^d))$ a.s.
\end{theorem}
Additionally, the estimates and continuity properties stated in Theorem \ref{thm:varglobal} hold, including $L^2$-moments as discussed in Subsection \ref{sss:p=2}. The case $d=1$ can also be considered by introducing a dummy variable. Since the growth estimates for $F$ and $G$ are subcritical when $d\leq 3$, Corollary \ref{cor:strongvarsub} further implies that for $\theta\in [0,1/2)$, $r\in [2, \infty)$, $\varepsilon\in (0,\theta)$,
\begin{align*}
  u&\in H^{\theta,r}_{\rm loc}((0,\infty);H^{2-2\theta}(\T^d))\cap C^{\theta-\varepsilon}_{\rm loc}((0,\infty);H^{2-2\theta}(\T^d)).
 \end{align*}
Here we used that $(A,B)\in \mathcal{SMR}_{p,\kappa}^{\bullet}$ for all $p\in [2, \infty)$ and $\kappa\in [0,p/2-1)\cup\{0\}$ (see \cite{AV21_SMR_torus}). Later, we will see that significantly more regularity can be achieved by applying Theorem \ref{thm:reactioncoerc}  (see Example \ref{ex:AllenCahnperiodicLpLq}).

\subsection{The stochastic Cahn--Hilliard equation}\label{ss:CahnHilliard}
A prominent case of an SPDE that is not weakly monotone (see the beginning of Section \ref{sec:Var}) is the Cahn--Hilliard equation, which was considered in many papers (see \cite{CarMil, DaPDeb} and references therein). For simplicity, we consider the equation on a bounded $C^2$-domain $\Dom\subseteq \R^d$. The arguments easily extend to unbounded domains (in case the elliptic result \eqref{eq:ellipCH} below holds).

On $\Dom$ consider the following equation with a trace class gradient noise term:
\begin{equation}
\label{eq:CahnHilliard}
\left\{
\begin{aligned}
\dd  u +\Delta^2 u \,\dd t &= \textstyle \Delta (f(u)) \,\dd t + \sum_{n\geq 1} g_n(\cdot, u, \nabla u)  \,\dd W^n_t, &\quad &\text{ on }\Dom,
\\ \nabla u \cdot n&=0 \quad \text{and} \quad \nabla (\Delta u) \cdot n=0,  &\quad &\text{ on }\partial \Dom,
\\ u(0)&=u_0, &\quad &\text{ on }\Dom.
\end{aligned}\right.
\end{equation}
Here, $n$ is the outer normal vector of $\Dom$, and $(W^n)_{n\geq 1}$ are independent Brownian motions on a given probability space $\O$.

Below, we closely follow the presentation of \cite[Section 5.1]{AVvar}.

We make the following assumptions on $f$ and $g$:
\begin{assumption}\label{ass:CH}
Let $d\geq 1$ and $\rho\in [0,\frac{4}{d}]$. Suppose that $f\in C^1(\R)$ and
\begin{align*}
 |f(y)-f(y')|
& \leq L (1+|y|^{\rho}+|y'|^{\rho})|y-y'|,
\\ f'(y)&\geq -C,
\end{align*}
and suppose that $g:\R\times \R\times \R^d\to \ell^2$ is such that $g(\cdot, 0,0)\in L^2(\Dom;\ell^2)$ and
\[\|g(x,y,z)-g(x,y',z')\|_{\ell^2}\leq L|y-y'| + L|z-z'|, \ \ \  x\in \Dom, y,y'\in \R, z,z'\in \R^d.\]
\end{assumption}
For example the classical double well potential $f(y) = y(y^2-1) = \partial_y [\frac14(1-y^2)^2]$ satisfies the above conditions for $d\in \{1, 2\}$.
Let $H = L^2(\Dom)$ and set
\[V:=H^2_N(\Dom) = \{u\in H^2(\Dom): \partial_n u|_{\partial\Dom}  =0\},\]
where $\partial_n u = \nabla u \cdot n$ and $n$ denotes the outer normal vector field on $\partial\Dom$.

We define $A\in \calL(V,V^*)$ and $F:V\to V^*$  by
\[\langle v, Au\rangle = (\Delta v, \Delta u)_{L^2}, \quad \text{and} \quad \langle v, F(u)\rangle = (\Delta v, f(u))_{L^2}.
\]
Of course, we need to ensure $f(u)\in L^2(\Dom)$ and this will be done below.
Let $B =0$, and let $G:V\to \calL_2(\mathcal{U},H)$ be defined by (well-definedness is checked below)
\[(G_n(u))(x) = g_n(x,u(x), \nabla u(x)).\]

To apply Theorem \ref{thm:varglobal} we first check Assumption \ref{ass:varsetting}.
By the smoothness of $\Dom$ and standard elliptic theory (see \cite[Theorem 8.8]{GT83}) there exist $\theta, M>0$ such that
\begin{align}\label{eq:ellipCH}\|u\|_{H^2(\Dom)}^2\leq \theta \|\Delta u\|_{L^2(\Dom)}^2 + M\|u\|_{L^2(\Dom)}^2, \ \ u\in V.
\end{align}
Hence, $\langle u, A u\rangle = \|\Delta u\|_{L^2(\Dom)}^2 \geq \theta\|u\|_{V}^2-M\|u\|_{H}^2$ for all $u\in V$, which gives the required coercivity condition for $A$.
Without loss of generality, we may assume $f(0) = 0$. For the local Lipschitz estimate, note that with $\rho_1=\rho$,
\begin{align*}
& \|F(u) - F(v)\|_{V^*} \lesssim \|f(\cdot, u) - f(\cdot, v)\|_{L^2(\Dom)}
\\ & \lesssim  \|(1+|u|^{\rho_1}+|v|^{\rho_1}) (u-v)\|_{L^{2}(\Dom)} & \text{(by Assumption \ref{ass:CH})}
\\ & \lesssim (1+\|u\|^{\rho_1}_{L^{2(\rho_1+1)}(\Dom)}+\|v\|^{\rho_1}_{L^{2(\rho_1+1)}(\Dom)})\|u-v\|_{L^{2(\rho_1+1)}(\Dom)} & \text{(by H\"older's inequality)}
\\ & \lesssim (1+\|u\|^{\rho_1}_{H^{4\beta_1-2}(\Dom)}+\|v\|^{\rho_1}_{H^{4\beta_1-2}(\Dom)})  \|u-v\|_{H^{4\beta_1-2}(\Dom)} & \text{(by Sobolev embedding).}
\end{align*}
In the Sobolev embedding we need $4\beta_1 -2- \frac{d}{2} \geq -\frac{d}{2(\rho_1+1)}$. Therefore, the condition \eqref{eq:subcriticalvar} leads to $\rho_1\leq \frac{4}{d}$.
Moreover, we can consider the critical case $2\beta_1 = 1+\frac{1}{\rho_1+1}$.
The function $G$ satisfies Assumption \ref{ass:varsetting} with $\rho_2 = 0$ and $\beta_2 =\frac34$. Indeed,
\begin{align*}
\|G(u) - G(v)\|_{L^2(\Dom;\ell^2)}   & \lesssim \|u - v\|_{L^2(\Dom)} + \|\nabla u - \nabla v\|_{L^2(\Dom)}
\lesssim \|(u-v)\|_{H^{1}(\Dom)}\eqsim \|u-v\|_{\frac34}.
\end{align*}
To check the coercivity condition of Theorem \ref{thm:varglobal}, integrating by parts it follows that for all $v\in V$,
\begin{align*}
\textstyle \langle F(v),v\rangle = -(\nabla v, \nabla (f(v)))_{L^2}
 \leq  -\int_{\Dom} f'(v) |\nabla v|^2 \,\dd  x
 \leq C \|\nabla v\|_{L^2(\Dom)}^2\leq \varepsilon \|v\|_{V}^2 + C_{\varepsilon} \|v\|_{H}^2,
\end{align*}
where we used standard interpolation estimates and where $\varepsilon>0$ is arbitrary. By Remark \ref{rem:sufficientcoercivity} this implies the required coercivity condition, and thus by Theorem \ref{thm:varglobal} we deduce the following.
\begin{theorem}[Global well-posedness]
\label{thm:CH}
Suppose that Assumption \ref{ass:CH} holds. Let $u_0\in L^0_{\F_0}(\Omega; L^2(\Dom))$.
Then \eqref{eq:CahnHilliard} has a unique global solution
\begin{equation*}
u\in C([0,\infty);L^2(\Dom))\cap L^2_{\rm loc}([0,\infty);H^2_N(\Dom)) \ a.s.
\end{equation*}
\end{theorem}
Also, $L^2$-moment bounds hold by Subsection \ref{sss:p=2}, and the continuous dependency of Theorem \ref{thm:varglobal} holds in the above setting. Furthermore, if $\rho_1 <4/d$, then one can check that Proposition \ref{prop:parabreg3} gives that
\begin{align*}
  u&\in H^{\theta,r}_{\rm loc}((0,\sigma);H^{2-4\theta}(\Dom))\cap C^{\theta-\varepsilon}_{\rm loc}((0,\sigma);H^{2-4\theta}(\Dom)), \ \ \theta\in [0,1/2), r\in [2, \infty), \varepsilon\in (0,\theta).
 \end{align*}

\begin{remark}
In case $d\geq 3$, one can apply $L^p(L^q)$-theory to establish global well-posedness for \eqref{eq:CahnHilliard} in case of the classical double well potential; however, this will not be considered in this survey. Regularity can also be obtained through $L^p(L^q)$-theory. Alternatively, for $d\in \{1, 2\}$ one could apply Theorem \ref{thm:strongvar} and Proposition \ref{prop:parabreg3} to obtain regularity.
\end{remark}

\subsection{Fluid dynamics models via $L^2$-setting}\label{ss:Fluid}
In this subsection, we consider several models from fluid dynamics which fit into a unified abstract framework, mostly two-dimensional. In particular, this includes 2D Navier Stokes, quasigeostrophic, and 2D Boussinesq equations, all of which we examine in more detail. A similar setting was also considered in \cite{ChueshovMillet}, where it was shown that the following models are included: 2D magneto-hydrodynamic equations, 2D magnetic B\'enard problem, 3D Leray $\alpha$-model for Navier--Stokes equations, shell models of turbulence.
While we do not delve into these latter models, our framework accommodates them as well.
Models which do not fit into this the $L^2$-setting are considered in Subsection \ref{ss:quasiGSLp} and \ref{subsec:SNSRd}.

All of the above equations are typically treated using Galerkin approximation since they do not conform to the classical variational setting. However, in our theory, they do not present any additional difficulties and can be seamlessly incorporated. Notably, our approach does not require compactness of embeddings, allowing us to include unbounded domains. Finally, it is worth mentioning that most 2D fluid dynamics models are critical in the $L^2$-setting.

\subsubsection{Abstract formulation}
The general problem we consider has the form
\begin{equation}\label{eq:abstractfluid}
\begin{aligned}
\dd u + A u \,\dd t = \Phi(u, u) \,\dd t + (B u + G(u)) \,\dd W, \quad u(0) = u_0.
\end{aligned}
\end{equation}
Here, $\Phi:V_{\beta_1}\times V_{\beta_1}\to V^*$ is assumed to be bilinear and satisfies certain estimates (see below). A key condition will be that $\lb u, \Phi(u, u)\rb = 0$ for all $u\in V$. In many models, $\Phi$ is of the form $\Phi(u,v) = \div( u\otimes v)$ with $V$ a first-order Sobolev space.

\begin{assumption}\label{ass:bilinearFluid}
\
\begin{enumerate}[{\rm (1)}]
\item\label{it1:bilinearFluid} $(A,B)$ is coercive, i.e.\ there exist $\theta>0$ and $M\geq 0$ such that for all $v\in V$,
    \[\lb v, Av\rb  - \tfrac{1}{2} \|B v\|_{\calL_2(\mathcal{U};H)}^2 \geq \theta\|v\|_V^2- M \|v\|_H^2.\]

\item\label{it2:bilinearFluid} For some $\beta_1\in (1/2, 3/4]$,
$\Phi:V_{\beta_1}\times V_{\beta_1}\to V^*$ is bilinear and satisfies
\[\|\Phi(u, v)\|_{V^*}\leq C\|u\|_{\beta_1} \|v\|_{\beta_1}, \ \ \lb u,\Phi(u,u)\rb = 0, \ \ u,v\in V.\]

\item\label{it3:bilinearFluid} For some $\beta_2\in (1/2, 1)$, $G:V_{\beta_2}\to \calL_2(\mathcal{U},H)$ is globally Lipschitz.
\end{enumerate}
\end{assumption}

Let $F:V_{\beta}\to V^*$ be given by $F(u) = \Phi(u,u)$. Then $F$ satisfies Assumption \ref{ass:varsetting} with $\rho_1=1$ and $\beta_1$ as in Assumption \ref{ass:bilinearFluid}. Note that $\beta_1 = 3/4$ is the critical case of Assumption \ref{ass:varsetting}. In many applications with $d=2$, one is forced to take $\beta_1 = 3/4$ as will be explained in a simplified setting in the following remark.
\begin{remark}\label{rem:gradientformburger}
Let $d=2$. If $V = H^1$, $H = L^2$ and $V^* = H^{-1}$, then $V_{\beta} = H^{2\beta-1}$. Thus
\[\|\nabla(u^2) - \nabla(v^2)\|_{H^{-1}} \lesssim \|u^2-v^2\|_{L^2} \leq \|u+v\|_{L^4}\|u-v\|_{L^4}\leq C(\|u\|_{\beta}+\|v\|_{\beta})\|u-v\|_{\beta}\]
if $2\beta-1 -\frac{2}{2}\geq -\frac{2}{4}$. This leads to $\beta\geq \frac34$, which means criticality cannot be avoided. One can also check that the scaling $u(\lambda\cdot)$ by $\lambda>0$ of the nonlinearity in the $H^{-1}$-norm and $\|\cdot\|_{3/4} = \|\cdot\|_{H^{1/2}}$-norm both coincide with $\lambda^{-1/2}$ for $\lambda\downarrow 0$. The criticality in the above is one of the reasons that the more classical variational settings do not apply to many standard fluid dynamics models.

Another problem of the classical variational setting is that the nonlinearity $F$ is not weakly monotone (see the beginning of Section \ref{sec:Var}), which, for instance, follows from a scaling argument considering $u(\lambda\cdot)$ and let $\lambda\to \infty$.
\end{remark}

\begin{theorem}\label{thm:fluidabstract}
Suppose that Assumption \ref{ass:bilinearFluid} holds,
Then for every $u_0\in L^0_{\F_0}(\Omega;H)$ there exists a unique global solution $u\in L^2_{\rm loc}([0,\infty);V)\cap C([0,\infty);H)$ a.s.\ to \eqref{eq:abstractfluid}. Moreover, for all $T\in (0,\infty)$ and $p\in (0,2]$,
\begin{align*}
\textstyle \E \sup_{t\in [0,T]}\|u(t)\|_{H}^{p} + \E\Big|\int_0^T \|u(t)\|^2_V \,\dd t\Big|^{p/2} & \leq C_{T,p} (1+\E\|u_0\|_{H}^{p}).
\end{align*}
Furthermore, the following continuous dependency on the initial data holds: if $u_0^n \in L^0_{\F_0}(\Omega;H)$ are such that $\|u_0-u_0^n\|_H\to 0$ in probability, then for every $T\in (0,\infty)$,
\[\|u - u^n\|_{L^2(0,T;V)} + \|u - u^n\|_{C([0,T];H)}\to 0 \ \ \text{in probability},\]
where $u^n$ is the unique global solution to \eqref{eq:abstractfluid} with initial data $u_0^n$.

Finally, if $B(v)^* v = 0$ for all $v\in V$ and $\lim_{\|v\|_V\to \infty} \frac{\|G(v)\|_{\calL_2(\mathcal{U},H)}}{\|v\|_V} = 0$, then there are constants $C, C_T>0$ such that for all $p\in [2, \infty)$,
\begin{align}\label{eq:fluidpindep}
\textstyle\big\|\sup_{t\in [0,T]}\|u(t)\|_{H}\big\|_{L^p(\Omega)} + p^{-1/2} \|u\|_{L^p(\Omega;L^2(0,T;V))} & \leq C(C_T+\|u_0\|_{L^p(\Omega;H)}).
\end{align}
\end{theorem}
In fluid dynamics, $B(v)^* v = 0$ (or equivalently $(B(v)h,v)_{H}=0$ for $h\in \mathcal{U}$) often follows from the fact that the coefficients appearing in $B$ are divergence-free.
\begin{proof}
It suffices to check the conditions of Theorem \ref{thm:varglobal} and its extension to $p=2$ of Subsection \ref{sss:p=2}.  Assumption \ref{ass:varsetting}\eqref{it1:varsetting} is immediate. To check Assumption \ref{ass:varsetting}\eqref{it2:varsetting} note
\begin{align*}
\|F(u)-F(v)\|_{V^*} &  = \|\Phi(u, u-v)+\Phi(u-v,v)\|_{V^*}
\leq C(\|u\|_{\beta_1} +\|v\|_{\beta_1})\|u-v\|_{\beta_1},
\end{align*}
which gives the desired estimate with $\rho_1=1$, which indeed satisfies \eqref{eq:subcriticalvar}.
Since $G$ is globally Lipschitz, we can take $\rho_2=0$.

The coercivity condition of Theorem \ref{thm:varglobal} follows from Remark \ref{rem:sufficientcoercivity} and the fact that
\begin{align}\label{eq:coercivityfluid}
\lb u, F(u)\rb = \lb u,\Phi(u,u)\rb=0.
\end{align}

Now, we are in the situation that Theorem \ref{thm:varglobal} applies. Note that the moment estimates extend to $\gamma=1$ by the conditions on $G$ and Subsection \ref{sss:p=2}. The final assertion on $p$-th moments follows from \eqref{eq:preciseestpind}.
\end{proof}

\begin{remark}
From \eqref{eq:coercivityfluid} and Remark \ref{rem:sufficientcoercivity} it is clear that the assumption $\lb u,\Phi(u,u)\rb = 0$ can be relaxed to: for all $\varepsilon>0$ there exists a constant $C_{\varepsilon}>0$ such that for all $u\in V$,
\[\lb u,\Phi(u,u)\rb \leq \varepsilon\|u\|_V^2 + C_{\varepsilon} \|u\|_H^2.\]
\end{remark}

\subsubsection{Regularity through a stronger setting}

Via Theorem \ref{thm:strongvar} we can upgrade the regularity of Theorem \ref{thm:fluidabstract} under suitable conditions.
\begin{theorem}[Regularity through the strong setting]\label{thm:fluidstrongvar}
Suppose that the conditions of Theorem \ref{thm:fluidabstract} hold. Suppose that $A$ and $B$ also define bounded operators  $A\in \calL(\wt{V},H)$ and $B\in \calL(\wt{V}, \calL_2(\mathcal{U},V))$.  Suppose that there exist constants $\wt{\theta}>0$ and $M\geq 0$ such that for all $u,v\in \wt{V}$
\begin{align*}
\lb v, Av\rb_{\wt{V},\wt{V}^*}  - \tfrac{1}{2} \|B v\|_{\calL_2(\mathcal{U};V)}^2 & \geq \wt{\theta}\|v\|_{\wt{V}}^2- \wt{M} \|v\|_V^2,
\\ \|\Phi(u,v)\|_{H}& \leq M \|u\|_{[H,\wt{V}]_{3/4}} \|v\|_{[H,\wt{V}]_{3/4}},
\\ \|G(u) - G(v)\|_{\calL_2(\mathcal{U},V)}& \leq M(1+\|u\|_{[H,\wt{V}]_{3/4}}+\|v\|_{[H,\wt{V}]_{3/4}})\|u-v\|_{[H,\wt{V}]_{3/4}},
\\ \|\Phi(v,v)\|_{H}& \leq M\|v\|_{\wt{V}}^{1/2} \|v\|_{V} \|v\|_{H}^{1/2},
\\ \|G(v)\|_{\calL_2(\mathcal{U},V)}& \leq C(1+\|v\|_{V}).
\end{align*}
Let $u_0\in L^0_{\F_0}(\Omega;V)$. Then the solution provided by Theorem \ref{thm:fluidabstract} satisfies $u\in L^2_{\rm loc}([0,\infty);\wt{V})\cap C([0,\infty);V)$ a.s. Moreover, if there is an $\varepsilon\in (0,1)$ with $u_0\in [V,\wt{V}]_{\varepsilon}$ a.s., there exist $\gamma\geq 0$,  $\delta\in (0,2]$
such that  \[\|\Phi(v,v)\|_{H} \leq M \|v\|_{[H,\wt{V}]_{3/4}}^{2-\delta} \|v\|_{V}^{\gamma}, \ \  v\in V,\]
and $(A,B)\in \mathcal{SMR}_{r,\alpha}^{\bullet}$ in the $(\wt{V}, V, H)$-setting for all $r\in (2, \infty)$ and $\alpha\in [0,r/2-1)$, then
\begin{align*}
  u&\in H^{\theta,r}_{\rm loc}((0,\infty);[H,\wt{V}]_{1-\theta}) \cap C^{\theta-\varepsilon}_{\rm loc}((0,\infty);[H,\wt{V}]_{1-\theta}), \ \ r\in (2, \infty), \theta\in [0,1/2), \varepsilon\in (0,\theta).
  \end{align*}
\end{theorem}
\begin{proof}
In the same way as in Theorem \ref{thm:fluidabstract} one can check that Assumption \ref{ass:varsetting} holds in the tilde setting.  Therefore, the first assertion is immediate from Theorem \ref{thm:strongvar}. The second assertion follows from Corollary \ref{cor:strongvarsub} applied in the tilde-setting.
\end{proof}

\subsubsection{Helmholtz decomposition}\label{sss:Helm}
Let $\Dom\subseteq \R^d$ be an open set, possibly unbounded. To introduce the Helmholtz decomposition, we begin by defining some useful spaces.
Let $\Ls^2(\Dom)$ denote the solenoidal subspace of $L^2$, i.e.\ the $L^2$-closure of all $u\in C^\infty_{c}(\Dom;\R^d)$ such that $\div u = 0$. Let $\Gs^2(\Dom)$ denote the space of weak gradients, i.e.
\[\Gs^2(\Dom)=\{\nabla p\in L^2(\Dom): p\in L^2_{\rm loc}(\Dom)\}\]
equipped with the norm $\|\nabla p\|_{L^2}$. Clearly,  $\Ls^2(\Dom)$ is a closed subspace of $L^2(\Dom;\R^d)$. From de Rham's theorem in differential geometry, one can deduce that also $\Gs^2(\Dom)$ is a closed subspace of $L^2(\Dom;\R^d)$ and $\Ls^2(\Dom)^{\bot} = \Gs^2(\Dom)$. Details can be found in \cite[p14]{Temam} (see also \cite[p81]{Sohr}).

Let $\pr:L^2(\Dom;\R^d)\to \Ls^2(\Dom)$ be the orthogonal projection. This projection is called the {\em Helmholtz projection}. The orthogonal decomposition $u = \pr u + \nabla p$ with $\nabla p = u-\pr u\in \Gs^2(\Dom)$ is referred to as the {\em Helmholtz decomposition}.

\subsubsection{Navier--Stokes with no-slip condition}\label{ss:SNS}
We consider the Navier--Stokes system on an arbitrary open set $\Dom\subseteq \R^d$ with $d=2$. We do not assume any specific regularity for $\Dom$ and it may also be unbounded. For simplicity, we focus on the no-slip boundary condition; however, the techniques presented are not limited to that setting. For example, the Navier boundary condition case can be handled with only minor adjustments. Additionally, periodic boundary conditions can also be considered, and they typically offer a simpler framework for analysis:

The problem we study on $\Dom$ is as follows.
\begin{equation}
\label{eq:Navier_Stokes}
\left\{
\begin{aligned}
\dd u &=\big[\nu \Delta u -(u\cdot \nabla)u -\nabla P  \big] \,\dd  t
+\textstyle{\sum_{n\geq 1}}\big[(\btwod_{n}\cdot\nabla) u +g_n(\cdot,u) -\nabla \wt{P}_n\big] \,\dd W_t^n,
\\
\div \,u&=0,
\\ u&=0 \ \text{on $\partial \Dom$},
\\ u(0,\cdot)&=u_0.
\end{aligned}\right.
\end{equation}
Here, $u:=(u^1,u^2):[0,\infty)\times \O\times \Dom\to \R^2$ denotes the unknown velocity field, $P,P_n:[0,\infty)\times \O\times \Dom\to \R$ the unknown pressures, $(W_t^n:t\geq 0)_{n\geq 1}$ a given sequence of independent standard Brownian motions and
\begin{equation*}
\textstyle (\btwod_{n}\cdot\nabla) u:=\Big(\sum_{j\in \{1,2\}} \btwod_n^j \partial_j u^k\Big)_{k=1,2},
\qquad (u\cdot \nabla ) u:=\Big(\sum_{j\in \{1,2\}} u^j \partial_j u^k\Big)_{k=1,2}.
\end{equation*}
One can also cover the Stratonovich formulation of the noise in \eqref{eq:Navier_Stokes}, but for simplicity, we will not do this here. The reader is referred to \cite[Appendix A]{AV20_NS} to see which additional terms need to be considered.

\begin{assumption}\label{ass:SNS}
Let $d=2$. Let $b^j = (b^j_{n})_{n\geq 1}:\Dom\to \ell^2$ be measurable and bounded and suppose that there exists a $\mu\in (0,\nu)$
such that for all $x\in \Dom$
\begin{align*}
\textstyle \frac{1}{2}\sum_{n\geq 1} \sum_{i,j\in \{1,2\}} b_n^i(x) b_n^j(x) \xi_i \xi_j \leq \mu |\xi|^2 \ \ \text{ for all }\xi\in \R^d.
\end{align*}
Moreover, $g:\Dom\times\R^2\to \ell^2$ satisfies $g(\cdot,0)\in L^2(\Dom;\ell^2)$ and
\begin{align*}
\|g(x,y) - g(x,y')\|_{\ell^2}\leq L_g |y-y'|, \ \ \ x\in \Dom, y,y'\in \R^2.
\end{align*}
\end{assumption}
In the above, $g_n:\R\to \R^2$, so $\ell^2$ is understood as an $\R^2$-valued sequence space.

To rewrite \eqref{eq:Navier_Stokes} as \eqref{eq:abstractfluid}, we will apply the theory of Subsection \ref{sss:Helm}. Using the notation introduced there let $\mathcal{U} = \ell^2$ with standard basis $(e_n)_{n\geq 1}$,
\[H = \Ls^2(\Dom), \ \  V = \Hs^1_0(\Dom) = H^1_0(\Dom;\R^2)\cap \Ls^2(\Dom)  \ \ \text{and} \ \   V^* := \Hs^{-1}(\Dom) = (\Hs^1_0(\Dom))^*.\]
Let $J:\Hs^1_0(\Dom)\to H^1_0(\Dom;\R^2)$ be the canonical embedding. Then $J^*:H^{-1}(\Dom;\R^2)\to V^*$. We claim that $J^* f = \pr f$ for all $f\in L^2(\Dom;\R^2)$. Indeed, for all $v\in V$, \[\lb v, J^*f\rb = (v, f)_{L^2(\Dom)} = (\pr v, f)_{L^2(\Dom)} = (v, \pr f)_{L^2(\Dom)} = \lb v, \pr f\rb.\]

By the divergence free condition $(u\cdot \nabla)u=\div(u\otimes u)$, where $u\otimes u$ is the matrix with components $u_{j} u_k$.
Assuming $u_0\in \Ls^2(\Dom)$, after applying the Helmholtz projection $\pr$ to \eqref{eq:Navier_Stokes} (using the same notation $\pr$ for $J^*$ as introduced in Subsection \ref{sss:Helm})
we can write \eqref{eq:Navier_Stokes} in the form \eqref{eq:abstractfluid} with \[A = -\nu \pr\Delta,\ \ \Phi(u,v) =- \pr\div[u\otimes v], \ (B u)e_n = \pr[(\btwod_{n}\cdot\nabla) u], \ \text{and} \  G(u) e_n = \pr g_n(\cdot, u).\]
We will say that $u$ is a solution to \eqref{eq:Navier_Stokes} if $u$ is a solution to \eqref{eq:abstractfluid} with the above choices. As it is known, if $u$ is sufficiently regular, then $\nabla p$ can be recovered from $u$, see Subsection \ref{sss:Helm}.

Below we check Assumption \ref{ass:bilinearFluid} for each of these mappings. Clearly, $A$ and $B$ have the required mapping properties. Indeed, for $A$ this follows from
\[|\lb v, A u\rb|  = \nu |(\nabla v, \nabla u)_{L^2(\Dom)}| \leq \nu \|v\|_{\Hs^1(\Dom)} \|u\|_{\Hs^1(\Dom)}\]
for all $u,v\in \Hs^1(\Dom)$. For $B$ this follows from
\[\textstyle \|Bu\|_{\calL_2(\mathcal{U},H)}^2 = \int_{\Dom}\sum_{n\geq 1} |\pr[(b_n \cdot \nabla) u]|^2 \,\dd x\leq \int_{\Dom} \|b\|^2_{\ell^2} |\nabla u|^2 \,\dd x\leq \|b\|^2_{L^\infty(\Dom;\ell^2)} \|u\|_{V}^2.\]
The coercivity in Assumption \ref{ass:bilinearFluid}\eqref{it1:bilinearFluid} follows from the ellipticity condition in Assumption \ref{ass:SNS} in a similar way as we have seen for the Allen--Cahn equation in Subsection \ref{sss:AllenCahnweak}.

Since $d=2$, by the Sobolev embedding $V_{3/4}\hookrightarrow L^4$ (see \cite[Lemma A.7]{AV20_NS}), it follows that for all $u,v,z\in V$,
\begin{align*}
|\lb z,\Phi(u,v)\rb| &= \textstyle \Big|\sum_{j,k=1}^2 (\partial_j z_k, u_j v_k)_{L^2(\Dom)}\Big| \lesssim  \|z\|_{\Hs^1(\Dom)} \|u\|_{L^4(\Dom;\R^2)} \|v\|_{L^4(\Dom)}\lesssim C_d' \|z\|_{\Hs^1(\Dom)} \|u\|_{\frac34} \|v\|_{\frac34},
\end{align*}
which proves the desired mapping property. The bilinearity is clear from the definition.

Note that for $u\in C^\infty_c(\Dom;\R^2)$ with $\div(u) =  0$ (and by density for $u \in V$) we can write
\begin{align*}
\textstyle \lb u,\Phi(u,u)\rb = \sum_{j,k\in \{1, 2\}} (\partial_j u_k, u_j u_k)_{L^2(\Dom)}
= -\frac12\sum_{j,k\in \{1, 2\}} \int_{\Dom} (\partial_j u_j) u_k^2 \,\dd x = 0.
\end{align*}

For $G$ we note that $G(0)\in \calL_2(\mathcal{U},H)$, and
\begin{align*}
\textstyle  \|G(u) - G(v)\|_{\calL_2(\mathcal{U},H)}^2
= \int_{\Dom}\sum_{n\geq 1}|g_n(\cdot, u) -  g_n(\cdot, v)|^2 \,\dd
x \leq L^2_g \int_{\Dom}|u -  v|^2 \,\dd x =  L_g^2 \|u-v\|_H^2.
\end{align*}
Thus we are in a position to apply Theorem \ref{thm:fluidabstract} to obtain the following result.

\begin{theorem}[Global well-posedness of 2D Navier--Stokes equations]\label{thm:SNS}
Let $d=2$. Suppose that Assumption \ref{ass:SNS} holds,
Then for every $u_0\in L^0_{\F_0}(\Omega;\Ls^2(\Dom))$ there exists a unique global solution $u\in L^2_{\rm loc}([0,\infty);\Hs^1_0(\Dom))\cap C([0,\infty);\Ls^2(\Dom))$ to \eqref{eq:Navier_Stokes}. Moreover, for all $T\in (0,\infty)$ and $p\in (0,2]$
\begin{align}\label{eq:LpNS}
\textstyle \E \sup_{t\in [0,T]}\|u(t)\|_{L^2(\Dom;\R^2)}^{p}
+ \E\Big|\int_0^T \|u(t)\|^2_{H^1(\Dom;\R^2)} \,\dd t\Big|^{p/2} & \leq C_{T,p} (1+\E\|u_0\|_{L^2(\Dom;\R^2)}^{p}).
\end{align}
\end{theorem}
The continuous dependency as stated in Theorem \ref{thm:fluidabstract} holds as well.
Finally, note that for $u\in C^\infty_c(\Dom;\R^2)$,
\begin{align*}
\textstyle (B(u)e_n, u)_{H} = \sum_{j,k=1}^2 \int_{\Dom} u^k b^j_n \partial_j u^k  \,\dd x
= \frac12 \sum_{j,k=1}^2 \int_{\Dom} b^j_n \partial_j (u^k)^2  \,\dd x
= -\frac12 \sum_{k=1}^2 \lb (u^k)^2, \div (b_n)\rb_{\Di(\Dom)}.
\end{align*}
Therefore, if $\div(b_n) = 0$ in distributional sense, then $B(v)^* v =0$. The latter can be extended to all $v\in V$ by continuity. Moreover, if additionally $\lim_{|x|\to \infty}\frac{\|g(x)\|_{\ell^2}}{|x|} = 0$, then Theorem \ref{thm:fluidabstract} implies that \eqref{eq:fluidpindep} holds, and thus, in particular,  \eqref{eq:LpNS} holds for all $p\in (0, \infty)$. The estimate \eqref{eq:LpNS} also extends to all $p\in (0,\infty)$ if $b^j\in W^{1,\infty}_b(\Dom;\ell^2)$ by applying a refined version of the latter.

From Theorem \ref{thm:fluidstrongvar} we obtain the following regularity result for the strong setting. We formulate the result for special domains only, since we need some elliptic regularity theory in the proof.
\begin{theorem}[Strong regularity]\label{thm:SNSstrong}
Let $d=2$ and suppose that $\Dom= \R^2$ (or the two-dimensional torus). Suppose that Assumption \ref{ass:SNS} holds and that
$b^1, b^2\in W^{1,\infty}(\Dom;\ell^2)$ and $u_0\in L^0_{\F_0}(\Omega, \Hs^1(\Dom))$, and
\begin{align*}
\|\partial_x g(x,y) - \partial_x g(x,y')\|_{\ell^2}+ \|\partial_y g(x,y) - \partial_y g(x,y')\|_{\ell^2}\leq L_g |y-y'|, \ \ \ x\in \Dom, y,y'\in \R^2.
\end{align*}
Then the solution $u$ to \eqref{eq:Navier_Stokes} provided by Theorem \ref{thm:SNS} satisfies
$u\in L^2_{\rm loc}([0,\infty);\Hs^2(\Dom))\cap C([0,\infty);\Hs^1(\Dom))$ a.s.
Moreover, if $b^1=b^2=0$, then
\begin{align*}
  u&\in H^{\theta,r}_{\rm loc}((0,\infty);\Hs^{2-2\theta}(\Dom))\cap C^{\theta-\varepsilon}_{\rm loc}((0,\infty);\Hs^{2-2\theta}(\Dom)), \ \ r\in (2, \infty), \theta\in [0,1/2), \varepsilon\in (0,\theta).
  \end{align*}
\end{theorem}
Moreover, if $\div(b^j) = 0$ for $j\in \{1, 2\}$, then $(B v)^* v = 0$, and thus using Remark \ref{rem:integstrong} we see that if $u_0\in L^\infty(\Omega;\Hs^1(\Dom))$, then there exists an $r\in (0,1)$ depending on $\|u_0\|_{L^\infty(\Omega;\Hs^1(\Dom))}$ and $T$ such that
\[\E\|u\|_{C([0,T];\Hs^1(\Dom))}^{2r} + \E\|u\|_{L^2(0,T;\Hs^2(\Dom)))}^{2r}<\infty.\]
\begin{proof}
By elliptic regularity theory $\|\Delta v\|_{L^2(\Dom;\R^2)} + \|u\|_{L^2(\Dom;\R^2)}\eqsim \|v\|_{H^2(\Dom;\R^2)}$. The coercivity of $(A,B)$ in the strong setting can be deduced from this.

Most of the other conditions are straightforward to check, and we only comment on the mapping properties of $\Phi$ and $G$. One has that
\begin{align*}
\|\Phi(u,v)\|_{L^2(\Dom;\R^2)} &= \textstyle \Big(\int_\Dom |(u\cdot \nabla) v|^2 \,\dd x\Big)^{1/2}
\leq \|u\|_{L^4(\Dom;\R^2)} \|v\|_{H^{1,4}(\Dom;\R^2)}
\\ & \lesssim \|u\|_{H^{1/2}(\Dom;\R^2)} \|v\|_{H^{3/2}(\Dom;\R^2)}
 \lesssim \|u\|_{L^2(\Dom;\R^2)}^{1/2} \|v\|_{H^{1}(\Dom;\R^2)} \|u\|_{H^{2}(\Dom;\R^2)}^{1/2},
\end{align*}
where we applied Sobolev embedding and standard interpolation inequalities.
The penultimate estimate shows the required estimate for $\Phi(u,v)$ since $[\wt{V}^*, \wt{V}]_{3/4} \subseteq H^{3/2}(\Dom;\R^2)$.
The last two estimates also show the required two estimates for $\Phi(v,v)$ of Theorem \ref{thm:fluidstrongvar}.

For $G$ by the Lipschitz properties of $g$ and $\partial_x g$ one has
\begin{align*}
\|G(u) - G(v)\|_{\calL_2(\ell^2,H^1(\Dom;\R^2))} & \lesssim \|u-v\|_{L^2(\Dom)}  + \|\partial_y g(\cdot, u)\nabla u -  \partial_y g(\cdot, v)\nabla v\|_{L^2(\Dom;\ell^2)}
\end{align*}
The first term can be bounded as before. The second term can be bounded as
\begin{align*}
\|\partial_y g(\cdot, u)\nabla u -  \partial_y g(\cdot, v)\nabla v\|_{L^2(\Dom;\ell^2)}& \leq \|\partial_y g(\cdot, u)(\nabla u -  \nabla v)\|_{L^2(\Dom;\ell^2)} + \|(\partial_y g(\cdot, u) - \partial_y g(\cdot, v))\nabla v\|_{L^2(\Dom;\ell^2)}
\\ & \lesssim \|\nabla u - \nabla v\|_{L^2(\Dom;\R^2)}+\|(u-v)\nabla v\|_{L^2(\Dom;\R^2)}
\\ & \lesssim \|\nabla u - \nabla v\|_{L^2(\Dom;\R^2)}+\|u-v\|_{L^4(\Dom;\R^2)} \|\nabla v\|_{L^4(\Dom;\R^2)}
\\ & \lesssim (1+\|v\|_{H^{3/2}(\Dom;\R^2)})\|u-v\|_{H^{3/2}(\Dom;\R^2)}.
\end{align*}
Similarly, $\|G(v)\|_{\calL_2(\ell^2,H^1(\Dom;\R^2))}\lesssim C(1+\|v\|_{H^1(\Dom;\R^2)})$.
\end{proof}
After Theorem \ref{thm:SNSstrong} one can bootstrap further regularity. This is not immediately possible in the setting of Theorem \ref{thm:SNS} due to the fact that the nonlinearity $\Phi$ is critical.
In the case of periodic boundary conditions, the regularity conditions on $b^1, b^2$ can be relaxed considerably if one applies $L^p(L^q)$-theory. Moreover, at the same time much stronger regularity assertions can be proved. For details the reader is referred to \cite{AV20_NS}.

In Subsection \ref{subsec:SNSRd}, we will come back to the Navier--Stokes equations in the case the domain is $\R^d$ with $d\in \{2, 3\}$. There $L^p(L^q)$-theory will be used to derive (local) well-posedness and regularity.

\begin{remark}
Related result on the Navier--Stokes can be found in \cite{GlattZiane,goodair2024improved,G24_Goodair}.

Higher order regularity could be of use in numerical schemes for stochastic Navier--Stokes equations as considered in \cite{BP23, breit2023mean}.
\end{remark}

\subsubsection{Boussinesq equation}
The Boussinesq equation is an extension of the Navier--Stokes system in which the temperature $\theta$ is added as an unknown. It is widely studied in the deterministic setting. Moreover, in the stochastic setting it is for instance studied in \cite{ChueshovMillet,DuMi}.

In this section, we consider the following Boussinesq equation on an arbitrary open set $\Dom\subseteq \R^d$ with $d=2$:
\begin{equation}
\label{eq:Navier_Stokes-temperature}
\left\{
\begin{aligned}
\textstyle \dd  u &=\textstyle\big[\nu_1 \Delta u -(u\cdot \nabla)u -\nabla p  + \theta e_2\big] \,\dd  t
+\sum_{n\geq 1}\big[(\btwod_{n}\cdot\nabla) u +g_n(\cdot, u,\theta)-\nabla \wt{p}_n\big] \,\dd W_t^n,
\\ \textstyle
\dd\theta &= \textstyle\big[ \nu_2 \Delta \theta - u \cdot \nabla \theta\big]\,\dd t +\sum_{n\geq 1}\big[(\wt{b}_{n}\cdot\nabla) \theta +\wt{g}_n(\cdot, u,\theta)\big] \,\dd W_t^n,
\\ \div \,u&=0,
\\ u&=0 \  \text{and} \ \theta=0\ \text{on $\partial \Dom$,}
\\ u(0,\cdot)&=u_0.
\end{aligned}\right.
\end{equation}
Here, $e_2 = (0,1)$ is a standard unit vector in $\R^2$.
Again for simplicity, we only consider Dirichlet boundary conditions. In case the temperature vanishes (i.e.\ $\theta=0$), the above equation reduces to \eqref{eq:Navier_Stokes}. Below we show that \eqref{eq:Navier_Stokes-temperature} also fits in the setting of Theorem \ref{thm:fluidabstract}. Note that $(u\cdot \nabla)u = \div(u\otimes u)$ and $u \cdot \nabla \theta = \div(u \theta)$ by the condition $\div(u) = 0$.

\begin{assumption}\label{ass:SNStemp}
Suppose that $b^j:\Dom\to \ell^2$ satisfies Assumption \ref{ass:SNS} with $\nu$ replaced by $\nu_1$. Suppose the same holds with $(b,\nu_1)$ replaced by $(\wt{b},\nu_2)$. Suppose that $g,\wt{g}:\Dom\times\R^2\times\R\to \ell^2$ and $g(\cdot,0),\wt{g}(\cdot, 0)\in L^2(\Dom;\ell^2)$ and
\begin{align*}
\|g(x,y) - g(x,y')\|_{\ell^2}+\|\wt{g}(x,y) - \wt{g}(x,y')\|_{\ell^2}\leq L_g |y-y'|, \ \ \ x\in \Dom, y,y'\in \R^2.
\end{align*}
\end{assumption}

In the same way as in Subsection \ref{ss:SNS} we can derive a well-posedness result for \eqref{eq:Navier_Stokes-temperature} from Theorem \ref{thm:fluidabstract}.
Let $\mathcal{U} = \ell^2$ as before, and set
\[H = \Ls^2(\Dom)\times L^2(\Dom), \ \  V = \Hs^1_0(\Dom)\times H^1_0(\Dom), \ \ \text{and} \ \   V^* := \Hs^{-1}(\Dom)\times H^{-1}(\Dom).\]

Let $u_0\in \Ls^2(\Dom)\times L^2(\Dom)$. After applying the Helmholtz projection $\pr$ we can
write \eqref{eq:Navier_Stokes-temperature} in the form \eqref{eq:abstractfluid} with $A = (A_1, A_2)$, $\Phi = (\Phi_1, \Phi_2)$, $B= (B_1, B_2)$ and $G = (G_1, G_2)$, where
\begin{align*}
 A_1(u,\theta) &= -\nu_1 \pr\Delta u -\pr\theta_1 e_2, &  \Phi_1((u, \theta_1),(v, \theta_2)) &=-\pr\div[u\otimes v],
\\
 A_2(u,\theta) &= -\nu_2 \Delta \theta, & \Phi_2((u,\theta_1), (v, \theta_2)) &=- \div(v \theta_1),
\\
  B_1 (u,\theta))e_n &= \pr[(\btwod_{n}\cdot\nabla) u, &  G_1(u,\theta) e_n &= \pr g_n(\cdot, u,\theta),
\\
 (B_2(u,\theta))e_n &= (\wt{b}_{n}\cdot\nabla) \theta, &  G_2(u,\theta) e_n &= \wt{g}_n(\cdot, u, \theta).
\end{align*}
As before Assumption \ref{ass:bilinearFluid} holds for each of these mappings.
Thus we are in a position to apply Theorem \ref{thm:fluidabstract} to obtain the following result.

\begin{theorem}[Well-posedness of stochastic 2D Boussinesq equations]\label{thm:SNStemp}
Let $d=2$. Suppose that Assumption \ref{ass:SNS} holds,
Then for every $(u_0,\theta_0)\in L^0_{\F_0}(\Omega;\Ls^2(\Dom)\times L^2(\Dom))$ there exists a unique global solution $(u,\theta)\in L^2_{\rm loc}([0,\infty);\Hs^1_0(\Dom)\times H^1_0(\Dom))\cap C([0,\infty);\Ls^2(\Dom)\times L^2(\Dom))$ to \eqref{eq:Navier_Stokes}. Moreover, for all $T\in (0,\infty)$ and $p\in (0,2]$
\begin{equation}\label{eq:LpNStemp}
\begin{aligned}
\E  \textstyle  \sup_{t\in [0,T]}\|u(t)\|_{L^2(\Dom;\R^2)}^{p}+  \E\Big|\int_0^T \|u(t)\|^2_{H^1(\Dom;\R^2)} \,\dd t\Big|^{p/2} &
\\  + \textstyle  \E \sup_{t\in [0,T]}\|\theta(t)\|_{L^2(\Dom)}^{p} + \E\Big|\int_0^T \|\theta(t)\|^2_{H^1(\Dom)} \,\dd t\Big|^{p/2} &   \leq C_{T,p} (1+\E\|u_0\|_{L^2(\Dom;\R^2)}^{p} + \E\|\theta_0\|_{L^2(\Dom)}^{p}).
\end{aligned}
\end{equation}
\end{theorem}
The continuous dependency as stated in Theorem \ref{thm:fluidabstract} holds as well. Moreover, extensions to $L^p$-moments (see below Theorem \ref{thm:SNS}) and higher regularity as in Theorem \ref{thm:SNSstrong} and the text below it hold as well.

\begin{remark}
In \eqref{eq:Navier_Stokes-temperature} one can replace $\theta e_2$ by a more complicated nonlinear function of $(u,\theta)$. However, in that case, one needs a variant of Theorem \ref{thm:fluidabstract}, where an additional $F$ is added to the equation. For this, one needs to check coercivity by hand.
\end{remark}

\subsubsection{Quasi-geostrophic systems with $\alpha=1$}\label{ss:quasiGS}
Quasi-geostrophic equations have a geophysical origin and have been proposed in \cite{held1995surface} as a two-dimensional incompressible model. On $\R^2$ they take the following form:
\begin{equation}
\label{eq:QG}
\left\{
\begin{aligned}
\textstyle \dd \theta &=\textstyle  -\big[ (-\Delta)^{\alpha} \theta + (u\cdot \nabla)\theta \big] \, \dd t
+\sum_{n\geq 1}\big[(\btwod_{n}\cdot\nabla) \theta +g_n(\cdot, \theta)\big] \, \dd W_t^n,
\\
u&=R^{\bot} \theta,
\\ \theta(0,\cdot)&=\theta_0.
\end{aligned}\right.
\end{equation}
Here, $\theta$ represents temperature again, and $R^{\bot} \theta:=(-R_2\theta, R_1 \theta)$, where $R_j = \partial_j(\Delta)^{-1/2}$ for $j\in \{1,2\}$, are the Riesz transforms. Note that $\wh{R_j \phi}(\xi) = \frac{i\xi}{|\xi|} \wh{\phi}(\xi)$ for $\xi\in \R$, where $\wh{\phi}$ denotes the Fourier transform of $\phi:\R^2\to \C$. The same model can be considered in the periodic setting, or on compact manifolds without boundary although some technical complications arise (see \cite{Prussgeo} for the deterministic case). For simplicity, we focus on the flat case, and since we apply $L^2$-theory in this section, we are restricted to $\alpha=1$. In Section \ref{ss:quasiGSLp}, we will show how $L^p(L^q)$-theory can be used to prove global existence, uniqueness, and regularity in the case $\alpha\in (1/2, 1)$.

In the deterministic case global well-posedness is obtained in \cite{ConWu99, Prussgeo} for $\alpha\in (1/2,1]$, and the critical case $\alpha=1/2$ was addressed in \cite{CafVas10,KisNazVol}, and the reader is referred to \cite{ConstQuasigeo} for recent progress on this case. In the stochastic case with periodic boundary conditions, global existence and uniqueness for \eqref{eq:QG} was considered in \cite{RZZ, walker2024surface} for $\alpha\in (1/2,1]$ using compactness methods,  though these methods are not applicable in our setting on the full space.

The proof below can also be extended to the two-dimensional torus or more general compact two-dimensional manifolds in dimension three.

\begin{theorem}[Well-posedness of stochastic quasi-geostrophic systems]\label{thm:QS2d}
Let $d=2$ and $\alpha=1$. Suppose that $b$ and $g$ satisfy Assumption \ref{ass:SNS} with $\Dom$ replaced by $\R^2$.
Then for every $\theta_0\in L^0_{\F_0}(\Omega;L^2(\R^2))$ there exists a unique global solution $\theta\in L^2_{\rm loc}([0,\infty);H^1(\R^2))\cap C([0,\infty);L^2(\R^2))$ to \eqref{eq:QG}. Moreover, for all $T\in (0,\infty)$ and $p\in (0,2]$
\begin{align}\label{eq:LpQS}
\textstyle \E \sup_{t\in [0,T]}\|\theta(t)\|_{L^2(\R^2)}^{p} + \E\Big|\int_0^T \|\theta(t)\|^2_{H^1(\R^2)}\, \dd t\Big|^{p/2} & \leq C_{T,p} (1+\E\|\theta_0\|_{L^2(\R^2)}^{p}).
\end{align}
\end{theorem}
The continuous dependency as stated in Theorem \ref{thm:fluidabstract} holds as well. As in the previous examples, similar extensions to $p$-th moments and on regularity can be made.
\begin{proof}
In order to write \eqref{eq:QG} for $\alpha=1$ in the form \eqref{eq:abstractfluid}, let $V = H^1(\R^2)$, $H = L^2(\R^2)$ and $\mathcal{U} = \ell^2$. Let
\begin{align*}
\lb u, A v\rb &=  \textstyle{\int_{\R^2}\nabla u\, \nabla v\,\dd x }, & \Phi(\theta_1,\theta_2) &= - (R^{\bot} \theta_1 \cdot \nabla) \theta_2
\\ (B \theta)e_n &= (b_n\cdot \nabla) \theta, &  G(\theta)e_n  &= g_n(\theta).
\end{align*}
The coercivity of Assumption \ref{ass:bilinearFluid}\eqref{it1:bilinearFluid} can be checked as in previous examples. Concerning the first part of Assumption \ref{ass:bilinearFluid}\eqref{it2:bilinearFluid}, it suffices to observe that $\Phi(\theta_1,\theta_2) = -\div(\theta_2 R^{\bot} \theta_1)$ since $\div(R^{\bot} \theta_1) = 0$. Since $R^{\bot}$ is bounded on $L^4$ (by Mikhlin's Fourier multiplier theorem \cite[Theorems 4.3.7 and 6.2.7]{Grafakos1}), as in the previous examples, it follows that $\Phi$ has the required mapping properties with $\beta_1=3/4$. Also, the mapping properties of $B$ and $G$ can be checked as before.

The second part of Assumption \ref{ass:bilinearFluid}\eqref{it3:bilinearFluid} follows from
\[\textstyle  \lb \theta, \Phi(\theta,\theta) \rb = (\theta R^{\bot} \theta, \nabla \theta) = \frac12\int_{\R^2} R^{\bot} \theta \cdot \nabla (\theta^2) \,\dd x = \frac12\int_{\R^2} \div(R^{\bot} \theta) \theta^2 \,\dd x = 0\]
which is clear for $\theta\in C^1_c(\R)$ and extends to $\theta\in V$ by density and continuity.
It remains to apply Theorem \ref{thm:fluidabstract}.
\end{proof}

\subsubsection{Other fluid dynamics models}\label{ss:otherfluid}

Besides the models already discussed, numerous other fluid dynamics models can be addressed within the critical variational setting. For example, the so-called tamed Navier--Stokes equation in the unbounded domain $\R^3$ can be treated, as detailed in \cite[Subsection 5.2]{AVvar}, where the strong setting is required. In contrast, a different method based on probabilistic weak solutions was used in \cite{RockZha}.

It is also noteworthy that certain simplified 2D liquid crystal models can be included in our framework. For instance, \cite{RSZ} considered such models on bounded domains, and the same can be done in our setting. There are several advantages of our setting, such as not requiring the domains to be bounded, and the ability to consider gradient noise under the optimal condition given in \eqref{eq:stochastic_parabolicity_example}.

A particularly important model for ocean dynamics are the 3D primitive equations. The well-posedness of the deterministic case remained an open problem for many years before being resolved in the landmark paper \cite{CaoTiti}. One of the challenges with the primitive equations is that they are not coercive in the sense of Theorem \ref{thm:varglobal}. However, the linear part is coercive, which makes it relatively straightforward to obtain local existence and uniqueness using Theorem \ref{thm:varloc}. Through our methods, global existence and uniqueness were recently established in \cite{Primitive1, Primitive2}, where various stochastic versions of the primitive equations with transport noise were considered. To verify the blow-up criteria for global existence, several sophisticated energy bounds were derived. Moreover, various stochastic Gronwall lemmas were employed, so moment bounds could not be obtained in this case. Finally, the case of rough transport noise and higher-order regularity was addressed in \cite{agresti2023primitive} using
$L^p(L^q)$-theory.

\section{Selected applications through  $L^p(L^q)$-theory}\label{sec:LpLq}
In this section, we will give several applications of the local well-posedness theory of Section \ref{sec:loc-well-posed} and the blow-up criteria discussed in Section \ref{sec:blowup}, specifically in situations where the $L^2$-theory turns out to be insufficient. Such limitations may arise due to factors like boundary conditions or rough (transport) noise. These scenarios are illustrated in Subsections \ref{ss:AllenCahnLpLq}, \ref{ss:quasiGSLp}, \ref{ss:reaction}, and \ref{subsec:SNSRd}, respectively. Notably, in the latter subsections, we demonstrate how rough noise and high-order polynomial-type nonlinearities necessitate the use of $L^p(L^q)$-techniques.
Additionally, in these cases, the distinction between $q$ and $p$ becomes crucial, as evidenced in Subsection \ref{subsec:SNSRd} and the deterministic analyses in \cite{G86_NS,PW18}.

In Subsection \ref{ss:AllenCahnLpLq}, we extend our investigation of the stochastic Allen--Cahn equation in a weak setting, building on the analysis initiated in Subsection \ref{ss:AllenCahn}. Recall that in the $L^2$ framework, the dimensional restriction limited us to $d=1$. In this subsection, we consider the Allen--Cahn equation on $\Dom\subseteq \R^3$ with Dirichlet boundary conditions, employing an even weaker setting. By allowing $q\neq 2$,
we eliminate the dimensional restriction for well-posedness. Following a general local well-posedness result, we provide a self-contained proof of global existence for specific parameter choices. Furthermore, we present several regularization results. Moreover, we indicate how some of the more advanced techniques can be used to obtain global well-posedness for rough initial data as well.

In Subsection \ref{ss:reaction}, we follow the works \cite{AVreaction-local, AVreaction-global} to study reaction-diffusion equations with periodic boundary conditions under rough transport noise. We establish general results on local well-posedness and derive a global well-posedness result for systems satisfying a specific coercivity condition. For these results, we outline the proofs, which are based on It\^o's formula for $\|\cdot\|^\zeta_{L^\zeta(\T^d)}$
with sufficiently large $\zeta\geq 2$.
Additionally, we include an application to a simple predator-prey system that lacks coercivity in the variational sense (see Theorem \ref{thm:varglobal}). Here, we provide a detailed proof of global well-posedness, showcasing a technique that could be applied to other non-coercive systems.

In Subsection \ref{ss:quasiGSLp}, we revisit the quasi-geostrophic equation on $\T^2$. Compared to Subsection \ref{ss:quasiGS}, the use of $L^p(L^q)$-techniques enables us to handle fractional regularity $\alpha$ up to the critical value $1/2$. For this model, we establish results on global well-posedness in critical spaces and regularity of the solution.

Finally, in Subsection \ref{subsec:SNSRd}, we address the Navier--Stokes equations with transport noise in the full space $\R^d$. We prove local well-posedness, Serrin’s blow-up criteria, and regularity results. The endpoint Serrin criterion in Theorem \ref{thm:serrin}\eqref{it2:serrin} seems a new result, even in the periodic case.

\subsection{3D Stochastic Allen--Cahn equation with quadratic diffusion}\label{ss:AllenCahnLpLq}
Consider the Allen--Cahn equation on a smooth and bounded domain $\Dom\subseteq \R^3$:
\begin{align}\label{eq:AllenCahnLpLq}
\dd u  & =  \textstyle   \big( \Delta u +u-u^3\big) \, \dd t+ \sum_{n\geq 1}  g_{n}(\cdot,u)\, \dd W^n_t, \qquad u=0 \text{ on }\partial\Dom,
\end{align}
with the initial value condition $u(0,\cdot) = u_0$. In Subsection \ref{ss:AllenCahn}, this equation was analyzed in an $L^2$-setting. In the weak setting, the analysis was constrained to $d=1$ due to the growth of the nonlinearity, while in the strong setting, it was limited to $d\leq 4$ with periodic boundary conditions to avoid compatibility conditions on $g$.

To see this, note that the strong setting for \eqref{eq:AllenCahnLpLq} is $X_1=H^{2,2}(\Dom)\cap H^{1,2}_0(\Dom)$ and $X_0=L^2(\Dom)$ as detailed in Example \ref{ex:extrapolated_Laplace_dirichlet}. Therefore, $X_{1/2}=H^{1,2}_0(\Dom)$ and verifying Assumption \ref{ass:FGcritical} requires $g_n(\cdot,0)|_{\partial\Dom}=0$, as $u|_{\partial\Dom}=0$. In general, such compatibility conditions are often unnatural and lead to various difficulties, which can be circumvented using $L^p(L^q)$-theory.

To keep the discussion as self-contained as possible, we do not consider gradient noise, as this introduces additional complications in the case of Dirichlet boundary conditions (see Subsection \ref{ss:reaction} for the periodic case). Finally, let us mention that the results below extend also to unbounded domains with slight modifications to the conditions on $g$.

\smallskip

Unlike the following subsections and our work \cite{AVreaction-global}, our goal here is to establish the global well-posedness of \eqref{eq:AllenCahnLpLq} in the simplest possible manner, as shown in Theorem \ref{thm:globalAllenCahn} below. While more sophisticated methods allow for deeper results, we focus on presenting a relatively straightforward case to familiarize the reader with the core arguments.

In particular, we limit the application of the instantaneous regularization results from Section \ref{subsec:reg} and instead rely on the subcritical blow-up criteria in Theorem \ref{thm:subcriticalblowup}, whose proof is comparatively simpler than that of Theorem \ref{thm:criticalblowup}. Later, in Theorem \ref{thm:global_extrapolation_AC}, we demonstrate how the global well-posedness of \eqref{eq:AllenCahnLpLq} can be extrapolated to a critical setting using instantaneous regularization.

As a first step, it is instructive to examine the scaling of \eqref{eq:AllenCahnLpLq}.

\subsubsection{Scaling, criticality and setting}
\label{sss:scaling_AC_3D_discussion}
The Allen--Cahn nonlinearity was previously analyzed in an $L^q$-weak setting in Subsection \ref{ss:tracecriticalF}, where it was treated as \eqref{eq:SEE} with the choices $X_{0}=H^{-1,q}(\Dom)$ and $X_1=H^{1,q}_0(\Dom)$. As noted in that subsection, the Allen--Cahn equation shares the same (local) scaling as the Navier--Stokes equations analyzed in Subsection \ref{ss:scaling_intro}. Therefore, in light of our theory, we expect local well-posedness and blow-up criteria in the critical space $B^{3/q-1}_{q,p}(\Dom)$, accommodating potential boundary conditions.

The analysis in Subsection \ref{ss:tracecriticalF} demonstrates that (see \eqref{eq:condition_parameters_AC} and the comments below it), by applying Theorem \ref{thm:localwellposed}, one can establish local well-posedness of \eqref{eq:AllenCahnLpLq} in the critical space $\BD^{3/q-1}_{q,p}(\Dom)$ for $\frac{3}{2}<q<3$ (see Example \ref{ex:extrapolated_Laplace_dirichlet} for the notation). In particular, the smoothness of the critical space satisfies $\frac{3}{q}-1>0$.
This smoothness condition is linked to restrictions on the time weight $\a\in [0,\frac{p}{2}-1)\cup\{0\}$, see \eqref{eq:critical_condition_AC}-\eqref{eq:critical_condition_AC_2}.
Positive smoothness of the critical space can impose significant constraints, especially in studying global well-posedness, where global bounds must be proven in a space with positive smoothness to verify blow-up criteria, as discussed in Subsection \ref{subsec:blow_up}.
Since this weight range is natural in the context of SPDEs (see Subsection \ref{ss:stoch_max_reg}), we adopt an even \emph{weaker} setting to address these limitations:
\begin{equation}
\label{eq:choice_X0X1_AllenCahn}
X_0 = \HD^{-\delta, q}(\Dom) \quad \text{ and }X_1 = \HD^{2-\delta, q}(\Dom),
\end{equation}
with $\delta\in [1,2)$. Here we avoid $\delta> 2$, as in this case, $X_1$ would become larger than $L^q(\Dom)$, resulting in distributions rather than functions (actually we will use an even smaller interval for $\delta$, due to restrictions related to Sobolev embeddings).
Regarding the diffusion coefficients $g_n$, the discussion at the end of Subsection \ref{ss:tracecriticalF} highlights that quadratic growth in
$g$ is natural, ensuring that the deterministic and stochastic components of \eqref{eq:AllenCahnLpLq} maintain the same scaling. Specifically, we assume that the measurable mapping $g=(g_n)_{n\geq 1}:\Dom\times \R\to \ell^2$ satisfies for all
$x\in \Dom$ and $y,y'\in \R$,
\begin{equation}
\label{eq:assumption_g_growth_1_AC3D}
g(\cdot, 0)\in L^\infty(\Dom;\ell^2) \quad \text{ and }\quad
\|g(x,y)-g(x,y')\|_{\ell^2} \lesssim (1+|y|+|y'|)|y-y'|.
\end{equation}

\subsubsection{Local regularity and regularity}
Before addressing the global well-posedness of \eqref{eq:AllenCahnLpLq}, our first objective is to establish its local well-posedness. To this end, we apply Theorem \ref{thm:localwellposed}. To begin, we reformulate \eqref{eq:AllenCahnLpLq} as a stochastic evolution equation \eqref{eq:SEE}. More precisely, we interpret \eqref{eq:AllenCahnLpLq} as \eqref{eq:SEE} with
$(X_0,X_1)$ given by \eqref{eq:choice_X0X1_AllenCahn}, $\mathcal{U}=\ell^2$, and for $u\in X_1$,
\begin{equation}
\begin{aligned}
\label{eq:choice_ABFG_AC3D}
A u =-\Delta_D u , \qquad  F(u)= u-u^3,
\qquad B u& = 0,\qquad G(u) = (g_n (\cdot,u))_{n\geq 1},
\end{aligned}
\end{equation}
using the notation introduced in Example \ref{ex:extrapolated_Laplace_dirichlet}.

We say that $(u,\sigma)$ is a (unique) $(p,\a,\s,q)$-solution to \eqref{eq:AllenCahnLpLq} if $(u,\sigma)$ is a $L^p_\a$-solution to \eqref{eq:AllenCahnLpLq} with the choices \eqref{eq:choice_X0X1_AllenCahn},  \eqref{eq:choice_ABFG_AC3D} and $\mathcal{U}=\ell^2$.

\smallskip

By Appendix \ref{appendix:extra} and Theorem \ref{thm:SMRHinfty}, $(-\Delta_D, 0)\in  \mathcal{SMR}_{p,\kappa}^{\bullet}$ for all $p>2$ and $\kappa\in [0,\frac{p}{2}-1)$.
To apply Theorem \ref{thm:localwellposed} for obtaining local well-posedness, it remains to check the local Lipschitz condition of Assumption \ref{ass:FGcritical} with $F$ and $G$ as above. Similar as in Subsection \ref{sss:tracecriticalF_AC}, for $u,v\in X_1=\HD^{2-\delta,q}(\Dom)$,
\begin{align*}
\|u^3-v^3\|_{X_0} &\leq \|u^3-v^3\|_{L^r(\Dom)}
\leq 2(\|u\|_{L^{3r}(\Dom)}^2+\|v\|_{L^{3r}(\Dom)}^2) \|u-v\|_{L^{3r}(\Dom)}.
\end{align*}
Here we used Sobolev embeddings with $-\frac{3}{r} = -\delta-\frac{3}{q}$, where we impose
$2\leq q<6/\delta$ and $\delta\in [1, 3/2)$ in order to guarantee $r\in (1, \infty)$. It remains to embed $X_{\beta_1} = \HD^{2\beta_1-\delta,q}\hookrightarrow L^{3r}(\Dom)$ for suitable $\beta_1\in (1/2, 1)$. Sobolev embedding gives the required embedding for $2\beta_1 - \delta-\frac{3}{q} = -\frac{3}{3r} = -\frac{\delta}{3}-\frac{1}{q}$. This implies $\beta_1 =  \frac{\delta}{3} + \frac{1}{q}$ which lies in $(\frac{1}{2}, 1)$ due to  $\delta\in [1,\frac{3}{2})$ and $2\leq q<\frac{6}{\delta}$. Thus, we conclude
\[\|F(u) - F(v)\|_{X_0}\leq C (1+\|u\|_{X_{\beta}}^2+\|v\|_{X_{\beta}}^2) \|u-v\|_{X_{\beta_1}}.\]
Since $\rho_1=2$, the criticality condition \eqref{eq:subcritical} becomes:
\begin{align}
\label{eq:criticality_condition_AC_3D}
\textstyle
\frac{1+\kappa}{p}\leq \frac{3}{2} (1-\frac{\delta}{3} -\frac{1}{q})= \frac{3}{2} - \frac12\big(\delta + \frac{3}{q}\big).
\end{align}
This illustrates the advantage of $L^p(L^q)$-theory, as this condition always holds for sufficiently large $p$ and $q$. Some admissible cases include $(\delta,q,p,\kappa)$ equal $(1,q,4,0)$ and $(1,3,p,\kappa)$, provided the restrictions $q\in [2, \frac{6}{\delta})$, $p\in [2, \infty)$ and $\kappa\in [0,p/2-1)\cup\{0\}$ are satisfied.

For $G$, by \eqref{eq:assumption_g_growth_1_AC3D}, it follows that
\begin{align}
\label{eq:computation_G_3D_AC}
\|G(u) - G(v)\|_{\gamma(\ell^2, X_\frac12)} &\leq \|g(\cdot, u) - g(\cdot,v)\|_{\gamma(\ell^2, L^r(\Dom))}
\\
\nonumber
& \eqsim \|g(\cdot, u) - g(\cdot, v)\|_{L^r(\Dom;\ell^2)}
\\
\nonumber
& \leq L \|(1+|u|+|v|)(u-v)\|_{L^r(\Dom)}
\\
\nonumber
& \leq \wt{L}(1+\|u\|_{L^{2r}(\Dom)}+\|v\|_{L^{2r}(\Dom)}) \|u-v\|_{L^{2r}(\Dom)}.
\end{align}
Here we used Sobolev embedding with $-\frac{3}{r} = 1-\delta-\frac{3}{q}$.  It remains to embed $X_{\beta_2} = \HD^{2\beta_2-\delta,q}\hookrightarrow L^{2r}(\Dom)$ for suitable $\beta_2\in (1/2, 1)$. Sobolev embedding gives the embedding for $2\beta_2 - \delta-\frac{3}{q} = -\frac{3}{2r} =
\frac{1}{2}-\frac{\delta}{2}-\frac{3}{2q}$. This implies $2\beta_2 = \frac{1}{2} + \frac{\delta}{2}+ \frac{3}{2q}$, which ensures $\beta_2\in (1/2,1)$. Since $\rho_1=1$, the criticality condition \eqref{eq:subcritical} gives the same inequality  \eqref{eq:criticality_condition_AC_3D}.
This proves the required local Lipschitz property for $G$. Since $g(\cdot,0)\in L^\infty(\Dom;\ell^2)$, $G$ also satisfies the necessary mapping property.

\smallskip

Theorem \ref{thm:localwellposed} and the above discussion yield
\begin{proposition}[Local well-posedness in critical spaces]
\label{prop:reaction_diffusion_global}
Suppose that $\delta\in [1, \frac{3}{2})$, $q\in [2, \infty)$
and that $p\in (2, \infty)$ and $\a\in [0,\frac{p}{2}-1)$ satisfy
\begin{align}
\label{eq:reaction_diffusion_globali}
\textstyle
q<\frac{6}{\delta} \ \  \ \text{and} \ \ \ \frac{1+\a}{p}+\frac{1}{2}(\reg+\frac{3}{q})\leq \frac{3}{2}.
\end{align}
Suppose $g:\Dom\times \R\to \ell^2$ is measurable and satisfies \eqref{eq:assumption_g_growth_1_AC3D}.
Then for any $u_0\in L^0_{\F_0}(\O;\BD^{2-\reg-2\frac{1+\a}{p}}_{q,p}(\Dom))$,
\eqref{eq:AllenCahnLpLq} has a (unique) $(p,\kappa,\delta,q)$-solution $(u,\sigma)$ satisfying a.s.\ $\sigma>0$ and
\begin{equation}
\label{eq:regularity_u_reaction_diffusion_critical_spaces_1_AC}
u\in H^{\theta,p}_{\rm loc}([0,\sigma),w_{\a};\HD^{2-\delta-2\theta,q}(\Dom))\cap C([0,\sigma);\BD^{2-\delta-2\frac{1+\a}{p}}_{q,p}(\Dom))\  \text{ a.s.\ for all }\theta\in [0,\tfrac{1}{2}).
\end{equation}
Furthermore, the setting is subcritical whenever strict inequality holds in \eqref{eq:reaction_diffusion_globali}.
\end{proposition}

Before proceeding further, let us discuss the scenario where the setting is critical, i.e.\ \eqref{eq:reaction_diffusion_globali} holds with equality. When \eqref{eq:reaction_diffusion_globali} holds with equality, the condition $\frac{1+\a}{p}<\frac{1}{2}$ implies
$q<\frac{3}{2-\s}$. Combining this restriction with the first condition in \eqref{eq:reaction_diffusion_globali}, it follows that the optimal constraint is $q<\frac{3}{2-\s}$ for $\delta\in(1,\frac{4}{3}]$, and no new critical spaces emerge for $\delta>\frac{4}{3}$. Therefore, when $q<\frac{3}{2-\s}$ and $\delta\in(1,\frac{4}{3}]$, one can set $\a=\frac{p}{2}(3-\reg-\frac{3}{q})-1$ resulting in the space for the initial data:
$$
\BD_{q,p}^{2-\delta-2\frac{1+\a}{p}}(\Dom)= \BD_{q,p}^{\frac{3}{q}-1}(\Dom),
$$
which is critical for \eqref{eq:AllenCahnLpLq}, see Subsection \ref{sss:scaling_AC_3D_discussion}.
By choosing $\delta= \frac{4}{3}$ and letting $q\uparrow \frac{9}{2}$, one can reach critical spaces with smoothness up to $-\frac{1}{3}$. In particular, by choosing $q=3$, $p\geq 3$ and $\a=\frac{p}{2}(2-\reg)-1$, the above result yields local well-posedness of \eqref{eq:AllenCahnLpLq} in the critical space
$
L^3(\Dom)\subseteq \BD^0_{3,p}(\Dom).
$

\subsubsection{Global well-posedness}
As announced at the beginning of this subsection, we aim to establish a global well-posedness result for \eqref{eq:AllenCahnLpLq} under specific choices of the parameters $(p,\a,\s,q)$, employing the relatively simple blow-up criterion of Theorem \ref{thm:subcriticalblowup}. To motivate the parameter selection, let us recall some key considerations.

As we have seen in Subsection \ref{ss:AllenCahn}, the Allen--Cahn nonlinearity $u-u^3$ is dissipative, enabling an a priori $L^q$-bound that facilitates global well-posedness through a blow-up criterion. However, as noted earlier, the choice $q=2$ is insufficient for (even) the local well-posedness of \eqref{eq:AllenCahnLpLq} in three dimensions. Based on the comments below  \ref{prop:reaction_diffusion_global} and the requirement of a subcritical setting, we aim to establish a bound in $L^q$ with $q>3$.
Once an $L^q$-bound is achieved, to derive global well-posedness using Theorem \ref{thm:subcriticalblowup}, it is necessary to ensure that $L^q(\Dom)\embed \BD^{2-\s-2\frac{1+\a}{p}}_{q,p}(\Dom)$. Since $\frac{1+\a}{p}<\frac{1}{2}$, this embedding condition forces $\s>1$. Consequently, we choose $\s=\frac{4}{3}$. This choice is somehow arbitrary but motivated by the discussion on critical spaces following Proposition \ref{prop:reaction_diffusion_global}.

Finally, as in Subsection \ref{ss:AllenCahn} (see also Remark \ref{rem:AC_variational_quadratic_growth}), we require a condition ensuring that the dissipation of $-u^3$ counterbalances the energy production from $g(u)$. Therefore, we assume that
\begin{equation}
\label{eq:assumption_diffusion_AC_subquadratic_global}
\|g(x,y)\|^2_{\ell^2}\leq C(y^2+1)+ \gamma y^4, \ \ x\in \Dom, y\in \R,
\end{equation}
where $\g\in (0,1)$. As we will demonstrate in the proof below, the condition $\g<1$ is crucial for obtaining an $L^q$-bound with $q>3$ and managing the It\^o-correction. The case $\gamma=1$ can still be handled, but it requires a more sophisticated argument, see  \cite[Theorem 3.2]{AVreaction-global}).

\begin{theorem}[Global well-posedness]\label{thm:globalAllenCahn}
Let $g:\Dom\times \R\to \ell^2$ and suppose that \eqref{eq:assumption_g_growth_1_AC3D} and \eqref{eq:assumption_diffusion_AC_subquadratic_global} hold for some $\g<1$.
Let $q = 1+ \frac{2}{\gamma}>3$ and $\delta=\frac43$. Fix $p\in (2,3)$ and $\a>0$ such that $\frac{1+\a}{p} < \frac{5}{6} - \frac{3}{2q}$ and $\frac{1+\a}{p}>\frac{1}{3}$.
Then for every $u_0\in L^0_{\F_0}(\Omega;L^q(\Dom))$, \eqref{eq:AllenCahnLpLq} has a \emph{global (in time)} $(p,\a,\delta,q)$-solution $u$ such that \eqref{eq:regularity_u_reaction_diffusion_critical_spaces_1_AC} holds with $\sigma=\infty$.
\end{theorem}

The condition $\frac{1+\a}{p} < \frac{5}{6} - \frac{3}{2q}$ is used to enable $\s=\frac{4}{3}$ in \eqref{eq:reaction_diffusion_globali}. The condition on $(p,\kappa)$ ensures that $\frac{2}{3}<2\frac{1+\a}{p}$ and therefore the validity of the embedding:
\begin{equation}
\label{eq:Lq_embed_BD_Allen_cahn}
L^q(\Dom)\subseteq \BD^{\frac{2}{3}-2\frac{1+\a}{p}}_{q,p}(\Dom).
\end{equation}
We use a non-trivial weight $\a>0$ in order to apply the `easy' regularization result of Theorem \ref{thm:parabreg}. The need for regularization is explained in Step 1 below.

\begin{proof}
The existence of a $(p,\a,\s,q)$-solution to \eqref{eq:AllenCahnLpLq} with $(p,\a,\s,q)$ as in the statement of Theorem \ref{thm:globalAllenCahn} follows from Proposition \ref{prop:reaction_diffusion_global}. It remains to show $\sigma=\infty$ a.s.\

{\em Step 1: We prove the a priori bound \eqref{eq:apriori_bound_Lq_AC3D} below for $u$ up to the blow-up time $\sigma$ from It\^o's formula.} Note that if we merely know that  \eqref{eq:regularity_u_reaction_diffusion_critical_spaces_1_AC} holds, then $\nabla u$ is only defined as a distribution (recall that $\s=\frac{4}{3}$).
Thus, to apply the It\^o formula to compute $\|u\|_{L^q}^q$, we need more regularity. To this end, we employ the results in Subsection \ref{subsec:reg}.
 Let $0<s<T<\infty$ and $(\sigma_n)_{n\geq 1}$ be a localizing sequence and fix $n\geq 1$. Note that $\sigma_n<\sigma$ a.s.\ (see Proposition \ref{prop:predsigma}). By Theorem \ref{thm:parabreg} a.s.\ on $\{\sigma_n>s\}$,
\[u\in C([s,\sigma_n];\HD^{2-\delta-2\theta,q}(\Dom)),  \ \ \theta\in (0,1/2).
\]
In particular, $u(s)\in \BD^{1-2/\wt{p}}_{q,\wt{p}}(\Dom))$ a.s. on $\{\sigma_n>s\}$ where $\wt{p}\in (2, 6)$ is fixed. For It\^o's formula, we will need $C([s,T];L^q(\Dom))$-regularity of a suitable process. To obtain this, the random variable $\one_{\{\sigma>s\}}u(s)$ will be our initial value to a linear stochastic evolution equation in the $L^{\wt{p}}$-setting for $\wt{X}_0 = \HD^{-1,q}(\Dom)$.
We check that $F(u)$ and $G(u)$ have the required integrability and regularity to apply maximal $L^{\wt{p}}$-regularity. By Sobolev embedding using $\theta$ small enough, $\delta=\frac{4}{3}$ and $q\geq 3$ we find that a.s.
$
u\in C([s,\sigma_n];L^{2q}(\Dom)).
$
In particular, by Sobolev embedding, we find that
\[F(u)\in C([s,\sigma_n], L^{2q/3}(\Dom))\subseteq C([s,\sigma_n], \HD^{-1,q}(\Dom)).\]
Since $g$ grows at most quadratically, we find that $G(u)\in C([s,\sigma_n];L^{q}(\Dom;\ell^2))$.

Let $\Gamma \subseteq \{\sigma_n>s\}$ be $\F_s$-measurable.
Since $(-\Delta,0)\in \mathcal{SMR}_{\wt{p},0}^{\bullet}$ (see Definition \ref{def:SMRbullet}) on the space $\wt{X}_0 = \HD^{-1,q}(\Dom)$, we know that there is a unique $L^{\wt{p}}$-solution
\begin{equation}
\label{eq:vn_u_regularity_AC_3D_proof}
v_n\in L^{\wt{p}}_{\rm loc}([s,\infty);\HD^{1,q}(\Dom))\cap C([s,\infty);\BD^{1-\frac{2}{\wt{p}}}_{q,\wt{p}}(\Dom))
\end{equation}
to the linear problem
\begin{align}
\label{eq:eqvallencahn}
\begin{aligned}
\dd v  + A v \, \dd t = \one_{[s,\sigma_n)\times \Gamma}F(u)\, \dd t + \one_{[s,\sigma_n)\times \Gamma}G(u)\, \dd W,
\qquad
v(s) = \one_{\Gamma} u(s).
\end{aligned}
\end{align}
Clearly, $v_n$ is an $L^{\wt{p}\wedge p}$-solution on $X_0 = \HD^{-\delta,q}(\Dom)$ to this problem as well. Since $\wt{u} := \one_{\Gamma} u$ is an $L^{\wt{p}\wedge p}$-solution on $X_0$ to the same problem on $[s,\sigma_n\vee s]$, by $(\Delta, 0)\in \mathcal{SMR}_{p\wedge \wt{p},0}^{\bullet}$ on the space $X_0$ it follows that $\wt{u} = v_n$ on $[s,\sigma_n\vee s]$, and thus $u = v_n$ on $[s,\sigma_n]\times \Gamma$.

By It\^o's formula applied to $\|v_n\|^{q}_{L^{q}(\Dom)}$ we obtain that a.s.\ for all $t\geq s$
\begin{align*}
\|v_n(t)\|^{q}_{L^{q}(\Dom)} = & \textstyle\|\one_{\Gamma}u(s)\|^{q}_{L^{q}(\Dom)} - q\int_s^t \int_{\Dom} |v_n|^{q-2}  |\nabla v_n|^2 \,\dd x \,\dd r
\\ & \textstyle+ q\int_s^t \int_{\Dom} |v_n|^{q-2}  v_n \one_{[s,\sigma_n]} F(u) \,\dd x \,\dd r
\\ & \textstyle+ q\int_s^t \int_{\Dom} |v_n|^{q-2} v_n \one_{[s,\sigma_n]} G(u) \,\dd x \,\dd W
\\ & \textstyle+ \frac{q(q-1)}{2}\int_s^t \int_{\Dom} |v_n|^{q-2} \one_{[s,\sigma_n]} \|G(u)\|_{\ell^2}^2 \,\dd x \,\dd r.
\end{align*}
In the above, one can replace $F(u)$ and $G(u)$ by $F(v_n)$ and $G(v_n)$ since $u = v_n$ on $[s,\sigma_n]\times\{\sigma_n>s\}$. Let $f(y)=y-y^3$ be the Allen--Cahn nonlinearity. The assumption on $g$ and the choice of $q$ yields
\[\textstyle
y f(y) + \tfrac{q-1}{2}\|g(x,y)\|_{\ell^2}^2 \leq -y^4 + y^2 +  \tfrac{q-1}{2} C(y^2+1)+ \frac{q-1}{2} \gamma y^4 \leq C' (1+y^2), \ \ y\in \R,\]
where $C'$ depends only on $C$ and $\g$. Therefore, we can conclude that
\begin{align*}
\textstyle \|v_n(t)\|^{q}_{L^{q}(\Dom)} \leq
\|\one_{\Gamma}u(s)\|^{q}_{L^{q}(\Dom)}  +C'' \int_s^t (1+\|v_n(r)\|^{q}_{L^{q}(\Dom)}) \,\dd r  + M_t,
\end{align*}
where $C''$ depends on $q, C'$ and the Lipschitz constant of $G$, and $M$ is a continuous local martingale.
The stochastic Gronwall Lemma \ref{lem:Gronwall} applied to the process $t\mapsto \|v_n(t)\|^{q}_{L^{q}(\Dom)}$ implies that
\[\textstyle \E \sup_{t\in [0,T]} \|v_n(t)\|^{\lambda q}_{L^{q}(\Dom)}\leq C_{\lambda} (1+\E \|\one_{\Gamma}u(s)\|^{\lambda q}_{L^{q}(\Dom)}), \ \ \lambda\in (0,1).\]
Taking $\Gamma_{k,n} = \{\|u(s)\|_{L^{q}(\Dom)}\leq k\}\cap \{\sigma_n>s\}$ and letting $n\to \infty$, we can conclude
\[\textstyle \E \sup_{t\in [0,\sigma\wedge T)} \one_{\Gamma_k}\|u(t)\|^{\lambda q}_{L^{q}(\Dom)} \leq C_{\lambda} (1+\E \|\one_{\Gamma_k} u(s)\|^{\lambda q}_{L^{q}(\Dom)}),\]
where $\Gamma_k = \{\|u(s)\|_{L^{q}(\Dom)}\leq k\}\cap \{\sigma>s\}$.
In particular, $\sup_{t\in [s,\sigma\wedge T)}\|u(t)\|_{L^{q}(\Dom)}<\infty$ a.s. on $\Gamma_k$. Since $\Gamma_k$ increases to $\{\sigma>s\}$ a.s., it follows that
\begin{equation}
\label{eq:apriori_bound_Lq_AC3D}
\sup_{t\in [s,\sigma\wedge T)}\|u(t)\|_{L^{q}(\Dom)}<\infty \ \ \text{a.s. on $\{\sigma>s\}$.}
\end{equation}

{\em Step 2: Conclusion.}
From \eqref{eq:apriori_bound_Lq_AC3D} proven in Step 1 and \eqref{eq:Lq_embed_BD_Allen_cahn}, it follows that
\begin{align*}
\P(\sigma<\infty) = \lim_{s\downarrow 0, T\to \infty}\P(s<\sigma<T)
& = \lim_{s\downarrow 0,T\to \infty}\P\big(s<\sigma<T, \sup_{t\in [s,\sigma)} \|u(t)\|_{\BD^{\frac23-\frac{2}{p}}_{q,p}(\Dom)}<\infty\big)
\\ & =
\P\big(\sigma<\infty,\sup_{t\in [0,\sigma)} \|u(t)\|_{\BD^{\frac23-\frac{2}{p}}_{q,p}(\Dom)}<\infty\big) = 0,
\end{align*}
where in the last step we used Theorem \ref{thm:subcriticalblowup} and the fact that we are in the subcritical setting. Thus, $\sigma = \infty$ a.s.\ as desired.
\end{proof}

\subsubsection{Refining the global well-posedness and regularity}
\label{sss:extrapolation_AC_3D}
In Theorem \ref{thm:globalAllenCahn}, we established global well-posedness of \eqref{eq:AllenCahnLpLq} under the growth assumption \eqref{eq:assumption_diffusion_AC_subquadratic_global} and a very specific choice of the parameter $(p,\kappa,\s,q)$. Notably, we excluded the case of critical initial data. While the growth assumption on $g$ is natural and intrinsically tied to the energy dissipation properties of the SPDEs, the particular choice of parameters is somewhat restrictive.
This limitation can be addressed by \emph{extrapolating} the global existence through instantaneous regularization. This is connected to the fact that, if instantaneous regularization holds, the blow-up criteria become independent of the specific parameter setting. This result is formalized in Corollary \ref{cor:transfblowup} (see also Figure \ref{fig:diagram_abstract}).

The independence of the blow-up criteria from the choice of $(p,\kappa)$ implies that whenever a blow-up criterion guarantees global well-posedness in one parameter setting, the same conclusion holds across all other settings where local well-posedness can be established.
This is the content of the following result.

\begin{theorem}[Global well-posedness and regularization in critical spaces]
\label{thm:global_extrapolation_AC}
Let the assumptions of Proposition \ref{prop:reaction_diffusion_global} be satisfied and suppose that \eqref{eq:assumption_diffusion_AC_subquadratic_global} holds with $\gamma\in (0,1)$. Then for all $\delta\in [1, 3/2)$, $q\in [2, \infty)$, $p\in (2, \infty)$ and $\a\in [0,\frac{p}{2}-1)$ satisfying \eqref{eq:reaction_diffusion_globali}, there exists a global (unique) $(p,\a,\s,q)$-solution $u$ to \eqref{eq:AllenCahnLpLq} satisfying \eqref{eq:regularity_u_reaction_diffusion_critical_spaces_1_AC} with $\sigma=\infty$ and a.s.\
\begin{align}
\label{eq:reaction_diffusion_H_theta_AC_global_extrapolation}
u\in L^r_{\rm loc}((0,\infty);\HD^{1,r}(\Dom))
\cap C^{\theta/2,\theta}_{\rm loc}((0,\sigma)\times \overline{\Dom}), \ \ r\in (2,\infty), \theta\in (0,1).
\end{align}
\end{theorem}

As noted above Theorem \ref{thm:globalAllenCahn}, the case $\gamma=1$ is also valid.
From the discussion below Proposition \ref{prop:reaction_diffusion_global}, it follows that Theorem \ref{thm:global_extrapolation_AC} establishes global well-posedness of \eqref{eq:AllenCahnLpLq} with initial data belonging to the critical spaces $\BD^{3/q-1}_{q,p}(\Dom)$ with smoothness up to $-\frac{1}{3}$.

\begin{proof}
Let $(p,\s,\a,q)$ be as in the statement of Theorem \ref{thm:globalAllenCahn} and let $(u,\sigma)$ be the  $(p,\s,\a,q)$-solution of \eqref{eq:AllenCahnLpLq}. Arguing as in Step 1 in the proof of Theorem \ref{thm:globalAllenCahn}  and employing Theorem \ref{thm:parabreg2} in case $\a=0$, one can readily show that
 $(u,\sigma)$ instantaneously regularizes in time: a.s.\
\begin{align}
\label{eq:reaction_diffusion_H_theta_AC_global_extrapolation_proof0}
u \in H^{\theta,r}_{\rm loc}((0,\sigma);\HD^{2-\s-2\theta,q}(\Dom)), \ \ r\in (2,\infty), \theta\in (0,1).
\end{align}
From the above smoothness and iterating the argument in Step 1, one also obtains (see the proof of Proposition \ref{prop:locQS} for a detailed argument for another equation)
\begin{align}
\label{eq:reaction_diffusion_H_theta_AC_global_extrapolation_proof}
u \in L^r_{\rm loc}((0,\sigma);\HD^{2-\s,r}(\Dom))
\cap C^{\theta/2,\theta}_{\rm loc}((0,\sigma)\times \overline{\Dom}), \ \ r\in (2,\infty), \theta\in (0,1).
\end{align}
Hence, it remains to prove $\sigma=\infty$ a.s.
Recall that \eqref{eq:assumption_diffusion_AC_subquadratic_global} holds with $\g<1$.
In light of \eqref{eq:reaction_diffusion_H_theta_AC_global_extrapolation_proof}, one can repeat almost verbatim the proof of \cite[Theorem 2.10]{AVreaction-local} or following Corollary \ref{cor:transfblowup} (with the minor difference of varying also the integrability parameter $q$) to show that, for all $s>0$,
\begin{equation}
\label{eq:blow_up_criteria_AC_3D_extrapolated}
\textstyle \P(s<\sigma<\infty,\, \sup_{t\in [s,\sigma)}\|u(t)\|_{L^{\wt{q}}(\Dom)}<\infty)=0,
\end{equation}
where $\wt{q}=1+ \frac{2}{\gamma}$ (as in Theorem \ref{thm:globalAllenCahn}). Now, by \eqref{eq:reaction_diffusion_H_theta_AC_global_extrapolation_proof}, one can repeat the proof of Theorem \ref{thm:globalAllenCahn} for $(u,\sigma)$ and obtain the a priori bound \eqref{eq:apriori_bound_Lq_AC3D} for the $(p,\a,\s,q)$-solution $(u,\sigma)$.  Arguing as in Step 2 of Theorem \ref{thm:globalAllenCahn}, the latter a priori bound and \eqref{eq:blow_up_criteria_AC_3D_extrapolated} prove the claim of Theorem \ref{thm:global_extrapolation_AC}.
\end{proof}

Let us also mention that using \eqref{eq:reaction_diffusion_H_theta_AC_global_extrapolation_proof} and the argument of Corollary \ref{cor:comp} (see also Proposition \ref{prop:coincideRD} below), one can show that solutions to \eqref{eq:AllenCahnLpLq} with different choices of the parameters $(p,\a,\s,q)$ are compatible.

In the next subsection, we discuss further improvements of Theorems \ref{thm:globalAllenCahn} and \ref{thm:global_extrapolation_AC}, but we only present this in the case of periodic boundary conditions (see also Example \ref{ex:AllenCahnperiodicLpLq}). In particular, we present higher order regularity, a priori bounds, the critical setting with $q=3$ and $\gamma=1$, and continuous dependence on the initial data.

\subsection{Reaction-diffusion equations}\label{ss:reaction}
Systems of reaction-diffusion equations appear everywhere in models coming from applied science, e.g.\ chemistry, physics, biology, etc. The Allen--Cahn equation of the previous subsection is an example of a scalar reaction-diffusion equation. During the last two decades, a lot of work has been done on stochastic reaction-diffusion equations \cite{C03,Cer05,CR05,DKZ19,F91,KN19,KvN12,M18,S21,S21_dissipative,S21_superlinear_no_sign}. Physical motivations for stochastic perturbations of transport type can be found in \cite[Subsection 1.3]{AVreaction-local}.

In the deterministic setting, global well-posedness is known to hold for a large class of problems, which only satisfy a weak version of a coercivity condition. The survey \cite{P10_survey} provides an overview of a certain class of models, where some of them can be analyzed by maximal $L^p$-regularity techniques and backward equations. For quadratic systems with mass control, see  \cite{FMT, Kanel}.
It seems a difficult problem to prove global well-posedness in the same generality as was recently done in the deterministic setting. In this context, two open problems are formulated in Problems \ref{prop:triangular_structure} and \ref{pro:react2} at the end of this manuscript.

Recently, in \cite{AVreaction-local} based on Theorem \ref{thm:localwellposed}, a general local well-posedness theory was developed for systems of reaction-diffusion equations. Moreover, higher-order regularity and blow-up criteria were obtained as well. In the follow-up paper \cite{AVreaction-global} global well-posedness was proved for some classes of equations, and also some concrete models such as the Allen--Cahn equation (see Subsection \ref{ss:AllenCahn}), predator-prey models (see Subsection \ref{ss:LV}), coagulation dynamics, Brusselator (e.g.\ Gray Scott model). Some of these models are only weakly coercive and new arguments were needed for the global well-posedness.

Many of the previous models require $L^p(L^q)$-techniques for $p,q>2$. This is because Sobolev embeddings improve if $p,q$ are large. In Subsection \ref{ss:abstractreaction}, we explain a local well-posedness theory for reaction-diffusion equations on $\T^d$. In Subsection \ref{ss:coercivereaction}, we present some global well-posedness results in the case the system is coercive in an appropriate sense. Without such coercivity, the problem of global well-posedness can be much harder. Such an example is given by a Lotka--Volterra model. For the latter, we prove global well-posedness in Subsection \ref{ss:LV}. Many of the results below hold on smooth domains with suitable boundary conditions. However, if there is a gradient/transport noise, then it is not known when the leading linear part $(A,B)$ has stochastic maximal $L^p$-regularity (except if $B=0$). For this reason, we only consider periodic boundary conditions for now. In the rest of this subsection, we take
\begin{equation}
\label{eq:choice_X0_X1_reaction_diffusion_Td}
X_0 = H^{-\s,q}(\Tor^d;\R^\ell), \ \ \text{and} \ \ X_1 = H^{2-\s,q}(\Tor^d;\R^\ell),  \ \ \text{where} \ \ \delta\in [1, 2) \ \text{is fixed}.
\end{equation}
Note that $X_{1-\frac{1+\kappa}{p},p} = B^{2-\s-2\frac{1+\a}{p}}_{q,p}(\Tor^d;\R^{\ell})$.
As in Subsection \ref{ss:AllenCahnLpLq}, the parameters $q$ and $p$ will be used for spatial and time integrability, respectively, and the parameter $\delta$ is used to decrease spatial smoothness. The parameter $\kappa$ is related to the critical weight $t^{\kappa}$ used for the time-integrability.

\subsubsection{Local well-posedness, regularity and positivity}\label{ss:abstractreaction}
The results of this subsection are a special case of \cite{AVreaction-local} (see Remark \ref{rem:extreaction} for the generalization). We will only consider $d\geq 2$. The case $d=1$ can either be considered by adding a dummy variable, or by separate treatment (See \cite[Section 6]{AVreaction-local}).

Consider the following system of stochastic reaction-diffusion equations:
\begin{equation}
\label{eq:reaction_diffusion_system}
\left\{
\begin{aligned}
& \textstyle  \dd u_i -\nu_i \Delta u_i \,\dd t = f_i(\cdot, u)\,\dd t + \sum_{n\geq 1}  \Big[(b_{n,i}\cdot \nabla) u_i+ \rnoise_{n,i}(\cdot,u) \Big]\,\dd W^n, & \text{ on }\Tor^d,\\
&u_i(0)=u_{0,i},  & \text{ on }\Tor^d,
\end{aligned}
\right.
\end{equation}
where $i\in \{1,\dots,\ell\}$ and $\ell\geq 1$ is an integer.
Here, $u=(u_i)_{i=1}^{\ell}:[0,\infty)\times \O\times \Tor^d\to \R^\ell$ is the unknown process, $(W^n)_{n\geq 1}$ is a sequence of standard independent Brownian motions on the above mentioned filtered probability space.
The notation $(b_{n,i}\cdot\nabla) u_i:=\sum_{j=1}^d b^j_{n,i} \partial_j u_i$ is used, where the coefficients $b_{n,i}^j$ model small scale turbulent effects. Note that the SPDEs \eqref{eq:reaction_diffusion_system} are coupled only through the nonlinearities $f$ and $g$, but there are no cross interactions in the diffusion terms $-\nu_i \Delta u_i$ and $(b_{n,i}\cdot \nabla) u_i$, which is a standard assumption in reaction-diffusion systems.

The following is our main assumption on the coefficients and nonlinearities.
\begin{assumption}
\label{ass:reaction_diffusion_global}
Let $d\geq 2, \ell\geq 1$ be integers. We say that Assumption \ref{ass:reaction_diffusion_global}$(p,\kappa,q,h,\s)$ holds if $q\in [2,\infty)$, $p\in [2,\infty)$, $h>1$, $\s\in [1, 2)$ and for all $i\in \{1,\dots,\ell \}$ the following hold:
\begin{enumerate}[{\rm(1)}]
\item\label{it:regularity_coefficients_reaction_diffusion}
There is an $\alpha>\delta-1$ such that for all $j\in \{1, \ldots, d\}$,
$(\bm^{j}_{n,i})_{n\geq 1} \in C^{\alpha}(\Tor^d;\ell^2)$.
\item\label{it:ellipticity_reaction_diffusion} There exists $\mu>0$ such that for all $x\in \Tor^d$ and $\xi\in \R^d$,
$$ \textstyle
\sum_{j,k=1}^d \sum_{n\geq 1} b^j_{n,i}(x)b^k_{n,i}(x)
 \xi_j \xi_k
\geq (2\nu_i-\mu) |\xi|^2.
$$
\item\label{it:growth_nonlinearities} The maps
$f_i:\Tor^d\times \R\to \R$ and
$g_i:=(g_{n,i})_{n\geq 1}:\Tor^d\times \R\to \ell^2$,
are $\Borel(\Tor^d)\otimes \Borel(\R)$-measurable. Suppose that $f_i(\cdot, 0)\in L^\infty(\T^d)$, $g_i(\cdot,0)\in L^{\infty}(\Tor^d;\ell^2)$,
and for all $x\in \Tor^d$ and $y,y'\in\R$,
\begin{align*}
|f_i(x,y)-f_i(x,y')|
&\lesssim (1+|y|^{h-1}+|y'|^{h-1})|y-y'|,
\\ \|g_i(x,y)-g_i(x,y')\|_{\ell^2}
&\lesssim (1+|y|^{\frac{h-1}{2}}+|y'|^{\frac{h-1}{2}})|y-y'|.
\end{align*}
\item\label{it:reaction_diffusion_global_parameters_D} One of the following cases holds:
\begin{enumerate}[(i)]
\item\label{it:admissibleexp1} $p\in (2,\infty)$, $q\in [2,\infty)$, and $\s\in [1,\frac{h+1}{h})$,  satisfy
\begin{align*}
\frac{1+\kappa}{p}+\frac{1}{2}\Big(\reg+\frac{d}{q}\Big) = \frac{h}{h-1}, \ \text{and} \
\frac{d}{d-\reg} <q<\frac{d(h-1)}{h+1-\reg(h-1)};
\end{align*}
\item\label{it:admissibleexp2} $p=q=2$, $\kappa=0$, $\delta=1$ and $h\leq \frac{4+d}{d}$ with the additional restriction $h<3$ if $d=2$.
\end{enumerate}
\end{enumerate}
\end{assumption}

As commented at the end of Subsection \ref{ss:tracecriticalF} or in \cite[Subsection 5.3.4]{AV19_QSEE_1}, the relation between the growth of the drift and diffusion is optimal from a scaling point of view.

It is clear that if Assumption \ref{ass:reaction_diffusion_global}\eqref{it:regularity_coefficients_reaction_diffusion}-\eqref{it:growth_nonlinearities} hold for $\delta=1$ and $h$, then it also holds for \emph{some} $\delta>1$ and \emph{all} $\wt{h}>h$. However, the validity of Assumption \ref{ass:reaction_diffusion_global}\eqref{it:reaction_diffusion_global_parameters_D} depends on the specific choice of $(\delta,h)$ for which conditions \eqref{it:ellipticity_reaction_diffusion}-\eqref{it:growth_nonlinearities} are verified. In particular, if $h>1+\frac{4}{d}$, then Assumption \ref{ass:reaction_diffusion_global}\eqref{it:admissibleexp1} holds with $q=\frac{d(h-1)}{2}>2$, $\delta>1$ and \emph{any} $p\geq \frac{2}{2-\delta}$ (note that $\a=p(1-\frac{\delta}{2})-1\in [0,\frac{p}{2}-1)$). The integrability $q=\frac{d(h-1)}{2}$ is particularly relevant because the Lebesgue space $L^{d(h-1)/2}$ is invariant for the reaction-diffusion equations \eqref{eq:reaction_diffusion_system}. The reader is referred to \cite[Subsection 1.4]{AVreaction-local} for the scaling analysis of \eqref{eq:reaction_diffusion_system}. Although one can always increase the parameter $h$ in Assumption \ref{ass:reaction_diffusion_global}, this comes at the cost of obtaining smaller critical spaces. Moreover, the lower bound $h > 1 + \frac{4}{d}$ is not optimal for ensuring that Assumption \ref{ass:reaction_diffusion_global}\eqref{it:admissibleexp1} holds. Indeed, as shown in \cite[Lemma 2.5]{AVreaction-local} for $d = 3$, Assumption \ref{ass:reaction_diffusion_global}\eqref{it:admissibleexp1} holds for all $h > 2$ within a suitable range of parameters $(p, \kappa, q, \delta)$. The same suboptimality is observed for dimensions $d > 3$. For further discussion and a comparison with Fujita’s critical exponent introduced in \cite{F66}, we refer the reader to \cite[Remark 2.6]{AVreaction-local}.
Finally, by adding a dummy variable, one can increase the dimension and our discussion also includes the case $d=1$.

\smallskip

As usual, we view \eqref{eq:reaction_diffusion_system} as a stochastic evolution equation on $X_0$ given in \eqref{eq:choice_X0_X1_reaction_diffusion_Td}. More precisely, we say that $(u,\sigma)$ is a (unique) $(p,\a,q,\s)$-solution to \eqref{eq:reaction_diffusion_system} if $(u,\sigma)$ is a $L^p_\a$-solution to \eqref{eq:SEE} with $(X_0,X_1)$ as in \eqref{eq:choice_X0_X1_reaction_diffusion_Td}, $\mathcal{U}=\ell^2$ and
\begin{equation*}
\begin{aligned}
(A u)_i &=-\nu_i \Delta u_i , &  \qquad (B u)_i& = ((b_{n,i}\cdot\nabla)u_i)_{n\geq 1},\\
 (F(u))_i&=f_i(u),
 & \qquad  (G(u))_i &= (g_{n,i} (\cdot,u))_{n\geq 1}.
\end{aligned}
\end{equation*}
Here we stress the dependence on $(p,\a,q,\s)$ in the definition of solutions. However, in Proposition \ref{prop:coincideRD} we will see that the solutions to \eqref{eq:reaction_diffusion_system} for different choices of the parameters actually coincide.
In the following, we compare the above definition of solutions with the one used in \cite{AVreaction-local}.

\begin{remark}
Using  Assumption \ref{ass:reaction_diffusion_global}$(p,\kappa,q,h,\s)$ and \cite[Lemma 3.2]{AVreaction-local}, one can check that the above defined $(p,\a,q,\s)$-solution to \eqref{eq:reaction_diffusion_system} gives, for all $i\in\{1,\dots,\ell\}$,
\begin{equation*}
\begin{aligned}
f_i(\cdot, u)\in L^p_{\rm loc}([0,\sigma),w_{\a};H^{-\s,q}(\Tor^d)), \ \ \text{and} \ \
(g_{n,i}(\cdot,u))_{n\geq 1}\in L^p_{\rm loc}([0,\sigma),w_{\a};H^{1-\s,q}(\Tor^d;\ell^2)).
\end{aligned}
\end{equation*}
These were used to define a solution in \cite{AVreaction-local}, but are equivalent by stochastic maximal regularity.
\end{remark}

To simplify the notation in the sequel from now on, we write
\[H^{s,q} = H^{s,q}(\Tor^d;\R^{\ell}) \ \  \text{and} \ \ B^{s}_{q,p} = B^{s}_{q,p}(\Tor^d;\R^{\ell}).\]

The main result on local well-posedness reads as follows (see \cite[Theorem 2.6]{AVreaction-local}).

\begin{theorem}[Local existence and regularization in critical spaces]
\label{t:reaction_diffusion_global_critical_spaces}
Let Assumptions \ref{ass:reaction_diffusion_global}$(p,\kappa,q,h,\s)$ be satisfied.
Then for any $u_0\in  L^0_{\F_0}(\O;B^{\frac{d}{q}-\frac{2}{h-1}}_{q,p})$,
the problem \eqref{eq:reaction_diffusion_system} has a (unique) $(p,\a,\s,q)$-solution $(u,\sigma)$ such that a.s.\
 $\sigma>0$ and
\begin{equation}
\label{eq:regularity_u_reaction_diffusion_critical_spaces}
\begin{aligned}
u\in L^{p}_{{\rm loc}}([0,\sigma),w_{\a};H^{2-\s,q}))\cap C([0,\sigma);B^{\frac{d}{q}-\frac{2}{h-1}}_{q,p}).
\end{aligned}
\end{equation}
Moreover, $u$ instantaneously regularizes in space and time: a.s.,
\begin{align}
\label{eq:reaction_diffusion_H_theta}
u&\in L^r_{\rm loc}((0,\sigma);H^{1,r})
\cap C^{\theta/2,\theta}_{\rm loc}((0,\sigma)\times \Tor^d;\R^\ell), \ \ r\in (2,\infty), \theta\in (0,1).
\end{align}
\end{theorem}

As in Subsection \ref{ss:scaling_intro} or in \cite[Subsection 1.4]{AVreaction-local}, a scaling argument shows that the spaces $B^{\frac{d}{q}-\frac{2}{h-1}}_{q,p}$ are locally scaling-invariant and they are critical in the PDE sense for \eqref{eq:reaction_diffusion_system}.
It is interesting to note that the smoothness of $B^{\frac{d}{q}-\frac{2}{h-1}}_{q,p}$ does not depend on $(p,\kappa,\delta)$.
Moreover, we can take $q$ close to its upper bound $\frac{d(h-1)}{h+1-\reg(h-1)}$. Due to the Sobolev embedding, this gives the largest class of initial data which we can consider. In the limit for $q$ to its upper bound, this gives smoothness $1-\delta$. Taking $\delta$ close to its upper bound $\frac{h+1}{h}$, we almost reach smoothness $-\frac{1}{h}$. Of course, the latter is only possible if the regularity exponent of Assumption \ref{ass:reaction_diffusion_global} satisfies $\alpha\geq \frac{h+1}{h}-1 = \frac{1}{h}$.
As discussed below Assumption \ref{ass:reaction_diffusion_global}, if $h>1+\frac{4}{d}$, then Theorem \ref{t:reaction_diffusion_global_critical_spaces} can always be applied with $\delta>1$, $q=\frac{d(h-1)}{2}$ and any $p\geq \frac{2}{2-\delta}$. This choice results in the critical space $B^{\frac{d}{q}-\frac{2}{h-1}}_{q,p}= B^{0}_{q,p}$ with zero smoothness.
Thus, due to the elementary embedding
$$
L^{\frac{d(h-1)}{2}}\subseteq B_{\frac{d(h-1)}{2},p}^{0} \ \ \text{ if } \  p> \tfrac{d(h-1)}{2},
$$
applying the above with $\delta>1$, $q=\frac{d(h-1)}{2}$ and $p\geq \max\{q,\frac{2}{2-\delta}\}$ ensures local well-posedness of \eqref{eq:reaction_diffusion_system} for initial data in the critical Lebesgue space $L^{\frac{d(h-1)}{2}}$. See also \cite[Remark 2.8(c)]{AVreaction-local} for related discussion.
The regularity assertion \eqref{eq:regularity_u_reaction_diffusion_critical_spaces} can be further improved if the data $(b_{n,i}^j,f_i,g_i)$ have further smoothness in the $x$-variable. In this way, one can even obtain $C^\infty$-regularity in space if these coefficients are $C^\infty$ in space. The latter heavily relies on the periodic boundary conditions. For different boundary conditions compatibility conditions are required.

The proof of Theorem \ref{t:reaction_diffusion_global_critical_spaces} given in \cite{AVreaction-local} relies on Theorem \ref{thm:localwellposed}. Here, the stochastic maximal regularity is nontrivial due to the $b$-term. A detailed proof can be found in \cite{AV21_SMR_torus}. The estimates for $F(u) = f(u)$ and $G(u) = g(\cdot,u)$ are relatively straightforward extensions of what we have already seen in Subsection \ref{ss:tracecriticalF} (see \cite[Lemma 3.2]{AVreaction-local}). The time-regularity in \eqref{eq:reaction_diffusion_H_theta} follows from Theorems \ref{thm:parabreg} and \ref{thm:parabreg2}. After that, a classical bootstrap argument can be used to increase the regularity in space. Details in the case of the quasi-geostrophic equation are given in the proof of Proposition \ref{prop:locQS} below.

\begin{remark}[The need for $L^p(L^q)$-theory for rough Kraichnan model]
\label{rem:rough_kraichnan_LpLq_need}
In the case of a reaction-diffusion advected by a turbulent fluid, under a time-scale separation assumption, the fluid's influence can be modelled through transport noise as in \eqref{eq:reaction_diffusion_system}, see \cite[Subsection 1.3]{AVreaction-local}. In this context, a natural choice of $(b_n)_{n\geq 1}$ is given by the Kraichnan model \cite{K68}. In a turbulent regime, the regularity parameter $\alpha>0$ in Assumption \ref{ass:reaction_diffusion_global}\eqref{it:regularity_coefficients_reaction_diffusion}, which reflects the spatial correlation of the turbulent fluid, see e.g.\ \cite[Proposition 2.1]{agresti2023primitive}, is typically small. Therefore, the reader can readily check that if $\alpha$ is small and $h$ is large, then Assumption \ref{ass:reaction_diffusion_global}\eqref{it:reaction_diffusion_global_parameters_D} (which is equivalent to that local well-posedness of \eqref{eq:reaction_diffusion_system}) can only hold if the exponents $q $ and $p$ are sufficiently large.
\end{remark}

Although it seems an academic question, it can be important and nontrivial to show that solutions obtained from different settings coincide.
The following result is taken from \cite[Proposition 3.5]{AVreaction-local}, and is an extended version of Corollary \ref{cor:comp} where also the integrability $q$ and the growth $h$ are allowed to vary.

\begin{proposition}[Compatibility]\label{prop:coincideRD}
Suppose that Assumptions \ref{ass:reaction_diffusion_global} holds for two sets of parameters $(p,\kappa,q,h,\s)$ and $(p_0,\kappa_0,q_0,h_0,\s_0)$. Let $u_0\in L^0_{\F}(\Omega;B^{\frac{d}{q}-\frac{2}{h-1}}_{q,p}\cap B^{\frac{d}{q_0}-\frac{2}{h_0-1}}_{q_0,p_0})$. Let $(u,\sigma)$ and $(v, \tau)$ be $(p,\kappa,q,\s)$- and $(p_0,\kappa_0,q_0,\s_0)$-solutions to \eqref{eq:reaction_diffusion_system} respectively. Then $\sigma = \tau$ and $u = v$.
\end{proposition}

Via regularization arguments and the blow-up criteria of Theorems \ref{thm:criticalblowup}\eqref{it2:criticalblowup} and \ref{thm:subcriticalblowup} one can obtain the following.

\begin{theorem}[Blow-up criteria]\label{thm:blow_up_criteria}
Let the assumptions of Theorem \ref{t:reaction_diffusion_global_critical_spaces} be satisfied and let $(u,\sigma)$ be the $(p,\a,q,\s)$-solution to \eqref{eq:reaction_diffusion_system}.
Let $h_0\geq 1+\frac{4}{d}$. Suppose that $p_0\in (2,\infty)$, $h_0\geq h$, $\s_{0}\in (1,2)$  are such that Assumption \ref{ass:reaction_diffusion_global}$(p_0,\kappa_0,q_0,h_0,\s_0)$ holds.
Let $\zeta_0 = \frac{d}{2}(h_0-1)$. The following hold for all $0<s<T<\infty$:
\begin{enumerate}[{\rm(1)}]
\setcounter{enumi}{\value{nameOfYourChoice}}
\item\label{it:blow_up_not_sharp_L}
If $q_0 = \zeta_0$, then for all $\zeta_1>q_0$
\[\P\big(s<\sigma<T,\, \sup_{t\in [s,\sigma)}\|u(t)\|_{L^{\zeta_1}(\Tor^d;\R^{\ell})} <\infty\big)=0.
\]
\item\label{it:blow_up_sharp_L}
If $q_0>\zeta_0$, $p_0\in \big(\frac{2}{\s_0-1},\infty\big)$, $p_0\geq q_0$, and $\frac{d}{q_0}+\frac{2}{p_0}=\frac{2}{h_0-1}$, then
\[
\P\big(s<\sigma<T,\, \sup_{t\in [s,\sigma)}\|u(t)\|_{L^{\zeta_0}}+
\|u\|_{L^{p_0}(s,\sigma;L^{q_0})} <\infty\big)=0.
\]
\item\label{it:blow_up_sharp_2} If $p_0=q_0=2$, $\delta_0=1$, then
\[\P\big(s<\sigma<T,\, \sup_{t\in [s,\sigma)}\|u(t)\|_{L^{2}}+
\|u\|_{L^{2}(s,\sigma;H^{1,2})} <\infty\big)=0.
\]
\end{enumerate}
\end{theorem}
The parameters $(p,\kappa,q,\delta)$ can be different from the parameters for which we could provide energy bounds, i.e.\ $(\zeta_0, \zeta_1, p_0, q_0)$. This makes the above blow-up criteria very flexible in applications. Due to \eqref{it:blow_up_sharp_2} in some cases it can even suffice to prove $L^2$-bounds. The parameter $s$ is very useful in applications. It allows to take parabolic regularization of the solution into account in the formulation of blow-up criteria. As we have seen in Theorem \ref{thm:global_extrapolation_AC}, this allows us to obtain global well-posedness also for rough initial data.

Via the blow-up criteria, regularity and a maximum principle of \cite{Kry13}, the following result on positivity (in a distributional sense) was proved in \cite{AVreaction-local} through  $L^p(L^q)$-theory.

\begin{proposition}[Positivity]
\label{prop:positivity}
Let the assumptions of Theorem \ref{t:reaction_diffusion_global_critical_spaces} be satisfied.
Let $(u,\sigma)$ be the $(p,\a,\s,q)$-solution to \eqref{eq:reaction_diffusion_system} provided in Theorem \ref{t:reaction_diffusion_global_critical_spaces}. Suppose that
$u_0\in [0,\infty)^\ell$ a.s.,
and that for all $i\in \{1,\dots,\ell\}$, $n\geq 1$, $y=(y_{i})_{i=1}^{\ell}\in [0,\infty)^{\ell}$ and
$x\in \T^d$
\begin{align*}
f_i(x,y_1,\dots,y_{i-1},0,y_{i+1},\dots,y_{\ell})&\geq 0,\\
g_{n,i}(x,y_1,\dots,y_{i-1},0,y_{i+1},\dots,y_{\ell})&=0.
\end{align*}
Then a.s.\ for all  $x\in \Tor^d$ and $t\in [0,\sigma)$,
$u(t,x)\geq 0$.
\end{proposition}

\begin{example}
In the case of the Allen--Cahn equation $\ell=1$, and $f(y) = -y^3+y$. The latter satisfies the above condition. Therefore, to obtain the positivity of the solution of the stochastic Allen--Cahn equation with periodic boundary conditions, we need to consider a $g_n$ such that $g_n(x,0) = 0$. Besides some regularity, no conditions on the coefficients $b_{n}^j$ are needed.
\end{example}

\begin{remark}\label{rem:extreaction}
In \cite{AVreaction-local}  the following more general setting is considered. All coefficients are allowed to depend on $(t,\omega)$ as well. The terms $\nu_i \Delta u_i$ are replaced by a more general operator in divergence form to
model inhomogeneous conductivity and may also take into account the It\^o correction in the case of Stratonovich noise. An additional conservative term $\div [\Psi(u)]\, \dd t$ is added to the equation. The coefficients $a$ and $b$ can be assumed to have Sobolev regularity $H^{\alpha,r}(\Tor^d;\ell^2)$ for some $r\geq 2$ such that $\alpha-\frac{d}{r}>0$.
\end{remark}

\subsubsection{Global well-posedness for coercive systems}\label{ss:coercivereaction}

In \cite{AVreaction-global} we introduced a so-called $L^\zeta$-coercivity estimate, which gives sufficient conditions for the global existence of the solution in Theorem \ref{t:reaction_diffusion_global_critical_spaces}. If $\zeta=2$ and $\alpha_i=1$, the condition reduces to the classical coercivity estimates such as the one in the variational setting in Theorem \ref{thm:varglobal}.

In simplified form, the assumption is as follows.
\begin{assumption}[$L^\zeta$-coercivity]
\label{ass:dissipation_general}
Let $\zeta\in [ 2,\infty)$. We say that Assumption \ref{ass:dissipation_general}$(\zeta)$ holds if there exist constants $\theta,M,C,\alpha_1,\dots,\alpha_{\ell}>0$ such that a.e.\ in $[0,\infty)\times \Omega$ for all $v=(v_i)_{i=1}^{\ell}\in C^1(\mathbb{T}^d;\R^{\ell})$,
\begin{align*}
\textstyle \sum_{i=1}^{\ell} \alpha_i\int_{\T^d} |v_i|^{\zeta-2} \Big(\nu_i \nabla v_i\cdot \nabla v_i   - \frac{v_i f_i(\cdot, v)}{\zeta-1} -\frac12 \sum_{n\geq 1} \big[(b_{n,i} \cdot \nabla) v + g_{n,i}(\cdot, v) \big]^2 \Big) \, \dd x &
\\
\textstyle  \geq  \theta
\sum_{i=1}^{\ell}  \int_{\T^d} |v_i |^{\zeta-2}|\nabla v_i |^2 \,\dd x -  M
\sum_{i=1}^{\ell} \int_{\T^d}|v_i |^{\zeta} \, \dd x -C&.
\end{align*}
\end{assumption}
Moreover, if a local solution $u$ is known to take values in a subset $S$ of $\R^{\ell}$, then it is enough to consider $v$ which takes values in $S$ in the above.

To get a better intuition of what Assumption \ref{ass:dissipation_general} means, it is helpful to state a sufficient condition for it in the special case where $b=0$. More general conditions can be found in
\cite[Lemmas 3.3 and 3.5]{AVreaction-global}.

\begin{lemma}[Pointwise $L^\zeta$--coercivity if $b=0$]
\label{lem:dissipationI}
Let $\zeta\in [ 2,\infty)$. Suppose $b=0$. If there exist constants $\alpha_1, \ldots, \alpha_{\ell}, M>0$ such that for all $y\in \R^\ell$
\[  \textstyle  \sum_{i=1}^\ell \alpha_i |y_i|^{\zeta-2}  \Big[ \frac{y_i f_i(\cdot, y)}{\zeta-1} +\frac12 \|(g_{n,i}(\cdot, y))_{n\geq 1}\|^2_{\ell^2} \Big] \leq M(|y|^\zeta+1),\]
then Assumption \ref{ass:dissipation_general}$(\zeta)$ holds.
\end{lemma}
If $\ell=1$, the latter is equivalent to $\frac{y f(\cdot, y)}{\zeta-1} +\frac12 \|(g_{n}(\cdot, y))_{n\geq 1}\|^2_{\ell^2}\leq M(y^2+1)$. For instance, if $f(y) = -y^3$ this implies $g$ can grow at most quadratically, and the condition becomes more restrictive if $\zeta$ increases. Similar results can be checked for other classes of nonlinearities.

If Assumption \ref{ass:dissipation_general}$(\zeta)$ holds with sufficiently large $\zeta$, then one can obtain global existence using the following method as done in \cite[Theorems 3.2]{AVreaction-global}:
\begin{itemize}
\item apply It\^o's formula to rewrite $\|u_i\|_{L^{\zeta}(\T^d)}^{\zeta}$;
\item calculate $\sum_{i}\alpha_i \|u_i\|_{L^{\zeta}(\T^d)}^{\zeta}$;
\item estimate the latter through Assumption \ref{ass:dissipation_general};
\item apply the stochastic Gronwall Lemma \ref{lem:Gronwall} to find an a priori estimate for $u$;
\item apply the blow-up criteria of Theorem \ref{thm:blow_up_criteria} \eqref{it:blow_up_sharp_L} to get $\sigma=\infty$.
\end{itemize}

We summarize the result one obtains in the next theorem.

\begin{theorem}[Global well-posedness]\label{thm:reactioncoerc}
Suppose that Assumptions \ref{ass:reaction_diffusion_global}$(p,\kappa,q,h,\s)$  and \ref{ass:dissipation_general}$(\zeta)$ hold with $\zeta \geq \frac{d(h-1)}{2}\vee 2$.
Then for any $u_0\in  L^0_{\F_0}(\O;B^{\frac{d}{q}-\frac{2}{h-1}}_{q,p})$, there exists a unique global \emph{$(p,\a,\s,q)$-solution} $u:[0,\infty)\times\Omega\to \R^{\ell}$.
Moreover, the following additional assertions hold:
\begin{enumerate}[{\rm (1)}]
\item\label{it1:reactioncoerc} The regularity assertions \eqref{eq:regularity_u_reaction_diffusion_critical_spaces} and \eqref{eq:reaction_diffusion_H_theta} hold with $\sigma$ replaced by $\infty$.
\item\label{it2:reactioncoerc} For all $\lambda\in (0,1)$, there exist $N_{0,\lambda}>0$ such that for any $T<\infty$, $i\in \{1,\dots,\ell\}$,
\begin{align}
\label{eq:aprioribounds2}
\textstyle \E \sup_{t\in [0,T]} \|u_i(t)\|_{L^{\zeta}}^{\zeta \lambda} + \E \Big|\int_{0}^{T}\int_{\Tor^d}\one_{\Gamma}  |u_i|^{\zeta-2} |\nabla u_i|^{2}\,\dd x\,\dd t\Big|^{\lambda}
& \leq \textstyle N_{0,\lambda}\Big(1+\E\|u_0\|_{L^{\zeta}}^{\zeta \lambda}\Big).
\end{align}
\item\label{it3:reactioncoerc} Suppose that $p\geq q\geq \frac{d(h-1)}{2}$. The solution depends continuously on the initial data in the following sense:
    If $u_0^n\to u_0$ in $B^{\frac{d}{q}-\frac{2}{h-1}}_{q,p}$ in probability, then for all $T<\infty$, $u^{(n)}\to u$ in $C([0,T];B^{\frac{d}{q}-\frac{2}{h-1}}_{q,p})$ in probability.
\end{enumerate}
\end{theorem}
Parts \eqref{it2:reactioncoerc} and \eqref{it3:reactioncoerc} follow from \cite[Theorem 5.2 and Corollary 5.4]{AVreaction-global}, and require more delicate arguments which we cannot explain here.

Note that the condition $q\geq \frac{d(h-1)}{2}$ in \eqref{it3:reactioncoerc} implies that $\frac{d}{q}-\frac{2}{h-1}\leq 0$. Thus, the continuous dependency is formulated in a space of distributions. An exception is $\zeta=p=q=2$ if $d(h-1)\leq 4$, in which case the result is formulated for the space $L^2(\T^d;\R^\ell)$. However, in the latter case, it is better to use Theorem \ref{thm:varglobal} since it gives additional information.

The above result for instance applies to the stochastic Allen--Cahn equation in any dimension and with gradient noise, but also to many systems.
To make a comparison with Theorem \ref{thm:globalAllenCahn} we present some of the details in the case $b=0$ and $d=3$.

\begin{example}[Allen--Cahn with periodic boundary conditions for $d=3$ for $b=0$]\label{ex:AllenCahnperiodicLpLq}
Let $b=0$. For the Allen--Cahn equation one has $\ell=1$ and $f(y) = -y^3+y$ and $h=3$. Let $\zeta = q=p =3$, $\delta \in (1, \frac43)$, and $\kappa=2-\frac{3\delta}{2}$.  Suppose that $g(\cdot, 0)\in L^\infty(\T^d)$,
\begin{align*}
\|g(x,y)-g(x,y')\|_{\ell^2} \leq L(1+|y|+|y'|)|y-y'| \  \ \text{and}\  \
\|(g_{n}(\cdot, y))_{n\geq 1}\|^2_{\ell^2}\leq M(y^2+1) + y^4.
\end{align*}
It follows that Assumption \ref{ass:reaction_diffusion_global} holds. Moreover, from Lemma \ref{lem:dissipationI} we see that Assumption \ref{ass:dissipation_general}$(\zeta)$ holds. Therefore, the conditions of Theorem \ref{thm:reactioncoerc} are satisfied, and we obtain global well-posedness for any $u_0\in L^0_{\F_0}(\O;B^{0}_{3,3}(\T^3))$. In particular, $u_0\in L^0_{\F_0}(\O;L^3(\T^3))$ is allowed. Taking a larger $q$ and $\delta$, one can even go to spaces of initial data with negative smoothness down to $-1/3$.

The above condition on $g$ is slightly weaker than what we have encountered in Theorem \ref{thm:globalAllenCahn} in the case of Dirichlet boundary conditions. The difference is that with the help of the more powerful general theory presented here, we are allowed to consider a critical setting.
\end{example}

\subsubsection{Lotka--Volterra}\label{ss:LV}

As an example, we present a stochastic predator-prey model with diffusivity. For the deterministic theory, the reader is referred to \cite{CL84,GL94}. Stochastic perturbations can model uncertainties in the determination of the external forces and/or parameters, see e.g.\ \cite{CE89,NY21} and the references therein. Moreover, transport noise can model the `small-scale' effect of migration phenomena of the species, but for simplicity of the presentation, we leave it out of the discussion here. As a special case, our model below also includes the SIR model (susceptible-infected-removed). In the periodic stochastic setting global well-posedness was obtained for these models for the first time in \cite[Section 4.2]{AVreaction-global}.

Predator-prey models typically do not satisfy the coercivity estimates of Theorems \ref{thm:varglobal} and \ref{thm:reactioncoerc}. This is due to the system structure of the equation and because the diffusion equation for the ``predator`` is typically not dissipative. Fortunately, one can prove some control using the stochastic Gronwall Lemma \ref{lem:Gronwall} as soon as one can provide an a priori bound for the ``prey''.
Finally, we mention that since the positivity of the solution is required in predator-prey models, the noise has to be multiplicative.

Consider the following system on the $d$-dimensional torus $\T^d$ with $d\in \{1, 2, 3, 4\}$:
\begin{equation}
\label{eq:PPmodel}
\left\{
\begin{aligned}
\dd u -\nu \Delta u \, \dd  t&= \textstyle\big[ \lambda_1 u-\chi_{1,1}u^2 -\chi_{1,2} uv \big]\, \dd t+ \sum_{n\geq 1}  g_{n,1}(\cdot,u,v) \,\dd W^n,
\\ \dd v -\mu \Delta v \, \dd  t&= \textstyle  \big[\lambda_2 v-\chi_{2,2}v^2 +\chi_{2,1} uv \big]\, \dd t+ \sum_{n\geq 1}  g_{n,2}(\cdot,u,v) \,\dd W^n,
\end{aligned}\right.
\end{equation}
and initial data $u(0,\cdot)=u_0$ and $v(0,\cdot)=v_0$ .
The unknowns $u,v:[0,\infty)\times \O\times \T^d \to \R$ model the population of the prey and predator, respectively.

The main assumptions on $g_1 = (g_{n,1})_{n\geq 1}$ and $g_2 = (g_{n,2})_{n\geq 1}$ are as follows.
\begin{assumption}
\label{ass:lotka_volterra}
Let $1\leq d\leq 4$ and $\lambda_i\in \R, \chi_{i,j}\geq 0$ for $i,j\in \{1, 2\}$.    Let $g_1, g_2:\R^2\to \ell^2$ and suppose that there exists an $M>0$ such that for all $x\in \T^d$, $y,y',z,z'\geq 0$,
\begin{align}
\label{eq:PPpos} g_{1}(x,y,0)& =g_{2}(x,0,z)=0
\\ \label{eq:PPloc} \textstyle \sum_{i=1}^2\|g_i(x,y,z)-g_i(x,y',z')\|_{\ell^2} &\leq C(1+|y|^{1/2}+|y'|)^{1/2}|y-y'|,
\\ \label{eq:PPgrowth1} \frac12\|g_1(x,y,z)\|_{\ell^2}^2&\leq  C \big(1+y^2 \big)+ \chi_{1,1} y^3+\chi_{1,2} y^2 z,
\\
\label{eq:PPgrowth2}
\frac12\|g_2(x,y,z)\|_{\ell^2}^2&\leq C \big[1 +  (1+ y)z^{2} + y^2 z+ y^3 \big] +\chi_{2,2} z^3.
\end{align}
\end{assumption}
The condition \eqref{eq:PPpos} is to ensure the positivity of $u$ and $v$. Condition \eqref{eq:PPloc} is a local Lipschitz condition. The conditions \eqref{eq:PPgrowth1} and \eqref{eq:PPgrowth2} are certain growth conditions of which we do not know the optimality. The above problem  fits into the setting of Theorem \ref{t:reaction_diffusion_global_critical_spaces} with $p=q=2$, $\delta=1$, $h=2$, and $\kappa=0$. Thus, local well-posedness and regularity are immediate. Note that to include $d=1$ one can add a dummy variable.

Although the above model does not satisfy the usual coercivity conditions, we can still obtain global well-posedness.
\begin{theorem}[Global well-posedness -- Lotka--Volterra]
\label{t:Lotka_Volterra}
Suppose that Assumption \ref{ass:lotka_volterra} holds.
Then for every $u_{0},v_0\in L^0_{\F_0}(\Omega;L^2(\T^d))$ with $u_{0},v_0\geq 0$ a.s., there exists a (unique) global solution $(u,v)\in L^2_{\rm loc}([0,\infty);H^1(\T^d;\R^2))\cap C([0,\infty);L^2(\T^d;\R^2))$ a.s. of  \eqref{eq:PPmodel} and $u, v\geq 0$ a.s.

Moreover, $u$ and $v$ have the following regularity for all $r\in [1, \infty)$ and $\theta\in (0,1)$
\begin{align*}
u,v&\in L^r_{\rm loc}((0,\infty);H^{1,r}(\T^d))
 \ \text{and} \   u,v\in C^{\theta/2,\theta}_{\rm loc}((0,\infty)\times \Tor^d).
\end{align*}
Furthermore, if $u_0^n,v_0^n \in L^0_{\F_0}(\Omega;L^2(\T^d))$ with $u_{0}^n,v_0^n\geq 0$ a.s.\ are such that $u_0^n\to u_0$ and $v_0^n\to v_0$ in $L^2(\T^d)$ in probability, then $u^n\to u$ and $v^n\to v$ in $C([0,T];L^2(\T^d;\R^2))\cap L^2(0,T;H^1(\T^d;\R^2))$ in probability for every $T<\infty$, where $(u^n, v^n)$ are the solutions to \eqref{eq:PPmodel} with initial data $(u_0^n,v_0^n)$.
\end{theorem}

From the formulation of the above result, it almost seems that the above result can be proved through $L^2$-theory, and on more general domains with more general boundary conditions. This is indeed true if one only wants to show local well-posedness, which also follows from Theorem \ref{thm:varloc}. However, the global well-posedness is much harder to show. In the proof, we need positivity of the solution, and this follows from Proposition \ref{prop:positivity}, which requires $L^p(L^q)$-theory.
We explain in detail how the positivity is used to obtain the global existence. For this, we will use the blow-up criteria of Theorems \ref{thm:varloc} or \ref{thm:blow_up_criteria}. For simplicity we assume $\chi_{1,1} = \chi_{2,2} = \lambda_2 = 0$, as this does not change the argument much.  Suppose that $(u,v)$ is a maximal solution with maximal time $\sigma$.

Applying It\^o's formula to $\|u\|_{L^2(\T^d)}^2$ we find that
\begin{align*}
 \|u(t)\|_{L^2(\T^d)}^2  =  & \textstyle \|u_0\|_{L^2(\T^d)}^2 - \int_0^t \int_{\T^d} 2\nu|\nabla u(s)|^2 \,\dd x \,\dd s \\ & + \textstyle 2\int_0^t \int_{\T^d} \lambda_1 u(s)^2 -\chi_{1,2} u(s)^2v(s) \,\dd x \,\dd s+ \int_0^t \|g_{1}(\cdot,u(s),v(s))\|^2_{\ell^2} \,\dd s + M_t,
\end{align*}
where $M$ is a continuous local martingale. Using the positivity of $u$ and $v$ we know the signs in the above terms and thus we can apply the assumption \eqref{eq:PPgrowth1} to obtain the bound
\[\|u(t)\|_{L^2(\T^d)}^2 + 2\nu \|\nabla u\|^2_{L^2(0,t;L^2(\T^d))} \leq \|u_0\|^2_{L^2(\T^d)} + 2C t + \|u\|^2_{L^2(0,t;L^2(\T^d))} + M_t. \]
Now, one could take expectations and apply a classical Gronwall argument. Alternatively, one can apply the stochastic Gronwall Lemma \ref{lem:Gronwall} and immediately deduce that a.s.\ for all $T<\infty$
\begin{align}\label{eq:PPstochgronwallappli}
\sup_{t\in [0,\sigma\wedge T)} \|u(t)\|_{L^2(\T^d)}<\infty \ \ \text{and} \ \ \|\nabla u\|^2_{L^2(0,\sigma\wedge T;L^2(\T^d))}<\infty.
\end{align}
The above argument cannot directly be translated to $v$. We can still apply It\^o's formula to get
\begin{align*}
\|v(t)\|_{L^2(\T^d)}^2  =&  \textstyle \|v_0\|_{L^2(\T^d)}^2 - \int_0^t \int_{\T^d} 2\mu|\nabla v(s)|^2 \,\dd x \,\dd s \\ & + \textstyle 2\int_0^t \int_{\T^d} \chi_{2,1} u(s)v(s)^2 \,\dd x\,\dd s + \int_0^t \int_{\T^d}\|g_{2}(\cdot,u(s),v(s))\|^2_{\ell^2} \,\dd x\,\dd s + \wt{M}_t,
\end{align*}
where $\wt{M}$ is a continuous local martingale again. Now, one can repeat the application of the bound \eqref{eq:PPgrowth2} to obtain
\begin{align*}
&\|v(t)\|_{L^2(\T^d)}^2 + \textstyle 2\mu\int_0^t \|\nabla v(s)\|^2_{L^2(\T^d)} \,\dd s  \\ &\leq  \textstyle \|v_0\|_{L^2(\T^d)}^2  + 2\int_0^t \int_{\T^d} \big[ C(1+v(s)^2+u(s)^2v(s)+u(s)^3) + (C+\chi_{2,1}) u(s)v(s)^2 \big]\,\dd x\,\dd s  + \wt{M}_t
\\ &\leq \textstyle  \|v_0\|_{L^2(\T^d)}^2  + 2\wt{C}\int_0^t \int_{\T^d} \big(1+v(s)^2+u(s)^3 + u(s)v(s)^2 \big)\,\dd x\,\dd s  + \wt{M}_t,
\end{align*}
where in the last step we used $2u^2v = 2u^{3/2} (u^{1/2} v)\leq u^3 + uv^2$.
To bound the $L^2$-norm of $v$ we will use the stochastic Gronwall Lemma \ref{lem:Gronwall}. This time we need its full power. We claim that
\[\textstyle \int_{\T^d} \big(v(s)^2+u(s)^3 + u(s)v(s)^2\big) \,\dd x \leq N_{\varepsilon,u}(s) [1+\|v(s)\|_{L^2(\T^d)}^2]+ \varepsilon \|v\|^2_{H^{1,2}(\T^d)},\]
where $\varepsilon>0$ is arbitrary, and $N_{\varepsilon,u}\in L^1(0,\sigma\wedge T)$ a.s.\ and $N_{\varepsilon,u}(t)$ and depends on $u|_{[0,t]}$.
To prove this we use interpolation estimates, and consider each term separately.  For the mixed term, by Sobolev embedding with $d\leq 4$,
\begin{align*}
\textstyle \int_{\T^d}u(s)v(s)^2 \,\dd x & \leq \|u(s)\|_{L^4(\T^d)} \|v(s)\|_{L^{8/3}(\T^d)}^2
\\ & \leq \|u(s)\|_{H^{1,2}(\T^d)} \|v(s)\|_{H^{1/2,2}}^2
\\ & \leq \|u(s)\|_{H^{1,2}(\T^d)} \|v(s)\|_{L^2(\T^d)} \|v(s)\|_{H^1(\T^d)}
\\ & \leq C_{\varepsilon} \|u(s)\|_{H^{1,2}(\T^d)}^2 \|v(s)\|_{L^2(\T^d)}^2 + \varepsilon \|v(s)\|_{H^1(\T^d)}^2,
\end{align*}
which is of the required form due to $u\in L^2(0,\sigma\wedge T;H^1(\T^d))$. Similarly,
\begin{align*}
\textstyle \int_{\T^d}u(s)^3  \,\dd x & \leq \|u(s)\|_{H^{2/3,2}(\T^d)}^3 \leq \|u(s)\|_{L^2(\T^d)} \|u(s)\|_{H^{1,2}(\T^d)}^{2},
\end{align*}
which is also of the required form due to $u\in L^2(0,\sigma\wedge T;H^1(\T^d))\cap L^\infty(0,\sigma\wedge T;L^2(\T^d))$.

Choosing $\varepsilon>0$ small enough we can thus conclude
\begin{align*}
&\textstyle \|v(t)\|_{L^2(\T^d)}^2 + \mu\int_0^t \|\nabla v(s)\|^2_{L^2(\T^d)} \,\dd s \leq C \|v_0\|_{L^2(\T^d)}^2  + \int_0^t N_{\varepsilon,u}(s)[1+\|v(s)\|_{L^2(\T^d)}^2] + \wt{M}_t.
\end{align*}
By the stochastic Gronwall Lemma \ref{lem:Gronwall} this implies \eqref{eq:PPstochgronwallappli} with $u$ replaced by $v$. Thus, Theorem \ref{thm:varloc} implies $\sigma=\infty$ as required.

The final assertion on the continuous dependency in the norm of $C([0,T];L^2)$ follows from \cite[Theorem 5.2]{AVreaction-global}. However, it can be seen from the proof that the convergence holds in $L^2(0,T;H^1)$ as well. An alternative way would be to extend the proof of \cite[Theorem 3.8]{AVvar} to the above setting. For this, one needs to use the tail estimates for the above solutions which are provided by the stochastic Gronwall lemma.

\subsection{Quasi-geostrophic equations with $\alpha\in (1/2,1)$}\label{ss:quasiGSLp}
In Subsection \ref{ss:quasiGS}, we considered quasi-geostrophic systems on $\R^2$ but only in the case $\alpha=1$.
This is because $L^2$ is critical for the corresponding SPDE. In the case $\alpha<1$, critical spaces change accordingly. In particular, more regularity from the critical space is needed in this context, and thus $L^p(L^q)$-techniques come naturally. For details, the reader is referred to Subsection \ref{sss:scaling_quasigeostrophic} below.
Let us anticipate that $L^p(L^q)$-techniques are also handy in this context, as for quasi-geostrophic equations $L^\zeta$-bounds with $\zeta<\infty$ are typically available, see Lemma \ref{lem:a_priori_estimate_quasigeostrophic}.

Here we consider the quasi-geostrophic equation on the two-dimensional torus $\T^2$:
\begin{equation}
\label{eq:QGTalpha}
\left\{
\begin{aligned}
\dd \theta &=\textstyle -\big[ (-\Delta)^{\alpha} \theta + (u\cdot \nabla)\theta \big] \,\dd t
+\sum_{n\geq 1} g_n(\cdot, \theta)\,\dd W_t^n,
\\
u&=R^{\bot} \theta,
\\ \theta(0,\cdot)&=\theta_0.
\end{aligned}\right.
\end{equation}
In the following, we focus on the case $\alpha\in (1/2,1)$, as the situation where $\alpha=1$ is an exceptional case and it can be dealt with in the same way as in Subsection \ref{ss:quasiGS}. Proving regularity in the $\alpha=1$ setting is a bit more work but is certainly possible and can be done in a similar way as in \cite{AV20_NS} for $p=q=2$.

As before $R^{\bot} \theta = (-R_2 \theta, R_1 \theta) =  (-\partial_2 \Delta^{-1/2} \theta, \partial_1 \Delta^{-1/2} \theta)$ with $\theta$ a periodic distribution. Note that $\wh{R_j \theta}(n_1, n_2) = \frac{n_j}{|n|}\wh{\theta}(n)$ with $n\in \Z^2$, where we set $0/0 = 0$. The definition of $(-\Delta)^{\alpha}$ on the space of periodic distributions is given through the Fourier series as $\wh{(-\Delta)^{\alpha}\theta} (n)= |2\pi n|^{2\alpha} \wh{\theta}(n)$ for $n\in \Z^2$.

\subsubsection{Scaling}
\label{sss:scaling_quasigeostrophic}
Here, as in Subsections \ref{ss:scaling_intro} and \ref{ss:tracecriticalF}, we discuss the scaling of \eqref{eq:QGTalpha} by mainly looking at the deterministic part of the SPDE. We will afterwards make assumptions on the diffusion coefficients $g_n$ so that the nonlinear diffusion $(g_n(\cdot,u))_{n\geq 1}$ becomes as critical as the deterministic part, see the text below \eqref{eq:matching_scaling_AllenCahn}.
Of course, this is not a limitation of our framework. Indeed, the latter can be applied even if the diffusion dominates the scaling, but in this situation, one obtains critical spaces which differ from the one in the deterministic setting. For simplicity, we do not pursue this here.

Consider the deterministic quasi-geostrophic equation on $\R^2$:
\begin{equation}
\label{eq:deterministic_QG_scaling}
\partial_t \theta =- (-\Delta)^\alpha \theta + ( R^\bot\theta \cdot \nabla) \theta.
\end{equation}
One can easily check that global solutions to the above are invariant under the rescaling:
$$
\theta\mapsto \theta_{\lambda}\qquad \text{ where }\qquad
\theta_\lambda(t,x):= \lambda^{1-\frac{1}{2\alpha}} \theta(\lambda t,\lambda^{\frac{1}{2\alpha}} x) \ \text{ for }(t,x)\in \R_+\times \R^2,
$$
where $\lambda>0$. Thus, spaces which are invariant under the induced map on the initial data $\theta_0\mapsto \theta_{0,\lambda} :=\lambda^{1-\frac{1}{2\alpha}}  \theta_{0}(\lambda^{\frac{1}{2\alpha}} \cdot )$ are given by $L^{2/(2\alpha-1)}(\R^2)$ and $B^{1-2\alpha+2/q}_{q,p}(\R^2)$. Indeed, the corresponding homogeneous variants of these spaces satisfy
\[
\|\theta_{0,\lambda}\|_{\dot{B}^{1-2\alpha+2/q}_{q,p}(\R^2)}\eqsim \|\theta_0\|_{\dot{B}^{1-2\alpha+2/q}_{q,p}(\R^2)},
 \quad \text{ and }\quad
\|\theta_{0,\lambda}\|_{L^{2/(2\alpha-1)}(\R^2)}\eqsim \|\theta_0\|_{L^{2/(2\alpha-1)}(\R^2)}.
\]
The previous can be expressed by saying the Sobolev indexes of the spaces $L^{2/(2\alpha-1)}(\R^2)$ and $B^{1-2\alpha+2/q}_{q,p}(\R^2)$  are both equal to $1-2\alpha$.

Next, we aim to study the optimal scaling of the diffusion so that the stochastic equation \eqref{eq:QGTalpha} has the same scaling (and, consequently, critical spaces) as the deterministic counterpart \eqref{eq:deterministic_QG_scaling}. To this end, we argue as in the case of the Allen-Cahn equation \eqref{eq:matching_scaling_AllenCahn}: If $g_n(\cdot,\theta)=|\theta|^{m}$ for some $m\geq 1$, then
\begin{equation}
\label{eq:matching_scaling_quasigeostrophic}
\begin{aligned}
\textstyle \int_0^{t/\lambda} (R^\bot \theta_\lambda(s,x)\cdot \nabla)\theta_{\lambda}(s,x)\, \dd s
&= \textstyle \lambda^{1-\frac{1}{2\alpha}}\int_0^{t} (R^\bot \theta(s,\lambda^{\frac{1}{2\alpha}}x)\cdot \nabla)\theta(s,\lambda^{\frac{1}{2\alpha}}x)\, \dd s,\\
\textstyle \int_0^{t/\lambda} g_n(\cdot,\theta_{\lambda}(s,x)) \,\dd  \beta_{s,\lambda}^n
&= \textstyle \lambda^{m (1-\frac{1}{2\alpha})-\frac{1}{2}}
\int_0^{t} |\theta(s,\lambda^{\frac{1}{2\alpha}} x)|^m \,\dd  W^n_t.
\end{aligned}
\end{equation}
Thus, choosing $m=\frac{3\alpha-1}{2\alpha-1}$, we obtain that \eqref{eq:QGTalpha} has the same local scaling of the deterministic counterpart \eqref{eq:deterministic_QG_scaling} and therefore the same critical spaces.

\subsubsection{Local well-posedness and regularity}
In order to analyze \eqref{eq:QGTalpha} for all $\alpha\in (1/2, 1)$, it is not enough to use $L^2$-theory. Indeed, in the weak setting $\alpha=1$ already led to the critical setting in Subsection \ref{ss:quasiGS}, and in the strong setting, coercivity fails. Another disadvantage of the strong setting is that regularity in space is needed for the diffusion nonlinearity $(g_n(\cdot,u))_{n\geq 1}$. Motivated by the scaling argument in the previous subsection, we assume that  $g:=(g_n)_{n\geq1 }:\T^2\times \R\to \ell^2$ is a measurable mapping satisfying
$g(\cdot, 0)\in L^\infty(\T^2;\ell^2)$ and
\begin{align}\label{eq:Lipschitzggeo_local}
\|g(x,y) - g(x,y')\|_{\ell^2}\lesssim(1+|y|^{m_\alpha-1}+|y'|^{m_\alpha-1}) |y-y'|, \ \ \ x\in \T^2, y,y'\in \R.
\end{align}
where $m_\alpha:=\frac{3\alpha-1}{2\alpha-1}$. Of course, we could allow for higher growth of $g$. However, the latter will cause critical spaces for \eqref{eq:QGTalpha} to differ from the ones of \eqref{eq:deterministic_QG_scaling}. We leave the details to the interested reader.
Here we will consider \eqref{eq:QGTalpha} in weak PDE setting:
\begin{equation}
\label{eq:choice_X0X1_quasigeostrophic_LpLq}
X_0 = H^{-1,q}(\T^2) \quad \text{and} \quad X_1 = H^{-1+2\alpha,q}(\T^2).
\end{equation}
Note that $X_{\beta} = H^{-1+2\alpha\beta,q}(\T^2)$. As usual, we see \eqref{eq:QGTalpha} in the form \eqref{eq:QGTalpha} with $\mathcal{U}=\ell^2$, the above choice of $X_0$ and $X_1$, and
\begin{equation}
\label{eq:choice_ABFG_quasi_LpLq}
A \theta =  (-\Delta)^{\alpha} \theta,  \quad F(\theta) = \div (R^\bot \theta \,\theta ),\quad B\theta=0,  \quad G(\theta) = (g_n(\cdot, \theta))_{n\geq 1}.
\end{equation}
Restrictions on $q$ to make the above mappings well-defined are given below.
In the above, we also used that, at least formally,
$\div (R^\bot \theta \,\theta )=(R^\bot \theta\cdot\nabla) \theta$, as $\div\,R^\bot \theta=0$.
As is well-known, the conservative form of $F$ in \eqref{eq:choice_ABFG_quasi_LpLq} is handier for the weak analytic formulation of SPDEs.

As usual, we say that $(\theta,\sigma)$ is a (unique) $(p,\a,q)$-solution to \eqref{eq:QGTalpha} if $(\theta,\sigma)$ is a maximal $L^p_\a$-solution to \eqref{eq:SEE} with the above choice of $X_j$, $\mathcal{U}$ and $(A,B,F,G)$.

\smallskip

Next, we turn to check Assumption \ref{ass:FGcritical} in the current situation. Let $\Phi(\theta_1,\theta_2):=\div (R^\bot \theta_1 \,\theta_2) $. By H\"older's inequality and the boundedness of the Riesz transform, one has
\begin{align*}
\|\Phi(\theta_1, \theta_2)\|_{X_0}\lesssim \|\theta_2 R^\bot \theta_1\|_{L^q(\T^2)}\lesssim \|\theta_1\|_{L^{2q}(\T^2)} \|\theta_2\|_{L^{2q}(\T^2)}\lesssim \|\theta_1\|_{X_{\beta_1}}\|\theta_2\|_{X_{\beta_1}},
\end{align*}
where in the last step we used Sobolev embedding with $-1+2\alpha\beta_1 - \frac{2}{q}\geq -\frac{2}{2q}$. Thus we may choose $\beta_1 = \frac{q+1}{2\alpha q}$. Since we need $\beta_1<1$ this gives the restriction $q>\frac{1}{2\alpha-1}$.
We obtain
\[\|F(\theta_1) - F(\theta_2)\|_{X_0}\leq \|\Phi(\theta_1, \theta_1-\theta_2)\|_{X_0} + \|\Phi(\theta_1-\theta_2, \theta_2)\|_{X_0}\lesssim (\|\theta_1\|_{X_{\beta_1}} +\|\theta_2\|_{X_{\beta_1}})\|\theta_1-\theta_2\|_{X_{\beta_1}}.\]
On the other hand, since the nonlinearity is quadratic, the criticality condition \eqref{eq:subcritical} is $\frac{1+\kappa}{p}\leq 2(1-\beta_1)$. We can find $p>2$ and $\kappa\in[0,p/2-1)$ such that equality holds if and only if $2(1-\beta_1)<1/2$, which is equivalent to $\beta_1>3/4$. Using the above value of $\beta_1$ this leads to the condition $\frac{1}{q}>\frac{3\alpha}{2} - 1$.

Thus, if $\alpha\leq 2/3$ we choose any $q>\frac{1}{2\alpha-1}$, and if $\alpha\in(2/3,1)$ we choose $q\in (\frac{1}{2\alpha-1}, \frac{2}{3\alpha-2})$. Of course, we always have the restriction $q\geq 2$. The parameters $(p,\kappa)$ are then chosen such that
\begin{align}\label{eq:critcondquasigeo}
\frac{1+\kappa}{p} +\frac{1}{\alpha q} +\frac{1}{\alpha}= 2.
\end{align}
Next, let us turn our attention to $G$.
Due to \eqref{eq:Lipschitzggeo_local} and the H\"older inequality, for $\theta_1,\theta_2\in X_1$,
\begin{align*}
\|G(\theta_1)-G(\theta_2)\|_{\g(\ell^2,H^{-1+\alpha,q})}
&\stackrel{(i)}{\lesssim} \|g(\theta_1)-g(\theta_2)\|_{\g(\ell^2,L^{\zeta})}\\
&\lesssim (1+\|\theta_1\|_{L^{m_\alpha \zeta}}^{m_\alpha-1}+\|\theta_2\|_{L^{m_\alpha \zeta}}^{m_\alpha-1})
\|\theta_1-\theta_2\|_{L^{m_\alpha \zeta}}\\
&\stackrel{(ii)}{\lesssim} (1+\|\theta_1\|_{X_{\beta_2}}^{m_\alpha-1}+\|\theta_2\|_{X_{\beta_2}}^{m_\alpha-1})
\|\theta_1-\theta_2\|_{X_{\beta_2}}
\end{align*}
where in $(i)$ we used that $L^\zeta(\T^2)\embed H^{-1+\alpha,q}(\T^2)$ with $-\frac{2}{r}=-1+\alpha-\frac{2}{q}$ (note that $r\in (1,q)$, as we are assuming $q>\frac{1}{2\alpha-1}$, and thus $q>\frac{2}{1+\alpha}$) and in $(ii)$ we used $X_{\beta_2}=H^{\beta_2,q}(\T^2)\embed L^{m_\alpha \zeta}(\T^2)$ with $\beta_2=\frac{1}{2\alpha}+\frac{\alpha-1}{2\alpha m_\alpha}+\frac{1}{\alpha q}$. With standard computations, the reader can check that the (sub)criticality condition \eqref{eq:subcritical} with $\beta_2$ as above and $\rho_2=m_\alpha$ takes the form \eqref{eq:critcondquasigeo} in the critical case (this was indeed expected due to the scaling argument in \eqref{eq:matching_scaling_quasigeostrophic}).

In the case $(p,\kappa)$ satisfies the equality \eqref{eq:critcondquasigeo}, the critical space of initial values becomes
\[(X_0, X_1)_{1-\frac{1+\kappa}{p},p} = B^{-1+2\alpha -2\alpha\frac{1+\kappa}{p}}_{q,p}(\T^2) = B^{1-2\alpha +\frac{2}{q}}_{q,p}(\T^2),\]
which is also critical from a PDE point of view, see Subsection \ref{sss:scaling_quasigeostrophic}. Moreover, for $\alpha\in (1/2,2/3]$ we can allow smoothness down to $1-2\alpha$. If $\alpha\in (2/3, 1)$, we can allow smoothness down to $\alpha-1$. Interestingly, in all situations, the limiting smoothness is negative. As in Subsections \ref{ss:AllenCahnLpLq} and \ref{ss:reaction}, the above picture can be refined by choosing $X_j=H^{-\delta+2j\alpha,q}(\T^2)$ with $\delta\in (1,2]$ instead of the choice in \eqref{eq:choice_X0X1_quasigeostrophic_LpLq}.

Since $1+A$ can be shown to have an $H^\infty$-calculus by a periodic analogue of the argument in \cite[Theorem 10.2.25]{Analysis2}, it follows from Theorem \ref{thm:SMRHinfty} that $(A,0)\in \mathcal{SMR}_{p,\kappa}^{\bullet}$ for all $p\in (2, \infty)$ and $\kappa\in [0,p/2-1)$.
Now, Theorem \ref{thm:localwellposed} implies the following and is our first step in showing the existence and uniqueness of a global solution.  Here we say that $(u, \sigma)$ is a $(p,\kappa,q)$-(maximal) solution if it is an $L^p_{\kappa}$-(maximal) solution with spaces $X_j = H^{-1+2\alpha j, q}(\T^2)$ for $j\in \{0,1\}$.
\begin{proposition}[Local well-posedness in critical spaces]\label{prop:locQS}
Let $\alpha\in (1/2, 1)$. Let $q\geq 2$ be such that $q>\frac{1}{2\alpha-1}$, and such that $q\in (\frac{1}{2\alpha-1}, \frac{2}{3\alpha-2})$ if $\alpha>2/3$. Let $p>2$ and $\kappa\in [0,\frac{p}{2}-1)$ be as in \eqref{eq:critcondquasigeo}.
Let $g$ be as in \eqref{eq:Lipschitzggeo_local} and
$$
\theta_0\in L^0_{\F_0}(\Omega;B^{1-2\alpha +\frac{2}{q}}_{q,p}(\T^2)).
$$
Then \eqref{eq:QGTalpha} has an $(p,\kappa,q)$-solution $(\theta,\sigma)$. Moreover, a.s.
\begin{align*}
\theta& \in H^{\lambda,p}_{\rm loc}([0,\sigma), w_{\kappa};H^{-1+2\alpha-2\alpha\lambda,q}(\T^2)), \ \ \ \lambda\in [0,1/2),
\\ \theta&\in C([0,\sigma);B^{1-2\alpha +\frac{2}{q}}_{q,p}(\T^2)) \cap C((0,\sigma);B^{-1+2\alpha(1 -\frac{1}{p})}_{q,p}(\T^2)).
\end{align*}
\end{proposition}

From the scaling argument in Subsection \ref{sss:scaling_quasigeostrophic}, the above yields local well-posedness in critical spaces of Besov-type for \eqref{eq:QGTalpha}. Noticing that
$$
L^{2/(2\alpha-1)}(\T^2)\subseteq B^{0}_{2/(2\alpha-1),p}(\T^2)
$$
for $p\geq \frac{2}{2\alpha-1}$, the above also implies local well-posedness in the Lebesgue critical space $L^{2/(2\alpha-1)}(\T^2)$.

Next, we would like to show that the above solution exists globally, i.e.\ $\sigma = \infty$. To show this, we will use one of the blow-up criteria. However, to check the required bounds, we need to apply It\^o calculus to bound $\|\theta\|_{L^q(\T^2)}^q$, but for this, we need that $\theta$ is $H^1(\T^2)$-valued. To show that, we first bootstrap the regularity of $\theta$ locally in time on $(0,\sigma)$.

\begin{proposition}[Instantaneous regularization]\label{prop:QSregularity}
Let $\alpha\in (1/2, 1)$. Let $q\geq 2$ be such that $q>\frac{1}{2\alpha-1}$, and such that $q\in (\frac{1}{2\alpha-1}, \frac{2}{3\alpha-2})$ if $\alpha>2/3$. Let $p>2$ and $\kappa\in [0,\frac{p}{2}-1)$ be as in \eqref{eq:critcondquasigeo}.
Let $g$ be as in \eqref{eq:Lipschitzggeo_local} and $\theta_0\in L^0_{\F_0}(\Omega;B^{1-2\alpha +\frac{2}{q}}_{q,p}(\T^2))$.
Let $(\theta,\sigma)$ be the $(p,\kappa,q)$-solution of \eqref{eq:QGTalpha} given by Proposition \ref{prop:locQS}.
Then a.s.
\begin{equation}
\label{eq:high_order_regularity_quasigeostrophic_eq}
\theta\in L^{r}_{\loc}((0,\sigma);H^{\alpha,\zeta}(\T^2))\cap C_{\loc}^{\lambda/2,\alpha\lambda}((0,\sigma)\times \T^2) \ \ r,\zeta\in (2, \infty), \lambda\in (0,1).
\end{equation}
\end{proposition}

In case $g\in C^n(\T^2\times \R)$ for some $n\geq 1$, an iteration of the argument below shows more spatial smoothness for the solution of \eqref{eq:QGTalpha}. In particular, if $g\in C^\infty(\T^2\times \R)$, then $\theta\in H^{\lambda,r}_{\rm loc}((0,\sigma);C^{n}_b(\T^2))$ a.s.\ for any $n\in \N$ and $\lambda\in [0,1/2)$. All these higher-order regularity results appear to be new, cf.\ \cite[Theorem 2.7]{AV20_NS} and \cite[Theorem 4.2]{AV20_NS}.
Moreover, Proposition \ref{prop:QSregularity} also holds in the case $\alpha=1$. However, the proof of time regularization (see Step 1 in the proof below) requires more care because $L^2$-data is critical for the corresponding SPDEs, see Subsection \ref{sss:scaling_quasigeostrophic}. However, in the latter case, the time regularization follows almost verbatim from the argument in \cite[Proof of Theorem 4.1 part C]{AV20_NS}.
After that, one can use the argument of Proposition \ref{prop:QSregularity} again.

\begin{proof}
To improve the regularity of $(\theta,\sigma)$, we use a bootstrap technique. We first bootstrap the integrability in time, and then we use the latter to improve the integrability in space. After that, we can also improve the Sobolev regularity in space.
It is convenient to start by bootstrapping time regularity, as it can be done using abstract results such as the one in Section \ref{subsec:reg}.

In the proof below, for notational brevity, we write
\[
\theta\in {\rm S}^{r,\zeta,\gamma} \text{ a.s.} \quad \text{ if }\quad
\theta\in \textstyle{\bigcap_{\lambda\in [0,1/2)}} H^{\lambda,r}_{\rm loc}((0,\sigma);H^{\gamma-2\alpha\lambda,\zeta}(\T^2)) \text{ a.s.\ }
\]
Note that by Proposition \ref{prop:locQS} we know that $\theta\in {\rm S}^{p,q,2\alpha-1}$ a.s.
Below, $(\sigma_n)_{n\geq 1}$ is a localizing sequence for $(\theta,\sigma)$.

{\em Step 1: Integrability in time.} Since $p>2$ we can use Proposition \ref{prop:locQS} and Theorems \ref{thm:parabreg} and \ref{thm:parabreg2} to obtain that $\theta\in \bigcap_{r\in (2, \infty)}{\rm S}^{r,q,2\alpha-1}$ a.s.

{\em Step 2: Integrability in space.}
We will show that $\theta\in \bigcap_{r,\zeta\in (2, \infty)} {\rm S}^{r,\zeta,2\alpha-1}$ a.s.
For this, it is enough to show that there exists an $\delta>0$ such that for all $\zeta\in [q, \infty)$,
\[\textstyle \theta\in \bigcap_{r\in (2, \infty)} {\rm S}^{r,\zeta,2\alpha-1} \ \text{a.s.} \quad \Rightarrow\quad \theta\in \bigcap_{r\in (2, \infty)} {\rm S}^{r,\zeta+\delta,2\alpha-1} \ \text{a.s.}\]
To prove this let $\zeta\geq q$ be given. Let $s\in (0,\infty)$. Then on the set $\{\sigma_n>s\}\times (s, \sigma_n)$
\begin{align*}
\|F(\theta)\|_{H^{-1,\zeta+\delta}(\T^2)} &\lesssim  \|\theta R^\bot \theta\|_{L^{\zeta+\delta}(\T^2)}\lesssim \|\theta\|_{L^{2\zeta+2\delta}(\T^2)}^2
\lesssim \|\theta\|_{H^{2\alpha-1,\zeta}(\T^2)}^2,
\end{align*}
where the last step holds if we can show $2\alpha-1 - \frac{2}{\zeta} \geq -\frac{2}{2\zeta+2\delta} = -\frac{1}{\zeta+\delta}$ for a suitable $\delta$ independent of $\zeta$. For this, let $\delta = \min\{q^2(2\alpha-1 - \frac{1}{q}),1/8\}>0$ and note that
\[2\alpha-1 - \frac{1}{\zeta} \geq 2\alpha-1 - \frac{1}{q}\geq \frac{\delta}{q^2}\geq \frac{\delta}{\zeta(\zeta+\delta)}.\]
The latter implies the desired estimate.
By \eqref{eq:Lipschitzggeo_local}, for $G$ a similar argument works.

For the initial value, by the trace result of Proposition \ref{prop:tracespace} we have a.s.
\[\one_{\sigma>s} \theta(s)\in B^{-1+2\alpha(1 -\frac{1}{r})}_{\zeta,r}(\T^2)\subseteq B^{-1+2\alpha(1 -\frac{1+\wt{\kappa}}{r})}_{\zeta+\delta,r}(\T^2)\]
by Sobolev embedding if $\frac{2\alpha \wt{\kappa}}{r}-\frac{2}{\zeta}\geq -\frac{2}{\zeta+\delta}$. The latter holds with $\wt{\kappa} = r/4$ by the choice of $\delta$. Indeed,
\[\frac{2\alpha \wt{\kappa}}{r}\geq \frac{\wt{\kappa}}{r} = \frac14 \geq 2\delta\geq \frac{2}{\zeta}-\frac{2}{\zeta+\delta}.\]

Since $(A, 0)\in \mathcal{SMR}_{r,\wt{\kappa}}$ for $Y_0 = H^{-1,\zeta+\delta}(\T^2)$ and $Y_1 = H^{2\alpha-1,\zeta+\delta}(\T^2)$, the desired implication follows from stochastic maximal $L^r_{\wt{\kappa}}$-regularity from Proposition \ref{prop:localizationSMR} with deterministic and stochastic inhomogeneities $F(\theta)$ and $G(\theta)$ and initial value $\one_{\sigma>s} \theta(s)$.

{\em Step 3: Sobolev smoothness in space.} Here we prove \eqref{eq:high_order_regularity_quasigeostrophic_eq}.
For simplicity, here we assume $\alpha>\frac{2}{3}$. In the case $\alpha\in (\frac{1}{2},\frac{2}{3}]$, it is enough to apply the argument below twice.

Fix $s>0$ and $n\geq 1$. From Step 2, we know that, a.s.\ on $\{\sigma_n>s\}$,
\begin{equation}
\label{eq:theta_regularity_bootstrap_final_step}
\theta\in C([s,\sigma_n];H^{-1+2\alpha-\varepsilon,\zeta}(\T^2))
\end{equation}
for all $\varepsilon>0$ and $\zeta<\infty$. Next, we would like to estimate $F(\theta)=\div(R^\bot \theta\,\theta)$ and $G(\theta)$ on the stochastic interval $\{\sigma_n>s\}\times \O$ by knowing  \eqref{eq:theta_regularity_bootstrap_final_step}. To this end, we fix
$\varepsilon>0$ small and $\zeta<\infty$ large so that $-1+2\alpha-\varepsilon>\frac{2}{\zeta}$ and $-1+2\alpha-\varepsilon>1-\alpha$. As for the first condition, one can check that a choice of the parameters is always possible, while for the second one needs $\alpha>\frac{2}{3}$.
With $\varepsilon$ and $\zeta$ as before, we have the following embedding at our disposal:
\begin{equation}
\label{eq:theta_regularity_bootstrap_final_step_embeddings}
H^{-1+2\alpha-\varepsilon,\zeta}(\T^2)\subseteq L^\infty(\T^2)\quad \text{ and }\quad
H^{-1+2\alpha-\varepsilon,\zeta}(\T^2)\subseteq H^{1-\alpha,\zeta}(\T^2),
\end{equation}
where the first follows from Sobolev embeddings. Thus, \eqref{eq:theta_regularity_bootstrap_final_step} implies $G(\theta)\in{ L^\infty(s,\sigma_n;L^{\zeta}(\ell^2))}$ a.s.\ on $\{\sigma_n>s\}$. Moreover,
\begin{align*}
\|F(\theta)\|_{H^{-\alpha,\zeta}(\T^2)}\leq \|\theta R^\bot \theta\|_{H^{{1-\alpha},\zeta}(\T^2)}
\lesssim \|\theta\|_{H^{{1-\alpha},\zeta}(\T^2)\cap L^\infty(\T^2)}^2,
\end{align*}
where in the last step, we used standard multiplication results for functions with fractional Sobolev smoothness, see e.g.\ \cite[Proposition 2.1.1]{ToolsPDEs}. Thus, $F(\theta)\in { L^\infty(s,\sigma_n;H^{-\alpha,\zeta}(\T^2))}$ a.s.\ on $\{\sigma_n>s\}$.
As in Step 2, for the initial value we have
\[\one_{\{\sigma>s\}} \theta(s)\in H^{-1+2\alpha-\varepsilon,\zeta}(\T^2)
\subseteq B^{\alpha(1-2\frac{1+\wt{\a}}{r})}_{\zeta,r}(\T^2)\]
for any $r<\infty$ provided $\wt{\a}\in [0,\frac{r}{2}-1)$ is sufficiently large.

As $r<\infty$ and $\zeta<\infty$ can be chosen arbitrarily large, we can conclude the desired {Sobolev} regularity from $(A, 0)\in \mathcal{SMR}_{r,\wt{\kappa}}$ for $Y_0 = H^{-\alpha,\zeta}(\T^2)$ and $Y_1 = H^{\alpha,\zeta}(\T^2)$. {To establish the claimed H\"older regularity in \eqref{eq:high_order_regularity_quasigeostrophic_eq}, observe that for all $\lambda\in (0,\frac{1}{2})$ and sufficiently large $r,\zeta\in (2,\infty)$,  
$$
H^{\lambda,r}_{\loc}(0,\sigma;H^{\alpha(1-2\lambda),\zeta}(\T^2))\subseteq C^{\lambda-1/r}_{\loc}((0,\sigma);C(\T^2))
\cap C_{\loc}((0,\sigma);C^{\alpha(1-2\lambda)-2/\zeta}(\T^2)).
$$ The claimed regularity follows from the arbitrariness of $\lambda\in (0,1)$.}
\end{proof}

\subsubsection{Global existence and uniqueness}

After these preparations, we now have enough regularity of the local solution $(\theta, \sigma)$, to obtain an a priori bound by It\^o calculus in a similar way as presented in \cite[Lemma 3.12]{AVreaction-global}.
\begin{lemma}[A priori estimate]
\label{lem:a_priori_estimate_quasigeostrophic}
Let $\alpha\in (1/2, 1)$. Let $q\geq 2$ be such that $q>\frac{1}{2\alpha-1}$, and such that $q\in (\frac{1}{2\alpha-1}, \frac{2}{3\alpha-2})$ if $\alpha>2/3$. Let $p>2$ and $\kappa\in (0,p/2-1)$ be as in \eqref{eq:critcondquasigeo}. Suppose that $\theta_0\in L^0_{\F_0}(\Omega;B^{1-2\alpha +\frac{2}{q}}_{q,p}(\T^2))$ and let $g:\T^2\times \R\to \ell^2$ be as in \eqref{eq:Lipschitzggeo_local} and also of linear growth, i.e.\ for some $L_g\geq 0$,
$$
\|g(x,y)\|_{\ell^2} \leq L_g (1+|y|) \ \ \text{ for all } x\in\T^2 , \ y\in\R.
$$
Let $(\theta,\sigma)$ be the $(p,\kappa,q)$-maximal solution of \eqref{eq:QGTalpha} given by Proposition \ref{prop:locQS}. Then for all $T>s>0$ and $\zeta\geq 2$ one has
\begin{align}\label{eq:sboundas}
\sup_{t\in [s, \sigma\wedge T)}\one_{\sigma>s}\|\theta(t)\|_{L^\zeta(\T^2)}<\infty \ \ a.s.
\end{align}
Moreover,
\begin{align}\label{eq:sbound}
\E\sup_{r\in [s,\sigma\wedge T]}\one_{\Gamma_{k,s}}\|\theta(r)\|_{L^{\zeta}(\T^2)}^{\zeta}\leq C_T\big(1+ \E\one_{\Gamma_{k,s}}\|\theta(s)\|_{L^{\zeta}(\T^2)}^{\zeta}\big),
\end{align}
where $\Gamma_{{k,s}}= \{\sigma>s, \|{\theta}(s)\|_{L^{\zeta}(\T^2)}\leq k\}$. Furthermore, if there exists an $\varepsilon>0$ such that ${-1+2\alpha(1-\frac{1+\kappa}{p})>\varepsilon}$ and $\theta_0\in L^q(\Omega;H^{\varepsilon,q}(\T^2))$, then
\begin{align}\label{eq:sboundzero}
\E\sup_{r\in [0,\sigma\wedge T)}\|\theta(r)\|_{L^{q}(\T^2)}^{q}\leq C_T\big(1+ \E\|\theta_0\|_{L^{q}(\T^2)}^{q}\big).
\end{align}
\end{lemma}

\begin{proof}
By Proposition \ref{prop:QSregularity}, $\theta$ is locally regular on $(0,\sigma)$ and, in particular, $\theta\in C((0,\sigma)\times\T^2)\cap L^2_{\rm loc}((0,\sigma);H^{{\alpha,\eta}}(\T^2))$ a.s.\ {for any $\eta<\infty$. Now fix $\eta$ such that $\eta>2/\alpha$ and $\eta\geq \zeta$. Then $H^{\alpha,\eta}(\T^2)\subseteq L^\infty(\T^2)$ by Sobolev embedding}.

Let $(\sigma_{j})_{j\geq 1}$ be a localizing sequence for $(\theta,\sigma\wedge T)$. Let
\begin{align*}
\tau_j = \inf\{t\in [s, \sigma_j]: \|\theta(t)-\theta(s)\|_{C(\T^2)} + \|\theta\|_{L^2(s,t;H^{{\alpha,\eta}}(\T^2))}\geq j\}
\end{align*}
on the set $\{\sigma>s\}\cap \{\|{\theta}(s)\|\leq j-1\}$, and $\tau_j = s$ otherwise. The convention is that $\inf\emptyset =\sigma_j$. Let $\Gamma\in \F_s$ be such that $\Gamma \subseteq \{\sigma>s\}$. Let $\theta^{(j)}(t) = \one_{\Gamma} \theta(t\wedge \tau_j)$. Since $(\theta, \sigma)$ is a $(p,\kappa,q)$-maximal solution of \eqref{eq:QGTalpha}, it follows that
\begin{align*}
\textstyle \theta^{(j)}(t) - \theta^{(j)}(s) &=\textstyle -\int_0^t \one_{[s,\tau_j]\times \Gamma} \big[ (-\Delta)^{\alpha} \theta + F(\theta) \big] \,\dd  r
+\sum_{n\geq 1}  \int_s^{t} \one_{[s,\tau_j]\times \Gamma} g_n(\theta)\,\dd W^n.
\end{align*}
From the definition of $\tau_j$ we see that $\|{\theta}^{(j)}\|_{L^\infty((s,\tau_j)\times\T^2)}\leq 2j-1$ on the set $\{\tau_j>s\}$. Let $\xi\in C^2_b(\R)$ be such that $\xi(y) = |y|^{\zeta}$ for $|y|\leq 2j-1$. By an extended version of It\^o's formula (see \cite[Proposition A.1]{DHV16}) we obtain that a.s.\ for all $t\in [s,T]$,
\begin{align*}
\|\theta^{(j)}(t)\|_{L^{\zeta}(\T^2)}^{\zeta} = \|\theta^{(j)}(s)\|_{L^{\zeta}(\T^2)}^{\zeta} + \zeta \mathcal{D}(t) + \zeta\mathcal{S}(t),
\end{align*}
where
\begin{align*}
\mathcal{D}(t) &= \textstyle  \int_s^t \int_{\T^2}\one_{[s,\tau_j]\times \Gamma} \Big[ -  \theta |\theta|^{\zeta-2}  (-\Delta)^{\alpha}\theta - \theta |\theta|^{\zeta-2} R^{\bot} \theta \cdot \nabla \theta + \frac12 (\zeta-1) |\theta|^{\zeta-2} \|g(\theta)\|_{\ell^2}^2 \Big]\, \dd x \, \dd r,
\\ \mathcal{S}(t) &= \textstyle  \sum_{n\geq 1}\int_s^t \int_{\T^2}\one_{[s,\tau_j]\times \Gamma} \theta|\theta|^{\zeta-2} g_n(\theta) \,\dd x \, \dd W^n.
\end{align*}
Here, the terms $\int_{\T^2} |\theta|^{\zeta-2}\theta (-\Delta)^\alpha \theta\,\dd x  $ and $\int_{\T^2} \theta|\theta|^{\zeta-2} R^{\bot} \theta\cdot \nabla \theta\,\dd x $ are understood in the distributional sense. Indeed, $\theta |\theta|^{\zeta-2} \in H^{\alpha,\eta}(\T^2)$ by \cite[Corollary 10.5]{TayPDE3}. Moreover, as $\eta>2/\alpha$,
$H^{\alpha,\eta}(\T^2)$ is a Banach algebra by \cite[Proposition 10.2]{TayPDE3}. The distributional pairing is well defined by dividing $(-\Delta)^{\alpha}$ equally over both sides. For the convective term, the first derivative can be divided over the terms since $\alpha>1/2$. To prove the It\^o formula, one uses a mollifier argument combined with the above pairing.

Next, we need another approximation argument to estimate the first two spatial integrals in $\mathcal{D}(t)$, and then obtain the claimed energy estimate. Fix $\theta\in H^{\alpha,\eta}(\T^2)$ and let $(\theta_n)_{n\geq 1}\subseteq C^\infty(\T^2)$ be a sequence such that $\theta_n\to \theta$ in $H^{\alpha,\eta}(\T^2)$. Then, using \cite{TayPDE3} as before, one can check that $\sup_{n\geq 1}\||\theta_n|^{\zeta-2}\theta_n\|_{H^{\alpha,\eta}(\T^2)}<\infty$. Choosing a weakly convergent subsequence of the latter, the limit can be identified as $|\theta|^{\zeta-2}\theta$. Therefore, since $\theta_n\to \theta$ in $H^{\alpha,\eta}(\T^2)$ strongly, we obtain
\[\textstyle \int_{\T^2} |\theta|^{\zeta-2}  \theta (-\Delta)^{\alpha}\theta \, \dd x=\lim_{n\to \infty} \int_{\T^2} |\theta_n|^{\zeta-2}  \theta_n(-\Delta)^{\alpha}\theta_n \, \dd x\geq 0,\]
where the estimate follows from \cite[Proposition II.3.24 and Example II.3.26]{EN} and \cite[Lemma 2.1]{Prussgeo}. 

Next, for the convective term in $\mathcal{D}(t)$, one see that $\sup_{n\geq 1}\|\theta_n |\theta_n|^{\zeta-2} R^{\bot} \theta_n\|_{H^{\alpha,\eta}(\T^2)}<\infty$, and again by choosing a weakly convergent subsequence, we deduce that
\begin{align*}
\textstyle
\int_{\T^2} \theta |\theta|^{\zeta-2} R^{\bot} \theta \cdot \nabla \theta \,\dd x
&=
\textstyle \lim_{n\to \infty}\int_{\T^2} \theta_n |\theta_n|^{\zeta-2} R^{\bot} \theta_n \cdot \nabla \theta_n \,\dd x. \end{align*}
Integrating by parts, we get
\begin{align*}
\int_{\T^2} \theta_n |\theta_n|^{\zeta-2} R^{\bot} \theta_n \cdot \nabla \theta_n \,\dd x
\textstyle
=\frac{1}{\zeta}\int_{\T^2} R^{\bot} \theta_n \cdot \nabla |\theta_n|^{\zeta} \,\dd x =  -\frac{1}{\zeta}\int_{\T^2} \div(R^{\bot} \theta_n) |\theta_n|^{\zeta} \,\dd x = 0,
\end{align*}
where we also used that $\div(R^{\bot} \theta_n) = 0$. 

Therefore, using the linear growth of $g$ we can conclude that
\begin{align*}
\textstyle \|\theta^{(j)}(t)\|_{L^{\zeta}(\T^2)}^{\zeta} \leq \|\theta^{(j)}(s)\|_{L^{\zeta}(\T^2)}^{\zeta} + C\int_s^t \int_{\T^2} \one_{[s,\tau_j]\times \Gamma}  (|\theta|^{\zeta}+1)  \, \dd x \, \dd r + \zeta\mathcal{S}(t).
\end{align*}
Taking the supremum over time and then expectations, after applying the Burkholder-Davis-Gundy inequality, we obtain that
\begin{align*}
\textstyle  \E\sup_{r\in [s,t]}\|\theta^{(j)}(r)\|_{L^{\zeta}(\T^2)}^{\zeta}&\leq \E\|\theta^{(j)}(s)\|_{L^{\zeta}(\T^2)}^{\zeta} + C\E \int_s^t \one_{[s,\tau_j]\times \Gamma}  (\|\theta\|^{\zeta}_{L^{\zeta}(\T^2)}+1)\, \dd r \\ & + C' \E \Big(\sum_{n\geq 1}\int_s^t \one_{[s,\tau_j]\times \Gamma}  \Big| \int_{\T^2} |\theta|^{\zeta-1} |g_n(\theta)| \,\dd x\Big|^2 \, \dd r\Big)^{1/2}.
\end{align*}
Applying Minkowski's inequality, the linear growth of $g$ once more, and $ab\leq \varepsilon a^2 +C_{\varepsilon}b^2$ we obtain
\begin{align*}
\textstyle  \Big(\sum_{n\geq 1}\int_s^t \one_{[s,\tau_j]\times \Gamma}  & \textstyle  \Big| \int_{\T^2} |\theta|^{\zeta-1} |g_n(\theta)| \, \dd x\Big|^2 \, \dd r\Big)^{1/2} \leq
\Big(\int_s^t \one_{[s,\tau_j]\times \Gamma} \Big| \int_{\T^2}  |\theta|^{\zeta-1}  |g(\theta)\|_{\ell^2} \, \dd x\Big|^2 \, \dd r\Big)^{1/2}
\\ & \textstyle  \leq \Big(2\int_s^t \one_{[s,\tau_j]\times \Gamma}[ \|\theta\|_{L^{\zeta}(\T^2)}^{2\zeta}  +1  ]\, \dd r\Big)^{1/2}
\\ &\textstyle  \leq \Big( [2\sup_{r\in [s,t]}\|\theta^j\|_{L^{\zeta}(\T^2)}^{\zeta}+1] \int_s^t \one_{[s,\tau_j]\times \Gamma}[ \|\theta\|_{L^{\zeta}(\T^2)}^{\zeta} +1 \, \dd r]\Big)^{1/2}
\\ &\textstyle  \leq \varepsilon [\sup_{r\in [s,t]} \|\theta^j\|_{L^{\zeta}(\T^2)}^{\zeta}+1] + C_{\varepsilon} \int_s^t \one_{[s,\tau_j]\times \Gamma} [ \|\theta\|_{L^{\zeta}(\T^2)}^{\zeta} +1] \, \dd r.
\end{align*}
Taking expectation, combining the estimates, and choosing $\varepsilon>0$ small enough we can conclude
\begin{align*}
\textstyle \E\Big[\sup_{r\in [s,t]}\|\theta^{(j)}(r)\|_{L^{\zeta}(\T^2)}^{\zeta}+1\Big]&\leq \textstyle 2\E\Big[\|\theta^{(j)}(s)\|_{L^{\zeta}(\T^2)}^{\zeta}+1\Big] + C'_{\varepsilon}\E \int_s^t (\|\theta^j\|^{\zeta}_{L^{\zeta}(\T^2)}+1) \, \dd r
\end{align*}

Therefore, Gronwall's inequality applied to $y(t) = \E\Big[\sup_{r\in [s,t]}\|\theta^{(j)}(r)\|_{L^{\zeta}(\T^2)}^{\zeta}+1\Big]$ gives that
\begin{align*}
\textstyle \E\Big[\sup_{r\in [s,T]}\|\theta^{(j)}(r)\|_{L^{\zeta}(\T^2)}^{\zeta}+1\Big]\leq 2\E\Big[\|\theta^{(j)}(s)\|_{L^{\zeta}(\T^2)}^{\zeta}+1\Big] e^{(T-s)C'_{\varepsilon}}.
\end{align*}
Letting $j\to \infty$ we obtain that
\begin{align*}
\textstyle \E\Big[\sup_{r\in [s,T]}\one_{\Gamma}\|\theta(r)\|_{L^{\zeta}(\T^2)}^{\zeta}+1\Big]\leq 2\E\one_{\Gamma} \Big[\|\theta(s)\|_{L^{\zeta}(\T^2)}^{\zeta}+1\Big] e^{(T-s)C'_{\varepsilon}},
\end{align*}
which proves the required bound \eqref{eq:sbound}.
Taking $\Gamma_{k,s} = \{\|\theta(s)\|_{L^{\zeta}(\T^2)}\leq k\}\cap \{\sigma>s\}$, and noting that $\Gamma_{k,s}\uparrow \{\sigma>s\}$ it follows that a.s.\ on $\{\sigma>s\}$ one has $\sup_{r\in [s,T]}\|\theta(r)\|_{L^{\zeta}(\T^2)}<\infty$, which is \eqref{eq:sboundas}.

To derive the bound \eqref{eq:sboundzero} if $\theta_0\in L^q(\Omega;H^{\varepsilon,q}(\T^2))$, note that we can take $X_0 = H^{-1,q}(\T^2)$ and $X_1 = H^{-1+2\alpha,q}(\T^2)$ and let $\wt{\kappa}\geq \kappa$ be such that $0<-1+2\alpha(1-\frac{\wt{\kappa}+1}{p})<\varepsilon$. Then the trace space satisfies $H^{\varepsilon,q}(\T^2)\hookrightarrow B^{-1+2\alpha(1-\frac{\wt{\kappa}+1}{p})}_{q,p}(\T^2)\hookrightarrow L^{q}(\T^2)$. As in the proof of Proposition \ref{prop:locQS} one can see that there is a maximal $(p,\wt{\kappa}, q)$-solution $(\wt{\theta}, \wt{\sigma})$ to \eqref{eq:QGTalpha}. By Corollary \ref{cor:comp} $\wt{\sigma} = \wt{\sigma}$ and $\wt{\theta} = \theta$. Now, it remains to observe that a.s. $\theta=\wt{\theta}\in C([0,\sigma);B^{-1+2\alpha(1-\frac{\wt{\kappa}+1}{p})}_{q,p}(\T^2))\subseteq C([0,\sigma);L^{q}(\T^2))$.
Now, we can let $k\to \infty$ and $s\downarrow 0$ in \eqref{eq:sbound} for $\zeta=q$.
\end{proof}

Next, the energy bound will be used to derive $\sigma=\infty$ through blow-up criteria.
\begin{theorem}[Global existence and uniqueness]\label{thm:quasigeoglobal}
Let $\alpha\in (1/2, 1)$. Let $q\geq 2$ be such that $q>\frac{1}{2\alpha-1}$, and such that $q\in (\frac{1}{2\alpha-1}, \frac{2}{3\alpha-2})$ if $\alpha>2/3$. Let $p>2$ and $\kappa\in (0,p/2-1)$ be such that $\frac{1+\kappa}{p} + \frac{q+1}{\alpha q}= 2$.
Suppose $\theta_0\in L^0_{\F_0}(\Omega;B^{1-2\alpha +\frac{2}{q}}_{q,p}(\T^2))$ and $g:\T^2\times\R\to \R$ satisfies
\begin{align*}
g(\cdot, 0)\in L^\infty(\T^2;\ell^2), \ \ \text{and} \ \
\|g(x,y) - g(x,y')\|_{\ell^2}\leq L_g |y-y'|, \ \ \ x\in \T^2, y,y'\in \R.
\end{align*}
Then there exists a unique $(p,\kappa,q)$-solution $\theta$ of \eqref{eq:QGTalpha}. Moreover,
a.s.
\begin{align*}
\theta& \in H^{\lambda,p}_{\rm loc}([0,\infty), w_{\kappa};H^{-1+2\alpha-2\alpha\lambda,q}(\T^2))\cap C([0,\infty);B^{1-2\alpha +\frac{2}{q}}_{q,p}(\T^2)), & \lambda\in [0,1/2),
\\ \theta&\in H^{\lambda,r}_{\rm loc}((0,\infty);H^{1+\alpha-2\lambda, \zeta}(\T^2)), & r,\zeta\in (2, \infty), \lambda\in [0,1/2).
\end{align*}
Moreover, each of the bounds \eqref{eq:sboundas}, \eqref{eq:sbound}, and \eqref{eq:sboundzero} hold with $\sigma=\infty$.
\end{theorem}

\begin{proof}
On the set $\{s<\sigma<\infty\}$, for all $\zeta\geq 2$ one has a.s.\ $\sup_{t\in [s, \sigma\wedge T)}\one_{\sigma>s}\|\theta(t)\|_{L^\zeta(\T^2)}<\infty$. Therefore, on $\{s<\sigma<\infty\}$ a.s.
\begin{align*}
\|F(\theta)\|_{L^r(s,\sigma;H^{-1,q}(\T^2))}\lesssim \|\theta R^{\bot}\theta\|_{L^{q}(\T^2)}\leq \|\theta\|_{L^{2r}(s,\sigma;L^{2q}(\T^d))}^2<\infty
\end{align*}
and
\[\|g(\theta)\|_{L^r(s,\sigma;H^{-1+\alpha,q}(\T^2))}\lesssim \big\|(1+\|\theta\|_{L^{q}(\T^2)})\big\|_{L^r(s,\sigma)}<\infty.\]
Since $\one_{\sigma>s}\theta(s)\in B^{-1+2\alpha(1 -\frac{1}{r})}_{q,r}(\T^2)$ (see Proposition \ref{prop:QSregularity}), it follows from $(A,0) \in \mathcal{SMR}_{r,0}^{\bullet}$ and Proposition \ref{prop:localizationSMR} that on $\{s<\sigma<\infty\}$ a.s.
\[\theta\in C([s,\sigma];B^{-1+2\alpha(1 -\frac{1}{r})}_{q,r}(\T^2)).\]
In particular, choosing sufficiently large $r>2$, we find that $\lim_{t\uparrow \sigma} \theta(t)$ exists in $B^{-1+2\alpha(1 -\frac{1}{r})}_{q,r}(\T^2)\hookrightarrow B^{1-2\alpha+\frac2q}_{q,p}(\T^2) = X_{1-\frac{1+\kappa}{p},p}$, where we used that $-1+2\alpha>1-2\alpha+\frac2q$ due to $q>\frac{1}{2\alpha-1}$.
From Theorem \ref{thm:criticalblowup}\eqref{it1:criticalblowup} it follows that
\[\P(s<\sigma<\infty) = \P\big(s<\sigma<\infty,\lim_{t\uparrow \sigma} \theta(t) \ \text{exists in $X_{1-\frac{1+\kappa}{p},p}$} \big) = 0.\]
Since $\sigma>0$ a.s.\ we can conclude that $\P(\sigma<\infty) = \lim_{s\downarrow 0} \P(s<\sigma<\infty) = 0$.

The regularity assertions follow from Propositions \ref{prop:locQS} and \ref{prop:QSregularity}.
\end{proof}

\begin{remark}\
\begin{itemize}
\item
A priori estimates for  $\|\theta\|_{L^2(\Omega;L^2(s,\sigma;H^{\alpha}(\T^2)))}$ and $\|\theta\|_{L^p(\Omega;L^\infty(0,T;L^p(\T^2)))}$ were derived in \cite[(3.6) and Theorem 3.3]{RZZ} under the condition that $\theta_0\in L^p(\T^2)$. In the latter, the unique strong solution was obtained through probabilistic weak solutions and pathwise uniqueness.
\item
Theorem \ref{thm:quasigeoglobal} excludes the borderline case $\alpha=1/2$, which is considered hyperbolic instead of parabolic. In the deterministic setting, global well-posedness and regularity have been obtained in \cite{CafVas10,KisNazVol,ConstQuasigeo} on various domains. In the stochastic case, the global existence and uniqueness for \eqref{eq:QGTalpha} for $\alpha=1/2$ seems completely open.
\end{itemize}
\end{remark}

\subsection{Stochastic Navier--Stokes equations on the whole space}\label{subsec:SNSRd}
In this subsection, we analyze the following stochastic Navier--Stokes equations on $\R^d$ with $d\geq 2$:
\begin{equation}
\label{eq:Navier_StokesRd}
\left\{
\begin{aligned}
&\textstyle \dd u =\big[\nu \Delta u -(u\cdot \nabla)u -\nabla P  \big] \,\dd t
+\sum_{n\geq 1}\big[(\btwod_{n}\cdot\nabla) u +g_n(\cdot,u)-\nabla \wt{P}_n\big] \, \dd W_t^n,\\
&\div \,u=0,\\
&u(0,\cdot)=u_0.
\end{aligned}
\right.
\end{equation}
Here, $u$ denotes the unknown velocity field, $P$ the deterministic pressure and $\wt{P}_n$ the turbulent pressure.
Physical motivations for the model \eqref{eq:Navier_StokesRd} have been already discussed in Subsection \ref{ss:scaling_intro}.
Here, the noise is understood in the It\^o sense, $\nu>0$ and the mappings $b_n$ and $g_n$ will be specified below. However, the arguments below also extend to transport noise in Stratonovich formulation (see \cite{AV20_NS} for a similar situation) and variations of the Navier--Stokes equations such as the Boussinesq and magneto-hydrodynamic equations.

Stochastic Navier--Stokes equations on $\R^3$ have been recently investigated, see e.g.\  \cite{DZ20_critical_NS,kim2010strong,KWX24_SPDE}. The results below appear to be the first results in the case of a \emph{non-small} transport noise. Below, we show how some of the arguments in \cite{AV20_NS} to the whole space case. Nevertheless, we also obtain an endpoint Serrin-type blow-up criterion (see Theorem \ref{thm:serrin}\eqref{it2:serrin} below), which is \emph{new} even in the periodic case and complements \cite[Theorem 2.9]{AV20_NS} (the proof in the periodic case is analogous).  Moreover, we prove that the solution instantaneously regularizes in time and space as we previously did in the periodic case. However, the argument to prove space regularity on $\R^d$ differs from the one in our above-mentioned paper, since the latter used the boundedness of the underlying domain.
For completeness, let us mention that the `small data implies global' result of \cite[Theorem 2.11]{AV20_NS} also extends to the current situation. For brevity, we do not state this here.

Next, we collect the assumptions on coefficients $b=(b_n)_{n\geq 1}$ and the nonlinearity $g=(g_n)_{n\geq 1}$.

\begin{assumption}
\label{ass:NSD_LqLp}
We say that Assumption \ref{ass:NSD_LqLp}$(p,\s,q,\varepsilon)$ holds if $\s\in (-1,0]$, $\varepsilon\geq 0$, and either $\big[p\in (2,\infty),q\in [2,\infty) \big]$ or $\big[p=q=2\big]$, and the following are satisfied:
\begin{enumerate}[{\rm(1)}]
\item\label{it1:NSD_LqLp} There exists an $\alpha>-\delta$ such that $b\in C^{\alpha}(\R^d;\ell^2)$.
\item\label{it2:NSD_LqLp} $b$ satisfies the parabolicity condition, i.e., there exists $\nu_0\in (0,\nu)$ such that
$$
\textstyle \sum_{n\geq 1} (b_n\cdot\xi)^2 \leq 2\nu_0|\xi|^2\text{ for all }\xi\in\R^d.
$$
\item\label{it21:NSD_LqLp} $b$ is nearly constant at infinity, i.e., there exists $b_\infty\in \ell^2$ such that
$$
\textstyle \lim_{|x|\to \infty} \|b(x)-b_{\infty}\|_{\ell^2}\leq \varepsilon.
$$
\item\label{it3:NSD_LqLp} The mapping $g$ decomposes as $g=\sum_{i=1}^{\ell }g_i $ where $g_i:\R^d\times \R^d\to \ell^2$ satisfy $g_i(0,\cdot) \in( L^1\cap L^\infty)(\R^d;\ell^2)$ for all $i\in \{1,\dots,\ell\}$. Moreover, for all $i\in \{1,\dots,\ell\}$ there exists $\eta_i\in [0,1]$ such that and for all $x\in\R^d$ and $y,y'\in \R^d$,
$$
\|g_i(x,y)-g_i(x,y')\|_{\ell^2}\lesssim (|y|^{\eta_i}+|y'|^{\eta_i})|y-y'|.
$$
\end{enumerate}
\end{assumption}

Note that $\eta_i=1$ is allowed, and therefore $g_i$ can grow quadratically. In particular, scaling-invariant nonlinearities in the diffusion part are allowed, see \cite[Subsection 1.1]{AV20_NS} and Subsection \ref{ss:scaling_intro}. As we will see later in the proof, the decomposition of the diffusion $g$ is used in checking Assumption \ref{ass:FGcritical}. It is unclear to us whether this can be removed.

A prototype choice of the transport noise coefficients $b$ is the Kraichnan model on $\R^d$, see e.g.\ \cite[Section 2 and Appendix C]{GL23_well_posedness}. Note that the regularity assumption \eqref{it1:NSD_LqLp} holds for sufficiently high correlation, see \cite[Proposition 2.6 and Remark 2.7]{GL23_well_posedness}.
On the other hand, the condition \eqref{it21:NSD_LqLp} is artificial and does not hold for the Kraichnan noise.
However, one can choose a finite collection of divergence-free vector fields $b_n\in \S(\R^d;\R^d)$ for which the corresponding vector $b=(b_n)_{n\in \{1,\dots,N\}}$ is an approximation of the Kraichnan noise. The parameter $\varepsilon$ allows us to cover also the case of small transport noise which is sometimes employed, cf.\ \cite{BS21_primitive,kim2010strong,LLW24_small_transport}.
For later use, we gather some comments on the Assumption \ref{ass:NSD_LqLp}\eqref{it21:NSD_LqLp} in the following.

\begin{remark}[On Assumption \ref{ass:NSD_LqLp}\eqref{it21:NSD_LqLp}]
Condition \eqref{it21:NSD_LqLp} is only used to obtain stochastic maximal $L^p$-regularity for the couple $(A,B)$ in \eqref{eq:choice_ABFG_NS3D} below.
At the moment, it can be avoided in some cases. Indeed, it follows from \cite{M02} that \eqref{it21:NSD_LqLp} can be removed if $p=q$ (under possibly additional assumptions on $b$). Note that the results in \cite{M02} readily extend to the time-weighted case by employing Theorem \ref{thm:extrapol} at the beginning of his proof. Hence, the reader can check that some of the results below are also valid without \eqref{it21:NSD_LqLp} (for instance, the well-posedness in the critical space $L^d$, see Theorem \ref{t:NS_Rd_local} and \eqref{eq:Ld_well_posedness}). We choose not to state such results, as the unnatural restriction $q=p$ gives limitations on the application of the nonlinear theory developed in Subsections \ref{sec:loc-well-posed} and \ref{subsec:blow_up}, which we believe can be removed by a further investigation of the linear theory.
\end{remark}

Before stating the main results of this section, we define suitable function spaces in which we study \eqref{eq:Navier_StokesRd}. Let $\pr$ be the Helmholtz projection on $\S'(\R^d;\R^d)$:
$$
\textstyle (\widehat{\pr f})^k(\xi)=  \widehat{f^k}(\xi)-\sum_{j=1}^d \frac{ \xi^j\xi^k}{|\xi|^2} \widehat{f^j}(\xi),
$$
where $k\in \{1,\dots,d\}$, $f=(f^k)_{k=1}^d\in \S'(\R^d;\R^d)$ and $\widehat{\cdot}$ denotes the Fourier transform on $\R^d$.
By standard Fourier methods, it is clear that the $\pr$ restricts to a bounded linear operator on Bessel-potential and Besov spaces in the reflexive range. Therefore, we can define:
\begin{equation}
\Hs^{s,q}(\R^d):= \pr(H^{s,q}(\R^d;\R^d))
\quad \text{ and }\quad
\Bs^{s}_{q,p}(\R^d):= \pr(B_{q,p}^{s}(\R^d;\R^d))
\end{equation}
for all $s\in \R^d$ and $1<q,p<\infty$. By applying the Helmholtz projection to the first line in \eqref{eq:Navier_StokesRd}, one can readily check that \eqref{eq:Navier_StokesRd} is (formally) in the form of \eqref{eq:SEE} with the choice:
\begin{equation}
\begin{aligned}
\label{eq:choice_ABFG_NS3D}
A u &=-\nu \Delta u , &  \qquad B u& = (\pr [(b_n\cdot\nabla)u])_{n\geq 1},\\
 F(u)&= \pr[\div (u\otimes u)],
 & \qquad  G(u) &= (\pr [g_n (\cdot,u)])_{n\geq 1},
\end{aligned}
\end{equation}
where we used the conservative form of the Navier--Stokes nonlinearity $\div (u\otimes u)=(u\cdot\nabla)u$, as $\div\, u=0$. The latter choice allows one to accommodate for weaker settings in the spatial variables.

Next, we introduce the solution concepts for \eqref{eq:Navier_StokesRd}. Let Assumption \ref{ass:NSD_LqLp}$(p,\s,q,\varepsilon)$ be satisfied and let $\a\in [0,\frac{p}{2}-1)\cup\{0\}$. Then we say that $(u,\sigma)$ is a (unique) \emph{$(p,\s,\a,q)$-solution} to \eqref{eq:Navier_StokesRd} if $(u,\sigma)$ is a $L^p_\a$-maximal solution to \eqref{eq:SEE} with the above choice of the linear operators $(A,B)$ and the nonlinear mappings $(F,G)$ as in \eqref{eq:choice_ABFG_NS3D},
$\mathcal{U}=\ell^2$ and $X_j=\Hs^{-1+\s+2j,q}(\R^d)$ with $j\in \{0,1\}$. Note that compared to Subsections \ref{ss:AllenCahnLpLq} and \ref{ss:reaction}, we replaced $\s$ by $-\s$ in our definition of the spaces $X_0$ and $X_1$, to be coherent with the notation in \cite{AV20_NS}.

\begin{theorem}[Local well-posedness and regularization in critical spaces]
\label{t:NS_Rd_local}
Let Assumption \ref{ass:NSD_LqLp}$(p,\s,q,\varepsilon)$ be satisfied. Assume that $(p,\s,q)$ satisfies one of the following conditions:
\begin{itemize}
\item $\delta\in [-\frac{1}{2},0]$, $\frac{d}{2+\s} <q<\frac{d}{1+\s}$, and $\frac{2}{p}+\frac{d}{q}\leq 2+\delta$;
\item $\delta=0$, and $p =q =d =2$.
\end{itemize}
Then, there exists $\varepsilon>0$ for which the following assertions hold. Letting $\a:=-1+\frac{p}{2}( 2+\delta-\frac{d}{q})$, for all $u_0\in L^0_{\F_0}(\O;\Bs_{q,p}^{d/q-1}(\R^d))$ there exists a (unique) $(p,\s,\a,q)$-solution $(u,\sigma)$ to \eqref{eq:Navier_StokesRd} satisfying $\sigma>0$ a.s.\ and
\begin{align}
\label{eq:NSD_local1}
u&\in H^{\theta,p}_{\loc}([0,\sigma),w_{\a};\Hs^{1+\s-2\theta,q}) \text{ a.s. for all } \theta\in [0,1/2),\\
\label{eq:NSD_local2}
u&\in C([0,\sigma);\Bs^{d/q-1}_{q,p})\cap C([0,\sigma);\Bs^{1-2/p}_{q,p})\text{ a.s.}
\end{align}
Finally, if $\varepsilon=0$, then $(u,\sigma)$ instantaneously regularize in time and space: a.s.\
\begin{equation}
\label{eq:regularization_in_time}
u\in L^r_{\loc} ((0,\sigma);\Hs^{1,\zeta})\cap C^{\theta-\varepsilon}_{\rm loc}((0,\sigma);\Hs^{1-2\theta,\zeta}), \ \ \theta\in [0,1/2), \varepsilon\in (0,\theta), r<\infty, \zeta\in [q, \infty).
\end{equation}
\end{theorem}

Recall from Subsection \ref{ss:scaling_intro} that the Besov space $B^{d/q-1}_{q,p}(\R^d;\R^d)$ is scaling-invariant (or critical) for \eqref{eq:Navier_StokesRd}.
In case one can take $\delta=-\frac{1}{2}$ (the latter depending on the smoothness of $b$, see Assumption \ref{ass:NSD_LqLp}\eqref{it1:NSD_LqLp}), in Theorem \ref{t:NS_Rd_local} one can choose $q\sim \frac{d}{1+\s}\sim 2d$ and therefore the above establishes local well-posedness in critical spaces of smoothness $>-\frac{1}{2}$. It is unclear whether the previous threshold is optimal. For further discussion, the reader is referred to Problem \ref{prob:largest_invariant_space}.
Before going further, let us also note that the invariant space $L^d(\R^d;\R^d)$ is always included in Theorem \ref{t:NS_Rd_local}. Indeed, the case $d=2$ is clear while for $d\geq 3$ it is enough to apply Theorem \ref{t:NS_Rd_local} with $\delta\in (-[\alpha\wedge \frac{1}{2}],0)$, $q=d$ and $p\geq d$ and note that
\begin{equation}
\label{eq:Ld_well_posedness}
\Ls^d(\R^d)\subseteq \Bs^{0}_{d,p}(\R^d).
\end{equation}
The instantaneous regularization result of \eqref{eq:regularization_in_time} is new even in the well-known case $\s=0$ and $q=p=d=2$. The main difficulty in this case is related to the criticality of the energy space $L^2$ in two dimensions (the reader is referred to the comments below \cite[Theorem 2.4]{AV20_NS} for details).
As the proofs below show, the assumption $\varepsilon=0$ can be removed for \eqref{eq:regularization_in_time} to hold for a fixed but large $r$, however, $\varepsilon$ would also depend on such parameter $r$. Due to the unboundedness of the domain, the argument used in \cite[Theorem 2.4]{AV20_NS} to prove instantaneous regularization does not extend to the current situation. Instead, here we use the one employed in Theorem \ref{thm:quasigeoglobal}.
With the aid of the latter argument, the reader can check that higher-order regularity results as in \cite[Theorem 2.7]{AV20_NS} can also be obtained. Another issue related to the unboundedness of the domain is that in \eqref{eq:regularization_in_time} we need $\zeta\in [q, \infty)$.

\smallskip

The following shows that if $\sigma<\infty$ on a set of positive probability, then critical norms have to blow up a.s.\ on the same set.
\begin{theorem}[Stochastic Serrin-type criteria]
\label{thm:serrin}
Let the assumptions of Theorem \ref{t:NS_Rd_local} be satisfied with $\varepsilon=0$. Assume that $(p_0,\s_0)$ and $q_0\in [q, \infty)$ satisfy one of the following conditions:
\begin{itemize}
\item $\delta_0\in [-\frac{1}{2},0]$, $\frac{d}{2+\s_0} <q_0<\frac{d}{1+\s_0}$, and $\frac{2}{p_0}+\frac{d}{q_0}\leq 2+\delta_0$;
\item $\delta_0=0$, and $p_0 =q_0 =d =2$.
\end{itemize}
Then, for all $s>0$, the following blow-up criteria hold:
\begin{enumerate}[{\rm(1)}]
\item\label{it1:serrin} $\displaystyle{\P\big(s<\sigma<\infty,\, \|u\|_{L^{p}(s,\sigma;H^{\g_0,q_0})} <\infty\big)=0}$ where $\displaystyle{\g_0:=\frac{2}{p_0}+\frac{d}{q_0}-1}$ {\rm(Serrin-type criterion)};
\vspace{0.1cm}
\item\label{it2:serrin} $\displaystyle{\P\big(\sigma<\infty,\, \lim_{t\uparrow \sigma}u(t)\text{ exists in }B^{d/q_0-1}_{q_0,p_0}\big)=0}$
{\rm(endpoint Serrin-type criterion)}.
\end{enumerate}
\end{theorem}

As the proof below shows, if $\varepsilon>0$ is small, then \eqref{it1:serrin} and \eqref{it2:serrin} still hold with $(q_0,p_0,\s_0)=(q,p,\s)$.

The criteria \eqref{it1:serrin}-\eqref{it2:serrin} are \emph{sharp} as both the space $L^{p_0}(0,T;H^{\g_0,q_0})$ and $C([0,T];B_{q_0,p_0}^{d/q_0-1})$ respect the scaling of the Navier--Stokes equations. Indeed, their space-time Sobolev indexes coincide with $-1$ which is the Sobolev index of critical spaces $B^{d/q_0-1}_{q_0,p_0}$ as
\begin{equation}
\label{eq:criticality_indexes_NS3D}
\frac{2}{p_0}+\g_0-\frac{2}{q_0}=-1 \qquad \text{ and }\qquad \frac{2}{\infty}+\frac{d}{q_0}-1 -\frac{d}{q_0}=-1.
\end{equation}
Here, the factor `2' for the time integrability comes from the parabolic scaling.
The reader is referred to the comments below Theorem \ref{thm:subcriticalblowup} for comments on Sobolev indexes.
Let us stress that the endpoint criteria \eqref{it2:serrin} also holds in the periodic case analyzed in \cite{AV20_NS} and follows immediately from the results in the latter references and Theorem \ref{thm:criticalblowup}\eqref{it1:criticalblowup}.

Concerning the terminology, the criterion \eqref{it2:serrin} is a natural extension to the stochastic setting of Serrin-type criteria, see e.g.\ \cite[Theorem 11.2]{LePi}. Indeed, in the special case where $q_0>d$ and $p_0= 2/(1-\frac{d}{q_0})$, we have $\g_0=0$ and therefore \eqref{it1:serrin} coincides with
$$
\P\big(s<\sigma<\infty,\, \|u\|_{L^{p_0}(s,\sigma;L^{q_0})} <\infty\big)=0,
$$
which is a stochastic version of \cite[Theorem 11.2]{LePi}.
The blow-up criterion \eqref{it1:serrin} is therefore a ``endpoint'' version of \eqref{it2:serrin}, i.e., taking $p=\infty$ in the latter.
In the deterministic setting, there are available more refined versions of the endpoint Serrin-type criterion.
An extension of these results is the content of Problem \ref{prob:NS_blow_up}, to which the reader is referred for more details.

\smallskip

In Theorem \ref{thm:SNS} we have seen that for $p=q=d=2$, $\delta=0$, \eqref{eq:Navier_StokesRd}
has a unique global solution $u\in L^2_{\rm loc}([0,\infty);\Hs^1)\cap C([0,\infty);\Ls^2)$ a.s.
Combining this with Theorem \ref{thm:serrin}, we obtain higher-order regularity of this global solution.

\begin{corollary}[2D Global well-posedness and regularity]
\label{cor:2D_NS_global}
Let the assumptions of Theorem \ref{t:NS_Rd_local} be satisfied with $p=q=d=2$, $\varepsilon=0$, and $\eta_i\equiv 0$.  Then the global $(2,0,0,2)$-solution $u$ satisfies
\[u\in L^r_{\loc} ((0,\infty);\Hs^{1,\zeta})\cap C^{\theta-\varepsilon}_{\rm loc}((0,\infty);\Hs^{1-2\theta,\zeta}) \ \text{a.s.}, \ \ \theta\in [0,1/2), \varepsilon\in (0,\theta), r<\infty, \zeta\in [2, \infty).
\]
\end{corollary}

\subsubsection{Proof of Theorems \ref{t:NS_Rd_local} and \ref{thm:serrin}}
To prove the stated results, we employ the abstract results of Sections \ref{sec:loc-well-posed} and \ref{sec:blowup}.
The following is the $\R^d$-analogue of \cite[Lemma 4.2]{AV20_NS}, which serves to check Assumption \ref{ass:FGcritical} for the nonlinearities. Below, we employ the shorthand notation $X_j:= \Hs^{-1+\s+2j,q}:=\Hs^{-1+\s+2j,q}(\R^d)$ for $j\in \{0,1\}$ and
$$
X_{\theta}:=[X_0,X_1]_{\theta}=\Hs^{-1+\s+2\theta,q} \quad\text{ and }\quad X_{\theta,p}:=(X_0,X_1)_{\theta,p}=\Bs^{-1+\s+2\theta}_{q,p},
$$
 for $ \theta\in (0,1)$ and $1<q,p<\infty$.

\begin{lemma}
\label{lem:estimate_nonlinear_NS3D}
Assume that Assumption \ref{ass:NSD_LqLp}$(p,\a,q,\varepsilon)$ holds and
$
\frac{d}{2+\s}<q<\frac{d}{-\s}.
$
Let $F$ and $G$ be as in \eqref{eq:choice_ABFG_NS3D}.
Then $\beta:= \frac{1}{2}(1-\frac{\s}{2}+\frac{d}{2q})\in (0,1)$ and there exists a $C>0$ such that for all $v,v'\in X_1$,
\begin{align*}
\|F(v)-F(v')\|_{X_0} + \|G(v)-G(v')\|_{\g(\ell^2;X_{1/2})} &\leq C(\|v\|_{X_\beta}+\|v'\|_{X_\beta})\|v-v'\|_{X_{\beta}}.
\end{align*}
\end{lemma}

\begin{proof}
We start by showing the estimate for the $F$-part. As $F(0)=0$, it is enough to prove the local Lipschitz estimate. Now, note that for all $v,v'\in X_1$,
\begin{align}
\label{eq:estimate_F_NS3D}
\|F(v)-F(v')\|_{X_0}
&\lesssim \|v\otimes v - v'\otimes v'\|_{H^{\s,q}}\\
\nonumber
&\stackrel{(i)}{\lesssim} \|v\otimes v - v'\otimes v'\|_{L^\lambda}\\
\nonumber
&\lesssim  (\|v\|_{L^{2\lambda}}+\|v'\|_{L^{2\lambda}})\|v-v'\|_{L^{2\lambda}}\\
\nonumber
&\stackrel{(ii)}{\lesssim}  (\|v\|_{H^{\theta,q}}+\|v'\|_{H^{\theta,q}})\|v-v'\|_{H^{\theta,q}}\\
\nonumber
&\eqsim  (\|v\|_{X_\beta}+\|v'\|_{X_\beta})\|v-v'\|_{X_\beta},
\end{align}
where $\lambda:=\frac{dq}{d-\s q}$, $\theta:=\frac{d}{q}-\frac{d}{2\lambda}$, and
 in $(i)$ and $(ii)$ we used the Sobolev embedding $L^\lambda \embed H^{\s,q}$ and $H^{\theta,q}\embed L^{2\lambda}$.
Note that $q<\frac{d}{-\s}$ implies that $\lambda\in (\frac{q}{2},q]$ and therefore $\theta>0$. Moreover, $q>\frac{d}{2+\s}$ implies $\theta<1+\s$ and $\beta<1$.

Next, we discuss the $G$-parts. By triangular inequality, it is enough to show the corresponding estimate for a fixed $g_i$ with $i\in \{1,\dots,\ell\}$. Moreover, we only discuss the local Lipschitz estimate, as the growth assumption follows from $g_i(\cdot,0)\in (L^1\cap L^\infty)(\R^d;\R^d)$ and the argument below, see Assumption \ref{ass:NSD_LqLp}\eqref{it3:NSD_LqLp}. Thus, for all $v,v'\in X_1$,  we have
\begin{align*}
\|G_i(v)-G_i(v')\|_{\g(\ell^2,H^{\s,q})}
&\ \leq
\|\pr [g_i(\cdot,v)-g_i(\cdot,v')]\|_{\g(\ell^2,H^{\s,q})} \\
&\stackrel{(iii)}{\lesssim}
\|g_i(\cdot,v)-g_i(\cdot,v')\|_{\g(\ell^2,L^{\lambda_i})}\\
&\ \eqsim \|g_i(\cdot,v)-g_i(\cdot,v')\|_{L^{\lambda_i}(\ell^2)} \\
&\ \lesssim \big\|(|v|^{\eta_i}+|v'|^{\eta_i})|v-v'|\big\|_{L^{\lambda_i}}\\
&\ \lesssim (\|v\|_{L^{(\eta_i+1)\lambda_i}}^{\eta_i}+\|v'\|_{L^{(\eta_i+1)\lambda_i}}^{\eta_i})\|v-v'\|_{L^{(\eta_i+1)\lambda_i}},
\end{align*}
where in $(iii)$ we used $L^{\lambda_i}\embed H^{\s,q}$ where $\lambda_i\in [\lambda,q]$ satisfies $\lambda_i(1+\eta_i)\geq q$ and $\lambda$ is as below \eqref{eq:estimate_F_NS3D}.
Note that the choice of $\lambda_i$ is always possible as $\eta_i\in [0,1]$.
Now, since $\lambda_i\leq \lambda$ and $\eta_i\leq 1$, the claimed estimate for $G_i$ follows from the Sobolev embeddings $H^{\theta,q}\embed L^{\lambda_i(1+\eta_i)}$ as in \eqref{eq:estimate_F_NS3D}.
\end{proof}

With this preparation, we are now ready to prove Theorems \ref{t:NS_Rd_local} and \ref{thm:serrin}.

\begin{proof}[Proof of Theorem \ref{t:NS_Rd_local}]
For the reader's convenience, we divide the proof into three steps.

\emph{Step 1: The existence of a (unique) $(p,\a,\s,q)$-solution $(u,\sigma)$ to \eqref{eq:Navier_StokesRd}, where $\a=-1+\frac{p}{2}( 2+\delta-\frac{d}{q})$}. To prove Step 1, we employ Theorem \ref{thm:localwellposed}. Let us first note that the stochastic $L^p$-maximal regularity for $(A,B)$ follows from \cite[Remark 3.3]{AV20_NS} with $\varepsilon=0$, and in case $\varepsilon>0$, then it follows from the $\varepsilon=0$ case and the perturbation result of \cite[Theorem 3.2]{AV21_SMR_torus}. Thus, it remains to check Assumption \ref{ass:FGcritical}. From Lemma \ref{lem:estimate_nonlinear_NS3D}, it follows that \eqref{eq:subcritical} is satisfied with $m=1$, $\rho_1=1$ and $\beta_1=\frac{1}{2}(1-\frac{\s}{2}+\frac{d}{2q})$ provided $\a$ and $p$ satisfy
\begin{equation}
\label{eq:criticality_condition_NS_nonlinear}
\frac{1+\a}{p}\leq \frac{(1+\rho_1)}{\rho_1}(1-\beta_1)= 1-\frac{d}{2q}+\frac{\s}{2}.
\end{equation}
Note that the application of Lemma \ref{lem:estimate_nonlinear_NS3D} is legitimate, as for $\s\geq -\frac{1}{2}$ we have $\frac{d}{1+\s}\leq \frac{d}{-\s}$. It is clear that if $p=2$ and $\a=0$, then the above can only hold if $d=q=2$ and $\s=0$.
Next, we focus on the case $p>2$. Note that, due to the choice $\a=-1+\frac{p}{2}( 2+\delta-\frac{d}{q})$, our assumptions immediately give $\a\in [0,\frac{p}{2}-1)$ and that the condition  \eqref{eq:criticality_condition_NS_nonlinear} is satisfied. Moreover, let us point out that the restrictions $\delta\geq -\frac{1}{2}$  and $q<\frac{d}{1+\s}$ come from enforcing an equality in \eqref{eq:criticality_condition_NS_nonlinear}. Indeed, to allow the equality \eqref{eq:criticality_condition_NS_nonlinear}, as $\frac{1+\a}{p}<\frac{1}{2}$ and $p>2$, it is necessary that $1-\frac{d}{2q}+\frac{\s}{2}<\frac{1}{2}$ which implies $q<\frac{d}{1+\s}$.
Thus, the restriction $\s\geq -\frac{1}{2}$ is a consequence of the restriction $q<\frac{d}{-\s}$ in Lemma \ref{lem:estimate_nonlinear_NS3D}: $\frac{d}{1+\s}\leq \frac{d}{-\s}$ if and only if $\s\geq -\frac{1}{2}$.

Therefore, the assumptions of Theorem \ref{thm:localwellposed} are satisfied and therefore it ensures the existence of a (unique) $(p,\a,\s,q)$-solution to \eqref{eq:Navier_StokesRd} with space of initial data given by
\begin{equation}
\label{eq:trace_equal_critical_NS}
X_{1-\frac{1+\a}{p},p}= \Bs^{1+\s-2\frac{1+\a}{p}}_{q,p} =\Bs^{\frac{d}{q}-1}_{q,p};
\end{equation}
where the last equality follows from the choice of the weight $\a=-1+\frac{p}{2}( 2+\delta-\frac{d}{q})$.

\emph{Step 2: (Instantaneous time regularization) If $\varepsilon=0$, then }
\begin{equation}
\label{eq:regularization_in_time_step_2}
u\in H^{\theta,r}_{\loc} ((0,\sigma);\Hs^{1+\s-2\theta,r}) \text{ a.s.\ for all }\theta\in [0,1/2),\ r<\infty.
\end{equation}
Arguing as in Step 1, if $\varepsilon=0$, then the stochastic maximal $L^p$-regularity for $(A,B)$ follows from \cite[Remark 3.3]{AV20_NS}.
To prove \eqref{eq:regularization_in_time}, we distinguish three cases.
\begin{itemize}
\item If $p>2$ and $ \frac{2}{p}+\frac{d}{q}\leq  2+\delta$, then $\a>0$ and \eqref{eq:regularization_in_time_step_2} follows from Theorem \ref{thm:parabreg}.
\item If $p>2$ and $ \frac{2}{p}+\frac{d}{q}=  2+\delta$, then $\a=0$ and \eqref{eq:regularization_in_time_step_2} follows from Theorem \ref{thm:parabreg2} where we used that the constant $C$ in Lemma \ref{lem:estimate_nonlinear_NS3D} is independent of $v,v'$.
\item If $p=q=2$, then $d=2$, $\s=0$ and $\a=0$. In this case, \eqref{eq:regularization_in_time_step_2} follows
almost verbatim from the arguments in \cite[Part (C) of Theorem 4.1, p.\ 43-44]{AV20_NS} by using \cite[Proposition 6.8]{AV19_QSEE_2} and Lemma \ref{lem:estimate_nonlinear_NS3D}.
\end{itemize}

\emph{Step 3:  If $\varepsilon=0$, then \eqref{eq:regularization_in_time} holds}. From Step 2, we know that the solution instantaneously lies in a subcritical space. Hence, one can apply the classical bootstrap argument (via Proposition \ref{prop:localizationSMR}) to prove space regularity as we did in the proof of Theorem \ref{thm:quasigeoglobal}. For brevity, we omit the details.
\end{proof}

\begin{proof}[Proof of Theorem \ref{thm:serrin}]
By Corollary \ref{cor:transfblowup} and the proof of Theorem \ref{t:NS_Rd_local}, it is enough to show the claim of Theorem \ref{thm:serrin} with $(p_0,\s_0,q_0)=(p,\s,q)$. In the latter situation,
\eqref{it2:serrin} follows from Theorem \ref{thm:criticalblowup}\eqref{it2:criticalblowup} and \eqref{eq:trace_equal_critical_NS}. While for \eqref{it1:serrin}, note that
$
\frac{\a}{p}
= 1+\frac{\delta}{2}
-\frac{1}{p} -\frac{d}{2q}
$
and
$$
X_{1-\frac{\a}{p}}
=\Hs^{1+\s-2\frac{\a}{p},q}
=\Hs^{\frac{2}{p}-\frac{d}{q}-1,q}.
$$
Therefore, \eqref{it1:serrin} with $(p_0,\s_0,q_0)=(p,\s,q)$ is a consequence of \cite[Theorem 4.11]{AV19_QSEE_2} and the fact that the constant $C$ in Lemma \ref{lem:estimate_nonlinear_NS3D} is independent of $v,v'$.
\end{proof}

\appendix

\section{Technical results}

\subsection{Extrapolation spaces}\label{appendix:extra}

In this appendix, we present results on interpolation-extrapolation scales for sectorial operators, which we use at several places. For a more detailed presentation, we refer to \cite[Chapter 5]{Am} and \cite[Section 6.3]{Haase:2}.

Let $\Aop$ be a sectorial operator on a Banach space $X$ such that $0\in \rho(\Aop)$. The latter implies that $(X,\|\Aop^{-1}\cdot\|_{X})$ is a normed space. Define the extrapolated space $X_{-1,\Aop}$ as the completion of $(X,\|\Aop^{-1}\cdot\|_{X})$, i.e.\
\begin{equation*}
X_{-1,\Aop}:=(X,\|\Aop^{-1}\cdot\|_{X})^{\sim},
\end{equation*}
where $\sim$ denotes the completion of the space.
Clearly, $X\hookrightarrow X_{-1,\Aop}$ and if $x\in X$, then we have
$$
\|x\|_{X_{-1,\Aop}}=\|\Aop^{-1}x\|_{X}\leq \|\Aop^{-1}\|_{\calL(X)}\|x\|_{X}.
$$
Since $\overline{\Do(\Aop)}=X$, this also shows that $\Do(\Aop)\ni x\mapsto \Aop x\in X$ extends to a linear isometric isomorphism from $X$ onto $X_{-1,\Aop}$. The extension of this map will be denoted by $T_{-1,\Aop}$ or simply $T_{-1}$ if no confusion seems likely.

To proceed further, let us note that $\Aop$ induces a closed linear operator $\Aop_{-1}$ on $X_{-1,\Aop}$ given by
\begin{equation}
\label{eq:A_first_extrapolation}
\Aop_{-1}:=T_{-1}\Aop T_{-1}^{-1}.
\end{equation}
One can check that $\Aop_{-1}|_{X}=\Aop$. By \eqref{eq:A_first_extrapolation}, $\Aop$ is similar to $\Aop_{-1}$. These simple observations lead to the following.
\begin{proposition}
\label{prop:A_first_extrapolation}
Let $\Aop$ be a sectorial operator on $X$ such that  $0\in \rho(\Aop)$. Then $\Aop_{-1}$ is the closure of $\Aop$ in $X_{-1,\Aop}$ with $\Do(\Aop_{-1})=X$.
Moreover, the following hold:
\begin{enumerate}[{\rm(1)}]
\item $\Aop_{-1}$ is a sectorial operator on $X_{-1,\Aop}$ and $\om(\Aop_{-1})=\om(\Aop)$;
\item\label{it:A_extrapolated_BIP} If $\Aop\in \BIP(X)$, then $\Aop_{-1}\in \BIP(X_{-1,\Aop })$;
\item If $\Aop$ has a bounded $H^{\infty}$-calculus on $X$, so does $\Aop_{-1}$ and $\angH(\Aop_{-1})=\angH(\Aop)$.
\end{enumerate}
\end{proposition}

Recall that $\BIP$ was introduced in Subsection \ref{ss:Hinfty} (see also \cite{Analysis3,Haase:2}).

The previous proposition shows that if $\Aop_{-1}$ is sectorial, then the fractional powers $(\Aop_{-1})^{\alpha}$ for $\alpha>0$ are well-defined closed linear operators on $X_{-1,\Aop}$. Let us denote by $X_{\alpha-1,\Aop}$ the domain of $(\Aop_{-1})^{\alpha}$,
\begin{equation}
\label{eq:complete_scale}
X_{-1+\alpha,\Aop}:=\big(\Do((\Aop_{-1})^{\alpha}),\|(\Aop_{-1})^{\alpha}\cdot\|_{X}\big),\qquad \alpha\geq 0.
\end{equation}
By Proposition \ref{prop:A_first_extrapolation} one has $\Do(\Aop_{-1})=X$ and thus $X_{0,\Aop}=X$.

Let $\alpha\geq -1$ and let $\Aop_{\alpha}$ be the realization of $\Aop_{-1}$ on $X_{\alpha,\Aop}$, i.e.\
\begin{align*}
\Do(\Aop_{\alpha})&:=\{x\in X_{\alpha,\Aop}\,:\,\Aop_{-1}x \in X_{\alpha,\Aop} \},\\
\Aop_{\alpha}x&:=\Aop_{-1}x,\text{ if }x\in \Do(\Aop_{\alpha}).
\end{align*}
Note that $\Aop_{0}=\Aop$ and $\Aop_{\alpha}=\Aop_{-1}$ if $\alpha=-1$. Under suitable assumptions, $(X_{\alpha,\Aop})_{\alpha\geq -1}$ becomes an interpolation scale with respect to complex interpolation (see \cite[Theorem 6.6.9]{Haase:2}).
\begin{proposition}
\label{prop:interpolation_extrapolation_scale}
Let $\Aop\in \BIP(X)$ be such that $0\in \rho(\Aop)$. Let $(X_{\alpha,\Aop})_{\alpha\geq -1}$ be as above. Then
\begin{equation*}
X_{\alpha(1-\theta)+\beta\theta,\Aop}=[X_{\alpha,\Aop},X_{\beta,\Aop}]_{\theta}, \quad \alpha,\beta\geq -1,\;\theta\in (0,1).
\end{equation*}
isomorphically.
\end{proposition}

In the above setting $(X_{\alpha,\Aop},\Aop_{\alpha})_{\alpha\geq -1}$ is called the interpolated-extrapolated scale of $\Aop$.

The following alternative characterization of the spaces $X_{-\alpha,\Aop}$ for $\alpha\in (0,1)$ is given in \cite[Section 5.1.4]{Am}, where $A^*$ denotes the adjoint of $A$.
\begin{theorem}
\label{t:duality}
Let $\Aop$ be a sectorial operator on a reflexive Banach space $X$ such that $0\in \rho(\Aop)$. Then for each $\vartheta\in [0,1]$, $X_{-\vartheta,\Aop}$ is isomorphic to the dual of the space
$$X_{\vartheta,\Aop^*}=(\Do((\Aop^*)^{\vartheta}),\|(\Aop^*)^{\vartheta}\cdot\|_{X^*})
$$ with the duality pairing induced by the $(X,X^*)$-pairing. More concisely,
$
X_{-\vartheta,\Aop}=(X_{\vartheta,\Aop^*})^*.
$
\end{theorem}

In the following examples, we look at operators which will be used later.
\begin{example}[Dirichlet Laplacian]
\label{ex:extrapolated_Laplace_dirichlet}
In this example we specialize the above construction to the strong Dirichlet Laplacian $\Dd_{q}$ where $q\in (1,\infty)$. In this example, we assume that $\Dom\subseteq \R^d$ is a bounded $C^2$-domain. However, many of the results below also hold on domains with less smoothness and unbounded domains. The results for $\Dom = \R^d$ or $\Dom = \T^d$ are much simpler and can be presented for all smoothness parameters. Let $W^{1,q}_0(\Dom):=\{u\in W^{1,q}(\Dom)\,:\,u|_{\partial \Dom}=0\}$. The strong Dirichlet Laplacian is defined as
\begin{equation}
\label{eq:strong_Dirichlet_Laplacian}
\Dd_q:W^{2,q}(\Dom)\cap W^{1,q}_0(\Dom)\subseteq L^q(\Dom) \to L^q(\Dom),\qquad \Dd_q f:=\Delta f.
\end{equation}
By \cite{DDHPV}, $A_q:=-\Dd_q$ has a bounded $H^{\infty}$-calculus on $L^q(\Dom)$ with angle $\angH(A_q)=0$. Thus, $A_q$ generates an extrapolated-interpolated scale, which will be denoted by
$$\big(\HD^{2\alpha,q}(\Dom),A_{2\alpha,q}\big)_{\alpha\in [-1,\infty)}.$$
Therefore, $\HD^{2,q}(\Dom)=W^{2,q}(\Dom)\cap W^{1,q}_0(\Dom)$ and $\HD^{0,q}(\Dom)=L^q(\Dom)$.
By Proposition \ref{prop:interpolation_extrapolation_scale}, for all $ -2\leq s_1<s_2<\infty$ one has
\begin{equation}
\label{eq:HD_complex_interpolation}
\HD^{s,q}(\Dom)=[\HD^{s_1,q}(\Dom),\HD^{s_2,q}(\Dom)]_{\theta},\qquad \vartheta\in (0,1),s:=(1-\theta) s_1+\theta s_2.
\end{equation}
Moreover, by Theorem \ref{t:duality}, $\HD^{-s,q}(\Dom)=(\HD^{s,q'}(\Dom))^*$ for $s\in(0,2)$.
We define the extrapolated Dirichlet Laplacian as
$\Dd_{s,q}:=- A_{s,q}$ with $ s\geq -2$. Note that $\Dd_{0,q}=\Dd_{q}$.

By \cite{Se} and \eqref{eq:HD_complex_interpolation} one has the following identification:
\begin{equation}
\label{eq:HD_identification}
\HD^{s,q}(\Dom)=
\begin{cases}
H^{s,q}(\Dom)\qquad &\text{ if }s\in (0,1/q),\\
\{H^{s,q}(\Dom)\,:\,u|_{\partial\Dom}=0\}\qquad &\text{ if }s\in (1/q,2).
\end{cases}
\end{equation}
Here, $H^{s,q}(\Dom)$ denotes the Bessel potential spaces on domains (see \cite[Section 4.3.1]{Tri95}).  We avoided $s=1/q$, as in this case, the description is more complicated. The identification \eqref{eq:HD_identification} implies $\HD^{1,q}(\Dom)=W^{1,q}_0(\Dom)$, and by the above duality one has
$$
\HD^{-1,q}(\Dom)=(W^{1,q'}_0(\Dom))^*=:W^{-1,q}(\Dom).
$$
The above identities and integration by parts arguments show that $\Dd_{-1,q}$ is the `weak Dirichlet Laplacian', i.e.\ $\Dd_{-1,q}:W^{1,q}_0(\Dom)\subseteq W^{-1,q}(\Dom)\to W^{-1,q}(\Dom)$ and (see \cite[Example 3]{CriticalQuasilinear} for details)
\begin{equation}
\label{eq:weak_dirichlet_laplacian}
\textstyle -\l g,\Dd_{-1,q}f\r=\int_{\Dom}\nabla g\cdot \nabla f\,\dd x,\quad  f\in W^{1,q}_0(\Dom),g\in W^{1,q'}_0(\Dom).
\end{equation}
The same integration by parts arguments allows us to consider the divergence operator $\div:=\sum_{j=1}^d \partial_j$ as a map $\div: L^q(\Dom;\R^d)\to \HD^{-1,q}(\Dom)$ defined by
\begin{equation}
\label{eq:weak_divergence_operator}
\textstyle - \l g,\div F\r:=\int_{\Dom} F\cdot \nabla g\,\dd x, \qquad \forall F\in L^q(\Dom;\R^d),g\in W^{1,q'}_0(\Dom).
\end{equation}

Finally, we identify the real interpolation spaces as Besov spaces. For any $\theta\in (0,1)$ and $q,p\in (1,\infty)$ we define
\begin{equation*}
\BD^{-2+4\theta}_{q,p}(\Dom):=(\HD^{-2,q}(\Dom),\HD^{2,q}(\Dom))_{\theta,p}.
\end{equation*}
By \cite[Theorem 4.7.2]{BeLo} and \eqref{eq:HD_complex_interpolation} for $-2\leq s_0<s_1$ and $\theta\in (0,1)$ one has
\begin{equation}
\label{eq:def_BD}
\BD^{s}_{q,p}(\Dom)=(\HD^{s_0,q}(\Dom),\HD^{s_1,q}(\Dom))_{\theta,p},\qquad \text{where }s:=(1-\theta)s_0+\theta s_1.
\end{equation}
The notation $\BD$ is motivated by the following identification (see \cite{Grisvard}):
\begin{equation}
\label{eq:BD_identifications}
\BD^{s}_{q,p}(\Dom)=
\begin{cases}
B^{s}_{q,p}(\Dom),&\qquad s\in(0,1/q),\\
\{u\in B^{s}_{q,p}(\Dom)\,:\,u|_{\partial\Dom}=0\},&\qquad s\in(1/q,2).
\end{cases}
\end{equation}
Here, $B^{s}_{q,p}(\Dom)$ denotes the usual Besov spaces on domains (see \cite[Section 4.3.1]{Tri95}).
For $s_0,s_1\geq -2$ the following embedding results will be frequently used without further reference:
\begin{align*}
\HD^{s_0,q_0}(\Dom)& \hookrightarrow \HD^{s_1,q_1}(\Dom),  \ \ \ 1<q_0\leq q_1<\infty, s_0-d/q_0\geq s_1-d/q_1,
\\ \BD^{s_0}_{q_0,p_0}(\Dom)&\hookrightarrow \BD^{s_1}_{q_0,p_1}(\Dom), \ \ \ 1<q_0\leq q_1<\infty, s_0-d/q_0\geq s_1-d/q_1, 1\leq p_0\leq p_1\leq \infty,
\\ \BD^{s_0+\varepsilon}_{q,p_0}(\Dom)&\hookrightarrow \BD^{s}_{q,p_1}(\Dom), \ \ \varepsilon>0, q\in (1, \infty), p_0,p_1\in [1, \infty],
\\ \BD^{s_0+\varepsilon}_{q,\infty}(\Dom)&\hookrightarrow \BD^{s}_{q,p_1}(\Dom), \ \ \varepsilon>0, q\in (1, \infty), p_1\in [1, \infty],
\\ \BD^{s_0}_{q,1}(\Dom)&\hookrightarrow \HD^{s_0,q}(\Dom)\hookrightarrow \BD^{s_0}_{q,\infty}(\Dom), \ \ q\in (1, \infty).
\end{align*}
\end{example}

\subsection{It\^o's formula}

The following infinite-dimensional version of It\^o's formula for $\|\cdot\|^2_H$ is completely standard, and can be traced back to the early work on the variational setting (see \cite{KR79, Par2}), where also more general versions can be found. A shorter proof in a special case was obtained in \cite{Kry13}. For an overview, the reader is referred to the recent paper \cite{gyongy2024once}.
\begin{lemma}[It\^o's formula in Hilbert spaces]\label{lem:ItoHilbert}
Let $(V, H, V^*)$  be a Gelfand triple of Hilbert spaces as in Section \ref{sec:Var}.
Let $u_0 \in L^0_{\F_0}(\Omega;H)$. Consider the strongly progressively measurable processes $\Phi\in L^{2}(0,T;V^*)$ a.s., and $\Psi\in L^{2}(0,T;\calL_2(U,H))$ a.s. Let $u \in C([0,T], H) \cap L^2(0,T; V)$ a.s.\ be a progressively measurable process such that, a.s.\ for all $t\in [0,T]$,
\begin{align}
\textstyle u(t) = u_0 + \int_0^t \Phi(s) \,\dd s + \int_0^t \Psi(s) \,\dd  W(s).
\end{align}
Then, a.s.\ for all $t\in [0,T]$,
\begin{align*}
\textstyle\| u(t) \|_H^2 = \| u_0 \|_H^2 + 2 \int_0^t \lb \Phi(s), u(s) \rb \,\dd s + 2 \int_0^t \Psi(s)^* u(s) \,\dd W(s)  + \int_{0}^t |\tr(\Psi(s))|^2 \,\dd s.
\end{align*}
\end{lemma}
The above form can be proved through the orthonormal basis expansion $\|u\|_H^2 = \sum_{n\geq 1} |(u, h_n)|^2$ by reducing to the scalar-valued It\^o formula. Here it helps to take $h_n\in V$.

Next we state a relatively standard It\^o formula for $\|u\|_{L^q(\Dom)}^q$. It is also stated in  \cite[Lemma 3.3]{DG15_boundedness} under slightly different conditions, and we follow their proofs below.
\begin{lemma}\label{lem:Ito_generalized}
Suppose that $\Dom\subseteq \R^d$ is an open set. Let $q\in [2, \infty)$. Let $v,\phi, \Phi_j, \psi_n:[0,T]\times\Omega\to L^q(\Dom)$ be progressively measurable for $j\in \{1, \ldots, d\}$ and $n\geq 1$, and suppose that
\begin{itemize}
\item $v_0\in L^0_{\F_0}(\Omega;L^q(\Dom))$;
\item $v\in L^2(0,T;H^{1,q}_0(\Dom))\cap C([0,T];L^q(\Dom))$ a.s.;
\item $\phi\in L^1(0,T;L^q(\Dom))$ a.s.;
\item $\Phi\in L^2(0,T;L^q(\Dom;\R^d))$ a.s.;
\item $(\psi_n)_{n\geq 1}\in L^2(0,T;L^q(\Dom;\ell^2))$ a.s.;
\end{itemize}
and a.s.\ for all $t\in [0,T]$, the following identity holds in $H^{-1,q}(\Dom)$:
\begin{equation*}
\textstyle v(t) = v_0 + \int_0^t \phi(s)\,\dd s + \int_0^t \div (\Phi(s))\,\dd s +\sum_{n\geq 1}  \int_0^t \psi_n(s)\,\dd W^n_s.
\end{equation*}
Then, a.s.\ for all $t\in [0,T]$,
\begin{align*}
\|v(t)\|_{L^q(\Dom)}^q
=& \textstyle \|v_0\|_{L^q(\Dom)}^q + q\int_{0}^t \int_{\Dom} |v(s)|^{q-2} \big[v(s)\phi(s) - (q-1) \nabla v(s) \cdot \Phi(s)  \big] \,\dd x \,\dd s\\
& + \textstyle q\sum_{n\geq 1}\int_{0}^t \int_{\Dom}  |v(s)|^{q-2} v(s) \psi_n(s) \,\dd x \,\dd  W^n_s
\\ & \textstyle +\frac{q(q-1)}{2} \int_{0}^t \int_{\Dom} |v(s)|^{q-2} \|\psi(s)\|^2_{\ell^2} \,\dd x \,\dd s.
\end{align*}
The same formula holds if $\Dom$ is replaced by the $d$-dimensional torus $\T^d$.
\end{lemma}
The integrability of each of the terms follows from H\"older's inequality. In particular, the stochastic integral term defines a continuous local martingale $(M_t)_{t\in [0,T]}$ where the quadratic variation satisfies
\begin{align*}
\textstyle [M]_t = \sum_{n\geq 1}\int_{0}^t \Big|\int_{\Dom}  |v(s)|^{q-2} v(s) \psi_n(s) \,\dd x\Big|^2 \,\dd s
\leq \int_{0}^t \|v(s)\|_{L^q(\Dom)}^{2q-2}  \|\psi(s)\|_{L^q(\Dom;\ell^2)}^2 \,\dd s.
\end{align*}

\begin{proof}
By a cut-off argument, it suffices to consider the case of bounded domains.
Let $\Theta_n$ be the approximation of $|\cdot|^q$ as in \cite[Lemma 3.3]{DG15_boundedness} using the It\^o formula of \cite[p.\ 62]{Par75}. Then, as explained there, one has
\begin{align*}
\textstyle \int_{\Dom}\Theta_n(v(t))\,\dd x
=& \textstyle \int_{\Dom}\Theta_n(v_0) \,\dd x + \int_{0}^t \int_{\Dom} \Theta_n'(v(s)) \phi(s)  \,\dd x \,\dd s
\\ & \textstyle- \int_{0}^t \int_{\Dom}  \Theta_n''(v(s)) \nabla v(s)\cdot \Phi(s) \,\dd x \,\dd s
\\ & \textstyle+ \sum_{n\geq 1}\int_{0}^t \int_{\Dom}  \Theta_n'(v(s))  \psi_n(s) \,\dd x \,\dd  W^n_s
\\ & \textstyle +\frac{1}{2} \int_{0}^t \int_{\Dom} \Theta_n''(v(s))  \|\psi(s)\|^2_{\ell^2} \,\dd x \,\dd s.
\end{align*}
It remains to let $n\to \infty$. All terms converge by the dominated convergence theorem and the pointwise convergence properties of $\Theta_n$ and its derivatives. To obtain a dominating function it suffices to use H\"older's inequality and the bounds $|\Theta_n(z)|\leq C|z|^q$, $|\Theta_n'(z)|\leq C|z|^{q-1}$ and $|\Theta_n''(z)|\leq C|z|^{q-2}$, where $C$ is independent of $n$ and $z\in \R$.
\end{proof}

\subsection{Stochastic Gronwall lemma}

When analyzing S(P)DEs one often has to rely on Gronwall arguments. In the case of nonlinear equations, it can happen that the integrability of moments (even in $L^1$) does not hold, and more sophisticated results are required. Stochastic Gronwall lemmas can provide such results, and be independently traced back to \cite{GlattZiane} and \cite{vonRenSch}.

Below we present a special case of the version of the stochastic Gronwall lemma of \cite[Corollary 5.4]{geiss2024sharp}. We should mention that in the latter work only right-continuity is assumed. A surprising feature is that $M$ does not appear on the right-hand sides of the concluded estimates. For $Z:[0,\infty)\times\Omega\to [0,\infty)$ we will write $Z^*_t := \sup_{s\in [0,t]} Z_s$ below. Moreover, $Z$ is called {\em increasing} if a.s.\ for all $s\leq t$ one has $Z_s\leq Z_t$.
\begin{lemma}\label{lem:Gronwall}
Suppose that the following conditions hold:
\begin{itemize}
\item $X:[0,\infty)\times\Omega\to [0,\infty)$ is a.s.\ continuous and adapted;
\item $A:[0,\infty)\times\Omega\to [0,\infty)$ is a.s.\ continuous, increasing, adapted, and $A_0 = 0$;
\item $M:[0,\infty)\times\Omega\to \R$ is a continuous local martingale with $M_0=0$;
\item $H:[0,\infty)\times\Omega\to [0,\infty)$ is continuous, increasing and adapted;
\item For all $t\geq 0$, a.s.,
\begin{equation}
\label{eq:estimate_grownall_claim}
\textstyle
X_t \leq \int_0^t X_s \,\dd A_s + M_t + H_t.
\end{equation}
\end{itemize}
Then, for all $T,u,w,R>0$,
\[\textstyle \P(X_T^*>u)\leq \frac{e^R}{u} \E(H_T\wedge w) + \P(H_T\geq w) +  \P(A_T>R).\]
Moreover, the following $L^p$-estimate holds for all $p\in (0,1)$ and $T>0$,
\[\textstyle \big\|e^{-A_T}X_T^*\big\|_{L^p(\Omega)}\leq (1-p)^{-1/p} p^{-1} \|H_T\|_{L^p(\Omega)}.\]
\end{lemma}

The above result also holds if the processes are only defined on a random time interval $[0,\tau]$ and follows by extending them constantly on $(\tau, \infty)$. The tail estimate can be seen as a weak $L^1$-estimate. The moment estimate does not extend to $p=1$.

In applications, we take $A_t = \int_0^t a_s \,\dd s$, where $a\in L^1_{\rm}([0,\infty))$ is a nonnegative progressively measurable process, and thus $\int_0^t X_s \,\dd A_s = \int_0^t a_s X_s\,\dd s$.
\begin{proof}
For completeness, we include the short proof of \cite{geiss2024sharp} in the special case of continuous processes. By continuity, we may assume the estimate \eqref{eq:estimate_grownall_claim} holds a.s.\ for all $t\geq 0$. Let $Y_t := X_t e^{-A_t}$, and let $\wt{X}_t$ denote the right-hand side of \eqref{eq:estimate_grownall_claim}.
By It\^o's formula applied to $\wt{Y}_t = \wt{X}_t e^{-A_t}$ we find that
\[\textstyle \wt{Y}_t = \int_0^t e^{-A_s} \,\dd \wt{X}_s - \int_0^t \wt{X}_s e^{-A_s}   \,\dd A_s = \int_0^t e^{-A_s} X_s \,\dd A_s + N_t + \wt{H}_t - \int_0^t \wt{X}_s e^{-A_s}   \,\dd A_s \leq N_t + \wt{H}_t,\]
where $N_t = \int_0^t \exp(-A_s) \,\dd M_s$ and $\wt{H}_t = \int_0^t \exp(-A_s) \,\dd H_s\leq H_t$. Therefore, we find that, a.s. for all $t\geq 0$, $Y_t\leq \wt{Y}_t \leq N_t + H_t$.

Next we will show that
\begin{align}\label{eq:StochGronY}
\textstyle \P(Y_T^*>u)\leq \frac{1}{u}\E(H_T\wedge w) + \P(H_T\geq w).
\end{align}
By a localization argument, we may suppose that $N$ is a martingale.

Let $\tau = \inf\{t\geq 0: H_t\geq w\}$  and $\sigma = \inf\{t\geq 0: Y_t\geq u\}$.
Note that $u\one_{\{\sup_{s\in [0,t]} Y_s>u\}} \leq Y_{t\wedge \sigma}\wedge u$.
On the set $\{H_0<w\}$ one has (due to the nonnegativity of $X$ and thus $Y$)
\begin{align*}
Y_{t\wedge \sigma}\wedge u &= Y_{t\wedge \tau\wedge \sigma}\wedge u + (Y_{t\wedge \sigma}\wedge u-Y_{t\wedge \tau\wedge \sigma}\wedge u
\\ & \leq N_{t\wedge \tau\wedge \sigma} + H_{t\wedge \tau\wedge \sigma} \wedge u + u \one_{\{\tau<t\}}
\\ & \leq N_{t\wedge \tau\wedge \sigma} + H_{t}\wedge w + u \one_{\{H_{t}\geq w\}},
\end{align*}
where we used that $H$ is increasing. On the set $\{H_0\geq w\}$ one has $Y_{t\wedge \sigma}\wedge u\leq u \one_{\{H_t\geq w\}}$ again since $H$ is increasing. Below we write $\E(\xi;A) = \int_{A} \xi \dd \P$. From the above, it follows that
\begin{align*}
\E(u\one_{\{\sup_{s\in [0,t]} Y_s>u\}}) &\leq  \E(Y_{t\wedge \sigma}\wedge u)
\\  &\leq \E(Y_{t\wedge \sigma}\wedge u;H_0<w) + \E(Y_{t\wedge \sigma}\wedge u;H_0\geq w)
\\ &\leq \E(H_{t}\wedge w;H_0<w) + \E(u \one_{H_{t}\geq w};H_0<w) + \E(u \one_{H_t\geq w};H_0\geq w)
\\ &\leq \E(H_{t}\wedge w) + u \P(H_{t}\geq w).
\end{align*}
Here we used that $\E(N_{t\wedge \tau\wedge \sigma};H_0<w) = \E(N_0;H_0<w) = 0$ since $N$ and thus its stopped version are martingales starting at zero. Taking $t=T$, the above estimate gives \eqref{eq:StochGronY}.

The estimate \eqref{eq:StochGronY} applied with $u$ replaced by $e^{-R} u$ implies
\begin{align*}
\P(X_T^*>u)& \leq \P(X_T^*>u, A_T\leq R) + \P(A_T>R)
\\ & = \P(e^{-A_T}X_T^*>e^{-A_T}u, A_T\leq R) + \P(A_T>R)
\\ & \leq \P(Y_T^*>e^{-R}u) + \P(A_T>R)
\\ & \leq \textstyle \frac{e^R}{u}\E(H_T\wedge w) + \P(H_T\geq w)  + \P(A_T>R).
\end{align*}

It remains to prove the $L^p$ bound for $p\in (0,1)$. Taking $w = \lambda u$ with $\lambda>0$ to be determined, then \eqref{eq:StochGronY} and Fubini's theorem imply that
\begin{align*}
\E\|Y_T^*\|^p &= \textstyle\int_0^\infty p\, u^{p-1} \P(Y_T^*>u) \, \dd u
\\ & \leq \textstyle\int_0^\infty p \, u^{p-2} \E[H_T\wedge (\lambda u)]  \, \dd u + \int_0^\infty p\, u^{p-1} \P(H_T>\lambda u) \, \dd u
\\ & = \textstyle \frac{\lambda^{1-p}}{1-p}\E\|H_T\|^p + \lambda^{-p} \E\|H_T\|^p.
\end{align*}
Minimization leads to  $\lambda = p$, and thus we find that
\[\textstyle \E\|e^{-A_T} X_T^*\|^p \leq \|Y_T^*\|_p^p \leq \Big(\frac{p^{1-p}}{1-p} + p^{-p}\Big)\E\|H_T\|^p = \frac{p^{-p}}{1-p}\E\|H_T\|^p,\]
as desired.
\end{proof}

\setcounter{section}{15}
\section*{Open problems}

Below we include a list of open problems. It is a selection of problems that we believe could furhter advance the theory and understanding. It is in no way exhaustive, as there are many more open problems even for linear SPDEs.

\subsubsection*{Sufficient conditions for stochastic maximal $L^p$-regularity for nonzero $B$}

In Subsection \ref{ss:furtherSMR} and before, we have seen several concrete cases in which $(A,B)$ satisfies stochastic maximal $L^p$-regularity. However, there is no satisfactory operator-theoretic condition which covers most of these cases. It is obvious from a fixed point argument that if $(A,0)$ has stochastic maximal $L^p$-regularity and $\|B\|_{\calL(X_1, \gamma(\mathcal{U},X_{1/2}))}<\varepsilon_{A,p}$ with $\varepsilon_{A,p}$ depending on the maximal regularity constant,
then $(A,B)$ has stochastic maximal $L^p$-regularity. However, this usually leads to an unsatisfactory theory.

\begin{problem}[Sufficient conditions for stochastic maximal regularity]
Find reasonable sufficient abstract conditions for stochastic maximal $L^p$-regularity of $(A,B)$.
\end{problem}

Although the case $B = 0$ is well-studied, the following also remains open.
\begin{problem}[Optimal condition for stochastic maximal regularity]
Let $X_0 = L^q$ with $q\in (2, \infty)$. Let $A$ be a sectorial operator of angle $<\pi/2$. Find a necessary and sufficient condition for stochastic maximal $L^p$-regularity of $(A,0)$.
\end{problem}
A sufficient condition is given in terms of the $H^\infty$-calculus (see Theorem \ref{thm:SMRHinfty}). However, for $q=2$ no conditions on $A$ are needed (see Theorem \ref{thm:SMRHS}).

\subsubsection*{Extrapolation of stochastic maximal regularity}
The extrapolation Theorem  \ref{thm:extrapol} shows that if $B=0$, then stochastic maximal $L^p$-regularity can be extrapolated to all $r\in (2, \infty)$. In the case $B$ is not small the following is open:
\begin{problem}[Extrapolation of stochastic maximal regularity]
Suppose that $(A,B)\in \mathcal{SMR}_{p,0}^{\bullet}$ for some $p\in [2, \infty)$.
Does one have $(A,B)\in \mathcal{SMR}_{r,0}^{\bullet}$ for all $r\in (2, \infty)$ (or at least some large range for $r$)?
\end{problem}
An extrapolation result in a linear variational setting for $(A,B)$ with each $A$ self-adjoint, was recently obtained in \cite{BecVer}, where also the $(t,\omega)$-dependent case is covered. However, the extrapolation only gives $\mathcal{SMR}_{r,0}^{\bullet}$ for all $r\in [2, 2+\varepsilon)$, where $\varepsilon>0$ is small. We do not know whether the selfadjointness can be omitted. An alternative approach can be found in \cite[Appendix A]{agresti2024anomalous}, where a small improvement in space integrability was obtained.

\subsubsection*{$L^2$-moments in the variational setting}
In Theorem \ref{thm:varglobal} we obtained a bound for $\E \sup_{t\in [0,T]}\|u(t)\|_{H}^{p}$ for $p\in (0,2)$, but $p=2$ is excluded in general. In Section \ref{sss:p=2}, conditions are discussed under which one can extend this bound to $p=2$. However, we do not know whether this condition can be omitted or not.
\begin{problem}[Energy bounds for critical variational SPDEs]
\label{prob:p=2var}
In the variational setting of Theorem \ref{thm:varglobal}, is
there a constant $C_T$ such that for all $u_0\in L^2_{\F_0}(\Omega;H)$ one has
\[\E \sup_{t\in [0,T]}\|u(t)\|_{H}^{2} \leq C_T(1+\E\|u_0\|_{H}^{2})?\]
\end{problem}
A model example for which we do not know whether the $L^2$-bound holds is the Allen--Cahn equation on the one-dimensional torus $\T$ with a one-dimensional quadratic noise:
\[\dd u  =  \big( \Delta u +u-u^3\big) \,\dd t+ \gamma u^2 \,\dd W.\]
For $\gamma^2<2$, Section \ref{sss:p=2} gives an $L^2(\Omega)$-bound. But for $\gamma^2=2$ we do not know whether it holds.

In Theorem \ref{thm:varglobal} we stated a result on the continuous dependency on the initial data in terms of convergence in probability. In \cite[Theorem 3.8]{AVvar} it is also explained that the following holds as well: if $u_0^n\to u_0$ in $L^2(\Omega;H)$, then $u^n\to u$ in $L^p(\Omega;L^2(0,T;V)\cap C([0,T];H))$ for any $p\in (0,2)$. It is natural to ask when convergence takes place for $p=2$.
\begin{problem}[Continuity in the energy space for critical variational SPDEs]
Let the assumptions of Theorem \ref{thm:varglobal} be satisfied. Let $u_0^n, u_0\in L^2_{\F_0}(\Omega;H)$ be such that $u_0^n\to u_0$ in $L^2(\Omega;H)$. Does one have $u^n\to u$ in $L^2(\Omega;L^2(0,T;V)\cap C([0,T];H))$, where $u^n$ and $u$ are the solutions to \eqref{eq:SEE} corresponding to $u_0^n$ and $u_0$, respectively?
\end{problem}

The above seems open even if Assumption \ref{ass:varsetting} holds, $B = 0$, $G:V\to \calL_2(\mathcal{U},H)$ and $\lb v, F(v)\rb \lesssim 1+ \|v\|^2_H$.
Note that, in the latter case, one always has $L^2$-moments of the solution if $u_0\in L^2(\Omega;H)$ (see Subsection \ref{sss:p=2}).

\subsubsection*{Bootstrapping regularity in the critical case for $p=2$}
Theorems \ref{thm:parabreg} and \ref{thm:parabreg2} provide a way to bootstrap regularity for parabolic SPDEs, but only for $p>2$. The case $p=2$ can be included in the case that the nonlinearities are subcritical as we have seen in Proposition \ref{prop:parabreg3}. Under further restrictions, a regularization for $p=2$ was obtained in \cite[Proposition 6.8]{AV19_QSEE_2} which covers the critical case.
\begin{problem}[Time regularization in Hilbert spaces]
\label{pro:regularization}
Let $X_0$ be a Hilbert space. Does the regularization result of Theorem \ref{thm:parabreg2} extend to $p=2$ in the critical case (i.e.\ equality in \eqref{eq:subcriticalvar} holds for some $j$)?
\end{problem}

\subsubsection*{Stochastic reaction-diffusion equations}

In Subsections \ref{ss:reaction}  we discussed the global well-posedness of some coercive systems. We also included a non-coercive case in Subsection \ref{ss:LV}, where a predator-prey model was discussed. Some triangular models we are able to treat as well (see \cite{AVreaction-global} for the Brusselator and Gray-Scott model both with cubic $f$).
Unfortunately, we are quite far from the general theory in the deterministic setting (see  \cite{P10_survey} for triangular systems and \cite{FMT, Kanel} for (super)-quadratic models).

\begin{problem}[Reaction-diffusion with triangular structure]
\label{prop:triangular_structure}
Let $\Dom$ be $\T^d$ or $\R^d$ and $\ell\geq 2$. Consider the reaction-diffusion system \eqref{eq:reaction_diffusion_system} with $(b_{n,i} )_{n\geq 1}\in C^\gamma(\Dom;\ell^2)$ for some $\g>0$ and with possibly additional boundary conditions in case $\Dom\neq \T^d$.
Does global well-posedness hold for the reaction-diffusion system \eqref{eq:reaction_diffusion_system} under the following conditions for a fixed $h>1$ and each $i\in \{1, \ldots, \ell\}$:
 \begin{enumerate}[{\rm (1)}]
 \item $u_{0}\in [0,\infty)^{\ell}$;
 \item $g_i:\R^{\ell}\to \ell^2$ is globally Lipschitz and $g_i(u) = 0$ for all $u\in [0,\infty)^{\ell}$ with $u_i=0$;
 \item $f_i(u)\geq 0$ for all $u\in [0,\infty)^{\ell}$ with $u_i=0$;
 \item $|f_i(u)-f_i(u')|\lesssim (1+|u|^{h-1}+|u'|^{h-1})|u-u'|$ for all $u,u'\in [0,\infty)^\ell$ (polynomial growth);
 \item $\exists K_1, K_2\geq 0$ such that $\sum_{j=1}^{\ell} f_j(u)\leq K_0 + K_1 \sum_{j=1}^{\ell} u_j$ for all $u\in [0,\infty)^{\ell}$ (mass conservation);
\item There exist $r \in\R^\ell$ and a lower triangular invertible $\ell\times \ell$ matrix $R$ with nonnegative entries such that
$
R f(y)\leq [1+\sum_{i=1}^\ell y_i] r
$
componentwise for all $y\in [0,\infty)^\ell$
(triangular structure).
 \end{enumerate}
Does the global well-posedness hold with $g_i$ having superlinear growth? Is it possible to allow for optimal growth of $g$, i.e.\ Assumption \ref{ass:reaction_diffusion_global}\eqref{it:growth_nonlinearities}?
\end{problem}

In the deterministic case, the above result is due to \cite{Morgan1,Morgan2} (see also \cite[Theorem 3.5]{P10_survey}), and it is based on $L^p$-estimates and a duality argument. A stochastic version of the former fact is highly non-trivial in the presence of (non-small) noise coefficients $b_{n,i}$, even if the smoothness parameter $\g$ is large.

\smallskip

Another situation in which global well-posedness is known in the deterministic setting is the case of quadratic nonlinearity. A model example of the latter is the Lotka--Volterra type nonlinearity $f_i(u) = -u_i\tau_i  +  u_i \sum_{j=1}^\ell c_{ij} u_j$, where $c_{ij} = -c_{ji}$ (see \cite[p.\ 287]{FMT}).

\begin{problem}[Reaction-diffusion with quadratic growth]\label{pro:react2}
Let $\Dom$ be $\T^d$ or $\R^d$ and $\ell\geq 2$. Consider the reaction-diffusion system \eqref{eq:reaction_diffusion_system} with $(b_{n,i} )_{n\geq 1}\in C^\gamma(\Dom;\ell^2)$ for some $\g>0$ and with possibly additional boundary conditions in case $\Dom\neq \T^d$.
Does global well-posedness hold for the reaction-diffusion system \eqref{eq:reaction_diffusion_system} under the following conditions for each $i\in \{1, \ldots, \ell\}$:
 \begin{enumerate}[{\rm (1)}]
 \item $u_{0}\in [0,\infty)^{\ell}$;
 \item $g_i:\R^{\ell}\to \ell^2$ is globally Lipschitz and $g_i(u) = 0$ for all $u\in [0,\infty)^{\ell}$ with $u_i=0$;
 \item $f_i(u)\geq 0$ for all $u\in [0,\infty)^{\ell}$ with $u_i=0$;
 \item $\exists K_1, K_2\geq 0$ such that $\sum_{j=1}^{\ell} f_j(u)\leq K_0 + K_1 \sum_{j=1}^{\ell} u_j$ for all $u\in [0,\infty)^{\ell}$ (mass conservation);
 \item   $| f_i(u)-f_i(u')|\lesssim (1+|u|+|u'|) |u-u'|$ for all $u,u'\in [0,\infty)^{\ell}$
(quadratic growth).
 \end{enumerate}
Does the global well-posedness hold with $g_i$ having superlinear growth? Is it possible to allow for optimal growth of $g$, i.e.\ with $h=\frac{3}{2}$ in Assumption \ref{ass:reaction_diffusion_global}\eqref{it:growth_nonlinearities}?
\end{problem}

The proof of global well-posedness for quadratic nonlinearities in the deterministic setting given in \cite{FMT} relies on several non-trivial facts from PDE theory with bounded and measurable coefficients, cf.\ \cite[Appendix A]{FMT}. A stochastic version of the latter is at the moment not available even in the case of small noise coefficients $b_{n,i}$ (here we mean that $b_{n,i}\neq 0$ but $\|(b_{n,i})_{n\geq 1}\|_{L^\infty(\Dom;\ell^2)}\leq \varepsilon$ for some small $\varepsilon>0$).

\subsubsection*{Navier--Stokes equations with transport noise.}
Navier--Stokes equations with transport noise appeared in several instances of the present manuscript, see Subsection \ref{ss:scaling_intro}, \ref{ss:Fluid} and \ref{subsec:SNSRd}. At the moment the local well-posedness theory of Navier--Stokes equations with transport noise is not as mature as the corresponding theory in the deterministic case. Here we propose three problems concerning the extensions of deterministic results to the stochastic setting, which we believe are important steps in understanding stochastic fluids and connect well with the framework discussed in this manuscript.

In the study of fluid flows, it is well known that boundaries have important effects on the overall dynamics. In particular, the so-called \emph{no-slip boundary condition} $u|_{\partial\Dom}=0$, where $u$ denotes the velocity field, is one of the most relevant ones and plays an important role.
However, in contrast to the 2D case (see Subsection \ref{ss:Fluid}), the (local) well-posedness of the 3D Navier--Stokes equations with transport noise is unknown.

\begin{problem}[Stochastic Navier--Stokes equations with no-slip boundary conditions]
Let $\Dom$ be a bounded and sufficiently regular domain of $\R^3$. Is it possible to prove local well-posedness of the 3D Navier--Stokes equations on $\Dom$ (see either \eqref{eq:Navier_Stokes} or \eqref{eq:Navier_StokesRd}) with no-slip boundary conditions and \underline{non-small} transport noise? Or, in other words, does Theorem \ref{t:NS_Rd_local} still hold for the 3D Navier--Stokes equations on domains with no-slip boundary conditions for a certain choice of the parameters $(p,\a,\s,q)$?
\end{problem}

In the above, by {\em non-small transport noise}, we mean that the transport noise satisfies the natural condition \eqref{eq:AB_basic_example} and not $\|(b_n)_{n\geq 1}\|_{L^\infty(\Dom;\ell^2)}\leq \varepsilon$ for some small $\varepsilon>0$. Indeed, from the arguments in Subsection \ref{subsec:SNSRd}, it is clear that the above problems are only concerned with checking the stochastic maximal $L^p$-regularity used in Theorem \ref{thm:localwellposed}. In case of \emph{small} transport noise, the latter can be easily checked for appropriate values of $(p,\a,\s,q)$ (say $q>d$ and $\delta$ small) by combining Theorem \ref{thm:SMRHinfty} for the Stokes operator with no-slip boundary conditions and subsequently a perturbation argument \cite[Theorem 3.2]{AV21_SMR_torus} to include a small transport noise.
Let us mention that the $H^\infty$-calculus for the Stokes operators with no-slip boundary conditions on a smooth domain is well-known, see e.g.\ \cite[Theorem 9.17]{KKW}.

\smallskip

Next, we turn our attention to problems without boundaries. For convenience, we focus on the three-dimensional case. As commented below Theorem \ref{t:NS_Rd_local} (see also \cite[Theorem 2.4]{AV20_NS} for the periodic case), under suitable assumptions on the transport noise coefficients $b$, our framework provides well-posedness in the critical spaces $B^{3/q-1}_{q,p}$ with smoothness $>-\frac{1}{2}$ (corresponding to integrability $q<6$). However, in the deterministic setting, the celebrated Koch-Tataru result \cite{KT01} shows that even the endpoint smoothness $-1$ can be reached.
Therefore, it is natural to ask if such a threshold is natural in the stochastic case or is a matter of technique.

\begin{problem}[Largest critical space for stochastic Navier--Stokes equations]
\label{prob:largest_invariant_space}
Let $\Dom$ be either $\T^3$ or $\R^3$. Let $(b_n)_{n\geq 1}\in C^{\infty}(\Dom;\ell^2)$ be non-zero, and with possibly additional conditions at infinity if $\Dom=\R^3$. Can one prove local well-posedness for such stochastic Navier--Stokes equations with transport noise $b$ (see \eqref{eq:Navier_StokesRd} for $\Dom=\R^3$) in a critical space with smoothness less or equal to $-\frac{1}{2}$? What is the optimal smoothness threshold?
\end{problem}

In the above, by {\em scaling-invariant space}, we mean a function space with Sobolev index equal to $-1$, as it is for $B^{3/q-1}_{q,p}$ or $H^{3/q-1,q}$. Hence, the above question is not limited to Besov spaces. Possibly, other types of maximal regularity could be employed as discussed in Subsection \ref{sss:SMRdiff}.

\smallskip

Our final open problem concerns a refinement of the endpoint Serrin-type criterion for stochastic Navier--Stokes equations proved in Theorem \ref{thm:serrin}\eqref{it2:serrin} (see also the comments below \eqref{eq:criticality_indexes_NS3D} for the periodic case).

\begin{problem}[Refined endpoint Serrin-type criterion]
\label{prob:NS_blow_up}
Let $\Dom$ be either $\T^3$ or $\R^3$ and correspondingly let $(u,\sigma)$ be the maximal $(p,\s,\a,q)$-solution to the 3D Navier--Stokes equations with (non-trivial) transport noise $b$ on $\Dom$ (see \eqref{eq:Navier_StokesRd} for $\Dom=\R^3$) as provided by Theorem \ref{t:NS_Rd_local} for $\Dom=\R^3$ and \cite[Theorem 2.4]{AV20_NS} for $\Dom=\T^3$; with the appropriate regularity assumptions on the transport noise coefficients and the parameters $(p,\a,\s,q)$. Prove or disprove the following blow-up criterion:
$$
\textstyle
\P( \sigma<\infty,\, \sup_{t\in [0,\sigma)}\|u(t)\|_{B^{3/q-1}_{q,p}(\Dom;\R^3)}<\infty)=0.
$$
\end{problem}

The above problem is motivated by \cite{NS_sup_blowup1} where the deterministic case of the above is proven for the critical space $L^3$. The Besov space case was later obtained in \cite{MR3475661}. Improvements can be found in \cite{MR2551795,NS_sup_blowup2,NS_sup_blowup3}. A quantitative version was recently obtained by Tao \cite{NS_sup_blowup4}.

Besides the intrinsic interest in the above problem, we hope that a possible solution to Problem \ref{prob:NS_blow_up} could also suggest under which assumptions the blow-up criterion of Theorem \ref{thm:criticalblowup}\eqref{it1:criticalblowup} holds with the $\lim$-condition replaced by a $\sup$-one.

\bibliographystyle{alpha-sort}

\bibliography{literature}

\end{document}